\documentclass{amsbook}

\usepackage{amscd,amssymb,epsfig,epsf,pinlabel,graphicx,epic,eepic,nicefrac}
\usepackage[dvipsnames]{xcolor}
\usepackage[mathscr]{euscript}

\usepackage{hyperref}
\hypersetup{colorlinks=true,citecolor=brown,linkcolor=brown,urlcolor=black,filecolor=black}

\usepackage[all]{xy}
\SelectTips{cm}{}

\numberwithin{equation}{section}
\allowdisplaybreaks
\numberwithin{figure}{section}

\newtheorem{propo}{Proposition}[section]

\newtheorem{theor}[propo]{Theorem}
\newtheorem{lemma}[propo]{Lemma}

\theoremstyle{definition}

\theoremstyle{remark}

\newcommand{\aug}{\operatorname{aug}}

\newcommand{\bignu}{\hbox{\LARGE $\nu$}\!}
\newcommand{\calC}{\mathscr{C}}
\newcommand{\calD}{\mathscr{D}}
\newcommand{\calK}{\mathscr{K}}
\newcommand{\calL}{\mathscr{L}}
\newcommand{\calM}{\mathscr{M}}
\newcommand{\calN}{\mathscr{N}}
\newcommand{\calP}{\mathscr{P}}
\newcommand{\calR}{\mathscr{R}}
\newcommand{\card}{\operatorname{card}}

\newcommand{\conc}{\mathsf{c}}
\newcommand{\Com}{\operatorname{Com}}
\newcommand{\Der}{\operatorname{Der}}
\newcommand{\diag}{\operatorname{diag}}

\newcommand{\ev}{\operatorname{ev}}
\newcommand{\End}{\operatorname{End}}
\newcommand{\FF}{\mathbb{F}}

\newcommand{\Fun}{\operatorname{Fun}}
\newcommand{\g}{\mathfrak{g}}
\newcommand{\GL}{\operatorname{GL}}

\newcommand{\grAlg}{\mathscr{GA}}
\newcommand{\grCom}{\mathscr{CGA}}
\newcommand{\Hom}{\operatorname{Hom}}
\newcommand{\id}{\operatorname{id}}

\newcommand{\Int}{\operatorname{Int}}
\newcommand{\Ker}{\operatorname{Ker}}
\newcommand{\kk}{\mathbb{K}}
\newcommand{\Lie}{\mathscr{L}\!ie}
\newcommand{\Loop}{\mathbb{L}}
\newcommand{\LoopH}{\mathbb{H}}
\newcommand{\Mat}{\operatorname{Mat}}

\newcommand{\mediumdot}{{\displaystyle \mathop{ \ \ }^{\hbox{$\centerdot$}}}}

\newcommand{\Ob}{\operatorname{Ob}}

\newcommand{\perm}{\mathsf{p}}
\newcommand{\Perm}{\mathsf{P}}
\newcommand{\pr}{\operatorname{pr}}

\newcommand{\red}{\operatorname{red}}

\newcommand{\RR}{\mathbb{R}}

\newcommand{\StringH}{\mathscr{H}}
\newcommand{\tr}{\operatorname{tr}}
\newcommand{\ZZ}{\mathbb{Z}}
\newcommand{\varepsilonout}{\varepsilon_{\operatorname{out}}}
\newcommand{\varepsilonin}{\varepsilon_{\operatorname{in}}}

\newcommand{\stl}{{}^*\!}
\newcommand{\str}{^*}

\newcommand{\lb}{\left\langle}
\newcommand{\rb}{\right\rangle}

\newcommand{\losange}{\!\mathrel{{\triangleleft}{\triangleright}}\!}

\newcommand{\up}{\vspace{-0.5cm}}

\newcommand{\bracket}[2]{\left\{#1,#2\right\}}
\newcommand{\double}[2]{\left\{\!\!\left\{#1,#2\right\}\!\!\right\}}
\newcommand{\triple}[3]{\left\{\!\!\left\{#1,#2,#3\right\}\!\!\right\}}

\makeindex

\begin{document}

 \title[{Brackets in the {Pontryagin} algebras of manifolds}]{Brackets in the {Pontryagin} algebras of manifolds}


\author[Gw\'ena\"el Massuyeau]{Gw\'ena\"el Massuyeau}
\address{Gw\'ena\"el Massuyeau \newline
 \indent IRMA,    Universit\'e de Strasbourg \& CNRS \newline
\indent 67084 Strasbourg, France \newline
 \indent   \emph{and}   \newline
 \indent     {IMB}, Universit\'e  Bourgogne Franche-Comt\'e \& CNRS  \newline
  \indent   21000 Dijon, France   \newline
\indent   $\href{mailto:gwenael.massuyeau@u-bourgogne.fr}{\texttt{gwenael.massuyeau@u-bourgogne.fr}}$ 
}

\author[Vladimir Turaev]{Vladimir Turaev}
\address{ \newline
\indent Vladimir Turaev \newline
\indent   Department of Mathematics \newline
\indent  Indiana University \newline
\indent Bloomington IN47405, USA\newline
\indent $\href{mailto:vturaev@yahoo.com}{\texttt{vturaev@yahoo.com}}$
}

\maketitle

\begin{abstract}
Given a smooth oriented  manifold $M$   with  non-empty boundary,  we study   the {Pontryagin} algebra $A=H_\ast(\Omega )$ where $ \Omega $ is the space of loops  in~$M$  based at a   distinguished   point of $ \partial M$.
Using the ideas of string topology of Chas--Sullivan, we  define   a linear map $\double{-}{-}: A \otimes A \to A\otimes A$ which is
a double   bracket   in the sense of Van den Bergh satisfying a version of the Jacobi identity.
For $\dim(M)\geq 3$, the double bracket  $\double{-}{-}$ induces   Gerstenhaber brackets in the  representation algebras associated with $A$.
This  extends our    previous work on  the case  $\dim(M)=2$ where     $A= H_0(\Omega )$ is  the group algebra of the fundamental group   $\pi_1(M)$
and   the double bracket   $\double{-}{-}$   induces the  standard  Poisson brackets   on   the moduli spaces of representations   of $\pi_1(M)$.
\end{abstract}

\setcounter{tocdepth}{1}
\tableofcontents

\section* {Introduction}

   A remarkable   feature  of  an  oriented surface~$\Sigma$ discovered by  
   Goldman \cite{Go1, Go2} 
   is a  natural   Lie bracket in the   vector  
   space generated by the free homotopy classes of    loops in~$\Sigma$.  
   If~$\Sigma$ is  connected and closed, then Goldman's Lie bracket  arises from a  
  symplectic structure     on  the   moduli space  of representations    of the fundamental group $\pi=\pi_1(\Sigma)$ in a      Lie  group~$G$.  
 This    space   $\Hom(\pi, G)/G$   consists of the conjugacy classes of homomorphisms  $\pi\to G$.  The resulting  
symplectic structure       incorporates the classical K\"ahler forms    on the 
Teichm\"uller space  ($G=  \operatorname{PSL} (2,\RR)$), on the Jacobi variety ($G=  \operatorname{U}  (1)$),  and on the Narasimhan--Seshadri
moduli spaces of semistable vector bundles  ($G=  \operatorname{U}  (N)$ with $N\geq 1$).  
  Goldman's construction  also yields the Atiyah--Bott symplectic structure determined by   a compact Lie group 
and    a non-degenerate ${\rm ad}  $-invariant symmetric bilinear form on   its   Lie algebra. 
If~$\Sigma$ is  connected and $\partial \Sigma \neq   \varnothing  $, then  similar methods yield a weaker structure, namely,    
a Poisson bracket  in the algebra of conjugation-invariant smooth functions   on     $\Hom(\pi, G)$,  see \cite{FoR,  GHJW}.
This bracket extends to  a   quasi-Poisson bracket   in the algebra of all smooth functions on $\Hom(\pi, G)$, see \cite{AKsM}.
Analogous results hold for the general linear group  $G= \operatorname{GL}_N  $   over   any commutative ring 
provided     $\Hom(\pi, \operatorname{GL}_N)$ is  treated   as an  affine algebraic set and   smooth functions are traded for   regular functions,   see  \cite{MT}. 
 
Goldman's Lie bracket for      surfaces was generalized by Chas and Sullivan \cite{CS}, \cite{CS+}
     to manifolds of arbitrary dimensions.   Chas and Sullivan  call this area of study  the \lq\lq string topology". 
The present   memoir exhibits new phenomena in  string topology. We consider   the Pontryagin  algebras of manifolds with    boundary 
 and   construct  a bracket    in the associated  representation algebras.
 For surfaces, our bracket  is the quasi-Poisson bracket  on   $\Hom(\pi,  \operatorname{GL}_N  )$ mentioned above.
 In dimension  $\geq 3$,  the representation algebras are  graded,  and    our bracket  is a  Gerstenhaber bracket, 
 i.e.,   it satisfies  the axioms of a Poisson bracket   with  appropriate signs.
 In the rest of   the  Introduction   we focus on manifolds of dimension  $\geq 3$.

 We recall   the concept of a representation algebra following   \cite{Pr, LbW,Cb}.
Fix an integer $N\geq 1$ and a  field $\mathbb{F}$ which will be the ground field of the  algebras.
   Given  an algebra $A$ and a commutative algebra $B$,  consider  the set  $S=S(A,N,B)$ of all algebra homomorphisms from $A$
 to the   algebra $\Mat_N(B) $  of ${(N\times N)}$-matrices over $B$.
 Each   $a\in A$ and each pair of indices $i,j\in \{1, \dots , N\}$ determine a   mapping $a_{ij} \colon S  \to B$  which evaluates    a
  homomorphism    $A\to \Mat_N(B)$ at $a$ and takes the $(i,j)$-th entry of the resulting matrix.
 These  mappings    are   the \lq\lq coordinates"  on $S $,
generating an algebra of \lq\lq polynomial"    $B$-valued functions  on $S $.
 These coordinates  satisfy  various  polynomial relations   some of which are universal, i.e., hold for all $B$.
  By definition,   the  $N$-th \index{representation algebra} \emph{representation algebra} $A_N$ of $A$
 is generated by the symbols $\{a_{ij}\, \vert \, a\in A, 1\leq  i,j\leq N \}$ subject to those universal relations.
 One of the universal relations says that the generators commute, so that    $A_N$ is a commutative algebra.
 For   every commutative algebra $B$,  the algebra $A_N$ projects onto the algebra of polynomial $B$-valued  functions on $S(A,N,B)$  described above. 
 We view  $A_N$ as a  universal form  of    these polynomial algebras. If $A$ is graded, then so is  $A_N$. 

Our construction of brackets in the representation algebras $\{A_N\}_{N\geq 1}$  is based on the technique of 
Van den Bergh \cite{VdB}. He showed how to    construct such brackets     from a linear map $\double{-}{-}: A\otimes A \to A\otimes A$
satisfying certain conditions.
Van den Bergh  calls such    maps  \index{double Poisson bracket} \emph{double Poisson brackets}.  We 
  use the term \index{bibracket} \emph{bibracket} for the version of   double brackets used   here.  
   Also,  we work in the graded setting and rather consider  \index{bibracket!Gerstenhaber} \emph{Gerstenhaber bibrackets} 
   satisfying  a graded version of the Jacobi identity. 
 We show that a  Gerstenhaber bibracket  $\double{-}{-}$  in a graded algebra~$A$
 induces a Gerstenhaber bracket $\bracket{-}{-}$ in    $ A_N$ for all $N\geq 1$. 
In terms of the generators,    the   bracket $\bracket{-}{-}$  is defined as follows:
for  any $a,b\in A$, $i,j,u,v \in \{1,\dots,N\}$, and any  finite expansion
$\double{a}{b}=\sum_\alpha x_\alpha \otimes y_\alpha \in A\otimes A $,  we set 
$$
\bracket{a_{ij}}{b_{uv}} =\sum_{\alpha} \,\,  (x_\alpha)_{uj}   (y_\alpha)_{iv} .
$$
 The bracket  $\bracket{-}{-}$   is invariant under the natural actions of the group $\GL_N(\FF)$ and the Lie algebra $\Mat_N(\FF)$ on $A_N$.

Consider now a smooth oriented  manifold~$M$   of dimension $   \geq 3$ with   base point   $\star \in \partial M \neq   \varnothing  $. 
Let  $ \Omega=\Omega_\star$  be the space of loops   in~$M$  based at~$\star $.
The  graded vector space $A = H_\ast (\Omega  ; \mathbb{F})$ carries an  associative   multiplication  induced by  concatenation of loops. 
This  turns~$A$ into a graded algebra, the \index{Pontryagin algebra}  \emph{Pontryagin algebra} of~$M$.
 We define  a   so-called   \index{intersection bibracket} \emph{intersection bibracket}  in~$A$  as follows.
Pick an embedded path  $\varsigma:I=[0,1]\hookrightarrow  \partial M $ connecting  the  point  $ \star$  to another point~$\star' $.
Consider  any   singular cycles   $\kappa:K\to \Omega=\Omega_\star $ and    $\lambda: L\to \Omega'=\Omega_{\star'}$.
Let~$D$ be the set of all  tuples   $(k\in K, s\in  I, l\in L , t\in I)$   such $\kappa (k) (s)= \lambda (l) (t)$.
Each tuple  $(k, s, l , t ) \in D$    determines two loops in $M$ based at $\star$.
The first loop   goes along~$\varsigma$ from~$\star  $ to~$\star'$,
then along the path $\lambda (l)$ from $\star'= \lambda (l) (0)$ to $\lambda (l) (t)= \kappa (k) (s)$ and  then along
  the path  $\kappa (k)$ back to  $\kappa (k) (1)=\star$.  The second loop   goes along  the path $\kappa (k)$ from $\star= \kappa (k) (0)$
to $ \kappa (k) (s)=\lambda (l) (t)$, then along   $\lambda (l)$  to $\lambda (l) (1)= \star'$  and finally   along
 $\varsigma^{-1}$ back to  $ \star$.  Under appropriate transversality assumptions on $\kappa$ and $\lambda$,
 the resulting map    $D\to \Omega  \times \Omega $ is  a   singular cycle of dimension $$\dim (K)+\dim (L)+2-\dim(M).$$
 Passing to homology classes  and using the isomorphism $A    = H_\ast (\Omega; \mathbb{F})   \simeq   H_\ast (\Omega'; \mathbb{F})$
 determined by $\varsigma$, we obtain  the intersection bibracket  in~$A$.
  Our main result is the following theorem.    \\

\noindent
\textbf{Theorem.}
{\it The intersection bibracket  in the {Pontryagin} algebra  is   a  well-defined     Gerstenhaber bibracket. It  is   natural
with respect to diffeomorphisms of manifolds preserving the orientation  and the base point.}\\

  The intersection  bibracket    generalizes to higher dimensions the   bibracket  of a surface defined in~\cite{MT}.  
  By the   general   theory, the intersection   bibracket in  the {Pontryagin} algebra~$A$   induces  a 
  Gerstenhaber bracket   in   $A_N$ for all $N\geq 1$.
If the manifold~$M$ is simply connected     and~$\FF$ is a field   of characteristic zero, then the   Milnor--Moore   theorem identifies~$A$
with the universal enveloping algebra of the graded Lie algebra $\pi_\ast (M) =   \oplus_{p\geq 2} \,  \pi_p(M)$
(with the degree shifted by 1 and the Whitehead bracket  in the role of the Lie bracket).
In this case,  the   algebras $A_N$ can be viewed  as  the representation algebras of   $\pi_\ast (M)$.

Despite the simplicity of the underlying idea, a precise definition of the intersection    bibracket requires considerable efforts.
First of all, we    introduce  a version of  singular   homology   using  manifolds with corners instead of  simplices. 
Homology theories based on manifolds with corners were implicit already in \cite{CS}
and were since   considered by  several authors, see, for example,  \cite{CD} and \cite{Ci}.
  These theories are insufficient for our aims and  we  develop   our own approach.
For any topological space~$X$, we define  \index{polychain} \emph{polychains}  in~$X$  as oriented manifolds 
with corners endowed with  additional structure including an identification   of some faces,
 a map  to~$X$ compatible with this identification, 
 and $\mathbb F$-valued weights assigned to the connected  components (these weights play the role of   the  coefficients of singular simplices in singular chains).
We define a   \index{polychain!reduction of} \emph{reduction}   of polychains which  eliminates redundant      connected components (like, for example, components of weight zero).
Each  polychain   in~$X$ has  a well-defined reduced boundary.
If it is void, then  the polychain  is a  \index{polycycle} \emph{polycycle}.
The polycycles in~$X$ considered up to  disjoint unions with reduced boundaries
form a graded vector space $\widetilde{H}_\ast (X)$, the \index{face homology} \emph{face homology} of~$X$.
 The key theorem    enabling  our construction of  bibrackets says that 
the usual singular homology $ {H}_\ast (X)= {H}_\ast (X; \mathbb F)$ embeds in    $\widetilde{H}_\ast (X)$ as a direct summand.

 Given a manifold $M$ and a point $\star \in \partial M$ as above,
we   define smooth   polychains   in the loop space  $\Omega=\Omega_\star$ of~$M$ and   show that  any pair 
  of face homology classes of~$\Omega$    can be represented by transversal   smooth polycycles.
This allows us to   carry out   the intersection construction  outlined  above and to  obtain a linear map
$$
 \widetilde \Upsilon:   \widetilde{H}_{\ast} ( \Omega   ) \otimes \widetilde{H}_{\ast} (   \Omega) \to \widetilde{H}_{\ast} ( \Omega  \times \Omega ).
$$
This map   induces    a linear map in singular homology   $   \Upsilon:   A\otimes A \to H_{\ast} ( \Omega  \times \Omega )$ where $A=H_{\ast} ( \Omega   )$.
  The  K\" unneth theorem allows us to rewrite $   \Upsilon$ as  a    map 
$$  \double{-}{-}:   A\otimes A \longrightarrow A\otimes A$$
which turns out to be  a  Gerstenhaber   bibracket.    The assumption that the ground ring     is a field is   used   only in the K\" unneth   theorem;
 most of the exposition is therefore given over  an arbitrary commutative ring.
  Moreover, our constructions   can be generalized by replacing loops   based at~$\star $    
  with  paths  in~$M$ having both endpoints in  $\partial M$. This  leads us to    a  notion of a  \index{path homology category} \emph{path homology category}   of $M$ 
  and an extension  of the   intersection bibracket      to this category.

Given  a  smooth   oriented   manifold $W$
with   $\partial W= \varnothing $, we    can   remove  a  small open ball  from~$W$ and   obtain thus a manifold with boundary.
The  
intersection bibracket in its  {Pontryagin} algebra  
 and  the induced Gerstenhaber brackets  are invariants of~$W$.   Under further   assumptions on~$W$, we  
 obtain   an   $H_0$-Poisson  structure  \cite{Cb} on the {Pontryagin} algebra    of~$W$ itself.

 This work suggests a number of  questions.  
So far, we  do not have a general method  allowing   to compute the face homology, 
 and we do not know whether the face homology carries more information  than the singular homology. 
 Other questions concern the intersection bibracket. 
 Is it  sensitive   to the smooth structure of the manifold?
 Can it be  generalized to   PL-manifolds or to  topological manifolds?
Is it homotopy invariant and can it be defined    in homotopy-theoretic terms (cf.\ \cite{CJ})? 
  Note that the technique of face homology allows one to   define all the Chas--Sullivan operations \cite{CS}.
 It would be    useful to formally identify the resulting geometric operations with those in \cite{CJ}.
  Also, it would be interesting to provide algebraic models for the  intersection bibracket.
For instance, we do not know  how our geometric constructions 
are related to the  cobar constructions of \cite{BCER}
applied to the Poincar\'e duality model of \cite{LS},  see  \cite[Section 5.5]{BCER}. \\

  {\it Organization of the memoir.}  Chapters~\ref{Algebras, brackets,  and bibrackets}  and~\ref{Bibrackets in unital algebras and in categories} are purely algebraic: in Chapter~\ref{Algebras, brackets,  and bibrackets}  we define representation algebras   and   discuss brackets and bibrackets; in
  Chapter~\ref{Bibrackets in unital algebras and in categories}
 we discuss bibrackets in unital algebras and categories,   and we also consider  Hamiltonian reduction in this context. 
Chapter~\ref{Face homology} introduces the  face homology. In Chapter~\ref{Operations on polychains} we study   transversality of polychains and define  intersection operations in the homology  of path spaces.  In Chapters~\ref{Intersection bibrackets} and~\ref{More on brackets   and bibrackets} we construct 
the intersection  bibracket and discuss its   properties. \\ 

{\it Acknowledgements.}
  Part of this work  was  done while G. Massuyeau  visited  Bloomington,   Indiana  in spring 2013; he
  would like to thank Indiana University for hospitality and support. 
The work of V.\ Turaev on this memoir was partially supported by the NSF grants DMS-1202335 and  DMS-1664358​. 
The authors would like to  thank F.~Eshmatov for an explanation  of the paper  \cite{BCER}. \\

 {\it Conventions.} 
  Throughout   the memoir, the letter  $\kk$   denotes a    commutative ring which serves as the ground ring of all  modules and algebras.
  Thus, by   a module (respectively, an algebra, a linear map) we   mean a $\kk$-module (respectively, a $\kk$-algebra,  a $\kk$-linear map).
  By  the singular homology of a topological  space  we mean singular  homology with coefficients in $\kk$. 

Given a  smooth   oriented manifold $M$ and a smooth orientable    submanifold $N \subset M$,
an orientation of the normal bundle of $N$ in $M$ determines an orientation of $N$, and vice versa,
via  the following rule: a positive    frame   in the normal bundle   of $N$   followed
by   a positive  frame in the tangent bundle of    $N$ is a positive  frame in the tangent bundle of   $M$.
If $\partial {M}\neq   \varnothing   $, then the orientation of $M$  induces an orientation of   $\partial {M}$
using the ``outward vector first" rule. \\

 \chapter {Algebras, brackets,  and bibrackets} \label{Algebras, brackets,  and bibrackets}

\section{Algebras and brackets}\label{sect-prelim}

  We start by recalling some standard terminology.  

\subsection{Graded modules and  graded  algebras}\label{gral}

By    a  \index{graded module} {\it graded  module} we mean a $\ZZ$-graded module $A=\oplus_{p\in \ZZ} \, A^p$  (over $\kk$). 
An element $a$ of $A$ is \index{graded module!element of!homogeneous} {\it homogeneous} if   $a\in A^p$ for some $p$; we
write then $\vert a\vert =p$ and call $\vert a\vert$ the \index{graded module!element of!degree of} {\it degree} of~$a$.
By definition, the degree of $0\in A$ is an arbitrary integer. For any $d\in \ZZ$,
the  {\it $d$-degree} $\vert a\vert_d$ of a homogeneous element $a\in A$ is $\vert a\vert_d=  \vert a\vert+d$.

A \index{graded algebra} \index{graded algebra} {\it graded  algebra} is    a   graded module $A $ endowed with  an  associative
bilinear multiplication such that $ A^p A^q\subset A^{p+q}$ for all   $p,q\in \ZZ$.
Note that if the product of $k\geq 1$    homogeneous elements  $a_1, \dots, a_k $ of $ A$ is non-zero, then the degree of this product  is equal to   $
\vert a_1 \vert + \cdots + \vert a_k \vert$.  If $
  a_ 1 \cdots a_k=0$, then we set    $
\vert a_ 1 \cdots a_k\vert= \vert a_1 \vert + \cdots + \vert a_k \vert$. Similarly, for   $d\in \ZZ$,  we write
$ \vert a_ 1 \cdots a_k\vert_d$ for   $\vert a_1 \vert + \cdots + \vert a_k \vert +d $.

We do not require a graded algebra $A$ to have a  unit element.  If $ab=(-1)^{\vert a \vert \vert b\vert} ba $
for   some homogeneous
$a,b\in A$, then one says that $a$ and $b$ {\it commute}.
For  a graded  algebra $A$, we denote by $[A,A]$ the graded submodule of $A$ spanned by
the vectors $ab-(-1)^{\vert a \vert \vert b\vert} ba$ where $a,b$ run over all homogeneous
elements of $ A$.  The graded
algebra $A$  is \index{graded algebra!commutative} {\it commutative} if   $[A,A]=0$. Factoring any graded algebra
 $A$ by the 2-sided ideal generated by $[A,A]$ we obtain a
commutative graded algebra  $\Com(A)$.

Given graded algebras $A$ and $B$,   a \index{graded algebra!homomorphism of}  {\it  graded  algebra homomorphism}
$A\to B$ is   a  degree-preserving  algebra homomorphism from $A$ to $ B$.

We will consider any $\ZZ_{\geq 0}$-graded module  $A=\oplus_{p\geq 0} \, A^p$ as a
$\ZZ$-graded module by setting   $A^p=0$ for all   $p<0$.

\subsection{Representation algebras}\label{AMT0}

Each   graded algebra $A$ determines an infinite sequence of graded algebras $\tilde A_1, \tilde A_2,\dots$ as follows, cf$.$  \cite{LbW,Cb,VdB}.
The graded algebra ${\tilde A}_N$ with $N\geq 1$ is defined by the  generators   $a_{ij}$,
where $a$ runs over  all elements of $A$ and $i,j$ run over   $ \{1,2, \ldots, N\}$, and the following relations:
for all   $a,b\in A$, $k\in \kk$, and $i,j\in \{1,2, \ldots, N\}$,
\begin{equation}\label{eq:addid}
(ka)_{ij}=k a_{ij}, \quad  (a + b)_{ij}=a_{ij}+ b_{ij}, \quad (ab)_{ij}=    a_{il} b_{lj}.
\end{equation}
In the latter  formula and in the sequel we   always sum up
over repeating indices and drop the summation sign.
 A typical element of ${\tilde A}_N$ is represented by a non-commutative polynomial in the generators with zero free term.
 The grading in  $\tilde A_N$ is defined   by   $\vert a_{ij}\vert =p$ for   all   $a\in A^p$.

The construction of $\tilde A_N$ is functorial: a  graded    algebra homomorphism $f:{A\to A'}$ induces a
 graded  algebra homomorphism $\tilde f_N:\tilde A_N\to \tilde A'_N$ by
$ \tilde f_N(a_{ij})  = (f(a))_{ij}$ for all $a\in A$, $i,j \in \{1, \ldots, N\}$.
For $N=1$ we have $\tilde A_1=A$ and $\tilde f_1=f$.

The importance of   ${\tilde A}_N$ is due to the  following  fact.
For any graded algebra~$B$, let $\Mat_N(B)$ be the graded algebra of   $(N\times N)$-matrices   with entries in $B$.
(A~matrix  has  a grading $p\in \ZZ$ whenever all  its  entries belong to $B^p$.)
Then there is a canonical bijection
\begin{equation}\label{eq:adjunction}
\Hom_{ \grAlg } \big( \tilde A_N , B\big)
\stackrel{\simeq}{\longrightarrow} \Hom_{\grAlg} (A, \Mat_N(B))
\end{equation}
which is natural in $A$ and $B$. Here   ${\grAlg}$    stands for  the category of graded algebras  and graded algebra homomorphisms.
The bijection \eqref{eq:adjunction}  carries
a   graded algebra homomorphism $r:\tilde A_N\to B$    to the
 map   $ A \to \Mat_N(B) $ sending any $a \in A$ to the   $(N\times N)$-matrix   $(r(a_{ij}))_{i,j}$.
The inverse bijection carries     a graded algebra homomorphism  $s:A\to \Mat_N(B)$ to the
 graded algebra homomorphism   $\tilde A_N\to B$ sending a generator $a_{ij}$ to
the $(i,j)$-th term of the matrix $s(a)$ for all $a\in A$.
 Consequently,    the  endofunctor $ A\mapsto \tilde A_N$ of $\grAlg$
is left adjoint to the endofunctor $  B\mapsto \Mat_N(B)$ of $\grAlg$.

The commutative graded algebra
$A_N=\Com(\tilde A_N) $ is obtained from $\tilde A_N$ by adding the relations
$a_{ij} b_{kl}=(-1)^{\vert a \vert \vert b\vert} b_{kl} a_{ij}$ for any homogeneous $a,b\in A$ and any $i,j,k,l\in \{1, \ldots, N\}$.
We call $A_N$ the \index{representation algebra} \emph{$N$-th representation algebra} of $A$.
The construction of $  A_N$ is functorial: a morphism $f:A\to A'$ in $\grAlg$ induces a morphism
$\tilde f_N:\tilde A_N\to \tilde A'_N$ in $\grAlg$, which in its turn
induces a morphism $  f_N:  A_N\to   A'_N$ in the category of commutative graded algebras   $\grCom$.
For any commutative graded algebra $B$,
 \begin{equation*}\label{commeqa}
 \Hom_{\grCom} ({  A}_N, B) \simeq  \Hom_{\grAlg} (\tilde A_N, B) \simeq \Hom_{\grAlg} (A, \Mat_N(B)).
\end{equation*}
  Consequently,    the functor ${\grAlg}\to {\grCom}, A\mapsto A_N$ is left adjoint to the functor ${{\grCom} \to {\grAlg}}, B\mapsto \Mat_N(B)$.

\subsection{Brackets}\label{gral++++}

Let  $A$ be a graded module and $d \in \ZZ$. By a \index{bracket} {\it bracket} in $A$ we   mean a  linear map $\bracket{-}{-} :A\otimes A\to A$.
A bracket  $\bracket{-}{-}$ in $A$ \index{bracket!degree of} {\it has degree $d$} if $\bracket{A^p}{A^q}\subset A^{p+q+d}$
 for all $p,q\in \ZZ$.
A bracket  $\bracket{-}{-}$ in $A$  is \index{bracket!$d$-antisymmetric} {\it $d$-antisymmetric}  if for all homogeneous $a,b \in A$,
\begin{equation}\label{antis} \bracket{a}{b}=- (-1)^{ \vert a \vert_d \vert
b\vert_d} \bracket{b}{a}.
\end{equation}
A bracket  $\bracket{-}{-}$ in $A$  satisfies the \index{$d$-graded Jacobi identity} {\it  $d$-graded Jacobi identity} if
\begin{equation}\label{jaco}  (-1)^{ \vert a \vert_d \vert c \vert_d} \bracket {a}  {\bracket{b}{c}}
+ (-1)^{ \vert b \vert_d\vert a \vert_d} \bracket  {b} {\bracket{c}{a}}
+(-1)^{  \vert c \vert_d \vert b \vert_d  }   \bracket   {c} {\bracket{a}{b}}    =0
\end{equation}
for all homogeneous $a,b,c \in A$.  A degree $d$ bracket $\bracket{-}{-}$ in $ A$ satisfying   \eqref{antis} and \eqref{jaco}
is called a \index{$d$-graded Lie bracket}  {\it $d$-graded Lie bracket}, and the pair $(A, \bracket{-}{-})$ is called then a
\index{$d$-graded Lie algebra} {\it $d$-graded Lie algebra}.

For example, any  graded algebra $A$ gives rise to a   $0$-graded   Lie algebra of derivations in $A$. Recall that
a \index{derivation} \emph{derivation    in  $A$  of   degree}
$k\in \ZZ$ is a linear map $\delta :A \to A$ such that $\delta(A^p)\subset
A^{p+k}$   for any $p\in \ZZ$ and $\delta(ab)=\delta(a)b+(-1)^{k\vert a\vert}a\delta(b)$
for any homogeneous   $a \in A$ and any $b\in A$.
Derivations  of $A$ of degree $k$   form a module $ \Der^k(A) $. The graded module $\Der(A)=\oplus_{ k\in \ZZ } \Der^k(A)$
  carries a   $0$-graded   Lie bracket  defined by
$[\delta_1,\delta_2]= \delta_1 \delta_2-(-1)^{k_1k_2} \delta_2 \delta_1$
for any derivations $\delta_1$ and $\delta_2$ of $A$ of degrees $k_1$ and $k_2$ respectively.

A bracket $\bracket{-}{-} $ in a graded algebra $A$ satisfies the \index{$d$-graded Leibniz rules} {\it $d$-graded Leibniz rules} if for all
homogeneous $a,b,c\in A$,
\begin{eqnarray}
\label{poisson1} \bracket{a}{bc} &= &\bracket{a}{b}c + (-1)^{\vert a\vert_d \vert b\vert} b \bracket{a}{c},\\
\label{poisson2} \bracket{ab}{c} &= & a  \bracket{b}{c}   + (-1)^{\vert b\vert \vert c\vert_d}     \bracket {a}{c} b.
\end{eqnarray}
A \index{bracket!Gerstenhaber} {\it  Gerstenhaber  bracket of degree $d\in \ZZ$} in a  graded algebra~$A$
is a $d$-graded Lie bracket $\{ -, -\}$ in~$A$ which satisfies the $d$-graded Leibniz rules. The pair $(A, \{ -, -\})$ is called then
a  \index{graded algebra!Gerstenhaber} {\it  Gerstenhaber  algebra of degree $d$}.
For example, any  graded algebra~$A$ is a  Gerstenhaber  algebra of degree $0$
with respect to the   bracket (called the  \index{commutator} {\it commutator}) defined by $\bracket{a}{b} = ab - (-1)^{\vert a\vert \vert b\vert} ba$
for   homogeneous $a,b\in A$ and extended to all $a,b\in A$ by linearity.

\section{Bibrackets}\label{bibrackets}

    The rest of this chapter presents an extension of   Van den Bergh's  \cite{VdB} theory of double brackets
in algebras to   graded algebras.  
 Such an extension is outlined in \cite[Section 2.7]{VdB}  in the case of degree $ -1$.  
Fix throughout this section a graded algebra $A$ and an integer $d$.

\subsection{ Conventions}\label{conventions}

Any  $x \in A^{\otimes 2}=   A\otimes A  $ expands 
as a   sum $ x=\sum_\alpha x'_\alpha \otimes
x''_\alpha$ where $ x'_\alpha,
x''_\alpha$ are homogeneous elements of $A$ and the
index $\alpha$ runs over
a finite set. To simplify notation, we will drop  the summation sign and the index
and write simply $x =x'\otimes x''$.     Similarly, an
element $x$ of $A^{\otimes 3}=  A\otimes A \otimes A $  will be written as
$x'\otimes x''\otimes x'''$ with homogeneous $x', x'', x'''\in A$.

Unless explicitly stated otherwise, we endow $A^{\otimes 2}$ with
the ``outer'' $A$-bimodule structure  defined by  $ax b= ax'\otimes  x'' b $
for  any   $a,b \in A$ and $x\in A^{\otimes 2}$. We shall  also use the
``inner'' $A$-bimodule structure on $A^{\otimes 2}$   defined by
\begin{equation}\label{inner}
a*x*b= (-1)^{\vert a\vert \vert b\vert +\vert a\vert \vert x'\vert +\vert b\vert \vert x''\vert} \, x' b \otimes a x'' \quad
\end{equation}
for  homogeneous $a,b \in A$ and any  $x\in A^{\otimes 2}$.

Given a permutation $(i_1,\dots,i_n)$ of $(1,\dots,n)$ with $n\geq 1$,
we denote by $\Perm_{i_1\cdots i_n}$ the \index{graded permutation} \emph{graded permutation} $A^{\otimes n}\to A^{\otimes n}$
carrying  any     $a_1\otimes \cdots \otimes a_n $  with homogeneous $a_1,\dots,a_n\in A$
to $(-1)^{t} a_{i_1} \otimes a_{i_2} \otimes \cdots \otimes a_{i_n}$
 where    $t\in \ZZ$ is the sum of the products $\vert a_{i_k} \vert \vert a_{i_l} \vert$
over all pairs of indices  $ k < l $  such that $i_k>i_l$.
For any   $d\in \ZZ$, we similarly define  the  \index{$d$-graded permutation} {\it $d$-graded permutation}
$\Perm_{i_1\cdots i_n,d} :A^{\otimes n}\to A^{\otimes n}$  using the  $d$-degree 
 $\vert\!-\!\vert_d =\vert\!-\!\vert + d$   instead of    $\vert\!-\!\vert$.

\subsection{Bibrackets in~$A$}\label{bibr}

A   \index{bibracket} \emph{bibracket} in   $A$ is a  linear map $$\double{-}{-}:A \otimes A   \longrightarrow   A \otimes A.$$
A  bibracket $\double{-}{-}$ in $A$ has \index{bibracket!of degree $d$} {\it degree $d$} if for any integers $p,q$,
$$
\double{A^p}{A^q}  \subset \bigoplus_{{i+j=p+q+d}}\, A^i \otimes A^j .
$$
A  \index{bibracket} \emph{$d$-graded bibracket} in   $A$ is a bibracket $\double{-}{-}$ in $A$ of degree $d$
satisfying the following \index{$d$-graded Leibniz rules}  {\it $d$-graded Leibniz rules}:   for all homogeneous $a,b,c\in A$,
\begin{eqnarray}
\label{bibi} \double{a}{bc} &=& \double{a}{b}c + (-1)^{\vert a\vert_d \vert b\vert} b \double{a}{c},\\
\label{bibi+} \double{ab}{c} &=&  a* \double{b}{c}   + (-1)^{\vert b\vert \vert c\vert_d}   \double{a}{c}*b.
\end{eqnarray}

The following key  lemma shows that a $d$-graded bibracket in $A$ induces brackets of degree~$d$  in all  representation algebras $\{A_N\}_{N }$.

\begin{lemma}\label{fdbtopb}
Given  a $d$-graded bibracket $\double{-}{-}$ in  $A$
and an integer $N\geq 1$, there is a unique  bracket $\{ -, -\}$ in $A_N $ satisfying
  the $d$-graded Leibniz rules  \eqref{poisson1}, \eqref{poisson2}  and such that
\begin{equation}\label{db_to_qpb}
\bracket{a_{ij}}{b_{uv}}=  \double{a}{b }_{uj}'  \double{a }{b }_{iv}''
\end{equation}
for all  $a,b\in A$ and $i,j,u,v\in \{1,\ldots, N\}$. The bracket $\{ -, -\}$ has degree~$d$.
\end{lemma}

\begin{proof}  We extend \eqref{db_to_qpb} to a bilinear form
$\{ -, -\}:A_N\times A_N \to A_N$ satisfying \eqref{poisson1} and
\eqref{poisson2}. To see that this form is well-defined, we need to
verify the compatibility with the defining relations of $A_N$. That
the right-hand side of \eqref{db_to_qpb} is  linear  in  $a$ and $b$
follows from the   linearity   of $\double{-}{- }$.
We now verify the compatibility with
the third relation  in \eqref{eq:addid}.
Pick any homogeneous $a,b, c\in A$ and set
$x= \double{ a}{b} $ and $y= \double{ a}{c} $. Then
$$\double {a} {bc}=x c + (-1)^{\vert a\vert_d \vert
b\vert}  by  = x'\otimes x''c + (-1)^{\vert a\vert_d
\vert b\vert} b y'\otimes y'' .$$ Therefore, for  any  $i,j,u,v\in \{1,2,\dots,N\}$,
\begin{eqnarray*}
\bracket{a_{ij}}{(bc)_{uv}}&=& {\double{a} {bc}}'_{uj} {\double{a}
{bc}}''_{iv}\\
&=& x'_{uj} (x''c)_{iv} + (-1)^{\vert a\vert_d \vert
b\vert} (b y')_{uj} y''_{iv} \\
&=& x'_{uj} x''_{il} c_{lv}+ (-1)^{\vert a\vert_d \vert b\vert} b_{ul} y'_{lj}
y''_{iv}
\\
&=&\bracket{a_{ij}}{b_{ul}} c_{lv}+ (-1)^{\vert a\vert_d
\vert b\vert} b_{ul} \bracket{a_{ij}}{ c_{lv}}   \ = \   \bracket{a_{ij}}{b_{ul}c_{lv}}.
\end{eqnarray*}
To check that
$\bracket{(ab)_{ij}}{c_{uv}}= \bracket{a_{il} b_{lj}}{ c_{uv}}$, set
$z=\double{a}{c} $ and $t=\double{b}{c} $. Then
$$\double {ab} {c}
=a* t   + (-1)^{\vert b\vert \vert c\vert_d} z*b= (-1)^{\vert t'\vert
\vert a\vert} t'\otimes a t''+ (-1)^{\vert b\vert \vert c z'' \vert_d
} z'b \otimes z''.$$ Therefore \begin{eqnarray*}
\bracket{(ab)_{ij}}{c_{uv}} &=& {\double{ab} {c}}'_{uj} {\double{ab}
{c}}''_{iv}\\
&=& (-1)^{\vert t'\vert \vert a\vert} t'_{uj} (a t'' )_{iv} +
(-1)^{\vert b\vert \vert c z'' \vert_d } (z'b)_{uj} z''_{iv} \\
&=& (-1)^{\vert t'\vert \vert a\vert} t'_{uj} a_{il} t''_{lv} +
(-1)^{\vert b\vert \vert c z'' \vert_d } z'_{ul} b_{lj} z''_{iv}
\\
&=& a_{il} t'_{uj}  t''_{lv} + (-1)^{\vert b\vert \vert c
\vert_d }  z'_{ul} z''_{iv}  b_{lj} \\  &=& a_{il} \bracket{
b_{lj}}{c_{u v}} +  (-1)^{\vert b\vert \vert c\vert_d}
\bracket{a_{il}}{c_{u v}}  b_{lj}   \ = \   \bracket{a_{il} b_{lj}}{c_{u v}}.
\end{eqnarray*}
The last claim of the lemma follows from the definitions.
\end{proof}

\subsection{Antisymmetric bibrackets}\label{smbibr}

Consider the   linear involutions  $\Perm_{21}$ and $\Perm_{21,d}$ of $A^{\otimes 2}$
determined by the permutation $(21)$  as in Section~\ref{conventions}:
for   homogeneous $a,b\in A$, we have 
$$
\Perm_{21}(a\otimes b) = (-1)^{\vert a \vert \vert b \vert} b \otimes a\,\, \,\, 
 {\rm {and}} \,\, \,\, \Perm_{21,d}(a\otimes b) = (-1)^{\vert a \vert_d \vert b \vert_d} b \otimes a.
 $$ 
Given  $f\in \End( A^{\otimes 2})$, the \index{$d$-transpose} {\it $d$-transpose}  of $f$ is $f_d = \Perm_{21} f \Perm_{21,d} \in \End (A^{\otimes 2})$.

\begin{lemma}\label{bibrtranspose}
   A bibracket  $\double{-}{-}$   satisfies 
\eqref{bibi}
if and only if its $d$-transpose $\double{-}{-}_d$
satisfies 
 \eqref{bibi+}.
\end{lemma}

\begin{proof}
Assume that  a bibracket $\double{-}{-}$  in $A$  satifies \eqref{bibi}.
Pick   homogeneous $ a,b, c\in A$ and set
$x=\double{c}{a}  $, $y=\double{c}{b} $. Then
\begin{eqnarray*} \double{ab}{c}_d &=& (-1)^{\vert
ab\vert_d \vert c\vert_d}  \Perm_{21}( \double{c}{ab})\\
&=&  (-1)^{\vert
ab\vert_d \vert c\vert_d}   \Perm_{21}(\double{c}{a} b  +(-1)^{\vert
c \vert_d \vert a\vert}  a \double{c}{b})\\
&=& (-1)^{\vert
ab\vert_d \vert c\vert_d}  \Perm_{21} (  x' \otimes x''  b  +(-1)^{\vert
c \vert_d \vert a\vert}
ay' \otimes y'' )\\
&=&    (-1)^{\vert ab\vert_d \vert  c\vert_d+ \vert  x'  \vert \vert  x''  b \vert } x''  b\otimes
x'
+   (-1)^{\vert
b \vert_d \vert c\vert_d +\vert ay'  \vert \vert y'' \vert }
y'' \otimes ay' \\
&=&    (-1)^{\vert ab\vert_d \vert  c\vert_d  } \Perm_{21} (\double{c}{a})\ast  b
+   (-1)^{\vert
b \vert_d \vert c\vert_d   }
a\ast \Perm_{21} (\double{c}{b}) \\
&=& (-1)^{\vert
b\vert \vert c\vert_d} \double{a}{c}_d \ast b +a\ast  \double{b}{c}_d .
\end{eqnarray*}
  So, $\double{-}{-}_d$ satifies \eqref{bibi+}.
The converse is shown by a similar computation.
\end{proof}

A  bibracket  $\double{-}{-}$ in $A$   
is \index{bibracket!$d$-antisymmetric} {\it $d$-antisymmetric} if  $\double{-}{-}_d=- \double{-}{-}$.
By Lemma~\ref{bibrtranspose}, a $d$-antisymmetric  bibracket    satisfies
\eqref{bibi} if and only if it satisfies \eqref{bibi+}.
Note  for the record, that given  a  $d$-antisymmetric  bibracket  $\double{-}{-}$ in $A$,
we have for any homogeneous   $a,b\in A$,
\begin{equation}\label{flip}
\double{b}{a} = - (-1)^{\vert a \vert_d \vert b \vert_d+ \left\vert \double{a}{b}'\right\vert  \left\vert \double{a}{b}''\right\vert }
\double{a}{b}'' \otimes \double{a}{b}'\, .
\end{equation}

\begin{lemma}\label{fdbtopbsymm} If in Lemma~\ref{fdbtopb} the bibracket
$\double{-}{-}$ is $d$-antisymmetric,   then the induced   bracket  $\{-, -\}$ in $A_N$ is    $d$-antisymmetric,  i.e., satisfies \eqref{antis}.
\end{lemma}

\begin{proof}   Pick any homogeneous  $a,b\in A$
and  set $ x=\double{ a}{b}$. Then
\begin{eqnarray*}
\bracket{b_{uv}}{a_{ij}} &\stackrel{\eqref{flip}}{=}& - (-1)^{\vert a\vert_d \vert b\vert_d+\vert x'\vert \vert x''\vert} \, x''_{iv} x'_{uj}\\
&=& - (-1)^{\vert a\vert_d \vert b\vert_d } \, x'_{uj} x''_{iv}   \ = \   - (-1)^{\vert a\vert_d \vert b\vert_d } \bracket {a_{ij}}{b_{uv}}.
\end{eqnarray*}

\vspace{-0.5cm}
\end{proof}

\subsection{The Jacobi identity}\label{bibrJac}

The  bracket  in $A_N$ constructed in Lemma~\ref{fdbtopb} may not satisfy the  $d$-graded  Jacobi identity \eqref{jaco}.
To compute   the deviation from this identity, we observe that
any   bibracket $\double{-}{-}$ in $A$   induces a linear endomorphism $\triple{-}{-}{-}$ of $A^{\otimes 3}$,  called the
\index{tribracket!induced} {\it induced    tribracket},  by
\begin{equation}\label{tribracket}
\triple{-}{-}{-} =
\sum_{i=0}^2  \Perm_{312}^i ( \double{-}{-} \otimes \id_A) (\id_A \otimes \double{-}{-}) \Perm_{312,d}^{-i}
\end{equation}
where    $\Perm_{312}, \Perm_{312,d}\in \End(A^{\otimes 3})$ are as defined in Section \ref{conventions}.

\begin{lemma}\label{qP--}   Let $N \geq1$.   If $\double{-}{-}$ is a  $d$-antisymmetric $d$-graded bibracket in~$A$,
then  the associated bracket  $\{-, -\}$ in $A_N$ satisfies
\begin{eqnarray*}
&&\bracket {a_{pq}} {\bracket {b_{rs}} {c_{uv}}}+ (-1)^{\vert a
\vert_d \vert  bc\vert }
\bracket {b_{rs}}  {\bracket  {c_{uv}} {a_{pq}}}+ (-1)^{\vert ab
\vert  \vert  c\vert_d} \bracket  {c_{uv}}
{\bracket  {a_{pq}}  {b_{rs}}} \\
&= & {\triple {a}  {b}{c}}_{uq}'    {\triple {a}  {b}{c}}_{ps}''
{\triple {a}  {b}{c}}_{rv}''' - (-1)^{\vert   b \vert_d
\vert   c \vert_d } {\triple {a}  {c}{b}}_{rq}' {\triple {a}  {c}{b}}_{pv}'' {\triple {a}  {c}{b}}_{us}'''
\end{eqnarray*}
for any homogeneous $a,b,c\in A$,
any $p,q,r,s,u,v\in \{1,\dots, N\}$.
\end{lemma}

\begin{proof}
  It follows from the definitions that
\begin{eqnarray*}
\triple{a}{b}{c}&=& \double{a}{\double{b}{c}'}\otimes \double{b}{c}''
+  (-1)^{\vert a \vert_d \vert bc\vert} \Perm_{312} \left (\double{b}{\double{c}{a}'} \otimes \double{c}{a}''\right) \\
&&+ (-1)^{ \vert ab\vert \vert c \vert_d} \Perm_{312}^2 \left (\double{c}{\double{a}{b}'} \otimes \double{a}{b}''\right)\\
&=& \double{a}{\double{b}{c}'}'\otimes \double{a}{\double{b}{c}'}''\otimes \double{b}{c}''\\
&&+ (-1)^{\vert a \vert_d \vert bc\vert}
\Perm_{312} \left (\double{b}{\double{c}{a}'}' \otimes \double{b}{\double{c}{a}'}''   \otimes \double{c}{a}''\right)\\
&& + (-1)^{ \vert ab\vert \vert c \vert_d} \Perm_{312}^2 \left (\double{c}{\double{a}{b}'}'
\otimes \double{c}{\double{a}{b}'}''  \otimes \double{a}{b}''\right).
\end{eqnarray*}
Using the commutativity of $A_N$, we deduce that
\begin{eqnarray}
\label{abc} && {\triple {a}  {b}{c}}_{uq}'    {\triple {a}  {b}{c}}_{ps}'' {\triple {a}  {b}{c}}_{rv}'''\\
\notag &=&\double{a}{\double{b}{c}'}'_{uq} \double{a}{\double{b}{c}'}''_{ps} \double{b}{c}''_{rv}\\
\notag &&+ (-1)^{\vert a \vert_d \vert bc\vert}
\double{b}{\double{c}{a}'}'_{ps}  \double{b}{\double{c}{a}'}''_{rv}  \double{c}{a}''_{uq}\\
\notag && + (-1)^{ \vert ab\vert \vert c \vert_d}  \double{c}{\double{a}{b}'}'_{rv}
\double{c}{\double{a}{b}'}''_{uq}   \double{a}{b}''_{ps}.
\end{eqnarray}
Applying the transpositions $b\leftrightarrow c$, $r \leftrightarrow  u$,  and $s \leftrightarrow v$,
we obtain
\begin{eqnarray}
\label{acb} && {\triple {a}  {c}{b}}_{rq}'    {\triple {a}  {c}{b}}_{pv}'' {\triple {a}  {c}{b}}_{us}'''\\
\notag &=&\double{a}{\double{c}{b}'}'_{rq} \double{a}{\double{c}{b}'}''_{pv} \double{c}{b}''_{us}\\
\notag &&+ (-1)^{\vert a \vert_d \vert cb\vert}
\double{c}{\double{b}{a}'}'_{pv}  \double{c}{\double{b}{a}'}''_{us}  \double{b}{a}''_{rq}\\
\notag && + (-1)^{ \vert ac\vert \vert b \vert_d}  \double{b}{\double{a}{c}'}'_{us}
\double{b}{\double{a}{c}'}''_{rq}   \double{a}{c}''_{pv}.
\end{eqnarray}
Equalities \eqref{abc} and \eqref{acb} allow us to expand the right-hand side of the formula claimed in the lemma.
We next expand the left-hand side of this formula.
Set $x=\double {b} {c}\in A^{\otimes 2}$ and observe that
\begin{eqnarray*}
\bracket {a_{pq}} {\bracket {b_{rs}} {c_{uv}}}
&=&\bracket {a_{pq}}{x'_{us} x''_{rv}}\\
&=& \bracket {a_{pq}} {x'_{us} } {x}''_{rv} + (-1)^{\vert a\vert_d  \vert x'\vert}  x'_{us} \bracket {a_{pq}} {{x}''_{rv}}\\
&=& \double {a}  {x'}'_{uq}  \double{a}{ x'}''_{ps} {x}''_{rv}
+ (-1)^{\vert a\vert_d  \vert x'\vert}  x'_{us} \double {a}  {x''}'_{rq}  \double{a}{ x''}''_{pv}.
\end{eqnarray*}
We   rewrite  the second summand as follows.  Since $\double{-}{-}$   has degree $d$,
$$\vert \double {a}  {x''}'_{rq}
\double{a}{ x''}''_{pv} \vert = \vert \double {a}  {x''}'
\double{a}{ x''}''  \vert=\vert a \vert +\vert x''\vert +d =\vert a \vert_d +\vert x''\vert  .$$
The commutativity of $A_N$  implies that
$$(-1)^{\vert a\vert_d  \vert x'\vert}  x'_{us} \double {a}  {x''}'_{rq}
\double{a}{ x''}''_{pv}= (-1)^{   \vert x'\vert   \vert x''\vert  } \double {a}  {x''}'_{rq}
\double{a}{ x''}''_{pv} x'_{us}.$$
 The  $d$-antisymmetry of  $\double{-}{-}$ allows us to compute $ x=\double{b}{c}$  from $y=  \double{c}{b}$:
by \eqref{flip}, we have
$x'\otimes x'' =  -(-1)^{\vert b\vert_d \vert c\vert_d+\vert y'\vert \vert y''\vert} y''\otimes y'$. Hence,
$$(-1)^{   \vert x'\vert   \vert x''\vert  } \double {a}  {x''}'_{rq}
\double{a}{ x''}''_{pv} x'_{us}=  -(-1)^{\vert b\vert_d \vert c\vert_d}
\double {a}  {y'}'_{rq}
\double{a}{ y'}''_{pv} y''_{us}. $$
As a result, we obtain  that
\begin{eqnarray}
\bracket {a_{pq}} {\bracket {b_{rs}} {c_{uv}}}
\notag &=& \double {a}  {\double{b}{c}'}'_{uq}  \double{a}{ \double{b}{c}'}''_{ps} {\double {b}
{c}}''_{rv}\\
\label{abc_simple} && - (-1)^{\vert b\vert_d \vert c\vert_d}   \double {a}  {\double {c} {b}'}'_{rq}
\double{a}{ \double {c} {b}'}''_{pv} \double {c} {b}''_{us}.\end{eqnarray}
Cyclically permuting $a,b,c$ and   the indices, we obtain
\begin{eqnarray}
\notag \bracket {b_{rs}}  {\bracket  {c_{uv}} {a_{pq}}}
&=&\double {b}  {\double {c} {a}'}'_{ps}  \double{b}{ \double {c} {a}'}''_{rv} {\double {c}{a}}''_{uq}\\
\label{bca_simple} && - (-1)^{\vert c\vert_d \vert a\vert_d}   \double {b}  {\double {a} {c}'}'_{us}
\double{b}{ \double {a} {c}'}''_{rq} \double {a} {c}''_{pv}
\end{eqnarray}
  and
\begin{eqnarray}
\notag \bracket  {c_{uv}}  {\bracket  {a_{pq}}  {b_{rs}}}&=&
\double {c}  {\double {a} {b}'}'_{rv}  \double{c}{ \double {a} {b}'}''_{uq} {\double {a}{b}}''_{ps}\\
\label{cab_simple}&& - (-1)^{\vert a\vert_d \vert b\vert_d}   \double {c}
{\double {b} {a}'}'_{pv}
\double{c}{ \double {b} {a}'}''_{us} \double {b} {a}''_{rq}.
\end{eqnarray}
  The required formula directly   follows from the equalities \eqref{abc}--\eqref{cab_simple}.
\end{proof}

\subsection{Gerstenhaber bibrackets}\label{stdb-}

A  \index{bibracket!Gerstenhaber} {\it  Gerstenhaber   bibracket of degree~$d$} in~$A$    is a  $d$-antisymmetric  $d$-graded  bibracket $\double{-}{-}$   in~$A$
  such that the induced tribracket  \eqref{tribracket} is equal to zero.
The pair $(A,\double{-}{-})$ is called then a \index{double Gerstenhaber algebra} \emph{double Gerstenhaber algebra of degree~$d$}.
  This  structure  was first introduced by Van den Bergh \cite[Section 2.7]{VdB} for $d=-1$; 
see also \cite{BCER} in the setting of differential graded algebras.

\begin{lemma}\label{importantGers}
For any Gerstenhaber bibracket of degree $d$  in  $A$
and  $N\geq 1$,   the bracket $\bracket{-}{-}$ in $A_N$
given by Lemma~\ref{fdbtopb}   is a      Gerstenhaber  bracket of degree~$d$.
\end{lemma}

\begin{proof}
This follows from Lemmas~\ref{fdbtopb}, \ref{fdbtopbsymm}, and \ref{qP--}.
The  equality $$ \bracket {a_{pq}} {\bracket {b_{rs}} {c_{uv}}}+ (-1)^{\vert a
\vert_d \vert  bc\vert }
\bracket {b_{rs}}  {\bracket  {c_{uv}} {a_{pq}}}+ (-1)^{\vert ab
\vert  \vert  c\vert_d} \bracket  {c_{uv}}
{\bracket  {a_{pq}}  {b_{rs}}}=0 $$
provided by Lemma~\ref{qP--} implies   the  $d$-graded Jacobi identity    \eqref{jaco} in which  $a,b,c$ are replaced
with $a_{pq}, b_{rs},  c_{uv}$, respectively.
\end{proof}

\section{Equivariance}\label{sectionActions}

We show that the bracket  constructed   in Lemma~\ref{fdbtopb} is  equivariant 
under the natural actions of the general linear group and the Lie algebra of matrices on the representation algebra.
We begin with   terminology.

\subsection{Lie pairs}

By a \index{Lie pair} \emph{Lie pair} we mean  a pair $(G, \g)$   where $G$ is a group   and
 $\g$ is a  (non-graded) Lie algebra  endowed with a (left) action of $G$ on $\g$ by Lie algebra
 automorphisms. The  action  is denoted by  $ w\mapsto  {}^g\! w$ for $w\in \g$ and $g\in G$.  
 
Given a   Lie pair  $(G,\g)$, by a \index{$(G,\g)$-algebra} \emph{$(G,\g)$-algebra} we mean a
 graded  algebra $A$ endowed with  an  action of $G $   and  an action of
$\g $   such that  $ {}^g \! w\, a=g\, w (g^{-1}  a)  $ for all $g\in G $, $w\in \g $, $a\in A$.
Here   an action of $G$ on $A$ is a group homomorphism from $G$ to the group of graded algebra automorphisms of  $A$,
and an  action    of   $\g$  on $A$ is a   Lie algebra  homomorphism from $\g $
 to   the Lie algebra   of derivations of $A$ of degree  zero, cf$.$ Section~\ref{gral++++}.

\subsection{ Action on the representation algebras }\label{twoactionsAN+++} 

Fix  an integer   $N\geq 1$.
Let   $G_N=\GL_N(\kk)$  be the $N$-th general linear group over $\kk$  and let
$\g_N=\Mat_N(\kk)$ be the    Lie algebra   of $(N\times N)$-matrices with Lie bracket $[u,v]=uv-vu$.
The   pair $(G_N, \g_N)$ is a Lie pair where  $G_N$ acts on $\g_N$  by ${}^g\! w=gwg^{-1}$
for  any $g\in G_N$, $w\in \g_N$.  
 The representation   algebra ${\tilde A}_N$ associated with   a   graded algebra $A$   in Section \ref{AMT0}
 is a $(G_N,\g_N)$-algebra. Here  $G_N $
acts on ${\tilde A}_N$ as follows: for  a matrix $g=(g_{k,l})_{k,l=1}^N\in G_N$ and a generator  $a_{ij}\in {\tilde A}_N$,   set
\begin{equation}\label{firstform}
g a_{ij}= (g^{-1})_{i,k}\, g_{l,j}\, a_{kl}.
\end{equation}
In   this formula,   the numerical coefficients appear to the left of the   generator  $a_{kl}$. It is
easier to remember \eqref{firstform}  in the equivalent form
$g a_{ij}= (g^{-1})_{i,k}\, a_{kl} \, g_{l,j}$,
and   we will   use the latter form. Direct computations show
that these formulas are compatible with the   relations in
${\tilde A}_N$ and define an action  of $G_N$ on ${\tilde A}_N$.
 We verify the compatibility with the relation $(ab)_{ij}= a_{il} b_{lj}$:
\begin{eqnarray*}
g (ab)_{ij} \! & = & \! (g^{-1})_{i,k}\, (ab)_{kl} \, g_{l,j} \ = \
(g^{-1})_{i,k}\, a_{kp} b_{pl} \, g_{l,j} \\
\! & = & \! (g^{-1})_{i,k}\,  a_{kp} \delta_{pq} b_{ql} \, g_{l,j}
\ = \ (g^{-1})_{i,k}\, a_{kp} g_{p,r} (g^{-1})_{r,q} b_{ql} \, g_{l,j} \ = \ (g a_{ir}) (gb_{rj}).
\end{eqnarray*}
 The Lie algebra  $\g_N $  acts on ${\tilde A}_N$ as follows: for  a matrix   $w=(w_{k,l})_{k,l=1}^N\in \g_N$ and
a generator $a_{ij}\in \tilde A_N$, set
\begin{equation}\label{action}
wa_{ij}= a_{ik} w_{k,j}- w_{i,k} a_{kj}.
\end{equation}
This formula is compatible with the   relations in ${\tilde A}_N$
and defines an action of $\g_N$  on~${\tilde A}_N$.
 We verify the compatibility with the   relation $(ab)_{ij}= a_{il} b_{lj}$:
\begin{eqnarray*}
w(a_{il} b_{lj}) &= &w(a_{il}) b_{lj} +  a_{il} w(b_{lj})\\
&= &a_{ik} w_{k,l} b_{lj} - w_{i,k}a_{kl}b_{lj} +  a_{il} b_{lk} w_{k,j}- a_{il} w_{l,k} b_{kj}\\
&=&a_{il} b_{lk} w_{k,j}- w_{i,k} a_{kl} b_{lj}=(ab)_{ik} w_{k,j}- w_{i,k} (ab)_{kj} \ = \ w(ab)_{ij}.
\end{eqnarray*}
It is easy to check that these   actions turn ${\tilde A}_N$ into a $(G_N,\g_N)$-algebra.
Moreover,  these actions    descend to the   commutative   graded algebra   $A_N=\Com(\tilde A_N)$  and
turn it into  a  $(G_N, \g_N)$-algebra.  

 The next  lemma shows that the bracket  in $A_N$ provided by Lemma~\ref{fdbtopb} is  equivariant under the actions of $G_N$ and $ \g_N$.   

\begin{lemma}\label{fdbtopb+action}
  Let  $\double{-}{-}$ be  a $d$-graded bibracket  in a graded algebra $A$.
 For any $N\geq 1$,  the   bracket $\{ -, -\}$ in $ A_N $ defined in Lemma~\ref{fdbtopb}    satisfies
\begin{equation}\label{ssdM++}   g
\bracket{a} {b}  = \bracket{ga }{gb} \quad {\rm {and}} \quad
w \bracket{a} {b}  = \bracket{wa}{b}  +  \bracket{a}{wb}
\end{equation}
for all $g\in G_N $,
$w\in \g_N $  and $a,b\in  A_N$.
\end{lemma}

\begin{proof} Pick  $g=(g_{k,l})_{k,l} \in G_N$. It is easy to see that if
the identity     $ g \bracket{x} {y}  = \bracket{gx}{gy}  $
holds for  all the generators   of $ A_N$,
then it holds for any $x,y \in  A_N$.   Given  $a,b\in A$ and $i,j,u,v\in \{1,\dots,N\}$,
\begin{eqnarray*}
\bracket{ g a_{ij}}{g b_{uv} }
&=& \bracket{(g^{-1})_{i,k}  a_{kl} g_{l,j}}{(g^{-1})_{u,s} b_{st}  g_{t,v}}\\
&= & (g^{-1})_{i,k}  g_{l,j} (g^{-1})_{u,s} g_{t,v} \bracket{ a_{kl} }{ b_{st} }\\
&=& (g^{-1})_{i,k}  g_{l,j} (g^{-1})_{u,s} g_{t,v}  \double{a}{b}'_{sl} \double{a}{b}''_{kt}\\
&=& (g^{-1})_{u,s}  \double{a}{b}'_{sl} g_{l,j}  \,   (g^{-1})_{i,k} \double{a}{b}''_{kt}  g_{t,v} \\
&=& (g \double{a}{b}'_{uj}) (g \double{a}{b}''_{iv})
\ = \  g (\double{a}{b}'_{uj}\double{a}{b}''_{iv}) \ = \ g \bracket{   a_{ij}}{  b_{uv} }.
\end{eqnarray*}
Similarly, given $w=(w_{k,l})_{k,l}\in \g_N $, it is enough to check the
identity $ w \bracket{x} {y}  = \bracket{wx}{y}  +  \bracket{x}{wy}$
for   the generators  of $A_N$.    For   $a,b\in A$ and $i,j,u,v\in \{1,\dots,N\}$,
\begin{eqnarray*}
w \bracket{   a_{ij}}{  b_{uv} }
&=& w (\double{a}{b }_{uj}'  \double{a }{b }_{iv}'')\\
&=& w (\double{a}{b }_{uj}' ) \double{a }{b }_{iv}''
+ \double{a}{b }_{uj}' w( \double{a }{b }_{iv}'') \\
&=& \double{a}{b }_{uk}'  w_{k,j}  \double{a }{b }_{iv}''
- w_{u,k}  \double{a}{b }_{kj}'  \double{a }{b }_{iv}'' \\
&& +  \double{a}{b }_{uj}'  \double{a }{b }_{ik}''  w_{k,v}
-  \double{a}{b }_{uj}'  w_{i,k}  \double{a }{b }_{kv}'' \\
&=&    w_{k,j} \bracket{  a_{ik} }{  b_{uv} }
- w_{u,k}  \bracket{  a_{ij} }{  b_{kv} }   +     w_{k,v} \bracket{  a_{ij} }{  b_{uk} }
-   w_{i,k} \bracket{  a_{kj} }{  b_{uv} }\\
&=&  \bracket{  a_{ik} w_{k,j} - w_{i,k} a_{kj}}{  b_{uv} } +
\bracket{   a_{ij}}{ b_{uk} w_{k,v} -w_{u,k} b_{kv} }\\
&=&  \bracket{  w a_{ij}}{  b_{uv} } +  \bracket{   a_{ij}}{ w b_{uv} }.
\end{eqnarray*}

\vspace{-0.5cm}
\end{proof}

\section{The associated pairing   and the trace}\label{sectionActions-ee}

We   study the pairing $A \otimes A \to A$  induced by a bibracket in a graded algebra~$A$   and, in particular,  discuss its  behavior  under   the trace maps.

\subsection{The pairing  $\langle - , - \rangle$}\label{stdb}

  A  bibracket $\double{-}{-}$   in  a graded algebra~$A$   induces   
  an \index{associated pairing}  \emph{associated pairing}    $\langle -,-  \rangle: A \otimes A \to A$    by 
$$\langle a,b \rangle= {\double{a }{b}}'  {\double{a }{b}}''\in A  \quad {\text {for}} \quad a,b \in A .$$

\begin{lemma}\label{comparis-}
Let    $\double{-}{-}$  be a  $d$-antisymmetric $d$-graded  bibracket  in~$A$. 
Then  the associated pairing   $\langle -,-  \rangle$ has the following properties:
\begin{itemize}
\item[(i)] $\langle -,-\rangle$ has degree $d$ and satisfies the $d$-graded Leibniz rule  \eqref{poisson1},
\item[(ii)] $\langle a,b \rangle \equiv -(-1)^{\vert a\vert_d \vert b\vert_d } \langle b,a \rangle \
({\rm {mod}} \, [A,A])$ for all homogeneous $a,b\in A$,
\item[(iii)] $\langle [A,A], A \rangle=0$ and $\langle A, [A,A]\rangle \subset [A,A]$,
\item[(iv)] for any homogeneous $a,b,c\in A$,
\begin{eqnarray*}
& & \langle \langle a, b\rangle, c\rangle - \langle a , \langle b , c \rangle \rangle
+ (-1)^{\vert a\vert_d \vert b\vert_d} \langle b , \langle a , c \rangle \rangle \\
 &=&   m\big((-1)^{\vert a \vert_d \vert b \vert_d}\triple {b} {a} {c} - \triple {a} {b} {c}\big)
\end{eqnarray*}
where  $m\in \Hom(A^{\otimes 3}, A)$  carries  $x\otimes y\otimes z$ to $xyz $ for all $x,y,z\in A$.
\end{itemize}
\end{lemma}

\begin{proof}
  Claim (i) is straightforward. To check  (ii),      set
$ z= \double{a }{b} $. Then
 $\double{b }{a} \stackrel{}{=}
-(-1)^{\vert a\vert_d \vert b\vert_d+\vert z'\vert \vert z''\vert} z''\otimes z'$
by \eqref{flip}
and, modulo $[A,A]$,
$$\langle b,a \rangle= -(-1)^{\vert a\vert_d \vert
b\vert_d +\vert z'\vert \vert z''\vert} z''  z'  \equiv  -(-1)^{\vert a\vert_d
\vert b\vert_d } z'   z'' = -(-1)^{\vert a\vert_d \vert b\vert_d } \langle a,b
\rangle.$$

  To check  (iii), pick any homogeneous     $a,b,c\in A$ and set $x=\double{a }{c} $,
$y=\double{b }{c} $. We have
\begin{eqnarray*}
\double{ab}{c} &=& a* \double{b}{c} + (-1)^{\vert b\vert \vert c \vert_d} \double{a}{c}*b\\
&= &  (-1)^{\vert a \vert \vert y'\vert} y'\otimes ay'' +
(-1)^{  \vert b \vert \vert c   x''\vert_d }x'b\otimes x''
\end{eqnarray*}
so that
$$
\langle a b, c \rangle \ = \ {\double{ab }{c}}'  {\double{ab}{c}}''   =
(-1)^{\vert a\vert \vert y' \vert} y'  a y'' + (-1)^{\vert b\vert \vert c x'' \vert_d} x'  b x''.
$$
Transposing $a$ and $b$, we also obtain
$$ \langle ba, c \rangle = (-1)^{\vert b\vert \vert x' \vert} x'  b
x'' + (-1)^{\vert a\vert \vert c y'' \vert_d} y'  a
 y''  .$$
Since     $\double{- }{-}$ has degree $d$, we have $\vert c x'' \vert_d\equiv \vert ax' \vert \, ({\rm mod}\, 2)$ and
$\vert c y'' \vert_d\equiv \vert by' \vert \, ({\rm mod}\, 2)$. Therefore
$ \langle a b, c \rangle= (-1)^{\vert a\vert \vert b \vert}  \langle ba, c
\rangle$. Hence $\langle [A,A], A \rangle=0$.
This equality together with (ii) imply the inclusion  $\langle A, [A,A]\rangle \subset [A,A]$.

  We now prove (iv).
Set $x=\double{b}{c}$, $y=\double{a}{c}$, $\tilde y= \double{c}{a}$, $z=\double{a}{b}$,
and $\tilde z =\double{b}{a}$.
Then
\begin{eqnarray*}
\double{z'z''}{c}
&= &z' * \double{z''}{c} + (-1)^{\vert z''\vert \vert c \vert_d} \double{z'}{c}*z''\\
&=&(-1)^{\vert z'\vert\, \left\vert\double{z''}{c}' \right\vert } \double{z''}{c}' \otimes z'\double{z''}{c}''\\
&& + (-1)^{\vert z''\vert \, \left\vert c   \double{z'}{c}''  \right\vert_d } \double{z'}{c}'z'' \otimes \double{z'}{c}''.
\end{eqnarray*}
We deduce that
\begin{eqnarray}
\label{(ab)c} \langle \langle a, b\rangle, c\rangle \ = \  \langle z'z'', c\rangle &=&
(-1)^{\vert z'\vert\, \left\vert\double{z''}{c}' \right\vert } \double{z''}{c}'  z'\double{z''}{c}''\\
\notag && + (-1)^{\vert z''\vert \, \left\vert c   \double{z'}{c}''   \right\vert_d } \double{z'}{c}'z''  \double{z'}{c}'' .
\end{eqnarray}
 By  (i),   we  have
\begin{eqnarray}
\label{a(bc)}
\langle a , \langle b,c\rangle \rangle
&=& \langle a , x'x'' \rangle= \langle a ,x'\rangle x'' + (-1)^{\vert a\vert_d \vert x'\vert} x' \langle a ,x''\rangle
\end{eqnarray}
and
\begin{eqnarray}
\label{b(ac)}
\langle b , \langle a, c \rangle\rangle &=& \langle b , y'y'' \rangle=
\langle b ,y'\rangle y'' + (-1)^{\vert b\vert_d \vert y'\vert} y' \langle b ,y''\rangle.
\end{eqnarray}
By the definition of the tribracket \eqref{tribracket},
\begin{eqnarray*}
\triple{a}{b}{c}&=& \double{a}{x'}\otimes x''
+  (-1)^{\vert a \vert_d \vert bc\vert} \Perm_{312} \left (\double{b}{\tilde{y}'} \otimes \tilde{y}''\right) \\
&&+ (-1)^{ \vert ab\vert \vert c \vert_d} \Perm_{312}^2 \left (\double{c}{z'} \otimes z''\right)\\
&=& \double{a}{x'}\otimes x''
+ (-1)^{\vert a \vert_d \vert bc\vert}
\Perm_{312} \left (\double{b}{\tilde{y}'}' \otimes \double{b}{\tilde{y}'}''   \otimes \tilde{y}''\right)\\
&& + (-1)^{ \vert ab\vert \vert c \vert_d} \Perm_{312}^2 \left (\double{c}{z'}'
\otimes \double{c}{z'}''  \otimes z''\right)\\
&=& \double{a}{x'}'\otimes \double{a}{x'}''\otimes x'' \\
&& + (-1)^{\vert a \vert_d \vert bc\vert +  \vert b  \tilde{y}'\vert_d   \vert \tilde{y}'' \vert}
\tilde{y}'' \otimes \double{b}{\tilde{y}'}' \otimes \double{b}{\tilde{y}'}'' \\
&& + (-1)^{ \vert ab\vert \vert c \vert_d+ \left\vert \double{c}{z'}' \right\vert \, \left\vert \double{c}{z'}'' z''\right\vert}
\double{c}{z'}''  \otimes z'' \otimes \double{c}{z'}' \\
&\stackrel{\eqref{flip}}{=}& \double{a}{x'}'\otimes \double{a}{x'}''\otimes x'' \\
&& - (-1)^{\vert a \vert_d \vert bc\vert + \vert b\vert_d \vert y' \vert+ \vert a\vert_d \vert c \vert_d}
y' \otimes \double{b}{y''}' \otimes \double{b}{y''}'' \\
&& - (-1)^{ \vert ab\vert \vert c \vert_d+ \left\vert \double{z'}{c}'' \right\vert \, \vert z''\vert+ \vert c\vert_d \vert z'\vert_d}
\double{z'}{c}'  \otimes z'' \otimes \double{z'}{c}'' .
\end{eqnarray*}
Therefore
\begin{eqnarray}
\notag m\triple{a}{b}{c}&=&
\langle a,x'\rangle x''  - (-1)^{\vert a \vert_d \vert b\vert_d + \vert b\vert_d \vert y' \vert} y'  \langle b,y''\rangle \\
\label{mabc} && - (-1)^{\left\vert c  \double{z'}{c}'' \right\vert_d   \, \vert z''\vert} \double{z'}{c}'   z''  \double{z'}{c}''.
\end{eqnarray}
  Transposing  $a \leftrightarrow b$, we obtain
\begin{eqnarray}
\notag m\triple{b}{a}{c}&=&
\langle b,y'\rangle y''  - (-1)^{\vert a \vert_d \vert b\vert_d + \vert a\vert_d \vert x' \vert} x'  \langle a,x''\rangle \\
\notag && - (-1)^{\left\vert c  \double{\tilde{z}'}{c}''   \right\vert_d \, \vert \tilde{z}''\vert} \double{\tilde{z}'}{c}'   \tilde{z}''  \double{\tilde{z}'}{c}'' \\
\label{mbac} &\stackrel{\eqref{flip}}{=}&
\langle b,y'\rangle y''  - (-1)^{\vert a \vert_d \vert b\vert_d + \vert a\vert_d \vert x' \vert}
x'  \langle a,x''\rangle \\
\notag && +  (-1)^{\left\vert \double{z''}{c}' \right\vert \, \vert z'\vert  + \vert a \vert_d \vert b \vert_d}
\double{z''}{c}'   z'  \double{z''}{c}''.
\end{eqnarray}
  Then (iv)   follows from \eqref{(ab)c}--\eqref{mbac}.
\end{proof}

\subsection{The   trace}\label{hamham1--}

For a  graded algebra~$A$, consider   the  module   $\check A =A/[A,A]$ with the grading induced by that of~$A$. 
 Lemma~\ref{comparis-} implies that  the pairing  $\langle -, - \rangle \colon A \otimes A \to A$  
 associated with  $\double{-}{-}$  induces   a  pairing  $\check A \otimes \check A \to \check  A$.
 The latter pairing is  also denoted by $\langle -, - \rangle$.  It has degree~$d$ and  is 
  $d$-antisymmetric.  If  the induced tribracket of   $\double{-}{-}$ is zero, then $\langle -, - \rangle$ is a $d$-graded Lie bracket.  
 
Note that for  any      $N\geq 1$, the formula $\tr(a)= \sum_{i=1}^N a_{ii}$  defines a  linear map $\tr :A\to A_N$.  
Clearly,  $\tr([A,A])=0$ so that $\tr$ induces a   linear map    $\check A  \to A_N$.
This map is also denoted by~$\tr$ and  is     called the \index{trace} \emph{trace}.
 The graded subalgebra    of $A_N$  generated by    $\tr(\check A) \subset A_N$ is denoted   $A_N^t$ and    is called the \index{trace algebra} \emph{$N$-th trace  algebra} of~$A$.   We have     $A_N^t \subset A_N^{G_N}$ where    $A_N^{G_N}$
  is   the subalgebra of $A_N$ consisting of the elements invariant under the action of $G_N=\GL_N(\kk)$. 
When $A$ is finitely generated as an algebra and  $\kk$ is a field of characteristic zero,  $A_N^t=A_N^{G_N}$, see \cite{LbP}.

\begin{lemma}\label{comparis}
Under the conditions of Lemma~\ref{comparis-},    the map    $\tr:\check A\to A_N$  
carries  the pairing   $\langle -, - \rangle$ in~$\check  A$   into the
bracket $\bracket{-}{-}$ in $A_N$ induced by   $\double{-}{-}$. 
As a consequence,     $\bracket{A_N^t}{A_N^t}\subset A_N^t$ for all $N\geq 1$.  
\end{lemma}

\begin{proof} Pick   any   $a,b\in A$ and let $\check a, \check b$ be their projections to $\check A$.  We have 
\begin{eqnarray*}
\bracket{\tr(\check a)}{\tr(\check b)} &=&  \Big\{ \,\sum_{i}   a_{ii},\sum_{j}   b_{jj} \, \Big\} = \sum_{i,j}  \bracket{a_{ii}}{b_{jj}} \\
&\stackrel{\eqref{db_to_qpb}}{=}&    \sum_{i,j} \double{a}{b}'_{ji} \double{a}{b}''_{ij} =    \sum_{j} \left(\double{a}{b}' \double{a}{b}''\right)_{jj}\\
&=&  \tr \left(\double{a}{b}' \double{a}{b}''\right)
\  {=} \ \tr\big(\langle   a,  b \rangle\big)= \tr\big(\langle \check a, \check b \rangle\big).
\end{eqnarray*}

\vspace{-0.5cm}
\end{proof}

Note that  for $N=1$,  the trace      $\tr:A\to A_1 =\Com(A)$ is  the    canonical    projection and $A^t_1=A_1$.

\chapter {Bibrackets in unital algebras and in categories}\label{Bibrackets in unital algebras and in categories}

\section{Bibrackets in  unital algebras}\label{remremrem000}

  We define  a version  of  representation algebras  in the unital setting.

\subsection{Unital algebras}\label{remremrem}

A graded algebra      $A$ is \index{graded algebra!unital} \emph{unital} if it has a  two-sided unit $1_A\in A^0$.
Unital graded algebras and graded algebra  homomorphisms   carrying~$1$ to~$1$
 form a category  ${\grAlg\!}^+$.  Given  a unital graded algebra   $ A$,  we define a sequence of unital graded algebras $\tilde A_1^+, \tilde A_2^+, \dots$
  For $N\geq 1$,  $\tilde A_N^+$  is obtained   from the algebra $\tilde A_N$ defined in Section \ref{AMT0} as follows.
 First, we adjoin a unit   to $\tilde A_N$, that is consider the unital graded algebra $ \kk e\oplus  \tilde A_N$ with   two-sided   unit   $e$.
 By definition, $\tilde A_N^+$  is  the quotient of  $\kk e\oplus  \tilde A_N$  by   the   relations  $(1_A)_{ij}=\delta_{ij}e$ where $\delta_{ij}$ is
 the Kronecker delta and $i,j$ run over $1,\dots, N$.  For any    $B \in \Ob ({\grAlg\!}^+)$,   the bijection \eqref{eq:adjunction} induces a   natural   bijection
 \begin{equation}\label{eq:adjunctionunital}
\Hom_{ \grAlg\!^+\!\! } \big( \tilde A_N^+ , B\big) \simeq  \Hom_{\grAlg\!^+\!\! } \big(A, \Mat_N(B)\big).
\end{equation}
 Similarly, let  ${\grCom\!}^+$   be  the category   of commutative unital graded algebras  and graded algebra  homomorphisms   carrying $1$ to $1$.
  Set $A_N^+=\Com (\tilde A_N^+)\in \Ob ({\grCom\!}^+)$. Then for  any    $B\in \Ob ({\grCom\!}^+)$, we have a natural bijection
\begin{equation}\label{commeqaunital}
\Hom_{ {\grCom\!}^+ \!\! } \big(   A_N^+ , B\big) \simeq \Hom_{\grAlg\!^{+} \!\! } \big(A, \Mat_N(B)\big) .
\end{equation}
  We call $A_N^+$ the \index{representation algebra!unital} \emph{$N$-th unital representation algebra} of $A$.
From the viewpoint of algebraic geometry,  $A^+_N$ is   the   ``coordinate algebra''   of the   ``affine scheme''
whose set of $B$-points is the set of algebra homomorphisms $A\to \Mat_N(B)$  for any $B\in \Ob ({\grCom\!}^+)$.
 Here an ``affine scheme''  is a representable functor  from ${\grCom\!}^+$ to the category of sets.
The same graded algebra $A_N^+$   can be obtained from $A_N$ by adjoining a two-sided unit $e$
and quotienting by the relations $(1_A)_{ij}=\delta_{ij}e$ where   $i,j$ run over $1,\dots, N$.  For $N=1$, we have $\tilde A_1^+=A$ and $A_1^+=\Com (A)$.

\begin{lemma}\label{importantGerspp} Let $\double{-}{-}$ be a $d$-graded  bibracket  in a unital  graded algebra~$A$
and let     $\bracket{-}{-}$ be the induced bracket in $A_N$,  see   Lemma~\ref{fdbtopb}.
Then there is a unique bracket  $\bracket{-}{-}^+$   in  $A_N^+$  such that the projection $   A_N\to A_N^+$ is bracket-preserving.
If $\double{-}{-}$ is a  Gerstenhaber bibracket of degree $d$, then $\bracket{-}{-}^+$ is a  Gerstenhaber bracket of degree~$d$ in $A^+_N$. 
\end{lemma}

\begin{proof} Denote the projection $   A_N\to A_N^+$ by $p$. Clearly, $p$ is onto which implies the
  uniqueness of $\bracket{-}{-}^+$. To prove the existence, we   extend   $\bracket{-}{-}$ to a bracket $\bracket{-}{-}'$
  in the algebra  $  A'_N   =\kk e\oplus A_N$   by   $\bracket{e}{{ A'_N}}'=\bracket{{ A'_N}}{e}'=0$.
  The latter bracket  is easily checked to satisfy the $d$-graded Leibniz rules  \eqref{poisson1}, \eqref{poisson2}. Therefore,   it suffices  to verify that
  $ (1_A)_{ij}-\delta_{ij}e $ annhilates $\bracket{-}{-}'$ both on the left and on the right for all $i,j$.
  The   Leibniz rule \eqref{bibi+}   for  $\double{-}{-}$   implies   that  $\double{1_A}{A} =0$. Therefore for any $b\in A$ and $u,v\in \{1,\dots,N\}$,
$$\bracket{(1_A)_{ij}-\delta_{ij}e}{b_{uv}}'=\bracket{(1_A)_{ij} }{b_{uv}} - \delta_{ij} \bracket{ e}{b_{uv}}'=0.$$
Since   $ A'_N$ is generated by the set   $\{b_{uv}\, \vert\, b,u,v\}$,     the $d$-graded Leibniz rules  \eqref{poisson1}, \eqref{poisson2} imply  that
$\bracket{(1_A)_{ij}-\delta_{ij}e}{ A'_N}'=0$. Similarly, $\bracket{ A'_N}{(1_A)_{ij}-\delta_{ij}e}'=0$.
The last claim of the lemma follows from Lemma~\ref{importantGers}.
\end{proof}

The constructions and results given for $ \tilde A_N$  and $A_N$ in Sections \ref{twoactionsAN+++}  and \ref{hamham1--}
 easily extend to  $ \tilde A_N^+$   and $A_N^+$.

\subsection{The case of universal enveloping algebras}\label{AMT--1}

A rich source of unital   algebras is the theory of Lie algebras since  their universal enveloping algebras are  unital. In the graded setting 
one starts with  a $0$-graded Lie algebra  $L=(L,\{-,-\})$  as   in Section~\ref{gral++++}.
The \index{universal enveloping algebra} {\it universal enveloping algebra} $U(L)$ of~$L$  is the quotient of the
graded tensor algebra  $\oplus_{n\geq 0} \, L^{\otimes n}$ by the 2-sided ideal generated by the
vectors $$a\otimes b- (-1)^{\vert a \vert \vert b\vert} b\otimes a-\{a,b\}$$
 where $a,b$ run over all  homogeneous elements of $L$. The graded tensor algebra   is unital and so is $U(L)$.
 
  For any unital graded algebra~$ V $, the composition with the natural linear map $L\to U(L)$ determines a bijection
\begin{equation} \label{repalgLieone}
\Hom_{ \grAlg\!^+ } ( U(L) ,  V ) \simeq  \Hom_{ \Lie } (L,  V )
 \end{equation}
 where $\Lie$ is the category of $0$-graded Lie algebras
and, on the right hand-side,   $V $ is viewed as a graded Lie algebra with the commutator bracket.
 Section~\ref{remremrem}   yields  for  each  $N\geq 1$, a
commutative unital graded algebra  $ L_N=  (U(L))_N^+$.
 By \eqref{commeqaunital} and~\eqref{repalgLieone}, for any   $B\in \Ob({\grCom\!}^+)$,   we have a natural bijection
\begin{equation}\label{repalgLie}
\Hom_{ {\grCom\!}^+} (    L_N    , B) \simeq  \Hom_{  \Lie  } (L, \Mat_N(B)).
\end{equation}
  Note that $L_N$    is generated by the commuting symbols $a_{ij}$ where
$a$ runs over homogeneous elements of $ L$ and $i,j $ run over $  1, \ldots, N $, subject to the first two of the relations \eqref{eq:addid} and
the relation  $ \{a,b\}_{ij}=a_{il} b_{lj}- (-1)^{\vert a \vert \vert b\vert}b_{il} a_{lj}$ for all homogeneous  $a,b\in L$ and all $i,j $.
 Lemma~\ref{importantGerspp}  shows how to obtain  a bracket   in    $L_N$  from a bibracket in $U(L)$.

\section{Bibrackets in categories}\label{extenscatmain}

We   define representation algebras and bibrackets for  graded categories.
We   follow   Van den Bergh \cite[Section 7]{VdB} who did it for  non-graded  categories with finite sets of objects.

\subsection{Graded categories and   associated algebras}\label{extenscat}

 A  \index{graded category} \emph{graded category} is a small category~$\calC$
 such that for any objects $X,Y$ of $\calC$,
the set $\Hom_\calC (X,Y)$ is a graded module, the identity morphisms of all objects are homogeneous of degree zero,  and
the composition of morphisms is bilinear and
 degree-additive. The  latter condition means that for any homogeneous $f \in \Hom_\calC (X,Y)$
and $g\in \Hom_\calC(Y,Z)$, the morphism $g\circ f :X \to Z $ is homogeneous of degree $ \vert f\vert + \vert g\vert$.

With a graded category $\calC$ we associate a graded  algebra
$$
A= A(\calC)=  \bigoplus_{X,Y\in \Ob(\calC)}  \Hom_\calC (X,Y)
$$
where  $\oplus$ is the direct sum   of graded modules. The product   $fg\in A $ of   $f \in \Hom_\calC (X,Y)$ and
  $g \in \Hom_\calC (U,Z)$ is equal to $g\circ f  $ if $Y=U$
and to zero otherwise. For   $X \in \Ob(\calC)$, the   identity morphism of $X$   represents an element of $A $ denoted~$e_X$.
Clearly, $e_X e_X=e_X$ and  $e_X e_Y=0$ for $X\neq Y$. If   the set $\Ob(\calC)$ is  finite, then
  $1_A= \sum_{X \in \Ob(\calC)} e_X$ is a two-sided unit of $A$;   if the set $\Ob(\calC)$ is  infinite, then   $A $ is not unital.

For each integer $N\geq 1$, we introduce a unital graded algebra ${\tilde {\calC}}_N^+ $.
Consider  the unital graded algebra $ \kk e\oplus  \tilde A_N$ obtained by adjoining the   two-sided   unit~$e$
to the graded algebra $\tilde A_N$   associated with $A= A(\calC)$ in Section~\ref{AMT0}. Let ${\tilde {\calC}}_N^+ $
be the quotient of $ \kk e\oplus  \tilde A_N$ by the 2-sided ideal generated by the set $\{(   e_X  )_{ij}-\delta_{ij} e\}_{X,i,j}$
where  $X$ runs over all objects of $\calC$ and $i,j\in \{1,   \dots  , N\}$. The algebra ${\tilde {\calC}}_N^+ $ has the following  universal  property.
For each unital graded algebra $B$,   we  consider the  algebra $\Mat_N(B)$ of  $(N\times N)$-matrices over $B$ as a   category with a single object. 
This category is graded:   a   matrix   is homogeneous of degree~$p$  if  all its   entries belong to  $B^p\subset B$. 
There is  a natural bijection
$$
 \Hom_{{\grAlg\!}^+\!\! } \big({\tilde {\calC}}_N^+ , B\big)
 \stackrel{\simeq}{\longrightarrow} {\Fun}  ({\calC}, \Mat_N(B))
$$
where ${\grAlg\!}^+$ is the category of unital graded algebras and   ${\Fun}  ({\calC}, \Mat_N(B))$ is
the set of degree-preserving linear functors   $\calC\to \Mat_N(B) $. Note that such    functors    can be interpreted as   $N$-dimensional   $B$-representations of~$\calC$.

The  commutative unital graded algebra ${\calC}_N^+=\Com ({\tilde {\calC}}_N^+ )$ plays a similar role
in the category   ${\grCom\!}^+$  of commutative unital graded algebras:  for any  $B\in \Ob(\grCom^+)$, there  is a natural bijection
\begin{equation} \label{universality}
   \Hom_{{\grCom\!}^+\!\! } \big({  {\calC}}_N^+ , B\big)
 \stackrel{\simeq}{\longrightarrow} {\Fun}  ({\calC}, \Mat_N(B)) .
\end{equation}

\subsection{Double Gerstenhaber categories}\label{doublecat}

  Let $d$ be an integer.
A  \index{bibracket} \emph{$d$-graded bibracket} in a graded category $\calC$
is a $d$-graded bibracket $\double{-}{-}$ in the   graded algebra $ A=A(\calC)$
such that $\double{A}{e_X}=\double{e_X}{A}=0$ for all $X \in \Ob(\calC)$.
If such a bibracket   in $A $ is a  Gerstenhaber bibracket   of degree $d$,
then the pair $(\calC, \double{-}{-})$ is called  a  \index{double Gerstenhaber category} \emph{double Gerstenhaber category of degree $d$}.

\begin{lemma}\label{bracketcategories}
Let  $\double{-}{-}$ be a  $d$-graded  bibracket in a graded category  $\calC$.  Then
\begin{itemize}
\item for  any    $X,Y,   U, V     \in \Ob(\calC)$,
$$\double{\Hom_\calC (X,Y)}{\Hom_\calC (U, V  )} \subset \Hom_\calC (U,Y)\otimes \Hom_\calC (X,  V  );$$
\item for any integer $N\geq 1$,  the bracket  in   $A_N$   determined by   Lemma~\ref{fdbtopb} induces a bracket $\bracket{-}{-}$ in  ${  {\calC}}_N^+$
satisfying  the  Leibniz rules  \eqref{poisson1}, \eqref{poisson2};
\item if   $(\calC, \double{-}{-})$ is a double Gerstenhaber category of degree $d$, then the pair
  $({  {\calC}}_N^+, \bracket{-}{-})$   is a    unital    Gerstenhaber  algebra of degree~$d$ for all~$N\geq 1$.
\end{itemize}
\end{lemma}

\begin{proof} Using the identity $\double{A}{e_X}=\double{e_X}{A}=0$ and the Leibniz rules  for~$\double{-}{-}$,  
we obtain that for any $f\in \Hom_\calC (X,Y)$, $g\in \Hom_\calC (U,V)$,
\begin{eqnarray*}
\double{f}{g} &=&  \double{e_X f e_Y}{e_U g e_V} \\
&=&    e_U\,  \double{e_X f e_Y}{ g } \, e_V \\
&=&      e_U\, (e_X \ast \double{f }{ g } \ast e_Y ) \, e_V\\
&=&      e_U\, \left(e_X \ast (\double{f }{ g }'\otimes  \double{f }{ g }'') \ast e_Y \right) \, e_V\\
&=&   e_U \double{f }{ g }'e_Y \otimes e_X  \double{f }{ g }''e_V \ \in \Hom_\calC (U,Y)\otimes \Hom_\calC (X,V).
\end{eqnarray*}
Other claims of the lemma   follow from the definitions and Lemma~\ref{importantGers}.
\end{proof}

We conclude that a  double Gerstenhaber category $(\calC, \double{-}{-})$ of degree $d$
gives rise to a system of  unital  Gerstenhaber  algebras   $\{\calC_N^+\}_{N\geq 1}$   of degree $d$.
Moreover, for any full subcategory $ {\calC}' $ of ${\calC}$, the algebra $ A'  =A( {\calC}' )$
may be viewed as a subalgebra of $A=A({\calC})$ in the obvious way.
The first claim of Lemma~\ref{bracketcategories} implies that the    bibracket  $\double{-}{-}$  in $A$ restricts to a    bibracket in~$ A' $.
In this way, $ \calC' $ becomes a double Gerstenhaber category  of degree $d$.
In particular, any object $X$ of ${\calC}$ determines a full subcategory   ${\calC}_X$   of ${\calC}$ consisting of $X$ and all its endomorphisms.
  Then  the restriction of  $\double{-}{-}$  to  the unital graded algebra $A_X=A({\calC}_X) =  \End_{{\calC}} (X)$  
  is a Gerstenhaber bibracket of degree $d$, and   we have  $(\calC_X)^+_N=(A_X)^+_N$.
  
   \subsection{Remark} 
   
 In analogy with  non-unital  algebras,   one can consider  \lq\lq categories without identity morphisms".
However, such generalized categories do not appear in our geometric context   and we do not study them.

\section{Bibrackets in Hopf categories}

We   define    Hopf categories and we introduce a class of bibrackets in Hopf categories 
  called reducible bibrackets.

\subsection{Hopf categories}\label{prop11}   

Consider a   graded category~$\mathscr{C}$ and   the associated graded algebra $A=A(\mathscr{C})$, see  Section~\ref{extenscat}.
 For   $X  \in \Ob(\mathscr{C})$,   we  let $e_X\in A^0\subset A$ be the element   represented by the  identity morphism  of $X$.
  We view   $A \otimes  A$ as an algebra with multiplication  defined  by
  $$(a_1 \otimes a_2)(b_1 \otimes b_2)= (-1)^{\vert a_2 \vert\, \vert b_1 \vert} a_1b_1 \otimes a_2b_2$$
  for any homogeneous $a_1, a_2, b_1, b_2 \in A$.  
 A \index{Hopf category!comultiplication of} \emph{comultiplication}  in~$\mathscr{C}$
  is a   degree-preserving    algebra   homomorphism   $\Delta:A \to A\otimes A$    such that
    $$(\Delta \otimes \id_A) \Delta= (\id_A \otimes \Delta) \Delta, $$
 and   $\Delta (e_X)=e_X \otimes e_X$ for all $X\in \Ob (\mathscr{C})$.
As a consequence,    ${{\Delta}}$ must carry  $H=\Hom_{\mathscr{C}} (X,Y)  \subset A$ to $H\otimes H$   for any
 objects $X, Y$ of~$\mathscr{C}$.  The image of any ${a}\in H$ under ${{\Delta}}$ expands
(non-uniquely) as   a   sum $  \sum_i a^{(1)}_i \otimes
a^{(2)}_i$ where   $i$ runs over a finite set and $ a^{(1)}_i,
 a^{(2)}_i $ are homogeneous elements of~$H$. We  use
 Sweedler's notation, i.e.,
drop  the index $i$ and the summation sign and write simply
${{\Delta}}({a})=  a^{(1)} \otimes a^{(2)}$.
  The condition that   $\Delta $ is  degree-preserving means that $\vert a^{(1)} \vert +\vert a^{(2)} \vert =\vert a \vert$ for any homogeneous $a\in  A$.
  That $\Delta$ is an algebra homomorphism means the identity 
 $$
 \Delta(ab) = (-1)^{\vert  a^{(2)} \vert\, \vert b^{(1)} \vert} a^{(1)}b^{(1)} \otimes a^{(2)}b^{(2)}
 $$
 for any homogeneous   $a,b\in A$.

  An \index{augmentation} \emph{augmentation} of $\mathscr{C}$ is a      linear map    $\varepsilon :A \to \kk$  
  carrying  the  identity morphisms of all objects to~$1$, carrying $A^p$ to~$0$ 
    for all $p \neq 0$,     and satisfying   $\varepsilon(fg) = \varepsilon(f)\,  \varepsilon(g)$
     for any morphisms  $f, g$ in ${\mathscr{C}}$  with    $\hbox{target}(f)=\hbox{source}(g)$. 
 A \index{Hopf category!counit of} \emph{counit} for   a comultiplication  $\Delta:A\to A\otimes A$   is  an augmentation   $\varepsilon :A \to \kk$   of~$\mathscr{C}$
   such that
  $$   (\id_A \otimes \varepsilon) \Delta =\id_A = (\varepsilon  \otimes \id_A) \Delta \colon A \to A .$$
  Clearly,  if~$\varepsilon$ is a counit of~$\Delta $, then  $\Delta$ is a split injection with left inverses $\id_A \otimes \varepsilon$ and $\varepsilon  \otimes \id_A$.
  Also,  $\varepsilon$ induces  linear maps  $\varepsilonin, \varepsilonout \colon  A \to A$ such that  
 $$
 \varepsilonin(a) =\varepsilon(a) e_X \quad {\rm {and}} \quad   \varepsilonout(a) =\varepsilon(a) e_Y,
 $$
for all  $X, Y\in \Ob (\mathscr{C})$ and  $a\in \Hom_{\mathscr{C}} (X,Y) $.
 An  \index{Hopf category!antipode of}  \emph{antipode} in~$\mathscr{C}$ is a  degree-preserving  linear map   $s :A \to A$
 carrying  $\Hom_{\mathscr{C}}(X,Y)$ to $\Hom_{\mathscr{C}}(Y,X)$  for any $X,Y \in \Ob(\mathscr{C})$  and satisfying
 $$
 a^{(1)} s(a^{(2)})=    \varepsilonin(a),  \qquad s(a^{(1)} ) \, a^{(2)} =   \varepsilonout(a),
 $$
  for all  $a \in A$. It  \  follows immediately   that $s(e_X)= e_X$ for any $X\in \Ob(\mathscr{C})$.

   A graded category~$ \mathscr{C}$ endowed with a comultiplication  $\Delta$, a counit~$\varepsilon$,
   and an antipode  $s$  is called a \index{Hopf category} \emph{Hopf category}.
  When  $\mathscr{C}$ has a single object, we recover the usual notion of  a  \index{graded algebra!Hopf}   \emph{graded Hopf algebra}.
    A Hopf category  $ (\mathscr{C}, \Delta, \varepsilon, s)$  is  \index{Hopf category!cocommutative} \emph{cocommutative}
 if    $\Delta=\Perm_{21} \Delta$ and  is  \index{Hopf category!involutive}  \emph{involutive} if $s$ is an involution.

Basic  properties of Hopf algebras  (see, for instance, \cite[Theorem III.3.4]{Ka}) generalize to Hopf categories. We state the properties  used in the sequel.

\begin{lemma} \label{properties_antipode}
The antipode $s$ of a Hopf category $\mathscr{C}$ is an antiendomorphism of the underlying algebra of $A= A(\mathscr{C})$ in the sense that, for any homogeneous $a,b\in A$,
$$
s(ab)= (-1)^{\vert a \vert \vert b \vert} s(b) s(a).
$$
Also, $s$ is an antiendomorphism of the underlying coalgebra of $A$ in the sense that, for any $a\in A$,
$$
\varepsilon(s(a)) =\varepsilon(a) \quad \hbox{and} \quad (s(a))^{(1)} \otimes (s(a))^{(2)} = (-1)^{\vert a^{(1)} \vert \vert a^{(2)} \vert } s(a^{(2)}) \otimes s(a^{(1)}).
$$
Finally, the cocommutativity of $\mathscr{C}$  implies its involutivity, and the latter is equivalent to any of the following  two properties:
\begin{itemize}
\item[(i)] for all $a\in A$, $(-1)^{\vert a^{(1)} \vert \vert a^{(2)} \vert } s(a^{(2)}) a^{(1)} = \varepsilonout(a)$;
\item[(ii)] for all $a\in A$, $(-1)^{\vert a^{(1)} \vert \vert a^{(2)} \vert } a^{(2)} s(a^{(1)}) = \varepsilonin(a)$.
\end{itemize}
\end{lemma}

\begin{proof}
In the proof we will use the following notation. Recall that the algebra~$A$ is linearly generated by morphisms in $\mathscr{C}$.
  Given two expressions  linearly depending on one or several elements $a,b,  \ldots$ of $A$, we relate these expressions by the symbol~$\stackrel{\cdots}{=} $
  if they are   equal  for all $a,b,  \ldots$ and this equality follows from the axioms of a Hopf category  whenever  $a,b, \ldots$ are morphisms in~$\mathscr{C}$.

Let $C$ be the module  of degree-preserving linear maps  $A\otimes A \to A$. Note that the comultiplication in $A$ induces a degree-preserving coassociative comultiplication in $A \otimes A$ carrying $a\otimes b $  with $a,b \in A$ to
$$(-1)^{\vert b^{(1)}\vert \vert a^{(2)} \vert } \big(a^{(1)} \otimes b^{(1)}\big)\, \otimes \, \big(a^{(2)}\otimes b^{(2)}\big)\, .
$$
  This comultiplication induces  the \index{convolution product} \emph{convolution product}  $\ast$ in~$C$ by
$$
(f \ast g) (a \otimes b)
= (-1)^{\vert b^{(1)}\vert \vert a^{(2)} \vert } f\big(a^{(1)} \otimes b^{(1)}\big)\, g\big(a^{(2)}\otimes b^{(2)}\big)
$$
for any $f,g \in C$ and any $a,b\in A$.  We define elements   $l,r  $  of $ C$  by
 $l(a\otimes b) = s(ab)$ and $ r(a \otimes b) = (-1)^{\vert a \vert \vert b \vert} s(b) s(a)$ for any homogeneous $a,b\in A$.
 To prove the first claim of the lemma we must show that $l=r$. To this end we define     $m,u,v \in C$   by
$$ m(a \otimes b)=ab, \quad
u(a \otimes b)= \varepsilonout(ab), \quad
v(a \otimes b)=  \varepsilonin(ab)
$$ for any $a,b\in A$.
 Observe that
 \begin{eqnarray*}
(u \ast r) (a \otimes b) &=&  (-1)^{\vert b^{(1)}\vert \vert a^{(2)} \vert  + \vert b^{(2)}\vert \vert a^{(2)} \vert   }  \varepsilonout(a^{(1)}b^{(1)})\, \big(s(b^{(2)}) s(a^{(2)})\big) \\
& \stackrel{\cdots}{=}  &  (-1)^{\vert b\vert \vert a^{(2)} \vert    }  \varepsilon(a^{(1)}b^{(1)})\, s(b^{(2)}) s(a^{(2)}) \\
& \stackrel{\cdots}{=}  &  (-1)^{\vert b\vert \vert a^{(2)} \vert    }  \varepsilon(a^{(1)}) \varepsilon(b^{(1)})\, s(b^{(2)}) s(a^{(2)}) \\
&=&  (-1)^{\vert b\vert \vert a \vert    }   s\big( \varepsilon(b^{(1)}) b^{(2)}\big) s\big(\varepsilon(a^{(1)}) a^{(2)}\big) \ = \ r(a\otimes b)
\end{eqnarray*}
and
\begin{eqnarray*}
(l \ast m) (a \otimes b) &=& (-1)^{\vert b^{(1)}\vert \vert a^{(2)} \vert } s\big(a^{(1)} b^{(1)}\big)\, \big(a^{(2)} b^{(2)}\big) \\
 &=& s\big((a b)^{(1)}\big)\, \big((a b)^{(2)}\big)  \ = \  \varepsilonout( ab) \ = \ u (a\otimes b).
\end{eqnarray*}
Furthermore,
\begin{eqnarray*}
(m \ast r) (a\otimes b) &=&   (-1)^{\vert b^{(1)}\vert \vert a^{(2)} \vert  + \vert b^{(2)}\vert \vert a^{(2)} \vert} \big(a^{(1)}b^{(1)}\big)\, \big( s(b^{(2)}) s(a^{(2)})\big) \\
 &=&   (-1)^{\vert b \vert \vert a^{(2)} \vert  } a^{(1)} \big(b^{(1)}  s(b^{(2)})\big) s(a^{(2)}) \\
 &=&   (-1)^{\vert b \vert \vert a^{(2)} \vert  } a^{(1)} \varepsilonin (b) s(a^{(2)}) \\
 & \stackrel{\cdots}{=} & a^{(1)}  \varepsilonin (b) s(a^{(2)}) \\
 & = & a^{(1)}  \varepsilonin (b) s(a^{(2)}) \varepsilon(a^{(3)}) \\
 &  \stackrel{\cdots}{=} & a^{(1)}  s(a^{(2)}) \varepsilon(a^{(3)}b) \\
 &=& \varepsilonin (a^{(1)})  \varepsilon(a^{(2)}b) \ \stackrel{\cdots}{=}  \  \varepsilonin(ab) \ = \ v(a\otimes b).
\end{eqnarray*}
and
\begin{eqnarray*}
(l \ast v) ( a \otimes b) & = &  (-1)^{\vert b^{(1)}\vert \vert a^{(2)} \vert } s(a^{(1)}b^{(1)})\, \varepsilonin(a^{(2)}b^{(2)}) \\
&  \stackrel{\cdots}{=} &  (-1)^{\vert b^{(1)}\vert \vert a^{(2)} \vert } s(a^{(1)}b^{(1)})\, \varepsilon(a^{(2)}b^{(2)}) \\
&  \stackrel{\cdots}{=} &  (-1)^{\vert b^{(1)}\vert \vert a^{(2)} \vert } s(a^{(1)}b^{(1)})\, \varepsilon(a^{(2)}) \varepsilon(b^{(2)}) \\
& = &s\big(a^{(1)}\varepsilon(a^{(2)})\, b^{(1)}\varepsilon(b^{(2)})\big) \ = \ l(a\otimes b).
\end{eqnarray*}
Since $\ast$ is an associative operation, we deduce that
 $$
l = l \ast v = l \ast m \ast r = u \ast r = r.
$$

We now verify that $s$ is an antiendomorphism of the unital coalgebra $(A,\Delta,\varepsilon)$.
For this, we  consider the module $D$  of degree-preserving linear maps $A \to A\otimes A$, and we equip it with the  \index{convolution product} \emph{convolution product}  defined by
$$
(f \ast g) (a) = f(a^{(1)})\, g(a^{(2)})
$$
for any $f,g \in  D$ and any $a\in A$. Let  $l,r,u,v \in   D $ be  defined by
$$
l = \Delta s , \quad r =  (s \otimes s) \Perm_{21} \Delta , \quad u = \Delta \varepsilonout, \quad v = \Delta \varepsilonin.
$$
We must prove that $l=r$.
Observe that, for any $a\in A$,
\begin{eqnarray*}
( u \ast r)(a)  &= &  \Delta\varepsilonout(a^{(1)})\, \big((s \otimes s) \Perm_{21} \Delta(a^{(2)})\big) \\
&\stackrel{\cdots}{=} &  (-1)^{\vert a^{(3)} \vert \vert a^{(4)} \vert} \big( \varepsilonout(a^{(1)}) \otimes  \varepsilonout(a^{(2)})\big) \, \big(s(a^{(4)}) \otimes s(a^{(3)})\big)  \\
&=&  (-1)^{ \vert a^{(3)} \vert \vert a^{(4)} \vert}   \varepsilonout(a^{(1)}) s(a^{(4)}) \otimes  \varepsilonout(a^{(2)})  s(a^{(3)}) \\
&\stackrel{\cdots}{=} & (-1)^{ \vert a^{(2)}a^{(3)} \vert \vert a^{(4)} \vert}   \varepsilon(a^{(1)}) s(a^{(4)}) \otimes  \varepsilon(a^{(2)})  s(a^{(3)})  \\
&=&  (-1)^{ \vert a^{(2)} \vert \vert a^{(3)} \vert}   \varepsilon(a^{(1)}) s(a^{(3)}) \otimes  s(a^{(2)})   \\
&=&  (-1)^{\vert a^{(1)} a^{(2)}\vert \vert a^{(3)}\vert} s(a^{(3)} )  \otimes s\big(\varepsilon(a^{(1)}) a^{(2)}\big)  \\
&=&  (-1)^{\vert a^{(1)} \vert \vert a^{(2)} \vert} s(a^{(2)})  \otimes s(a^{(1)})  \ = \ r(a).
\end{eqnarray*}
and
$$
 ( l \ast \Delta) (a)  =  \Delta(s(a^{(1)}))\, \Delta(a^{(2)})
 =  \Delta \big( s(a^{(1)})\, a^{(2)} \big)   =  u(a).
$$
Furthermore,
\begin{eqnarray*}
(\Delta \ast r )(a) &= &  \Delta(a^{(1)})\, \big((s \otimes s) \Perm_{21} \Delta(a^{(2)})\big)  \\
&=& (-1)^{\vert a^{(3)} \vert \vert a^{(4)} \vert} \big(a^{(1)} \otimes a^{(2)}\big) \, \big(s(a^{(4)}) \otimes s(a^{(3)})\big)   \\
&=& (-1)^{\vert a^{(2)} a^{(3)} \vert \vert a^{(4)} \vert}  a^{(1)} s(a^{(4)}) \otimes  a^{(2)  } s(a^{(3)}) \\
&=& (-1)^{\vert a^{(2)}  \vert \vert a^{(3)} \vert}  a^{(1)} s(a^{(3)}) \otimes  \varepsilonin(a^{(2)}) \\
&\stackrel{\cdots}{=} &  a^{(1)} s(a^{(3)}) \otimes  \varepsilonin(a^{(2)})  \  \stackrel{\cdots}{=} \  v(a)
\end{eqnarray*}
and
\begin{eqnarray*}
(l \ast v)(a)  &=& \Delta s(a^{(1)}) \,    \Delta \varepsilonin(a^{(2)})\\
&=& \Delta \big(s(a^{(1)})\varepsilonin(a^{(2)})   \big)  \\
&\stackrel{\cdots}{=} &  \Delta \big(s(a^{(1)})\varepsilon (a^{(2)})   \big)  \ = \  \Delta \big(s\big(a^{(1)}\varepsilon(a^{(2)}\big))  \big) \ = l(a)
\end{eqnarray*}
Using the associativity of  $\ast$, we deduce that  $$l=l \ast v=l \ast \Delta \ast r =u \ast r=  r.$$
Also,~$s$  preverves the counit: for any $a\in A$, we have
\begin{eqnarray*}
\varepsilon(a)  \ \stackrel{\cdots}{=} \   \varepsilon(\varepsilonin(a)) & = &  \varepsilon\big(a^{(1)} s(a^{(2)})\big) \\
& \stackrel{\cdots}{=}&  \varepsilon(a^{(1)}) \varepsilon\big( s(a^{(2)})\big)
\ = \  \varepsilon\big( s\big(\varepsilon(a^{(1)})  a^{(2)}\big)\big)  \ = \ \varepsilon(s(a)).
\end{eqnarray*}

We now prove the part of the lemma concerning the involutivity. If $s^2=\id_A$, then \ the condition (i) is satisfied:
$$
(-1)^{\vert a^{(1)} \vert \vert a^{(2)} \vert } s(a^{(2)}) a^{(1)} = s \big(s(a^{(1)}) a^{(2)}\big) = s \big(\varepsilonout(a)\big) = \varepsilonout(a).
$$
Assume now  that the condition   (i) is met and consider the convolution product $\ast$
in the   module,  $E$,   of degree-preserving linear maps $A \to A$.
For any $a \in A$,
\begin{eqnarray*}
 (s\ast s^2)(a) = s(a^{(1)})\,  s\big(s(a^{(2)})\big) & = &  (-1)^{\vert a^{(1)} \vert \vert a^{(2)} \vert } s\big(s(a ^{(2)})  a^{(1)}  \big)  \\
&  {=}&  s( \varepsilonout(a) ) = \varepsilonout(a).
\end{eqnarray*}
  Thus, $s\ast s^2 =\varepsilonout$.  It follows from the axioms of a Hopf category that
$\id_A \ast s = \varepsilonin$ and $  \varepsilonin \ast f = f =f \ast \varepsilonout $
for each $f \in  E $  carrying the set  $\Hom_{\mathscr{C}}(X,Y)  $ into itself   for all  $X,Y \in \Ob(\mathscr{C})$.
Applying this to $f=s^2$ and to $f=\id_A$ and  using the   associativity of $\ast$, we  obtain  $$s^2=\varepsilonin \ast s^2 = \id_A \ast s \ast s^2= \id_A \ast\, \varepsilonout = \id_A.$$
This shows the equivalence between the involutivity and (i); the equivalence with (ii) is proved similarly.
Finally, if   $\mathscr{C}$  is cocommutative, then the identity $s \ast  \id_A   =\varepsilonout$ implies (i), so that $A$ is involutive.
\end{proof}

\subsection{Bibrackets re-examined}\label{prop2}

 Bibrackets in    a   Hopf category~$ ( \mathscr{C}, \Delta, \varepsilon, s)$ have   a useful  reformulation which we now describe.  Consider the associated graded algebra $A=A(\mathscr{C})$
  and   a $d$-graded bibracket $\double{-}{-} : A\otimes A  \to A\otimes A$  with $d\in \ZZ$.
We  define  a linear  map $ \Lambda= \Lambda(\double{-}{-})   :A \otimes A \to A\otimes A$   by
 \begin{equation} \label{Lambda}
\Lambda(a, b)  =   a^{(1)} s\Big( \double{a^{(2)}}{b^{(1)}}'  \Big) \otimes  \double{a^{(2)}}{b^{(1)}}''  s(b^{(2)})
\end{equation}
for  any  $a,b \in A$.  Note that, if    $a\in \Hom_{\mathscr{C}} (X,Y)$ and $b\in   \Hom_{\mathscr{C}} (U,V)$ with $X,Y,U,V \in \Ob (\mathscr{C})$,
then  $ \Lambda(a, b) \in \Hom_{\mathscr{C}} (X,U) \otimes  \Hom_{\mathscr{C}} (X,U).$

\begin{lemma} \label{F}
For any   $a  \in A$ and  any homogeneous $b,c\in A$, we have
\begin{eqnarray*}
\Lambda(a,bc)  &=&    \Lambda\big(a,b^{(1)}\big)\, \varepsilon\big(b^{(2)}c\big)   +    (-1)^{   \vert b  \vert\, \vert c\vert} \, \Lambda(a,c)  \, (s  \otimes  s)\!\big(\Delta (b) \big),\\
\Lambda(ab, c)  &=&  \Lambda(a^{(1)},c)\, \varepsilon (a^{(2)}b)  +   \Delta(a)   \, \Lambda(b,c).
\end{eqnarray*}
\end{lemma}

 \begin{proof}
Since both sides of the first identity are  linear in $b$ and $c$,  it suffices to consider  the case where
$b\in \Hom_{\mathscr{C}}(U,V)$ and $c \in \Hom_{\mathscr{C}}(W,Z)$ for some objects $U,V,W,Z$ of $\mathscr{C}$.
If $V\neq W$, then both sides of the identity are equal to zero. If  $V=W$,
then   $\Lambda\big(a,b^{(1)}\big)\, \varepsilon\big(b^{(2)}c\big) = \Lambda(a,b) \, \varepsilon(c)$ and
 \begin{eqnarray*}
 \Lambda(a,bc)  &=&   a^{(1)}   s\Big( \double{a^{(2)}}{(bc)^{(1)}}' \Big) \otimes  \double{a^{(2)}}{(bc)^{(1)}}''  s((bc)^{(2)})    \\
 &=& (-1)^{  \vert b^{(2)} \vert\,  \vert c\vert}  a^{(1)} s\Big( \double{a^{(2)}}{b^{(1)} c^{(1)} }' \Big) \otimes  \double{a^{(2)}}{b^{(1)} c^{(1)}}''  s(c^{(2)})\,  s(b^{(2)})\\
 &=& (-1)^{  \vert b^{(2)} \vert\,  \vert c\vert} a^{(1)} s\Big( \double{a^{(2)}}{b^{(1)}  }' \Big) \otimes  \double{a^{(2)}}{b^{(1)}  }''  c^{(1)} s(c^{(2)}) \, s(b^{(2)}) \\
 && +    \eta_1 a^{(1)} s\Big( \double{a^{(2)}}{ c^{(1)} }' \Big) \, s(b^{(1)}) \otimes  \double{a^{(2)}}{  c^{(1)}}''  s(c^{(2)}) \, s(b^{(2)}) \\
 &=&   \varepsilon(c) \, \Lambda(a,b)    +    \eta_2 \, \Lambda(a,c)  \, \big(s(b^{(1)}) \otimes  s(b^{(2)}) \big)
\end{eqnarray*}
where the signs $\eta_1,\eta_2 = \pm 1$ are  computed  by
\begin{eqnarray*}
\eta_1  &=&  (-1)^{ \vert b^{(2)} \vert\,  \vert c\vert+ \vert b^{(1)} \vert\, \left\vert    \double{a^{(2)}}{c^{(1)}}'   \right\vert  +
 \vert b^{(1)} \vert\, \vert    {a^{(2)}} \vert_d  }  ,\\
 \eta_2 &=& \eta_1 \cdot   (-1)^{ \vert b^{(1)} \vert\, \left\vert    \double{a^{(2)}}{c^{(1)}}'' c^{(2)}  \right\vert   } \ = \ (-1)^{   \vert b  \vert\,  \vert c\vert}.
\end{eqnarray*}

The second identity is proved similarly with the key case being  the one where $a,b$ are morphisms in $\mathscr{C}$ and the target  object of $a$ coincides with the source object of~$b$. Then
  $\Lambda\big(a^{(1)},c\big)\, \varepsilon\big(a^{(2)}b\big) = \Lambda(a,c)\,  \varepsilon(b)$ and
\begin{eqnarray*}
 \Lambda(ab, c)  &=&   (ab)^{(1)}   s\Big( \double{(ab)^{(2)}}{c^{(1)}}' \Big)   \otimes  \double{(ab)^{(2)}}{c^{(1)}}''  s(c^{(2)})  \\
 &=& (-1)^{  \vert a^{(2)} \vert\,  \vert b^{(1)}\vert}  \, a^{(1)} b^{(1)}   s\Big( \double{a^{(2)}b^{(2)}}{c^{(1)}}' \Big) \otimes  \double{a^{(2)}b^{(2)}}{c^{(1)}}''  s(c^{(2)})\\
 &=&\theta_1  a^{(1)} b^{(1)}   s\Big( \double{b^{(2)}}{c^{(1)}}' \Big) \otimes a^{(2)} \double{ b^{(2)}}{c^{(1)}}''  s(c^{(2)}) \\
 && +     \theta_2  \, a^{(1)} b^{(1)}   s\Big( \double{a^{(2)}}{c^{(1)}}' b^{(2)}\Big) \otimes  \double{a^{(2)} }{c^{(1)}}''  s(c^{(2)})\\
 &=& \big(  a^{(1)} \otimes a^{(2)}\big)\, \Lambda(b,c)  \\
 && +     \theta_3  \, a^{(1)} b^{(1)}  s( b^{(2)}) \,  s\Big( \double{a^{(2)}}{c^{(1)}}'\Big) \otimes  \double{a^{(2)} }{c^{(1)}}''  s(c^{(2)})\\
 &=&    \Delta(a)\,   \Lambda(b,c)     +   (-1)^{\vert a^{(2)}\vert\, \vert b \vert} \varepsilon (b)\, \Lambda(a,c)  \ = \ \Delta(a)\, \Lambda(b,c) + \varepsilon(b) \Lambda(a,c)
\end{eqnarray*}
where the signs $\theta_1, \theta_2,\theta_3  = \pm 1$ are  computed  by
\begin{eqnarray*}
\theta_1  &=&  (-1)^{ \vert a^{(2)} \vert \cdot \vert b^{(1)} \vert +  \vert a^{(2)} \vert \cdot \left\vert    \double{b^{(2)}}{c^{(1)}}' \right\vert } , \\
 \theta_2 &=&   (-1)^{\vert a^{(2)} \vert \cdot \vert b^{(1)} \vert +  \vert b^{(2)} \vert \cdot  \vert   c^{(1)} \vert_d+ \vert b^{(2)} \vert \cdot \left\vert    \double{a^{(2)}}{c^{(1)}}''   \right\vert   }, \\
 \theta_3 &=& \theta_2 \cdot  (-1)^{  \vert b^{(2)} \vert \cdot \left\vert    \double{a^{(2)}}{c^{(1)}}'   \right\vert   } \ = \  (-1)^{\vert a^{(2)}\vert\, \vert b \vert} .
\end{eqnarray*}

\up
\end{proof}

The bibracket $\double{-}{-}$ may be recovered from the map $ \Lambda$      at least in the case where the antipode~$s$ in~$\mathscr{C}$ is invertible. Indeed, for any $a,b \in A$,
    \begin{eqnarray}  \label{F_to_double}
 &&   s(a^{(1)}) \, \Lambda\big(a^{(2)},b^{(1)}\big)\, b^{(2)} \\
\notag   &=&  s(a^{(1)})  \,  a^{(2)} s\Big( \double{a^{(3)}}{b^{(1)}}'  \Big) \otimes  \double{a^{(3)}}{b^{(1)}}''  s(b^{(2)}) \, b^{(3)}  \\
\notag  &=&  \varepsilonout   (a^{(1)})  \,    s\Big( \double{a^{(2)}}{b^{(1)}}'  \Big) \otimes  \double{a^{(2)}}{b^{(1)}}''   \varepsilonout   (b^{(2)})   \\
\notag  &=&    \varepsilon (a^{(1)})  \,    s\Big( \double{a^{(2)}}{b^{(1)}}'  \Big) \otimes  \double{a^{(2)}}{b^{(1)}}''    \varepsilon (b^{(2)})    \\
 \notag &=&     (s\otimes \id_A)(\double{a}{b}).
\end{eqnarray}
  If follows  that, if  the antipode $s$ is invertible, then
\begin{eqnarray} \notag
\double{a}{b} &=&  (s^{-1}\otimes \id_A) \big (s(a^{(1)}) \, \Lambda\big(a^{(2)},b^{(1)}\big)\, b^{(2)} \big ) \\
 \notag  &=&   (-1)^{ \vert a^{(1)}\vert \vert a^{(2)} b^{(1)} \vert_d  }\,   (s^{-1}\otimes \id_A) \big(  \Lambda(a^{(2)},b^{(1)})\big) \,  \big(a^{(1)} \otimes  b^{(2)}\big).
\end{eqnarray}
 
\subsection{Reducible bibrackets}\label{prop3}  

  Let   $\double{-}{-} $ be a bibracket in    a   Hopf category $ \mathscr{C}= ( \mathscr{C}, \Delta, \varepsilon, s)$. 
It induces, in the notation of the previous subsection, a bilinear pairing   
$$
\lambda= \lambda(\double{-}{-}): A  \times  A \to A
$$
by   $\lambda =(\varepsilon \otimes \id_A) \Lambda$.
Explicitly, for any $a,b\in A$ we have 
\begin{equation} \label{double_to_T-pairing}
\lambda(a, b) =  \varepsilon \Big(  \double{a }{b^{(1)}}' \Big)    \double{a }{b^{(1)}}''  s(b^{(2)}).
\end{equation}
It follows from Lemma~\ref{F} that, for any $a \in A$  and any homogeneous $b,c \in A$,
\begin{eqnarray*}
\lambda(a,bc)  &=&    \lambda\big(a,b^{(1)}\big)\, \varepsilon\big(b^{(2)}c\big)   +    (-1)^{   \vert b  \vert\, \vert c\vert} \, \lambda(a,c)  \, s(b),\\
\lambda(ab, c)  &=&  \lambda(a^{(1)},c)\, \varepsilon (a^{(2)}b)  +      a\, \lambda(b,c).
\end{eqnarray*}

We  call  a  bibracket $\double{-}{-} $    in~$\mathscr{C}$   \index{bibracket!reducible} \emph{reducible} if  $\Lambda  (A\otimes A) \subset \Delta (A)$.  Then
$$
\lambda=(\varepsilon \otimes \id_A) \Lambda=(\id_A \otimes \varepsilon) \Lambda \colon A \times A   \longrightarrow   A 
 \quad {\rm {and}} \quad  \Lambda= \Delta \circ \lambda. 
$$
As a consequence, a reducible bibracket  in a Hopf category with invertible antipode  is fully   determined  by the associated     pairing~$\lambda$.

\begin{lemma} \label{Fcoc}  Suppose  that  the Hopf category $  \mathscr{C}  $ is   cocommutative.
\begin{itemize}
\item[(i)] If   $\double{-}{-}$ is reducible, then, for any   $a,b \in A$,
$$
\double{s(a)}{s(b)} = (s \otimes s)\Perm_{21} \double{a}{b};
$$
\item[(ii)]  If $\double{-}{-}$ is $d$-antisymmetric, then   $(s\otimes s)\Lambda = - \Perm_{21}  \Lambda \Perm_{21,d}$;
\item[(iii)] If $\double{-}{-}$ is reducible and $d$-antisymmetric,  then
$
 s \lambda  = -\lambda \Perm_{21,d}.
$
\end{itemize}
\end{lemma}

 \begin{proof} In the proof   we will often  use   that $s^2=\id_A$.
We begin with (i).
It easily follows from Lemma \ref{F} that $\Lambda(a,e_X)=0=\Lambda(e_X,a)$ for any $a\in A$ and $X \in \Ob(\mathscr{C})$.
Hence, for any $x,y \in A$,
\begin{eqnarray*}
0 & =&  \Lambda (x,\varepsilonout(y)) \ = \ \Lambda \big(x, s(y^{(1)})y^{(2)}\big)  \\
&=&  \Lambda\big(x,s(y^{(1)})\big)\, \varepsilon(y^{(2)}) + (-1)^{\vert y^{(1)}\vert \vert y^{(2)}\vert} \Lambda\big(x,y^{(2))}\big)\, (s\otimes s) \Delta\big(s(y^{(1)})\big)   \\
&=&  \Lambda\big(x,s(y)\big) + (-1)^{\vert y^{(1)}y^{(2)}\vert \vert y^{(3)}\vert + \vert y^{(1)}\vert \vert y^{(2)}\vert } \Lambda\big(x,y^{(3)}\big)\, \big( y^{(2)} \otimes y^{(1)}\big).
\end{eqnarray*}
 Therefore
\begin{equation} \label{F(x,s(y))}
 \Lambda\big(x,s(y)\big) = - (-1)^{\vert y^{(1)}y^{(2)}\vert \vert y^{(3)}\vert + \vert y^{(1)}\vert \vert y^{(2)}\vert } \Lambda\big(x,y^{(3)}\big)\, \big( y^{(2)} \otimes y^{(1)}\big).
\end{equation}
Similarly, for any $x,y \in A$,
\begin{eqnarray*}
0 &= & \Lambda(\varepsilonout(x),y) \ = \ (-1)^{\vert x^{(1)} \vert \vert x^{(2)} \vert }  \Lambda\big(s(x^{(2)}) x^{(1)},y\big)  \\
&=& (-1)^{\vert x^{(1)} \vert \vert x^{(2)} \vert }  \Lambda\big(s(x^{(2)}) ,y\big)\, \varepsilon(x^{(1)}) +
(-1)^{\vert x^{(1)} \vert \vert x^{(2)} \vert }  \Delta\big(s(x^{(2)})\big)\, \Lambda\big(x^{(1)},y\big) \\
&=&  \Lambda\big(s(x),y \big) + (-1)^{\vert x^{(1)}\vert \vert x^{(2)} x^{(3)}\vert + \vert x^{(2)}\vert \vert x^{(3)}\vert} \big(s(x^{(3)}) \otimes s(x^{(2)})\big) \Lambda\big(x^{(1)},y\big),
\end{eqnarray*}
and therefore  
\begin{equation} \label{F(s(x),y)}
 \Lambda\big(s(x),y \big)  =  -(-1)^{\vert x^{(1)}\vert \vert x^{(2)} x^{(3)}\vert + \vert x^{(2)}\vert \vert x^{(3)}\vert} \big(s(x^{(3)}) \otimes s(x^{(2)})\big) \Lambda\big(x^{(1)},y\big).
 \end{equation}
We have
\begin{eqnarray*}
(s \otimes {\id_A })  \double{s(a)}{s(b)}
&\stackrel{ \eqref{F_to_double}}{=}& s\big((s(a))^{(1)}\big)\, \Lambda\big((s(a))^{(2)},(s(b))^{(1))}\big)\, (s(b))^{(2)} \\
&=& (-1)^{\vert a^{(1)}\vert\, \vert a^{(2)}\vert + \vert b^{(1)}\vert\, \vert b^{(2)}\vert\ }  a^{(2)}\, \Lambda\big(s(a^{(1)}),s(b^{(2)})\big)\, s(b^{(1)})\\
&\stackrel{ \eqref{F(x,s(y))}}{=}&  \theta_1\, a^{(2)}\Big( \Lambda\big(s(a^{(1)}),b^{(4)}\big)\, * b^{(3)}\Big)\, b^{(2)} s(b^{(1)})\\
& = &  \theta_2\, a^{(2)}\Big( \Lambda\big(s(a^{(1)}), b^{(3)}\big)\, * b^{(2)}\Big)\, \varepsilonin(b^{(1)}) \\
& = &  \theta_3\, a^{(2)}\Big( \Lambda\big(s(a^{(1)}), b^{(3)}\big)\, * b^{(2)}\Big)\, \varepsilon(b^{(1)}) \\
& = &  \theta_4\, a^{(2)}\Big( \Lambda\big(s(a^{(1)}), b^{(2)}\big)\, * b^{(1)}\Big) \\
& \stackrel{ \eqref{F(s(x),y)}}{=} & \theta_5\, a^{(4)} s(a^{(3)}) \Big( s(a^{(2)}) * \Lambda\big(a^{(1)}, b^{(2)}\big)\, * b^{(1)}\Big)\\
&=&  \theta_6\,  \varepsilonin(a^{(3)})  \Big( s(a^{(2)}) * \Lambda\big(a^{(1)}, b^{(2)}\big)\, * b^{(1)}\Big) \\
&=&  \theta_7\,  \varepsilon(a^{(3)})  \Big( s(a^{(2)}) * \Lambda\big(a^{(1)}, b^{(2)}\big)\, * b^{(1)}\Big) \\
&=&  \theta_8\,  \Big( s(a^{(2)}) * \Lambda\big(a^{(1)}, b^{(2)}\big)\, * b^{(1)}\Big) \\
\end{eqnarray*}
where the signs $\theta_1,\theta_2, \dots$ are computed by
\begin{eqnarray*}
\theta_1 &=&  - (-1)^{\vert a^{(1)}\vert\, \vert a^{(2)}\vert + \vert b^{(1)}\vert\, \vert b^{(2)}b^{(3)} b^{(4)} \vert+\vert b^{(2)}b^{(3)}\vert \vert b^{(4)}\vert + \vert b^{(2)}\vert \vert b^{(3)}\vert }\\
&=&  - (-1)^{\vert a^{(1)}\vert\, \vert a^{(2)}\vert + \vert b^{(1)}\vert\, \vert b^{(2)}\vert +  \vert b^{(1)} b^{(2)} \vert\, \vert b^{(3)} b^{(4)} \vert+\vert  b^{(3)}\vert \vert b^{(4)}\vert }, \\
\theta_2 &=&  - (-1)^{\vert a^{(1)}\vert\, \vert a^{(2)}\vert  +  \vert b^{(1)}  \vert\, \vert b^{(2)} b^{(3)} \vert+\vert  b^{(2)}\vert \vert b^{(3)}\vert },\\
\theta_3 &=&  - (-1)^{\vert a^{(1)}\vert\, \vert a^{(2)}\vert  +\vert   b^{(2)} \vert \vert b^{(3)}\vert } \ = \    - (-1)^{\vert a^{(1)}\vert\, \vert a^{(2)}\vert  +\vert b^{(1)}   b^{(2)} \vert \vert b^{(3)}\vert }, \\
\theta_4 &=&  - (-1)^{\vert a^{(1)}\vert\, \vert a^{(2)}\vert  +\vert b^{(1)}  \vert \vert b^{(2)}\vert },\\
\theta_5 &=&  (-1)^{\vert a^{(1)}a^{(2)} a^{(3)} \vert\, \vert a^{(4)}\vert    + \vert a^{(1)}\vert \vert a^{(2)} a^{(3)}\vert + \vert a^{(2)}\vert \vert a^{(3)}\vert   +\vert b^{(1)}  \vert \vert b^{(2)}\vert } \\
&=& (-1)^{\vert a^{(3)} \vert\, \vert a^{(4)}\vert + \vert a^{(1)}a^{(2)}  \vert\, \vert a^{(3)} a^{(4)} \vert  + \vert a^{(1)}\vert \vert a^{(2)} \vert   +\vert b^{(1)}  \vert \vert b^{(2)}\vert   },\\
\theta_6 &=& (-1)^{ \vert a^{(1)}a^{(2)}  \vert\, \vert a^{(3)}  \vert  + \vert a^{(1)}\vert \vert a^{(2)} \vert   +\vert b^{(1)}  \vert \vert b^{(2)}\vert  },  \\
\theta_7 &=&   (-1)^{   \vert a^{(1)}\vert \vert a^{(2)} \vert    +\vert b^{(1)}  \vert \vert b^{(2)}\vert   } \  = \  (-1)^{  \vert a^{(1)}\vert \vert a^{(2)}a^{(3)} \vert + \vert b^{(1)}  \vert \vert b^{(2)}\vert  }, \\
\theta_8 &=& (-1)^{  \vert a^{(1)}\vert \vert a^{(2)} \vert + \vert b^{(1)}  \vert \vert b^{(2)}\vert }.
\end{eqnarray*}
Therefore, using the cocommutativity of  $\mathscr{C}$,   we obtain
\begin{eqnarray*}
 (s \otimes {\id_A })  \double{s(a)}{s(b)}  &=&   s(a^{(1)}) * \Lambda\big(a^{(2)}, b^{(1)}\big)\, * b^{(2)}  .  
\end{eqnarray*}
 Besides, 
\begin{eqnarray*}
(s \otimes {\id_A }) (s \otimes s)\Perm_{21} \double{a}{b} &=& \Perm_{21}   (s \otimes {\id_A })  \double{a}{b} \\
&\stackrel{ \eqref{F_to_double}}{=}& \Perm_{21}\big(  s(a^{(1)}) \, \Lambda\big(a^{(2)},b^{(1)}\big)\, b^{(2)}\big) \\
& =&  s(a^{(1)}) *  \big( \Perm_{21} \Lambda\big(a^{(2)},b^{(1)}\big)  \big)* b^{(2)}  \\
 &=& s(a^{(1)}) * \Lambda\big(a^{(2)},b^{(1)}\big) * b^{(2)}
\end{eqnarray*}
where the last equality uses the formula  $\Perm_{21} \Lambda= \Lambda$ which  follows from the reducibility of $\double{-}{-}$.
We conclude that $ \double{s(a)}{s(b)} =  (s \otimes s)\Perm_{21} \double{a}{b}$.

We now prove (ii). If  the bibracket $\double{-}{-}$ is $d$-antisymmetric, then   for any homogeneous $a,b\in A$,
\begin{eqnarray*}
  \Lambda \Perm_{21,d} (a \otimes b) & = &  (-1)^{\vert a \vert_d \vert b \vert_d}\, \Lambda(b,a)   \\
&=&  (-1)^{\vert a \vert_d \vert b \vert_d }\, b^{(1)}  s\Big( \double{b^{(2)}}{a^{(1)}}'  \Big) \otimes  \double{b^{(2)}}{a^{(1)}}''  s(a^{(2)}) \\
&=& \theta_1 b^{(1)} s\Big( \double{a^{(1)}}{b^{(2)}}''  \Big) \otimes  \double{a^{(1)}}{b^{(2)}}'  s(a^{(2)}) \\
&=& \theta_2  \Perm_{21}\Big(  \double{a^{(1)}}{b^{(2)}}'  s(a^{(2)})  \otimes   b^{(1)} s\Big( \double{a^{(1)}}{b^{(2)}}''  \Big) \Big) \\
&=& \theta_3  \Perm_{21}(s\otimes s) \Big( a^{(2)} s\Big(  \double{a^{(1)}}{b^{(2)}}' \Big) \otimes  \double{a^{(1)}}{b^{(2)}}''  s(b^{(1)}) \Big)\\
&=& - \Perm_{21} (s\otimes s) \Lambda(a,b)
\end{eqnarray*}
where the last equality is a  consequence   of the cocommutativity of~$  \mathscr{C}$  and
\begin{eqnarray*}
\theta_1 &=& - (-1)^{d\vert b^{(1)}\vert + d \vert a^{(2)}\vert + \vert a^{(1)}\vert \vert b^{(1)} \vert+ \vert a^{(2)}\vert \vert b^{(2)}\vert + \vert a^{(2)}\vert \vert b^{(1)} \vert
+ \left\vert\double{a^{(1)}}{b^{(2)}}'  \double{a^{(1)}}{b^{(2)}}'' \right\vert }, \\
\theta_2 &=& - (-1)^{d\vert b^{(1)}\vert + d \vert a^{(2)}\vert + \vert a^{(1)}\vert \vert b^{(1)} \vert+ \vert a^{(2)}\vert \vert b^{(2)}\vert
+ \vert a^{(2)} \vert \left\vert\double{a^{(1)}}{b^{(2)}}'' \right\vert + \left\vert\double{a^{(1)}}{b^{(2)}}'  \right\vert \vert b^{(1)} \vert  },  \\
\theta_3 &=& - (-1)^{ \vert a^{(2)}  \vert \vert a^{(1)}\vert  +    \vert b^{(2)} \vert \vert b^{(1)}\vert  }.
\end{eqnarray*}

 Finally, we deduce  (iii) from (ii):
 \begin{eqnarray*}
  s\lambda \ = \ s (\varepsilon \otimes \id_A) \Lambda &= &  (\varepsilon \otimes \id_A) (s\otimes s) \Lambda \\
&=&  -(\varepsilon \otimes \id_A)  \Perm_{21}  \Lambda \Perm_{21,d} \\
&=& - (\id_A \otimes \varepsilon)\Lambda \Perm_{21,d} \ = \ -\lambda \Perm_{21,d}.
\end{eqnarray*}

\up
\end{proof}

 \subsection{Remark} 
 
  Reducible bibrackets are   interesting from the  algebraic viewpoint  because
they induce brackets in more general representation algebras   associated with algebraic groups. 
This class of algebras includes the representation algebras considered here  and  associated with the general linear groups. For more on this, see~\cite{MT+}.  
The bibrackets   arising below  in the geometric context are  reducible.

\section{Hamiltonian reduction of bibrackets}\label{sectionActions-}

We  formulate   Hamiltonian reduction for Gerstenhaber bibrackets  based on  a notion of an  $H_0$-Poisson structure.
  In   the non-graded case,  the content of  this section is   due to   Crawley-Boevey~\cite{Cb} and Van den Bergh~\cite{VdB}.

 \subsection{$H_0$-Poisson structures}\label{hamham1}

 An   \index{$H_0$-Poisson structure} \emph{$H_0$-Poisson structure of degree~$d\in \ZZ$} 
 on a graded algebra $A$ is a $d$-graded Lie bracket   $\langle  -,- \rangle $   in the graded module $\check A=A/[A,A]$ such that,
 for all homogeneous   $x \in \check A$, the map $\langle  x, - \rangle : \check A \to \check A$   lifts to a derivation $A \to A$ of degree $\vert x \vert_d=\vert x \vert +d$.
 If~$A$ is a commutative graded algebra, then an $H_0$-Poisson structure of degree~$d $ in~$A$ is nothing but a   Gerstenhaber bracket of degree $d$ in $A$.

\begin{lemma} \label{VdB_to_Cb}
  Given a Gerstenhaber bibracket of degree $d$ in a graded algebra~$A$, the induced
   bracket $  \langle - , -\rangle$ in $\check  A$   is an  $H_0$-Poisson structure    of degree~$d$ on~$A$.
\end{lemma}

\begin{proof}   That    $\langle -, - \rangle$  is
 a  $d$-graded Lie bracket in $\check A$ follows from Lemma~\ref{comparis-}. 
 The same lemma shows that  the formula $x\mapsto \langle x , -\rangle$ defines a linear map $\check A \to \Der(A)  $ which
preserves the Lie   bracket  and   carries   $\check A^p$ to  $\Der^{p+d}(A)$ for all  $p\in \ZZ$.  This implies the claim of the lemma.
\end{proof}

\begin{theor}\label{H0_to_G}
Let  $\langle  -,- \rangle $ be an $H_0$-Poisson structure of degree~$d $  on a graded algebra $A$ and let $N\geq 1$.
Then  there is a unique Gerstenhaber bracket $\bracket{-}{-}$ of degree $d$ in the trace algebra $A_N^t \subset A_N$  such  that
$$\bracket{\tr(\check a)}{\tr(\check b)}  = \tr\, \langle \check a, \check b \rangle$$
for any $\check a, \check b \in \check A$.
\end{theor}

\begin{proof}
The  proof  follows the same lines as  in the non-graded case, see \cite[Theorem~4.5]{Cb}.
The uniqueness of $\bracket{-}{-}$ is obvious because the image of the trace map $\tr:\check A\to A_N$ generates $A^t_N$.
To prove the existence, consider   the commutative graded algebra $S=S(\check A)$ freely generated
by the graded module $\check A$ (the \index{graded algebra!symmetric} \emph{symmetric algebra} of $\check A$).
The $d$-graded Lie bracket $\langle-,-\rangle$ in $\check A$ uniquely extends  to a Gerstenhaber bracket $\langle-,-\rangle_S$ of degree $d$ in $S $.
The   map $\tr: \check A\to A_N$    uniquely extends to a graded algebra homomorphism $T  : S  \to A_N^t$, which is surjective.
Therefore, it suffices to prove the existence of a  map $\bracket{-}{-}: A_N^t \times A_N^t \to A_N^t$ such that the following diagram commutes:
\begin{equation*} \label{Cb}
\xymatrix{
S  \times S  \ar[rr]^-{\langle-,-\rangle_S} \ar[d]_-{T   \times T  } &&  S  \ar[d]^-{T  } \\
A_N^t \times A_N^t \ar[rr]^-{\bracket{-}{-}}&&  A_N^t\, .
}
\end{equation*}
 In other words, we need to show that the pairing $  T   \, \langle -,-\rangle_S:   S \times S  \to A_N^t$  annihilates   $\Ker (T  )$.
Since the bracket $\langle-,-\rangle_S$  is $d$-antisymmetric,  it suffices to show that $ T   \, \langle r , \Ker (T  ) \rangle_S =0$ for any $r\in S  $.
Since the bracket $\langle-,-\rangle_S$  satisfies the $d$-graded Leibniz rule  in the first variable and $T  $ is an algebra homomorphism,
 it suffices to consider the case  $r\in \check A$. By the definition of an $H_0$-Poisson structure, the map $\langle r, -\rangle : \check A\to \check A$ lifts to
a derivation $\delta:A \to A$.  There is a unique derivation $\delta_N: A_N \to A_N$
such that  $\delta_N(a_{ij})= (\delta(a))_{ij}$ for any $a\in A$ and $i,j \in \{1,\dots,N\}$.  Then,  for any $a\in A$,
$$
\delta_N(\tr(a)) = \delta_N\Big(\sum_i a_{ii}\Big) = \sum_i (\delta(a))_{ii} = \tr \delta(a) = \tr\, \langle r , \check a\rangle_S .
$$
It follows     that the maps  $\delta_N  T   :S  \to A_N $ and $ T   \, \langle r , -   \rangle_S: S\to A^t_N\subset A_N$ are equal on $\check A\subset S$.
Since $\check A$ generates the algebra $S$ and both these maps are derivations,   they must be equal. As a consequence, $ T   \, \langle r , \Ker (T  ) \rangle_S =0$.
\end{proof}

 Combining  Lemma \ref{VdB_to_Cb} and Theorem \ref{H0_to_G}, we obtain that  any Gerstenhaber bibracket of degree $d$ in $A$
induces a Gerstenhaber bracket  of degree $d$ in $A_N^t$.
Clearly this bracket is the restriction of the Gerstenhaber bracket in $A_N$ provided by Lemma \ref{importantGers}.

\subsection{Moment maps} \label{moment_maps}

Let $A$ be a unital graded algebra equipped with a Gerstenhaber bibracket $\double{-}{-}$  of degree $d$.
A \index{moment map} \emph{moment map} for   $\double{-}{-}$   is an element $\mu \in A^{-d}$
such that $\double{\mu}{a} = a \otimes 1_A - 1_A \otimes a$ for all $a\in A$ or, equivalently, $\double{a}{\mu}= a \otimes 1_A - 1_A \otimes a$ for all $a\in A$.
If $d\neq 0$, then there is at most one moment map. If $d=0$,   then for any moment map $\mu\in A^0$ and any  $k\in \kk$, the sum $\mu+k 1_A$ is a moment map.

\begin{lemma} \label{quotient}
Let  $\mu \in A^{-d}$ be a moment map. The  bracket $\langle -,-\rangle$ in $ A  $ associated with $\double{-}{-}$ induces  an $H_0$-Poisson structure of degree~$d$ on    ${B}= A/A \mu A$.
\end{lemma}

\begin{proof}
Let   $p:A \to {B}$ and $h: {B} \to \check{ {B}}= {B}/[{B},{B}]$  be   the canonical projections.
  Clearly,~$p$ carries $[A,A]$ to $[{B},{B}]$ and induces a linear map $\check p: \check A \to \check{ {B}}$.
Lemma~\ref{comparis-}(iii) shows that the bracket $ \langle -,-\rangle$ in~$A$ induces  a pairing $ \langle -,-\rangle: \check A \otimes  A \to  A$.
We claim that there are   linear maps $u,v $   such that the following diagram commutes:
\begin{equation}\label{AAAAAA----}
\xymatrix{
\check{A} \otimes A \ar[rr]^-{ \id \otimes p} \ar[d]_-{\langle -,-\rangle} && \check{A} \otimes {B} \ar[d]_-{u}  \ar[rr]^-{\check p \otimes h    }   && \check{{B}} \otimes \check {B} \ar[d]^-{v} \\
A \ar[rr]^-{ p}   &&   {B} \ar[rr]^-{ h   } &&  \check {B}\, .
}
\end{equation}
Such maps $u,v$ are necessarily unique because    $p $, $\check p$,   $h$ are onto. As a consequence, the following diagram commutes:
\begin{equation*}
\xymatrix{
\check{A} \otimes \check A \ar[d]_-{\langle -,-\rangle} \ar[rr]^-{\check p \otimes \check p} && \check{{B}} \otimes \check{{B}} \ar[d]^-{v}  \\
\check A  \ar[rr]^-{  \check p} && \check{{B}}\, .
}
\end{equation*}
Therefore $v$ is a $d$-graded Lie bracket in $\check B$. Since   $\langle x,-\rangle:A \to A$ is a derivation for all $x\in  \check A$ and  $\check p$ is onto,
the Lie bracket $v$     is an $H_0$-Poisson structure  on ${B}$.

It remains to verify the  claim  above. The definitions of the moment map $\mu$ and the bracket $\langle-,-\rangle$ in $ A$ imply that
$\langle A, \mu\rangle =0$. Hence,  $\langle   A, A\mu A \rangle \subset A\mu A=\Ker p$. This inclusion implies the existence of~$u$.
By Lemma~\ref{comparis-}(ii), $$\langle     A\mu A, A \rangle \subset A\mu A+[A,A]=\Ker h p.$$   This implies the existence of~$v$.
\end{proof}

  By Theorem \ref{H0_to_G},     we obtain   that under the assumptions of  Lemma \ref{quotient},
  the bibracket   in~$A$ induces    Gerstenhaber brackets of degree~$d$
on the trace algebras of~$B=A/A\mu A$.
As an exercise, the reader may extend Lemma~\ref{quotient}   to the  setting of graded categories   discussed in Section~\ref{extenscatmain}.

\chapter {Face homology}\label{Face homology}

 \section{Manifolds with faces and  partitions} \label{Manifolds with faces and partitions} 

 We    recall   manifolds with faces and   discuss  partitions on such manifolds.

\subsection{Manifolds with faces}\label{Manifolds with faces}

 We start with a bigger class of \index{manifold with corners} {\it manifolds with corners}, see \cite{Ce}, \cite{Do}, \cite{Ja},   \cite{MrOd},  and  \cite{Jo}.
An   {\it  $n$-dimensional manifold with corners}  with $n\geq 0$,
or, shorter, an   \index{manifold with corners} \emph{$n$-manifold with corners},    is a paracompact
Hausdorff topological space locally  differentiably ($C^\infty$)  modelled on   open subsets of    $[0,\infty)^n $.
  For a    definition in terms of local coordinate systems and for further details, see \cite{Jo}.
The underlying topological space of   an  $n$-manifold  with corners   $K$
is an $n$-dimensional topological manifold with boundary.
The topological boundary of $K$ is denoted by $\partial K$
(the symbol  $\partial K$  has a   different meaning in \cite{Jo}).
The  \index{dimension function} {\it dimension function} $d_K:K\to \ZZ$ carries a point of $K$  represented by a tuple $(x_1,\dots,x_n)$ in a local coordinate system to the number of non-zero terms in this tuple (this number does not depend on the choice of the local coordinate system).  For  $r\geq 0$, the set $$K_r=\{x\in K : d_K(x) \leq r\}$$ is a closed subset of $K$.
It is clear that
$$ K_0\subset K_1\subset   \cdots \subset K_{n-1}=\partial K \subset K_n=K.$$
Also, $K_0=d_X^{-1}(0)$  is a discrete set,  and $K_r\setminus K_{r-1}$ is a smooth $r$-dimensional manifold for all $r\geq 1$.

The set $ P(K) =\partial K\setminus K_{n-2}$  is an open subset of $\partial K$ and  any $x\in \partial K$ belongs to the closure of at most $n-d_K(x)$ connected components of  $P(K)$.
We call $K$ a \index{manifold with faces} {\it manifold with faces}  if
 $K$ is   compact and every $x\in \partial K$
 belongs to the closure of precisely $n-d_K(x)$ different components of   $P(K)$.
This condition implies that the closure   in $K$
 of any component of   $P(K)$  is   an $(n-1)$-dimensional  manifold
 with faces whose dimension function is the restriction of $d_K$.
We call  the closure of a  component of  $P(K)$  a  \index{face!principal} {\it principal face} of $K$.
We can now define recursively on  $n=\dim K$  the notion of a face of~$K$.
By definition, a  \index{manifold with faces!face of} \emph{face}
of~$K$ is  a connected component of~$K$, or a principal face of~$K$,
or a face of a principal face of $K$.
Clearly,~$K$ has only a  finite number of faces, and each face of $K$ is a connected manifold with faces.
The union of faces of  $K$ of dimension $\leq  r$ is equal to $K_{ r}$ for all $r\geq 0$.
The faces of $K$ contained in $\partial K$ are said to be \index{face!proper} {\it proper}.

Every point $x$ of $K$ lies in the interior of a unique face $F_x$ of $K$.
  If $d_K(x)\geq 1$, then $F_x$  is the closure of the component of  $K_r\setminus K_{r-1}$ containing $x$
for $r=d_K(X)$. If $d_K(x)=0$,  then $F_x=\{x\}$.
  Note that $F_x$ is the smallest face of $K$ containing $x$:
any face of $K$ containing $x$ contains $F_x$ as a face.

For example,  any compact smooth manifold   $M$ is a manifold with faces,
and  its faces are the components of $M$ and of $\partial M$.
 For any $n\geq 0$,
 an $n$-dimensional simplex is a  manifold with faces and its
  faces   are the usual combinatorial  faces. Finite     disjoint unions and finite   products of   manifolds with faces are manifolds with faces in the obvious way. 
  The  empty set is considered as an $n$-manifold with faces for any $n\geq 0$.

Following \cite{MrOd}, we call a map $f $ from   an $n$-manifold with faces $K$ to an $m$-manifold with faces   $L$     \index{manifold with faces!smooth map of}  \emph{smooth} if,
restricting $f$  to any local coordinate systems in these manifolds,
we obtain a map that extends to a $C^\infty$-map from an open subset of $\RR^n$ to $\RR^m$.
  (Such a map   $f$ is  said to be  ``weakly smooth'' in \cite{Jo}.)
A smooth map $f:K\to L$ is continuous and its restriction to any face $F$ of $K$ is a smooth map   $F\to L$.
A map $f:K\to L$   is a  \index{manifold with faces!diffeomorphism of} \emph{diffeomorphism} if it is a bijection and both~$f$ and~$f^{-1}$ are smooth.  Diffeomorphisms of manifolds with faces
preserve the dimension function and carry faces  onto faces.

We can define   smooth ($C^\infty$) triangulations of a manifold with faces
repeating word for word the standard definition of  a smooth triangulation   of an ordinary manifold  \cite[Section 8.3]{Mu}  and   requiring      all   faces to be subcomplexes. (The latter condition is probably   satisfied automatically but we prefer to spell it out.)
 The standard methods of the   theory of smooth triangulations \cite[Section 10.6]{Mu} apply in this setting and show that all manifolds with faces have smooth triangulations.

 A manifold with faces  $K$  is  \index{manifold with faces!oriented} {\it oriented} if its    underlying   topological manifold    is oriented.
The oriented manifold with faces obtained from $K$ by inverting the orientation
is denoted by $-K$.

\subsection{Partitions}\label{Partitions}
By a \index{manifold with faces!partition of} {\it  partition} $\varphi$   on a manifold with faces $K$ we mean
a partition of the set of   faces of $K$ into disjoint subsets, called \index{face!type of} {\it types},
and a family of diffeomorphisms $\{\varphi_{F,G} :F\to G\}_{(F,G)}$
numerated by ordered pairs  $(F,G)$  of      faces of $K$ of the same type  such that
\begin{itemize}
\item[(a)]    $\varphi_{F,F}=\id_F$ for any     face $F$ of $K$ and $\varphi_{G, H}\, \varphi_{F,G}=\varphi_{F, H}$
 for any   faces $F,G,H$ of $K$ of the same type;
\item[(b)] if $F,G$ are     faces of $K$ of the same type, then  $\varphi_{F,G}:F\to G$
carries  any face $F'$ of $F$ onto a face $G'$ of $G$ so that $F', G'$
have the same type as faces of $K$ and $\varphi_{F', G'}=\varphi_{F,G}\vert_{F'}:  F'\to G'$.
\end{itemize}
The  diffeomorphisms $\{ \varphi_{F,G} \}_{(F,G)}$ will be called  \index{partition!identification map of} \emph{identification maps}.
For example, every manifold with faces $K$  has a  \index{partition!trivial} \emph{trivial partition}  such that two faces have the same type if and only if  they coincide.

Given a partition $\varphi$ on $K$, we write $x\sim_\varphi y$ for points $x,y\in K$
if   there are     faces   $F, G$ of $K$  of the same type such that $x\in F$, $y\in G$, and $\varphi_{F,G}(x)=y $.
Clearly, $x\sim_\varphi y$ if and only if  the   faces  $F_x, F_y$   have the same type and $\varphi_{F_x,F_y}(x)=y $.
 Then  $\sim_{\varphi}$ is an equivalence relation on $K$.
 The quotient topological space $K_{\varphi}=K/\! \sim_{\varphi}$ may not be a   manifold. For any set $L \subset K$,
we  denote by $L_{\varphi}$ the image of $L$ under the   projection $K\to K_{\varphi}$.

A  smooth triangulation  $T$ of  $K$   \emph{fits a partition  $ \varphi $} on~$K$ if
the identification  map    $ \varphi_{F,G} :F\to G$ is a   simplicial isomorphism for any faces $F,G$ of the same type.

\begin{lemma}\label{partition}
For any partition $\varphi$ on $K$, there exists a smooth triangulation~$T$ of~$K$ which fits $\varphi$
and   projects to a triangulation, $T_\varphi$, of  the quotient space $K_\varphi$.
\end{lemma}

\begin{proof}
We   construct by induction  on  $r\geq 0$
a smooth triangulation $T^r$ of $K_r$   satisfying the following condition:  all the identification maps between faces of~$K$ of dimension $\leq r$ are simplicial isomorphisms.
The case $r=0$ is obvious: we just take $T^0=K_0$.
 Given $T^{r-1}$, we construct $T^{r}$   as follows:  pick one $r$-dimensional face of~$K$ in each type and  extend $T^{r-1}$ to the union of $K_{r-1}$ with
these faces using the theory of smooth triangulations \cite[Section 10.6]{Mu}.
The resulting triangulation of this union uniquely  extends to a  triangulation $T^r$ of $K_r$ satisfying the  condition above.
Set $n=\dim K$. Clearly,   $T=T^n$ is a smooth triangulation of~$K$  that fits~$\varphi$.

Let $T'$ and $T''$ be the first and second barycentric subdivisions of $T$, respectively. Both $T'$ and $T''$ fit $\varphi$.
We claim that (i)  the projection $\pi:K \to K_\varphi$ is injective on   each  simplex   of $T'$
and (ii) the images under $\pi$  of any two  simplices  of $T''$
(which by (i) are simplices) meet along a common face.
Thus, the triangulation  $T''$  of $K$ projects to a triangulation of $K_\varphi$ and   satisfies   the conditions of the lemma.

To prove (i), consider a  simplex $\tau$ of $T'$.  Since all simplices of $T'$ are faces of $n$-simplices, it suffices to consider the case where $\dim(\tau)=n$.
  Note that the restriction of  $\pi:K\to K_\varphi$ to the interior of any face of~$K$ is injective. Moreover,   for any faces $F \subset G$ of~$K$,  the restriction of $\pi$ to $\Int(F) \cup \Int(G)$ is injective.
Therefore, to prove the injectivity of $\pi\vert_\tau$,   it is enough to find   a   sequence of faces $F_0\subset F_1\subset \cdots  $ of $K$, possibly  with repetitions,  such that    $\tau \subset \cup_i \Int(F_i)$.
Let  $\sigma_0 \subset \sigma_1\subset \cdots    \subset \sigma_n  $  be the simplices of $T$ whose barycenters   are the vertices of $\tau$ where $\dim(\sigma_i)=i$ for all~$i$.
Let $F_i$ be the smallest face of $K$ containing $\sigma_i$.
 The inclusions $\sigma_{i-1} \subset \partial \sigma_i \subset F_i$ imply
  that $F_{i-1} \subset F_i$ for all $i$.
 Note that   $\Int(\sigma_i)\subset \Int(F_i)$  since $\partial F_i$ is a subcomplex of $T$.
Thus,  $\tau \subset \cup_i \Int(\sigma_i)  \subset \cup_i \Int(F_i)$.

 To prove (ii), observe first that for any simplex $\Delta $    of $T'$, the  set $ \pi^{-1} (\pi(\Delta ))$ is   a subcomplex of $T'$.
 Indeed,    this set is equal to $ \cup_{F,G} \,\varphi_{F,G}(\Delta \cap F)$  where $F,G$ run  over all faces of $K$ of the same type.
Since $T'$ fits   $\varphi$ and  both $   \Delta $ and $ F $ are subcomplexes of $  T'$,  so are the sets $\Delta \cap F$, $\varphi_{F,G}(\Delta \cap F)$,  and $\pi^{-1} (\pi(\Delta ))$.

Consider   any   simplices $\tau_1,\tau_2$ of the triangulation $T''$.  Let $\Delta_1$ and $\Delta_2$ be   simplices of $T'$ containing $\tau_1$ and $\tau_2$ respectively.
Set $R=\pi(\Delta_1 ) \cap \pi(\Delta_2 )$. Clearly,~$R$ is the image of the set $\Delta_1 \cap  \pi^{-1} (\pi(\Delta_2 ))$ under $\pi$.
By the above, the latter set   is   a subcomplex of $\Delta_1$. Therefore,  $R$
  is a subcomplex of the simplex $\pi(\Delta_1 )$.
We   claim that $R \cap \pi(\tau_1)$ is a face of the simplex $\pi(\tau_1)$.
This claim would imply that $R \cap \pi(\tau_1)=\pi(\tau_0)$ for a simplex $\tau_0$ of $T''$. Since $R \subset  \pi(\Delta_2 )$,
we can assume (replacing  if necessary  $\tau_0$ by some $\varphi_{F,G}(\tau_0)$)
that  $\tau_0\subset \Delta'_2$ where $\Delta'_2$ is the barycentric subdivision of $\Delta_2$. Then
$$
\pi(\tau_1)\cap \pi(\tau_2) =  R \cap \pi(\tau_1)  \cap \pi(\tau_2) = \pi(\tau_0) \cap \pi(\tau_2)
$$
is an intersection of two simplices of $\pi(\Delta_2')$. Hence, it is a simplex of $\pi(\Delta_2')$ and  a face of $\pi(\tau_2)$.
By symmetry between $\tau_1$ and $\tau_2$, the intersection  $\pi(\tau_1)\cap \pi(\tau_2)$ is also a face of $\pi(\tau_1)$.

To prove the   claim above, we need only to show that  any subcomplex~$R$ of   an arbitrary    simplex $\Delta$
meets any simplex $\tau$ of the first barycentric subdivision $\Delta'$ along a face of $\tau$.
Clearly, $R\cap \tau$ is an intersection of two subcomplexes of $\Delta'$ and therefore  a subcomplex of $\tau$.
Set $k=\dim \tau$ and let $\sigma_0 \subset \cdots \subset \sigma_k$ be the faces of~$\Delta$ whose barycenters $v_0,\dots,v_k$ are the vertices of $\tau$.
 Let $i$ be the largest integer such that $v_i \in R$.
Since $R$ contains an interior point of $\sigma_i$ and $R$ is a subcomplex of $\Delta$, we   have $\sigma_i\subset R$.
Then $R$ contains the face $\langle v_0,\dots,v_i \rangle$ of $\tau$.
Since $R\cap \tau$ is a subcomplex of $\tau$   not containing    $v_{i+1},\dots, v_k$,
we   have $R \cap \tau=\langle v_0,\dots,v_i \rangle $.
\end{proof}

 \section{Polychains, polycycles,  and face homology} \label{Polychains and homology}

 We  introduce the face homology of  a topological space~$X$.

\subsection{Polychains} \label{polychains} 

Given a partition $\varphi$ on a manifold with faces $K$, we say that a continuous map $ \kappa: K\to X$
 is  \index{map!compatible with a partition} {\it compatible}
with   $ \varphi $ if $\kappa \circ \varphi_{F,G} =  \kappa \vert_F: F\to X$
for any     faces $F, G$ of $K$ of the same type.
Every such $\kappa$ is obtained by composing the projection $K\to K_\varphi$
with a continuous map $K_\varphi\to X$.

An  \index{polychain} \emph{$n$-dimensional polychain} or, shorter, an \emph{$n$-polychain}
in $X$ with $n\geq 0$  is a quadruplet $\calK= (K,\varphi , u, \kappa)$ where
$K$ is an oriented   $n$-manifold   with faces,
$\varphi $ is a partition on~$K$,   $ u$ is a  map $  \pi_0(K) \to \kk$ called the  \index{partition!weight of} \emph{weight},
and $ \kappa: K \to  X$ is a continuous map compatible with $\varphi $.
 By convention,  for every $n\geq 0$, there is
an  \index{polychain!empty} \emph{empty $n$-polychain}  $\varnothing$
whose underlying $n$-manifold is the empty set.

A  \index{polychain!diffeomorphism of} \emph{diffeomorphism} of   $n$-polychains $ \calK=(K,\varphi , u ,\kappa)$ and
$ \calK'= (K',\varphi',u',\kappa')$ in~$X$ is    a diffeomorphism $f:K\to K'$ such that
\begin{itemize}
\item[(1)] $\kappa = \kappa' \circ f$;
\item[(2)]    faces $F,G$ of $K$ have the same type if and only the faces  $f(F),f(G)$ of $K'$ have the same type
and then $f\vert_{G} \circ \varphi_{F,G}  = \varphi'_{f(F),f(G)} \circ f\vert_F: F \to f(G)$;
\item[(3)] $u' (f(C))= \deg \big(f\vert_{C}: C\to f(C)\big) \, u (C )$
for any   connected  component $C$ of $K$ where $\deg$   denotes   the degree of a diffeomorphism.
\end{itemize}
We say that     $n$-polychains $ \calK $ and $ \calK' $ in $X$ are  \index{polychains!diffeomorphic} \emph{diffeomorphic}    and  we write $\calK\cong \calK'$
if there exists a diffeomorphism of $ \calK $ onto $ \calK' $.
It is clear that $\cong $ is an equivalence relation.
By definition, the diffeomorphism class of a polychain $\calK=(K,\varphi , u, \kappa)$
is preserved if one simultaneously inverts
the orientation of a   component   of $K$ and multiplies   the corresponding  weight   by $-1$.
  Therefore   the  \index{polychain!opposite} \emph{opposite} polychain $-\calK=(-K,\varphi,u, \kappa)$ is diffeomorphic to $(K,\varphi , -u, \kappa)$.

Examples of  polychains are provided by    \index{singular manifold}  \emph{singular manifolds} in $X$,
that is pairs (an oriented smooth compact   manifold $M$, a continuous map $\kappa:M\to X$).
Such a pair determines a polychain $  (M,\varphi , u, \kappa)$ where $ M$ is viewed as a manifold with faces as in Section~\ref{Manifolds with faces},
$\varphi$ is the trivial partition, and  $u=1\in \kk$ is the constant function on $\pi_0(M)$.
As   explained below, polychains in $X$   may be also extracted from   singular chains in $X$.
Thus,  we can    view  polychains   as common generalisations  of   singular manifolds and singular chains  in which the role of  source spaces  is played by manifolds with faces.

\subsection{Reduced polychains}\label{reducedpolychains--}

  A polychain  $\calK=(K,\varphi , u ,\kappa)$  in~$X$  is \index{polychain!reduced} {\it reduced} if
any   distinct  connected   components of $K$   have different types  with respect to $\varphi$
  and $u(C)\neq 0$ for any    connected   component $C$ of $K$.
We define two  transformations of an arbitrary polychain $ \calK=(K,\varphi , u ,\kappa)$ in $X$
 whose composition turns $\calK$ into a reduced polychain.

 To define the first transformation, pick a representative
in each type of connected components of $K$,   and let $K_+\subset K$ be the   union of these representatives.
Clearly, $K_+$ is a  manifold with faces which we endow  with orientation induced from that of~$K$.
Restricting $\varphi $ and $\kappa$ to $K_+$ we obtain a partition $\varphi_+$ on  $K_+$ and a map $\kappa_+:K_+\to X$ compatible with $\varphi_+$.
We define a weight $u_+$ on $K_+$ by
$$
  u_+(C)= \sum_{C'} \deg(\varphi_{C,C'}) \, u(C')
$$
where $C$ is a component of $K$ lying in $K_+$ and   $C'$   runs over all components of~$K$ of the same type as $C$. It is clear that
$(K_+,\varphi_+ , u_+ ,\kappa_+)$ is a polychain in $X$ whose distinct components  have different types.
This polychain, denoted $\red_+ (\calK)$, is   determined by $ \calK$ uniquely up to diffeomorphism.

The second  transformation of a polychain $ \calK=(K,\varphi , u ,\kappa)$  removes
  from $K$ all connected components   with zero weight  and restricts $\varphi , u ,\kappa$ to the remaining manifold with faces.
  The resulting polychain is denoted $\red_0 (\calK)$.

  The two-step  operation $\red=\red_0\red_+$  transforms  an arbitrary polychain into a reduced polychain defined uniquely up to diffeomorphism.
It is clear that a polychain $\calK$ is reduced if and only if $\red (\calK) \cong \calK$.

 \subsection{Operations} \label{reducedpolychains}

The  \index{polychain!boundary of} \emph{boundary} of an $n$-polychain $\calK=(K,  $ $ \varphi,  u, \kappa)$ in~$X$
is the $(n-1)$-polychain $\partial \calK=(K^\partial , \varphi^\partial, u^\partial,\kappa^\partial)$ in~$X$  defined as follows.
\begin{itemize}
\item The manifold with faces $K^\partial $ is    the    disjoint union of all principal faces  of~$K$ endowed with orientation induced from that of $K$
  (see the Introduction for our orientation conventions).

\item Let $\iota: K^\partial\to K$ be the natural map identifying each component of $K^\partial$ with its copy
 in~$K$. Two   faces $F,G$ of $K^\partial$  have the same type if
the faces $\iota (F), \iota(G)$ of $K$  have the same type and
$$\varphi^\partial_{F,G}= (\iota\vert_G)^{-1} \varphi_{\iota (F), \iota(G)} \iota:F\to G.$$

\item  For any  connected   component  $P$ of $K^\partial$, we  set    $u^\partial(P)=u  (K^P)$
where $K^P$ is the connected component of $K$ containing the principal face $\iota(P)$.

\item     We set   $\kappa^\partial=\kappa   \iota: K^\partial\to X$.
\end{itemize}
The boundary of a polychain  is well defined up to diffeomorphism,
and   diffeomorphic polychains have diffeomorphic boundaries.
The  \index{polychain!reduced boundary of} \emph{reduced boundary} $\partial^r \calK$  of a polychain $\calK$ is defined by  $\partial^r \calK=\red (\partial \calK)$.

\begin{lemma}
For any polychain $\calK$ in $X$, $\partial^r\! \red(\calK) = \partial^r \calK$ and $\partial^r \partial^r \calK =\varnothing$.
\end{lemma}

\begin{proof}
The first identity is clear. The second identity  follows from the first: 
 \begin{eqnarray*}
\partial^r \partial^r \calK = \partial^r\! \red(\partial \calK) = \partial^r \partial \calK = \red_0 \red_+( \partial \partial \calK) =\varnothing.
 \end{eqnarray*}

\up 
\end{proof}

The  \index{polychains!disjoint union of} \emph{disjoint union} of two $n$-polychains $\calK_1, \calK_2$  in~$X$ is defined in the obvious way
and is denoted   $\calK_1 \sqcup \calK_2$. Clearly,
$$\red (\calK_1 \sqcup \calK_2)= \red (\calK_1) \sqcup \red (\calK_2)
\quad {\text {and}}\quad \partial (\calK_1 \sqcup \calK_2)= \partial \calK_1 \sqcup \partial \calK_2 $$
  so that $  \partial^r (\calK_1 \sqcup \calK_2)= \partial^r (\calK_1) \sqcup \partial^r (\calK_2) $.

For $k\in \kk$ and   a polychain $\calK=(K,\varphi,u, \kappa)$ in $X$, set  $k\calK=(K,\varphi, k u, \kappa)$. Clearly,
$$
\red(k\calK)=  \red\left(k\red(\calK)\right)
\quad \hbox{and} \quad
\partial(k\calK) = k \partial \calK
$$
 so that $\partial^r (k\calK)=\red (k \partial^r \calK)$.
Note that the polychain $(-1) \calK$ is diffeomorphic to  the polychain  $-\calK$   opposite to $\calK$.

\subsection{Face homology}\label{facehomology}

 The diffeomorphism classes of  $n$-polychains in~$X$ may be added and multiplied by elements of $\kk$, 
 but  do not form a module because the  distributivity  relation $(k+l) \calK\cong k\calK   \sqcup   \l\calK$ fails.
Also, it is natural to throw in  relations identifying $\calK$ with $ \red(\calK) $ for all $\calK$. Quotienting the set  of diffeomorphism classes of  $n$-polychains in $X$
 by  the   relations of these two types, we obtain  the  \index{polychains!module of} \emph{$\kk$-module of $n$-polychains} in $X$.
 These modules together  with  the boundary maps induced by $\partial$ form the  \index{face chain complex} \emph{face chain complex}  of~$X$ whose  homology is  the \index{face homology}  \emph{face homology} of~$X$.  However, we prefer the following  more direct definition of  face homology.

We say that     $n$-polychains $\calK_1$ and $\calK_2$ in $X$ are \index{polychains!homologous}  \emph{homologous}, and write $\calK_1\simeq \calK_2$,
if there exist  $(n+1)$-polychains $\calL_1,\calL_2$ in $X$
such that  $$\red (\calK_1)\sqcup \partial^r \calL_1   \cong \red (\calK_2)\sqcup \partial^r \calL_2 .$$
Clearly, the homology relation $\simeq $ is an equivalence relation
(weaker than   the diffeomorphism relation   $\cong$).
  The homology class of an $n$-polychain $\calK$ in $X$ is denoted by $\lb \calK \rb$.  Note   that  $\lb \calK \rb=\lb \red(\calK)  \rb$.
 If   $\calK_1$ and $\calK_2$ are homologous, then $\partial^r \calK_1 \cong \partial^r \calK_2$.

 A polychain $\calK=(K,\varphi,u,\kappa)$ is   a \index{polycycle} \emph{polycycle} if $\partial^r \calK= \varnothing$.
 A polychain homologous to a polycycle is itself a polycycle. In particular, if $\calK$ is a polycycle, then so is $\red(\calK)$ and vice versa. Let
$$
\widetilde{H}_n(X)=\{\hbox{$n$-polycycles in $X$}\} /\simeq
$$
be the   set of homology classes of $n$-polycycles in $X$. Note that the  disjoint union of polycycles is a polycycle,
and multiplication of polycycles by elements of $\kk$  yield polycycles.

\begin{lemma}\label{additivityandrescaling}
Disjoint union  of polycycles together with multiplication of polycycles by elements of $\kk$
 turn   $\widetilde{H}_n(X)$  into a  module  (over $\kk$).
\end{lemma}

\begin{proof}
Clearly, the disjoint union of polychains is compatible with   $\simeq$ and induces   a binary operation in
 $\widetilde{H}_n(X)$. This operation  is associative and commutative
with $\varnothing$ representing the zero element. Thus,  $
\widetilde{H}_n(X)$ is an abelian monoid.

To prove that $\widetilde{H}_n(X)$  is a group, we use the   cylinder construction on polychains.
Consider an  $n$-polychain $\calK=(K,\varphi,u, \kappa)$ in $X$.
We define the  \index{polychain!cylinder} {\it cylinder   polychain} $\overline{\calK}= (\overline K, \overline \varphi, \overline u, \overline \kappa)$    as follows.
Set $\overline K=K \times I$ where    $I=[0,1]$ is viewed as a manifold with  faces   $I$, $\{0\}$, $\{1\}$ and  endow $\overline K$ with the product orientation.
Two   faces $F \times J$, $G \times J'$ of $K \times I$
are of the same type if   $F,G$ are   faces of $K$ of the same type and $J=J'$ is any face of $I$; then
  $\overline \varphi_{F \times J , G\times J}= \varphi_{F,G}\times \id_J $.
By definition,  $\overline u(C \times I)= u (C)$ for any connected  component $C$ of $K$, and $\overline \kappa: \overline K\to X$ is
the composition of the   cartesian   projection $\overline K\to K$ with $\kappa:K\to X$.
It follows from the definitions  that  $$\red (\overline \calK) \cong \overline {\red(\calK)}
\quad {\text {and}}\quad  \partial \overline \calK   \cong  \calK \sqcup (-\calK) \sqcup \overline {\partial   \calK}.$$
Therefore $$ \partial^r \overline \calK  =\red(  \partial \overline \calK )  \cong \red(\calK) \sqcup \red (-\calK)  \sqcup \red (\overline {\partial   \calK}) \cong \red(\calK) \sqcup \red (-\calK)  \sqcup   \overline { \partial^r   (\calK)} .$$ If $\calK$ is a polycycle, this gives  $ \partial^r \overline \calK  \cong \red(\calK) \sqcup \red (-\calK)$. Therefore $\calK \sqcup (-\calK) \simeq  \varnothing$.
We conclude that $\widetilde{H}_n(X)$  is an abelian group.

Given two homologous $n$-polycycles $\calK_1$ and $\calK_2$ in $X$, pick
  $(n+1)$-polychains $\calL_1,\calL_2$ in $X$
such that  $\red (\calK_1)\sqcup \partial^r \calL_1   \cong \red (\calK_2)\sqcup \partial^r \calL_2$.
Then, for any  $k\in \kk$,   $$\red\left(k \big(\red (\calK_1)\sqcup \partial^r \calL_1\big) \right) \cong
\red\left(k \big(\red (\calK_2)\sqcup \partial^r \calL_2\big) \right).$$
For each $i\in\{1,2\}$,
$$
\red\left(k \big(\red (\calK_i)\sqcup \partial^r \calL_i\big) \right)
= \red\left(k\red (\calK_i)\right) \sqcup \red \left(k\partial^r \calL_i\right)
= \red (k\calK_i) \sqcup \partial^r(k\calL_i).
$$
We deduce that $k\calK_1\simeq k\calK_2$.
Thus, the multiplication by $k\in \kk$  induces a   well defined map $  \widetilde{H}_n(X) \to \widetilde{H}_n(X)$.

The axioms of a $\kk$-module   are  straightforward  except   the linearity   in $k$. The latter is a consequence of
 the following   fact:
if $\calK_1=(K,\varphi , u_1,\kappa)$ and $\calK_2=(K,\varphi , u_2,\kappa)$
are $n$-polycycles in $X$ (with the same $K, \varphi, \kappa$),
then the $n$-polychain $\calK=(K,\varphi , u_1+u_2, \kappa)$
is a polycycle  homologous to  $\calK_1 \sqcup \calK_2$.
To see this,  consider the cylinder  polychain
$\overline {\calK_1 \sqcup   \calK_2} \cong \overline \calK_1 \sqcup \overline \calK_2$  (as defined   above)
and modify its partition   by additionally declaring  that,
for any face $F$ of $K$,
the  faces $(F \times \{0\})_1$ and $(F \times \{0\})_2$ of $(K \times I)_1 \sqcup (K \times I)_2$ have the same type and the corresponding identification map is the identity map.
This gives   an $(n+1)$-polychain   $\calL$  such that
$$
\red_+ (\partial   \calL ) \cong  \red_+(\calK_1)  \sqcup  \red_+(\calK_2) \sqcup \red_+(-\calK)
\sqcup  (\rm {a\,\, polychain \,\, with\,\,  zero\,\,  weight}).
$$
  Therefore
$$\partial^r   \calL  \cong  \red(\calK_1) \sqcup \red(\calK_2)\sqcup \red (-\calK) .$$
  Hence,  $\calK$ is a polycycle homologous to $ \calK_1 \sqcup \calK_2$.
\end{proof}

We call   $\widetilde{H}_n(X)$ the \index{face homology}  \emph{$n$-th face homology of $X$} (with coefficients in $\kk$).
The   face homology      extends to   a functor from the category of topological spaces to the category of modules:
a  continuous   map $f:X\to Y$ induces a linear map $ f_\ast: \widetilde{H}_n(X)\to   \widetilde{H}_n(Y)$
carrying the homology class of a polycycle  $\calK=(K, \varphi, u, \kappa)$ in $X$
to the homology class of the polycycle $  f_*(\calK) =(K, \varphi, u, f\kappa)$ in $Y$.

\subsection{Deformations}\label{homotopy_singular} 

   A  \index{polychain!deformation of} {\it deformation} of a  polychain $\calK=(K,\varphi,u,\kappa)$  in~$X$ 
is a family of polychains $\{ \calK^t  = (K,\varphi,u,\kappa^t) \}_{ t\in I}$  with the same $K, \varphi, u$  such that
$\{\kappa^t:K \to X\}_{t\in I}$ is a  (continuous) homotopy  of $\kappa^0=\kappa$.
 By the definition of a polychain,   the map $\kappa^t$ is   compatible with $\varphi$ for all $t\in I$.

\begin{lemma}\label{homotimplieshomol}
If  $\{ \calK^t \}_{ t\in I} $ is a deformation of a polycycle $\calK $,
then $\calK^1 $ is a polycycle homologous to $\calK=\calK^0$.
\end{lemma}

\begin{proof}
 Equality $\partial^r\calK^1=\varnothing$ is a direct consequence of the assumption $\partial^r\calK=\varnothing$.
Consider the  cylinder   polychain  $\overline{\calK}= (\overline K, \overline \varphi, \overline u, \overline \kappa)$
associated with $\calK=(K,\varphi,u,\kappa)$ in the proof of Lemma~\ref{additivityandrescaling}.
Let       $\widehat \kappa:   \overline K=K\times I \to   X$   be the map
 determined by the   homotopy  $\{\kappa^t\}_{t\in I}$   of $\kappa$.
 Then    $\calR= (\overline K, \overline \varphi, \overline u, \widehat \kappa)$  is a polychain such that
$$
\partial^r   \calR    \cong \red(\calK^1) \sqcup \red (-\calK^0).
$$
  This implies that $\calK^0\simeq \calK^1$.
\end{proof}

\begin{lemma}\label{homotimplieshomol+} Let $X,Y$ be topological spaces.
If maps $f, g:X\to Y$ are homotopic, then $f_\ast=g_\ast: \widetilde H_\ast(X) \to \widetilde H_\ast(Y)$.
\end{lemma}

\begin{proof} Pick a homotopy $\{f^t\}_{t\in I}$ between $f^0=f$ and $f^1=g$. For any polycycle  $\calK=(K,\varphi,u,\kappa)$ in $X$,
we have a deformation   $\{(K,\varphi,u, f^t \kappa ) \}_{ t\in I}$   relating the polycycles
$  f_*(\calK)   =(K,\varphi,u, f  \kappa )$ and $  g_*(\calK)  =(K,\varphi,u,  g \kappa )$.
Lemma~\ref{homotimplieshomol}  implies that these polycycles are homologous. Hence,  $f_\ast=g_\ast$.
\end{proof}

Lemma~\ref{homotimplieshomol+} implies that a   homotopy equivalence between topological spaces induces  an isomorphism of their face homology.

\subsection{Cross product} \label{cross_product}

The   cartesian   product $K\times L$ of two   manifolds with faces~$K$ and~$L$ can be  viewed as a manifold with faces in the obvious way.
The faces of $K\times L$ are the products $F\times G$ where~$F$ runs over   faces of~$K$ and~$G$ runs over   faces of~$L$.
When~$K$ and~$L$ are oriented, we always provide $K\times L$ with the product orientation.
  This    construction leads to a cross  product   in face homology as follows.

Let $X$ and $Y$ be topological spaces.
The \index{polychains!cross product of} \emph{cross product} of a $p$-polychain $\calK=(K,\varphi,u,\kappa)$  in $X$
and a $q$-polychain $\calL=(L,\psi,v,\lambda)$ in $Y$
is the $(p+q)$-polychain
$$
\calK \times \calL =(K \times L, \varphi \times \psi, u\times v, \kappa \times \lambda)
$$
 in $X\times Y$. Here $\varphi \times \psi$ is   the following partition on $K\times L$:
for  faces $F,F'$ of $K$ and $G,G'$ of $L$,
the faces $F \times G$ and $F'\times G'$ of $K\times L$ have the same type
if, and only if,  $F$ has the same type as $F'$ and $G$ has the same type as $G'$, and
then  $$(\varphi\times \psi)_{F\times G,F'\times G'}= \varphi_{F,F'} \times \psi_{G,G'}.$$
The weight $u\times v$ carries   $C\times D$ to $u(C) \, v(D)$
for any connected components~$C$  of~$K$ and~$D$ of~$L$. We also define the \index{polychains!reduced cross product of}\emph{reduced cross product} of $\calK$ and $\calL$ by
$$
\calK \times^r \calL = \red\left(\calK \times \calL\right).
$$
  Note that
\begin{equation}\label{rc=cr}
\calK \times^r \calL \cong   \calK \times^r \red(\calL)
\cong \red(\calK) \times^r \calL \cong \red(\calK) \times^r \red(\calL).
\end{equation}

\begin{lemma}\label{properties_cross}
(i) For any $p$-polychains $\calK_1, \calK_2$   in $X$ and $q$-polychain~$\calL$ in~$Y$, 
$$
(\calK_1 \sqcup \calK_2) \times^r \calL \cong (\calK_1 \times^r \calL) \sqcup (\calK_2 \times^r \calL).
$$
(ii) For any $p$-polychain $\calK$  in $X$
and  for any $q$-polychain  $\calL$ in $Y$,
$$
\partial^r(\calK \times^r \calL)
\cong \left(\partial^r \calK \times^r \calL\right) \sqcup (-1)^p \left( \calK \times^r   \partial^r \calL\right).
$$
\end{lemma}

\begin{proof}
Clearly,
$
(\calK_1 \sqcup \calK_2) \times \calL = (\calK_1 \times \calL) \sqcup (\calK_2 \times \calL)
$
so that
$$
(\calK_1 \sqcup \calK_2) \times^r \calL
= \red \left((\calK_1 \sqcup \calK_2) \times \calL\right)
\cong \red (\calK_1 \times \calL) \sqcup \red (\calK_2 \times \calL)
$$
which proves (i).  We now prove (ii).  Let $K$ and $L$ be the oriented manifolds with faces underlying $\calK$ and $\calL$ respectively.
A principal face of $K\times L$  has either the form  $P\times D$,
where $P$ is a principal face of $K$ and $D$ is a component of $L$,
or    the form   $C \times Q$, where $C$ is a component of $K$ and $Q$ is a principal face of $L$.
The orientation of $P\times D \subset \partial(K \times L)$
inherited from $K \times L$ coincides with the product orientation of $P\times D$ where   $P\subset \partial K$  inherits orientation  from $K$.
The orientation of $C \times Q \subset \partial(K \times L)$
inherited from $K \times L$ differs from the product orientation of $C\times Q$,
 where $Q\subset \partial L$  inherits orientation  from $L$,  by the sign $(-1)^p$.  So
$$
\partial(\calK \times \calL)
\cong \left(\partial \calK \times \calL\right) \sqcup (-1)^p \left( \calK \times   \partial \calL\right).
$$
Therefore
\begin{eqnarray*}
\partial^r(\calK \times^r\calL) \ =  \ \partial^r\red(\calK \times \calL) &=&  \partial^r(\calK \times \calL)\\
&=& \red \partial (K \times L)\\
&\cong& \left(\partial \calK \times^r \calL\right) \sqcup (-1)^p \left( \calK \times^r   \partial \calL\right).
\end{eqnarray*}
We conclude thanks to \eqref{rc=cr}.
\end{proof}

\begin{lemma}
The cross product of polychains  induces  a bilinear map
\begin{equation} \label{homological_cross} \times:\widetilde H_*(X) \times \widetilde H_*(Y)   \longrightarrow   \widetilde H_*(X \times Y).\end{equation}
\end{lemma}

\begin{proof}
Let $\calK$ be a  polycycle in $X$
and   $\calL$ be a  polycycle in $Y$.
  Lemma \ref{properties_cross}.(ii) implies that $\calK \times^r \calL$ is   a   polycycle.
 This polycycle is the reduction of $\calK \times \calL$, and therefore  $\calK \times \calL$ also is a polycycle.
  We claim that assigning to    $(\calK, \calL)$   the homology class
  $\lb \calK \times^r \calL \rb= \lb \calK \times  \calL \rb$   one obtains a well defined   pairing \eqref{homological_cross}.
Let us prove the independence of the choice of $\calK$ in its homology class (the second variable is treated similarly).
Consider     two homologous polycycles $\calK_1$ and $\calK_2$    in $X$, and
let  $\calP_1,\calP_2$ be  polychains in $X$
such that  $$\red (\calK_1)\sqcup \partial^r \calP_1   \cong \red (\calK_2)\sqcup \partial^r \calP_2 .$$
 Lemma \ref{properties_cross}.(i) implies  that
$$
(\red (\calK_1)\times^r \calL ) \sqcup (\partial^r \calP_1 \times^r \calL )  \cong
(\red (\calK_2) \times^r \calL ) \sqcup (\partial^r \calP_2 \times^r \calL).
$$
For  $i\in\{1,2\}$,   formula   \eqref{rc=cr} gives
$\red (\calK_i)\times^r \calL \cong  \red( \calK_i \times^r \calL)$.
Since $\partial^r \calL=\varnothing$,     Lemma \ref{properties_cross}.(ii) gives
$\partial^r \calP_i \times^r \calL=\partial^r( \calP_i \times^r  \calL )$.
Therefore $\calK_1 \times^r \calL \simeq \calK_2 \times^r \calL$.

The linearity of   \eqref{homological_cross}
  in the first variable follows from   Lemma \ref{properties_cross}.(i) and the equality $(k\calK)\times \calL = k(\calK\times \calL)$
for all $k\in \kk$. The linearity in the second variable is proved similarly. \end{proof}

\subsection{Remarks}\label{singsing}

1. A  polychain  derived from a  singular manifold $  \kappa:  M\to X$ (see Section~\ref{polychains}) is a polycycle if and only if $\partial M =  \varnothing  $.
The  oriented   bordism classes of $n$-dimensional singular manifolds $  \kappa:   M\to X$
with $\partial M =  \varnothing   $ form an abelian group $\Omega_n(X)$,   called the   \index{oriented  bordism group}  \emph{$n$-dimensional   oriented  bordism group} of $X$.
Treating  singular manifolds as polychains, we obtain an additive map  $\Omega_n(X)  \to   \widetilde H_n(X)$.
 By   Remark~\ref{rerererer}.2 below, this map is not surjective for some $n,X$, and $\kk=\ZZ$. Thus,
 some   face homology classes over $ \ZZ$ are not representable by singular manifolds.

2. For a  topological pair $(X,Y)$ and an integer $n\geq 0$, we   define the $n$-th \index{face homology!relative} \emph{relative face homology} $\widetilde{H}_n(X,Y)$ as follows.
Given  $n$-polychains $\calK_1,\calK_2$ in $X$, we write $\calK_1 \simeq_Y \calK_2$,
if there exist  $(n+1)$-polychains $\calL_1,\calL_2$ in $X$ and $n$-polychains $\calN_1,\calN_2$ in $Y$ such that
$$ \red (\calK_1)\sqcup \partial^r \calL_1  \sqcup \iota_\ast (\calN_1)  \cong \red (\calK_2)\sqcup \partial^r \calL_2 \sqcup  \iota_\ast ( \calN_2) $$
where   $\iota: Y\hookrightarrow X$ is the inclusion map.
An $n$-polychain $\calK$ in $X$  is a \index{polycycle!relative} \emph{polycycle relative to $Y$} if $\partial^r \calK$ is the image   of an $(n-1)$-polychain  in $Y$ under  $\iota $. Set
$$
\widetilde{H}_n(X,Y)=\{\hbox{$n$-polycycles in $X$ relative to $Y$}\} /\simeq_Y.
$$
The properties of the face homology of  topological spaces  stated above directly extend to the face homology of topological pairs.

\section{Face homology versus singular homology}\label{sect-rell}

 In this section, we construct   two natural transformations   $[-]:\widetilde {H}_\ast\to H_\ast$ and   $\lb - \rb: H_\ast\to \widetilde {H}_\ast$    relating
  face homology   to singular homology.

\subsection{Preliminaries}\label{singular_homology}
For an  integer   $n\geq 0$, the symbol  $\Delta^n$  denotes the standard $n$-simplex that is
the convex hull of the standard basis   $(e_0,\dots, e_n)$    of~$\RR^{n+1}$.
We   endow $\Delta^n$ with orientation induced by the
order of its vertices,   i.e$.$ the orientation represented by the basis
  $(\overrightarrow{e_0e_1},\overrightarrow{e_1e_2}, \dots,\overrightarrow{e_{n-1}e_n})$
in the tangent space of $\Delta^n$ at any point.
Each subset $A=\{i_0, i_1,\dots,i_r\}$ of $\{0,\dots,n\}$
with $i_0<i_1<\dots < i_r$ and  $0\leq r\leq n$
determines an   affine   map $e_A: \Delta^r \to \Delta^n$ carrying the vertices $e_0, e_1, \dots, e_r$ of $\Delta^r$
to the vertices $e_{i_0}, e_{i_1}, \dots, e_{i_r}$ of $\Delta^n$, respectively;
 the image of   the map $e_A$  is   the \index{face!combinatorial} \emph{combinatorial face} of $\Delta^n$ corresponding to $A$.

A  \index{singular simplex} {\it singular $n$-simplex} in a topological space $X$ is  a continuous map $\Delta^n \to X$.
 A \index{singular chain} {\it singular $n$-chain} in $X$ is a  finite formal linear combination
 of singular $n$-simplices with coefficients in $\kk$.
The \index{singular simplex!boundary of} \emph{boundary} of a singular $n$-simplex $\sigma:\Delta^n \to X$ is the singular $(n-1)$-chain
\begin{equation}\label{boundary}
\partial \sigma = \sum_{a=0}^n (-1)^a \cdot \sigma   e_{\hat a}
\quad \hbox{where} \ \hat a=\{0, 1, \dots, n\} \setminus \{a\}.
\end{equation}   The boundary of singular  simplices extends to singular chains by linearity.
The  modules of singular chains together with the boundary homomorphisms
form the \index{singular chain complex} \emph{singular chain complex} $C_\ast (X)$ of $X$.
Its homology is the \index{singular homology} \emph{singular homology} $H_\ast(X)$ of $X$  (with coefficients in $\kk$).

\subsection{ The   transformation     $[-]$}  \label{fundcla}
 Consider an  $n$-dimensional   oriented manifold with faces $K$.
Each weight $u : \pi_0(K) \to  \kk$ determines a homology class
$$
[K,u]=\sum_{C} u(C) \big[C\big] \ \in  \bigoplus_C  H_n (C,\partial C) =H_n (K,\partial K)
$$
where $C$ runs over all connected components of $K$
and $[C]\in H_n (C,\partial C)$
is the   fundamental class of $C$.
 We say that  a  partition $\varphi$ on $K$ is  \index{partition!compatible with weight} {\it compatible} with   $u$
if for any   principal face  $P$ of $K$,
\begin{equation}\label{compatibility}
  \sum_Q \deg(\varphi_{P,Q} ) \, u(K^Q) =0
\end{equation}
where $Q$ runs over all (principal) faces of $K$ of the same type as $P$ and  $K^Q$
is   the connected component of $K$ containing $Q$.

 \begin{lemma}\label{partit} 
  Let $\varphi$ be a  partition  on $K$ compatible with a
weight $u : \pi_0(K) \to  \kk$.  Then there  is  a unique homology class   $[K_{\varphi},u]\in   H_n(K_{\varphi})$
whose image in $   H_n(K_{\varphi},    (\partial K)_{\varphi})$ is equal to the image of
$[K,u]$ under the map
$ H_n (K,\partial K)\to H_n(K_{\varphi},  (\partial K)_{\varphi})$  induced by the projection $ K\to K_{\varphi} $.
\end{lemma}

\begin{proof}
Consider the commutative diagram
$$
\xymatrix{
H_n(\partial K)  \ar[r]  \ar[d]^-{\pi_*} & H_n(K) \ar[r]  \ar[d]^-{\pi_*}  & H_n(K,\partial K) \ar[r]^-{\partial_*} \ar[d]^-{\pi_*} & H_{n-1}(\partial K) \ar[d]^-{\pi_*}\\
H_n\left((\partial K)_\varphi\right)  \ar[r]  & H_n\left(K_\varphi\right) \ar[r]
&  H_n\left(K_\varphi,(\partial K)_\varphi \right) \ar[r]^-{\partial_*} & H_{n-1}\left((\partial K)_\varphi\right),
}
$$
where  the vertical maps are   induced by the projection $\pi:K\to K_\varphi$
and   each row is a part of the long exact sequence of a topological pair.
We have $H_{n}( (\partial K)_{\varphi})=0$ since $(\partial K)_{\varphi}$ is an $(n-1)$-dimensional polyhedron.
Hence,  it is enough to show that
$$
\pi_* \partial_*([K,u])=0 \in H_{n-1}\left((\partial K)_\varphi\right).
$$
 Consider   the  commutative square
$$
\xymatrix{
 \partial_*([K,u])\in \!\!\! \!\!\!\!\!\! \!\!\! \!\!\! \!\!\!
 &H_{n-1}(\partial K) \ar[r]^{j} \ar[d]^-{\pi_*} & H_{n-1}\left(\partial K, K_{n-2}\right)  \ar[d]^-{\pi_*} \\
& H_{n-1}\left((\partial K)_\varphi\right) \ar[r]^-{j_\varphi} & H_{n-1}\left((\partial K)_\varphi, (K_{n-2})_\varphi\right)
}
$$
where $j$ and $j_\varphi$ are the inclusion homomorphisms.
Since $(K_{n-2})_\varphi$ is an $(n-2)$-dimensional polyhedron, $\Ker {j_\varphi}=0$   and it suffices to prove that ${j_\varphi} \pi_* \partial_*([K,u])=0 $ or,  equivalently, that $  \pi_* j \partial_*([K,u])=0 $.
We have
$$
\partial_*([K,u]) = \sum_C u(C) \partial_*([C]) = \sum_C u(C) [\partial C]
$$
where the sum runs over the connected components $C$ of $K$.
Then
$$
j \partial_*([K,u])=\sum_{C} u(C) \sum_{P \subset C} [P] = \sum_{P} u(K^P)\, [P]
$$
where $P$ runs over  all principal faces of $K$ and   $K^P$ is the connected component of~$K$ containing $P$.
Pick  a face $P_i \in i$   in each type $i$ of principal faces of~$K$.
Then
\begin{eqnarray*} \pi_* j \partial_*([K,u]) &=& \sum_{P} u(K^P)\, \pi_*([P])\\
 &=&\sum_{i} \sum_{Q\in i}  u({K^Q})\, \pi_*([Q] )\\
&=& \sum_{i} \sum_{Q \in i }  u({K^Q})\, \pi_* \big(\deg(\varphi_{P_i,Q})\cdot  (\varphi_{P_i,Q})_{*}([P_i]) \big)\\
&=& \sum_{i} \Big(\sum_{Q \in i}  u({K^Q}) \deg(\varphi_{P_i,Q})\Big) \, \pi_* \left([P_i] \right)=0
\end{eqnarray*}
where at the last step we use the compatibility condition \eqref{compatibility}.
\end{proof}

It follows from the definitions that a polychain $ (K,\varphi,u, \kappa)$ in a topological space $X$ is a polycycle
if and only if $u$ and $\varphi$ are compatible in the sense of \eqref{compatibility}.
Therefore, given an  $n$-polycycle $\calK=(K,\varphi,u, \kappa)$ in $X$,  Lemma \ref{partit} gives   the homology class $[K_{\varphi},u]\in H_n(K_\varphi)$.
Since the map $  \kappa:K\to X $ is compatible with~$\varphi$, it induces a continuous map $K_{\varphi}\to X$   denoted by $\kappa_{\varphi}$.
  We define
$$
[\calK] = (\kappa_{\varphi})_\ast ([K_{\varphi},u])\in H_n (X).
$$

The homology class $[\calK]$ can be represented by   explicit singular cycles which are best described  in terms of locally ordered triangulations.  A \index{triangulation!local order on} \emph{local order} on a triangulation $T$ of a topological space is a  binary relation   on the set of vertices of $T$
 which restricts to a total order on the set of vertices of any simplex of $T$.
For example, any total order  on the set of vertices of $T$ is a local order on $T$.
A triangulation endowed with a local order is \index{triangulation!locally ordered} \emph{locally ordered}. We say that
a locally ordered smooth triangulation  $T$ of  $K$  \index{triangulation!fitting the partition} \emph{fits the partition  $ \varphi $}   if,
for any faces $F,G$ of $K$ of the same type,
the identification  map  $ \varphi_{F,G} :F\to G$ is a simplicial isomorphism   preserving the local order on the vertices.
To construct such a locally ordered triangulation  one can take a triangulation  $T$ of  $K$  provided by Lemma \ref{partition}
and lift an arbitrary total order $\leq$ on the set of vertices of $T_\varphi$ to $T$. More precisely,
 denote by $\pi:K \to K_\varphi$ the canonical projection and, for any   vertices $a,b\in T$,  declare that $a\leq b$ if  $\pi(a) \leq \pi(b)$.
Since any simplex of $T$ projects isomorphically onto a simplex of $T_\varphi$,
this gives a local order on $T$ which, obviously, fits $\varphi$.

  Pick a   
  locally ordered smooth triangulation $T$ of $K$ which fits   $\varphi$.
Each $r$-simplex $\Delta$ of $T$  with $r\geq 0$ determines
  a  singular simplex  in $K$ denoted by  $\sigma_\Delta$  and obtained as the composition of the affine isomorphism $\Delta^r \to \Delta$
  preserving the order of the vertices with the inclusion $\Delta \hookrightarrow K$.
We define  the \index{singular chain!fundamental} \emph{fundamental  $n$-chain}
\begin{equation} \label{sigma_T_u}
\sigma = \sigma(T,u) =\sum_\Delta \varepsilon_\Delta   u(K^\Delta)\,    \sigma_\Delta  \ \in C_n(K)
\end{equation}
where $\Delta$ runs over all $n$-simplices of~$T$,
  $K^\Delta$ is the connected component of~$K$ containing~$\Delta$,
$\varepsilon_\Delta=+1$ if the orientation of $\Delta$ induced by that of~$K$ is compatible
 with the order of the vertices of $\Delta$ and $\varepsilon_\Delta=-1$ otherwise.
 Clearly, the image of~$\sigma$ in $C_n(K,\partial K)$ is a relative $n$-cycle representing $[K,u]$.
 Projecting $\sigma$ to $K_\varphi$, we obtain a singular   $n$-chain   $\sigma_{\varphi}\in C_n(K_\varphi) $.  
The compatibility of $\varphi$ and $u$ implies that  $\sigma_{\varphi} $ is an $n$-cycle.
Therefore   
$[\sigma_\varphi]\in H_n(K_\varphi)$     satisfies the requirements of   Lemma~\ref{partit}
 so that   $[K_\varphi,u]=[\sigma_\varphi]$.  It follows that  $[\calK]$ is represented by the singular $n$-cycle
$$
(\kappa_\varphi)_*(\sigma_\varphi) = \kappa_*(\sigma)= \sum_\Delta \varepsilon_\Delta u(K^\Delta)\, \kappa \sigma_\Delta \ \in C_n(X).
$$

\begin{lemma}\label{from_sc_to_pcdebut}
The formula $   \lb \calK \rb   \mapsto [\calK]$   defines a  linear map $[-]:\widetilde{H}_n(X) \to {H}_n(X)$.
  Moreover, $[-]$ is a natural transformation from $\widetilde H_n$ to $H_n$.
\end{lemma}
\begin{proof} It follows   from the definitions that $[\calK]\in H_n(X)$  depends only
on the diffeomorphism class of $\calK$, that $[\calK]=[\red (\calK)]$, and that $[k\calK]=k[\calK]$ for any $k\in \kk$.
Moreover, $[\calK_1 \sqcup \calK_2]=[\calK_1]+ [\calK_2]$
for any $n$-polycycles $\calK_1,\calK_2$ in $X$.
Therefore,   to prove the first  claim  of the lemma,
it is enough to show that  $[\partial \calL]=0$ for any $(n+1)$-polychain $\calL=(L, \psi, v,  \lambda)$ in $X$.
  For this, pick a    locally ordered smooth triangulation $T$ of $L$ that fits  $\psi$ and   consider the  singular chain
$$
  \sigma =\sigma(T,v)  = \sum_\Delta \varepsilon_\Delta v (L^\Delta)\,    \sigma_\Delta   \ \in C_{n+1}(L) .
$$
Here $\Delta$ runs over all $(n+1)$-dimensional simplices  of   $T$,
$\varepsilon_\Delta$ is the sign determined by the orientation of $L$ and the order of the vertices of $\Delta$,
and $L^\Delta $ is the     component of~$L$ containing~$\Delta$.
Projecting   $\sigma$   to $L_\psi$ we obtain a singular chain   $\sigma_{\psi}$   in $ L_{\psi}$.
Next we consider the   $n$-polycycle  $\partial \calL=( L^\partial , \psi^\partial, v^\partial,\lambda^\partial)$.
The triangulation $T$ of $L$ induces a triangulation $T^\partial$ of $L^\partial$. The local  order on the set of  vertices
of $T$ restricts to a  local  order on the set of vertices of   $T^\partial$.
Consider the fundamental   $n$-chain $\tau   =\sigma (T^\partial, v^\partial)$   in $L^\partial$ as defined before the statement of the lemma.
Projecting   $\tau$   to the quotient space   $(L^\partial)_{\psi^\partial}$
we obtain a singular $n$-cycle   $\tau_{\psi^\partial}$
representing  $$\big[(L^\partial)_{\psi^\partial}, u^\partial\big]  \in H_n\big((L^\partial)_{\psi^\partial}\big).$$
The natural map $\iota:  L^\partial \to L$ induces a map $\iota_\psi:  (L^\partial)_{\psi^\partial}   \to L_\psi$
carrying   $\tau_{\psi^\partial}$  to  $\partial \sigma_{\psi}$.
 By   definition,  $\lambda^\partial=\lambda \iota: L^\partial\to X$. Therefore $(\lambda^\partial)_{\psi^\partial}=
\lambda_\psi \iota_\psi :  (L^\partial)_{\psi^\partial}  \to  X$
where   $\lambda_\psi:L_\psi\to X$ is the map  induced by $\lambda: L\to X$.
Hence,
$$
[\partial \calL]
=  \big((\lambda^\partial)_{\psi^\partial}  \big)_*  ([ \tau_{\psi^\partial}])
=(\lambda_\psi \iota_\psi)_* ([ \tau_{\psi^\partial} ])
=(\lambda_\psi  )_*( [\partial \sigma_{\psi}] ) =  0.
$$

  To prove the second  claim   of the lemma,  consider a continuous map $f:X\to Y$.
For any $n$-polycycle $\calK=(K,\varphi,u,\kappa)$ in $X$, we have
$$
f_*\left([\calK]\right) =  f_*(\kappa_{\varphi})_\ast ([K_{\varphi},u]) = (f \kappa_{\varphi})_\ast ([K_{\varphi},u]) =\big( (f \kappa)_{\varphi}\big)_\ast ([K_{\varphi},u])
= [f_*(\calK)]
$$
since $f_*(\calK)=(K,\varphi,u,f\kappa)$ by definition.
\end{proof}

\subsection{The    transformation $\lb - \rb$}\label{firsttransf}

Let $X$ be a topological space and   let $n\geq 0$ be an integer.
We  associate with each singular $n$-chain $\sigma$ in $X$ an    $n$-polychain   $\calP(\sigma)$ in $X$.
Pick an expansion $\sigma=\sum_i k_i \sigma_i  $,
where $i$ runs over a finite   set of indices,
$k_i\in \kk   $,  and $\{\sigma_i \}_i$ are   singular $n$-simplices in $X$.
 Let $K$ be    the    manifold with faces obtained as a  disjoint union of copies  $(\Delta^n)_i$
of   $\Delta^n$ numerated by  all  $i $.
We define a partition $ \varphi $ on      $K$ as follows:
a      face $F$ of $(\Delta^n)_i$ corresponding to a  set
$A   \subset   \{0,\dots,n\}$
 and a      face   $F'$   of $(\Delta^n)_{i'}$
corresponding to a  set $A'   \subset   \{0,\dots,n\}$ are   declared to be of the same type
if, and only if,  $A$ and $A'$ have the same   cardinality  
$r\leq n+1$,   and $\sigma_i   e_A = \sigma_{i'}   e_{A'}:\Delta^{r-1} \to X$ (where $e_A,e_{A'}$ are the maps defined   in Section~\ref{singular_homology}).
  Then we set $\varphi_{F,F'}= e_{A'}   e_A^{-1}: F \to F'$.
Clearly, the map   $\kappa=  \coprod_i  \sigma_i:K \to X$  is compatible with    $\varphi$.
We   define a weight $u:\pi_0(K) \to \kk$ by $u((\Delta^n)_i) = k_i$ for all~$i$.
The tuple  $ (K,\varphi,u,\kappa)$ is an   $n$-polychain in $X$   depending on the choice of the expansion $
\sigma=\sum_i k_i \sigma_i  $. However, the polychain $\calP(\sigma)= \red (K,\varphi,u,\kappa)$
does not depend on this choice. Indeed, any two expansions of $\sigma$ may be related by the following operations:
replacement of   $k  \sigma_{\bullet} +l \sigma_{\bullet}$ by $(k+l) \sigma_{\bullet}$ for any $k,l \in \kk$ and any singular $n$-simplex $\sigma_{\bullet}$ in $X$;
addition  of a term  $0 \sigma_{\bullet}$ for an arbitrary singular $n$-simplex $\sigma_{\bullet}$  in $X$; the inverse operations.
It is easy to see that $\calP(\sigma)$ is preserved under these transformations. By definition, if $\sigma=0$, then $\calP(\sigma)=\varnothing$.
The face   homology class $\lb \calP(\sigma) \rb$   of the polychain $\calP(\sigma)$ will be     denoted by $\lb \sigma \rb$.

 \begin{lemma}\label{mapP} If   $\sigma$ is a cycle, then   $  \calP(\sigma)$ is a polycycle.  The formula
 $[\sigma] \mapsto \lb \sigma \rb$, applied to singular $n$-cycles in $X$, defines a linear map $\lb-\rb:H_n(X) \to  \widetilde{H}_n(X)$.
  Moreover, $\lb -\rb$ is a natural transformation from $H_n$ to $\widetilde H_n$.
\end{lemma}

\begin{proof} We check first that for any singular $n$-chain $\sigma$ in $X$,   \begin{equation}\label{pureboundaries} \partial^r \calP(\sigma)  \cong  \calP(\partial \sigma) .\end{equation} Pick an expansion   $\sigma=\sum_i k_i \sigma_i  $ such that the simplices $\{\sigma_i\}_i$ are pairwise distinct and  $k_i\neq 0$ for all $i$. Then the associated polychain $ (K,\varphi,u,\kappa)$   is reduced and   $\calP(\sigma)=  (K,\varphi,u,\kappa)$.   A connected component $P$ of $\partial \calP(\sigma)=(K^\partial, \varphi^\partial, u^\partial, \kappa^\partial)$ is nothing but a
  principal face   of   $(\Delta^n)_i\subset K$ for some $i=i(P)$
corresponding to the complement    of a singleton $ a_P\in \{0,\dots,n\}$. By the definition of $u^\partial$, we have
$u^\partial (P)=k_{i(P)}$. We compute $\red_+ (\partial \calP (\sigma))$ as described  in Section \ref{polychains}. Pick a representative
$P$ for each type of connected components of $K^\partial$,   and let $K^\partial_+ \subset K^\partial$ be the    union of these representatives.
 Restricting $\varphi^\partial $ and $\kappa^\partial$ to $K^\partial_+$ we obtain a partition $ \varphi^\partial_+$ on  $K^\partial_+$ and a compatible map $\kappa^\partial_+:  K^\partial_+\to X$.
 The  weight $  u^\partial_+ $ on $K^\partial_+$
 is  evaluated  on each component $P$ of $K^\partial_+$  by
$$
 u^\partial_+ (P)= \sum_Q  \deg(\varphi_{P,Q}) \, u^\partial (Q)
 = \sum_Q (-1)^{a_P+a_Q} k_{i(Q)} = (-1)^{a_P} \sum_Q (-1)^{a_Q} k_{i(Q)}
$$
where $Q$ runs over all components of  $K^\partial$   of the same type as $P$.
Note that $\sum_Q (-1)^{a_Q} k_{i(Q)}$
is the total coefficient of the singular simplex   $\sigma_{i(P)}  e_{\widehat {a_P}}:\Delta^{n-1}\to X$  in $\partial \sigma$. Also,
$(-1)^{a_P}$   is the degree of   $e_{\widehat {a_P}}:\Delta^{n-1} \to P$
(recall that  $\Delta^{n-1}$ is oriented as in Section \ref{singular_homology}
while $P \subset \partial (\Delta^n)_i$ inherits orientation  from $(\Delta^n)_i$   where $i=i(P)$).
We conclude that the polychain  $\red_+(\partial \calP(\sigma))$ consists of $\calP (\partial \sigma)$ and eventually several connected  components of  weight zero. Hence
$$\partial^r \calP(\sigma) =\red(\partial \calP(\sigma)) = \red_0  \red_+(\partial  \calP(\sigma)) \cong \calP (\partial \sigma).$$
This proves \eqref{pureboundaries}.  The first assertion of the lemma follows.

  Next   we claim that
\begin{equation}\label{additivity}
\calP(\sigma+\tau) \simeq \calP( \sigma) \sqcup \calP(\tau)
\end{equation}
for any    singular $n$-cycles  $\sigma,\tau$ in $X$.
To see this, pick   expansions $
\sigma=\sum_{i } k_i \sigma_i$, $\tau=\sum_{j } l_j \tau_j
 $ and  let $\calK=(K,\varphi,u,\kappa)$, $\calL=(L,\psi,v,\lambda)$ be the associated polychains, respectively.
 Consider the cylinder  polychain $\overline {\calK  \sqcup   \calL} \cong \overline \calK  \sqcup \overline \calL$
  (as defined in the proof of Lemma \ref{additivityandrescaling})
and modify its partition   by additionally declaring that
for  any face~$F$ of   $(\Delta^n)_i\subset K$ corresponding to  $A   \subset  \{0,\dots,n\}$
and for any face $G$  of $(\Delta^n)_j \subset L$ corresponding to $B \subset \{0,\dots,n\}$
such that $\sigma_i  e_A = \tau_j   e_B$, the faces $ F \times \{0\} $ and $ G \times \{0\} $ of  $\overline {\calK  \sqcup   \calL}$  are of the same type, and the corresponding identification map  is
  $e_B  e_A^{-1}\times \id_{\{0\}}:F\times \{0\}\to G\times \{0\} $.
  The resulting  $(n+1)$-polychain, $\mathscr M$, in~$X$   satisfies
  $$
  \red_+\partial    \mathscr M  \cong  \red_+(\calK)  \sqcup  \red_+(\calL)
  \sqcup \red_+(-\calR) \sqcup  (\rm {a\,\, polychain \,\, with\,\,  zero\,\,  weight})
  $$
where $\calR$ is the polychain associated with the expansion
$ \sum_{i } k_i \sigma_i + \sum_{j  } l_j \tau_j$ of $\sigma+\tau$.
Hence, $  \partial^r  \mathscr M   \cong  \calP( \sigma)  \sqcup \calP(\tau) \sqcup (-\calP(\sigma+\tau))  $ and our claim follows.

If $\kk$ has no zero-divisors, then $\calP(k\sigma) \cong  k \calP(\sigma) $
for any    singular $n$-chain $\sigma$ in $X$ and any non-zero $k\in \kk$. For   an  arbitrary $\kk$ and  all $k\in \kk$,   we have
\begin{equation}\label{pureboundaries+}  \calP(k\sigma) \cong \red (k \calP(\sigma))  \simeq k \calP(\sigma) . \end{equation}
  Equalities \eqref{pureboundaries} -- \eqref{pureboundaries+}
imply that the formula $[\sigma] \mapsto \lb \calP (\sigma) \rb$
defines   a linear map $\lb-\rb: H_n(X)\to \widetilde{H}_n(X)$.

  To prove the last claim of the lemma,   consider a continuous map $f:X\to Y$.
Let $\sigma$ be a singular $n$-cycle  in $X$,  and let $\calK=(K,\varphi,u,\kappa)$ be the $n$-polycycle associated to an expansion $\sum_i k_i \sigma_i$ of $\sigma$.
The $n$-polycycle associated to the expansion $\sum_i k_i (f \sigma_i)$ of $f_*(\sigma)$  has   the form $\calK'=(K,\varphi', u, f\kappa)$
and   differs from $f_*(\calK)= (K,\varphi,u,f\kappa)$ only in   the partition. Modifying appropriately  the partition of the cylinder polychain $\overline{f_*(\calK)}$,
we obtain an $(n+1)$-polychain $\calM$ in $X$ such that
$$
\red_+ \partial \calM = \red_+ f_*(\calK) \sqcup \red_+(-\calK') \sqcup  (\rm {a\,\, polychain \,\, with\,\,  zero\,\,  weight}).
$$
We deduce that $\partial^r  \calM = \red f_*(\calK) \sqcup  \red(-\calK')$ and
$$
f_*(\lb \calP(\sigma) \rb)= f_*(\lb \calK \rb) = \lb f_*(\calK) \rb = \lb \calK' \rb = \lb  \calP(f_*(\sigma)) \rb.
$$

\up
\end{proof}

  The next theorem  implies  that $H_n(X)$ is canonically isomorphic to
  a direct summand  of $\widetilde{H}_n(X)$.

\begin{theor}\label{mainonpolychains}
We have
$ [-]\circ \lb-\rb = \id:H_n(X)\to H_n(X)$.
\end{theor}
\begin{proof}
Let   $\sigma =\sum_i k_i \sigma_i$  be a singular $n$-cycle in $X$ and let
$\calP(\sigma)=(K, \varphi, u, \kappa)$  be the corresponding   reduced   $n$-polycycle. Then
$$
 \left[\lb [\sigma] \rb \right] = [\calP(\sigma)]= (\kappa_{\varphi})_*\left([K_{\varphi},u]\right) =
\Big[\sum_i k_i \sigma_i\Big]= [\sigma] \in H_n(X).
$$
  Here the third equality  is obtained by considering the  tautological locally ordered smooth triangulation $T$ of  $K$
and the corresponding fundamental $n$-chain $\sigma(T,u)$
(see  the paragraph preceding Lemma \ref{from_sc_to_pcdebut}).
\end{proof}

\subsection{Cross product re-examined} \label{cross_product_re-examined}
The following lemma shows that the transformation
$[-]: \widetilde H_*\to H_\ast$ carries the cross product $\times$ in face homology
to the standard cross product $\times$ in singular homology.

\begin{lemma}\label{twocrossproducts}
For any  topological spaces $X$, $Y$, the following diagram   commutes:
\begin{equation}\label{cross_to_cross}
\xymatrix{
\widetilde H_*(X) \times \widetilde H_*(Y)  \ar[r]^-{\times}  \ar[d]_{[-]\times [-]}
& \widetilde H_*(X \times Y) \ar[d]^{[-]}\\
 H_*(X) \times  H_*(Y)  \ar[r]^-{\times} &  H_*(X \times Y).
}
\end{equation}
\end{lemma}

We   first  recall  the definition of  the map $ \times : H_*(X) \times H_*(Y) \to H_*(X\times Y)$ and then prove Lemma~\ref{twocrossproducts}.
Fix integers $p,q\geq  0$.
Any $p$-element subset $S$ of $\{1,\dots,p+q\}$
determines    non-decreasing    maps $$\alpha=\alpha_S: \{0,\dots,p+q\} \to \{0,\dots,p\} \quad {\rm {and}} \quad \beta=\beta_S: \{0,\dots,p+q\} \to \{0,\dots,q\}$$
  such that  $ \alpha(0)=\beta(0) = 0 $   and for any $i=1,\dots, p+q$,
$$
(\alpha(i),\beta(i))
= \left\{\begin{array}{ll}
\big(\alpha(i-1)+1,\beta(i-1)\big) & \hbox{ if } i\in S\\
\big(\alpha(i-1),\beta(i-1)+1\big) & \hbox{ if } i\not\in S.
\end{array}\right.
$$
Let  $\omega_S \subset \Delta^p \times \Delta^q$ be the   convex hull
of the set   $\big\{(e_{\alpha(0)},e_{\beta(0)}), \ldots , (e_{\alpha(p+q)},e_{\beta(p+q)})\big\}$.

\begin{lemma}\label{EZ}
The set $\omega_S$ is an embedded $(p+q)$-simplex  in $\Delta^p \times \Delta^q$  with vertices $ \{(e_{\alpha(i)},e_{\beta(i)})\}_{i=0}^{p+q}$.  
The  simplices $\{\omega_S\}_S$ and their faces form   a triangulation of $\Delta^p \times \Delta^q$.
\end{lemma}

\begin{proof} This lemma is   well known but we give a proof for completeness.
For $n\geq 0$, denote by $\mathbb{A}^n$ the affine space  formed by the points of $\RR^{n+1}$ with sum of coordinates~$1$.
The basis   $(e_0,\dots,e_n)$ of $\RR^{n+1}$ is  an affine basis of $\mathbb{A}^n$   and ${\Delta^n \subset \mathbb{A}^n}$.
Consider the basis
  $(v_1,\dots,v_p)=(\overrightarrow{e_0e_1}, \overrightarrow{e_1e_2}, \ldots , \overrightarrow{e_{p-1}e_p})$
  of  the vector space underlying $\mathbb{A}^p$
  and the basis $(v_{p+1},\dots,v_{p+q})=(\overrightarrow{e_0e_1}, \overrightarrow{e_1e_2}, \ldots,  \overrightarrow{e_{q-1}e_q})$
  of  the vector space underlying $\mathbb{A}^q$.
Then $(v_1,\dots,v_{p+q})$ is a basis of the vector space underlying   the product affine space   $\mathbb{A}^p \times \mathbb{A}^q$.

Recall that a \index{shuffle} \emph{$(p,q)$-shuffle} is a permutation $s$ of $\{1,\dots,p+q\}$
such that $$s(1)<s(2)< \cdots <s(p)\quad {\rm{and}} \quad s(p+1)<\cdots < s(p+q).$$
  Any    $p$-element subset $S$ of $\{1,\dots,p+q\}$ determines a unique $(p,q)$-shuffle $s$ such that $S=s(\{1,\dots,p\})$.
The first claim of the lemma follows from the fact that the  vector   basis $(v_{s^{-1}(1)},\dots,v_{s^{-1}(p+q)})$
underlies  the   set of vertices of $\omega_S \subset  \mathbb{A}^p \times \mathbb{A}^q$:
\begin{equation}\label{sequence}
\xymatrix @!0 @R=1cm @C=3.0cm  {
 (e_{\alpha(0)},e_{\beta(0)}) \ar@{|->}[r]^-{v_{s^{-1}(1)}}
& (e_{\alpha(1)},e_{\beta(1)}) \ar@{|->}[r]^-{v_{s^{-1}(2)}} &
\ \cdots \ \ar@{|->}[r]^-{v_{s^{-1}(p+q)}} & (e_{\alpha(p+q)},e_{\beta(p+q)}).
}
\end{equation}

To prove the second claim, observe  that given $n+1$ affinely independent points $f_0, \dots, f_n$  in an $n$-dimensional affine space,
an arbitrary point $f_0  +\sum_{i=1}^n t_i\, \overrightarrow{f_{i-1}f_i}$ of this space
(with $t_1,\dots,t_n\in \RR$) belongs to  the affine simplex spanned by  $f_0, \dots, f_n$  if and only if $1 \geq t_1 \geq \cdots \geq t_n \geq 0$.
Therefore   any point $z \in \Delta^p \times \Delta^q$ expands uniquely as
$$ z=  (e_0,e_0) + z_1 v_1 + \cdots + z_{p+q}   v_{p+q}$$
with $$1\geq z_1 \geq \cdots \geq z_p \geq 0 \quad {\rm {and}} \quad 1\geq z_{p+1} \geq \cdots \geq z_{p+q} \geq 0.$$
  By the same observation and \eqref{sequence},
the inclusion $z\in \omega_S$ for a  $p$-element subset $S$ of $ \{1,\dots,p+q\}$ holds if and only if
$$z_{s^{-1}(1)} \geq z_{s^{-1}(2)}  \geq \cdots \geq  z_{s^{-1}(p+q)}$$ where $s$ is the $(p,q)$-shuffle  determined by~$S$.
Therefore $\Delta^p \times \Delta^q$ is the union
of the  simplices $\{\omega_S\}_S$, and any two of these simplices  meet along a common face.
\end{proof}

For each $p$-element   subset  $S \subset   \{1,\dots,p+q\}$, we turn $\omega_S$ into a singular  simplex   in $\Delta^p \times \Delta^q$
 by sending the ordered vertices  $e_0 < e_1 <\cdots < e_{p+q}$ of $\Delta^{p+q}$ to
\begin{equation}\label{order_omega_S}
(e_{\alpha(0)},e_{\beta(0)}) < (e_{\alpha(1)},e_{\beta(1)}) < \cdots < (e_{\alpha(p+q)},e_{\beta(p+q)})
\end{equation}
respectively.   Summing up over all such $S$ we obtain a singular chain
\begin{equation}\label{signforS}
 \omega_{p,q} = \sum_S  \varepsilon_S\,   \omega_S  \ \in C_{p+q}(\Delta^p\times  \Delta^q)
\end{equation}
where $\varepsilon_S$ is the sign comparing the orientation in $\omega_S$ determined by the order of its vertices \eqref{order_omega_S}
with the product orientation in $\Delta^p\times  \Delta^q$.
The   \index{Eilenberg--Zilber chain map} \emph{Eilenberg--Zilber chain map}
\begin{equation}\label{AWEZ}
{EZ}: C_\ast (X) \otimes C_\ast(Y)    \longrightarrow   C_\ast (X\times Y)
\end{equation}
  is defined  by
$
EZ(\sigma \otimes \tau)=(\sigma\times  \tau)_\ast (\omega_{p,q})
$ for any singular simplices $\sigma: \Delta^p \to X$ and $\tau:\Delta^q  \to   Y $.
Here    $$(\sigma\times  \tau)_\ast \colon  C_\ast(\Delta^p\times  \Delta^q)   \longrightarrow   C_\ast(X \times Y)$$   is the chain map
 induced by
 $\sigma \times \tau : \Delta^p\times  \Delta^q \to X \times Y$.

 The   \index{singular homology!cross product in} \emph{cross product} of   singular    homology classes $x  \in H_p(X)$ and $y  \in H_q(Y)$ is defined by taking any cycles
$\sigma\in C_p(X)$ and $\tau\in C_q(Y)$ representing $x$ and $y$ respectively, and letting $x \times y \in H_{p+q} (X\times Y)$ be
the homology class of $EZ(\sigma \otimes \tau)$.

  \begin{proof}[Proof of Lemma~\ref{twocrossproducts}.]
Let $\calK=(K,\varphi,u,\kappa)$ be a  $p$-polycycle in $X$
and let  $\calL=(L,\psi,v,\lambda)$  be a $q$-polycycle in $Y$. We must prove that
\begin{equation}\label{[KxL]_bis}
[\calK \times \calL] = [\calK] \times [\calL] \in H_{p+q}(X \times Y).
\end{equation}
 Fix   a locally ordered smooth triangulation $T$ of $K$ which fits $\varphi$ and consider the   fundamental $p$-chain
$ \sum_i \, \varepsilon_i \, u(K^i) \,  \sigma_i \in C_p(K)
$
where   $i$ runs over $p$-simplices   of $T$,
$\sigma_i:\Delta^p \to K$ is the  smooth singular simplex determined   by  $i$,
$\varepsilon_i$ is the sign comparing  the orientation induced by the order of the vertices of $i$ to the orientation of $K$,
  and   $K^i$ is the connected component of $K$  containing~$i$.
Then
$$
[\calK]
= \Big[\sum_i \, \varepsilon_i \, u(K^i) \,  \left(\kappa \circ \sigma_i\right) \Big] \in H_p(X)
$$
 where the square brackets on the right-hand side stand for the homology class of a singular  cycle.
 Similarly,  fixing  a locally ordered  smooth triangulation $W$ of $L$ which fits $\psi$, we obtain
$$
[\calL]
= \Big[\sum_j \, \varepsilon_j\, v(L^j) \, \left(\lambda \circ \tau_j\right) \Big] \in H_q(Y)
$$
where $j$ runs over $q$-simplices   of $W$, $\tau_j:\Delta^q \to L$ is the  smooth singular simplex determined by $j$,
 $\varepsilon_j$ is the sign comparing  the orientation induced by the order of the vertices of $j$ to the orientation of $L$,
  and    $L^j$ is the   component of $L$  containing~$j$.
By the definition of the cross product in singular homology,
\begin{eqnarray}
\notag [\calK]  \times [\calL] &=&  \Big[\sum_{i,j} \, \varepsilon_i\, \varepsilon_j\, u(K^i) \,  v(L^j) \,
 EZ(\kappa\sigma_i  \otimes   \lambda\tau_j) \Big]  \\
\label{[K]x[L]_bis}& = &
\Big[ \sum_{i,j,S} \,  \varepsilon_i\,  \varepsilon_j\,  \varepsilon_S \,  u(K^i) \,  v(L^j)
\,\,  (\kappa \sigma_i \times \lambda\tau_j)  \omega_S\Big]
\end{eqnarray}
where $S$ runs over $p$-element subsets of $\{1,\dots, p+q\}$.

For any  simplices $i,j$  as above,  we push forward via $\sigma_i \times \tau_j$
the triangulation  of $\Delta^p\times \Delta^q$
 provided by Lemma~\ref{EZ} to a triangulation of $i\times j \subset K \times L$.
This gives a smooth triangulation $Z$ of $K \times L$.
The  set of vertices $Z^0$ of $Z$ is the cartesian product of the sets of vertices $T^0$ and $W^0$ of $T$ and $W$, respectively;
we endow~$Z^0$ with the product of the binary relations on $T^0$ and $W^0$
determined  by the local orders on $T$ and $W$. This defines a local order on $Z$ which fits the partition $\varphi \times \psi$.
The simplices of $Z$ can be identified with the  triples $(i,j,S)$ as above,
and the corresponding singular simplices in $K\times L$ are  the maps $( \sigma_i \times \tau_j)  \omega_S:\Delta^{p+q}\to K\times L$.
(Here we use   the fact   that the order \eqref{order_omega_S} of the vertices of   $\omega_S$
  is given by the product binary relation on the set of vertices of $\Delta^p\times \Delta^q$
determined by    the natural orders on the   sets   of vertices of $\Delta^p$ and $\Delta^q$.)
Let $\varepsilon_{i,j,S}$ be the sign comparing  the orientation of the simplex $(i,j,S)$
induced by the order of the vertices   to the product orientation of $K\times L$.
 Let $(K\times L)^{i,j,S}$ be the component of $K\times L$ containing the simplex $(i,j,S)$.
Then
\begin{equation}\label{[KxL]}
[\calK \times \calL] =
\Big[\sum_{i,j,S} \varepsilon_{i,j,S}\cdot (u \times v)\left((K\times L)^{i,j,S}\right) \cdot
\big((\kappa \times \lambda)\circ ( \sigma_i \times \tau_j)  \omega_S\big) \Big].
\end{equation}
Clearly,
$$
(u \times v)\left((K\times L)^{i,j,S}\right) = u(K^i) \,  v(L^j).
$$
Note that
$
 \varepsilon_{i,j,S} = \varepsilon_i \varepsilon_j \varepsilon_S
$
since $\varepsilon_i \varepsilon_j$ is the degree of the diffeomorphism
$\sigma_i \times \tau_j: \Delta^p \times \Delta^q \to i\times j$
with respect to the product orientation in  $\Delta^p \times \Delta^q$ and the product orientation in $K \times L$ restricted to
$i\times j$.
Comparing \eqref{[K]x[L]_bis} to \eqref{[KxL]}, we obtain  \eqref{[KxL]_bis}.
\end{proof}

\subsection{Remarks}\label{rerererer} 1.
Though we shall not need it in the sequel, note that the sign $\varepsilon_S$ in \eqref{signforS} can be computed explicitly:
$\varepsilon_S= (-1)^{n_S}$ where  $n_S$ is the number of pairs   $i<j$ with
 $i  \in \{1, \dots , p+q\}\setminus S$ and  $j\in S$. Indeed, in the notation introduced in the proof of Lemma~\ref{EZ},
 the orientation of   $\omega_S$ determined by the sequence \eqref{order_omega_S} is represented by the $(p+q)$-vector
$$
v_{s^{-1}(1)}\wedge \cdots \wedge  v_{s^{-1}(p+q)} = (-1)^{m} v_{1}\wedge \cdots  \wedge v_{p+q}
$$
where $s$ is the $(p,q)$-shuffle  associated with $S$ and $m$ is the number of inversions in~$s$.
Therefore  $\varepsilon_S=(-1)^{m}$ and it remains to observe that $m=n_S$.

The definition of $n_S$ may be also reformulated in terms of the maps $\alpha=\alpha_S$ and $\beta=\beta_S$.
Namely,  $n_S$ is the number of pairs   $i<j$ such that
$\beta(i)=\beta(i-1)+1$ and $\alpha(j)=\alpha(j-1)+1$. This implies the following formula for $n_S$ used, for example, in \cite[Section 4(b)]{FHT}:
$$
n_S = \sum_{1 \leq i<j \leq p+q} \big(\beta(i)-\beta(i-1)\big)\,
\big(\alpha(j)-\alpha(j-1)\big).
$$

2. By a celebrated result of  
Thom, there are topological spaces $X$ and integers   $n>0$
such that some $n$-dimensional singular  homology classes of $X$ with coefficients in $\kk=\ZZ$ are not realizable  by  closed  singular manifolds.
For such $X$ and $n$,  the natural map from the $n$-dimensional   oriented  bordism group $\Omega_n(X)$  to $  H_n (X)$,
carrying  a closed singular manifold $\kappa:M \to X$ to $\kappa_*([M])$, is not surjective. This map splits as a
  composition of  the map  $\Omega_n(X) \to \widetilde H_n(X)$ described in Remark~\ref{singsing}
  with  the surjective map  $[-]:\widetilde {H}_n (X)\to H_n (X)$.
  Therefore, for such~$X$ and~$n$,  the map  $\Omega_n(X) \to \widetilde H_n(X)$   is not   surjective.

    3.  The face homology   seems to be   difficult to compute. As a consequence, the authors do not know
whether the  transformation   $[-]:\widetilde {H}_\ast\to H_\ast$  is injective, and,  equivalently,
whether  $\lb - \rb: H_\ast\to \widetilde {H}_\ast$  is surjective.  
 In fact,    the authors  even   do not know whether the face homology   of    a point is trivial in positive degrees.

4. The constructions and results of this section easily extend to the  face homology of topological pairs   (cf$.$  Remark \ref{singsing}).

 \section{Smooth polychains}\label{smooth111}

We  reformulate   face homology of the path spaces of manifolds  in terms of smooth polychains.
We start by studying    polychains in    manifolds. 

\subsection{Polychains in   manifolds}\label{smoothmaps}

  Recall from Section \ref{Manifolds with faces} that
a map $\kappa $ from an   $n$-dimensional    manifold with faces~$K$
to  a smooth    $m$-dimensional manifold~$M$  (possibly, with boundary) is said to be smooth if
restricting $\kappa$  to any local coordinate systems in $K$ and $M$ we obtain a map that extends to a $C^\infty$-map from an open subset of $\RR^n$ to $\RR^m$.
If ${{N}}\subset K$ is a union of (some) faces  of $K$, then  we call a map    $   {{N}} \to  M$    \index{map!smooth}  \emph{smooth} whenever its restrictions to all   faces of $K$ contained in~${{N}}$ are smooth. Such a map  $ {{N}} \to  M$ is  necessarily continuous. This terminology applies in particular to ${{N}}= \partial K$.

\begin{lemma}\label{smo1} Let $K$ be a   manifold with faces.
Any smooth map $ \partial K \to  M$    extends to a smooth map from a neighborhood of $\partial K$ in $K$ to $  M$.
\end{lemma}

\begin{proof}  Using a partition of unity on $K$ and local coordinates on $M$, we easily reduce the lemma to the case where $K=\RR^n_+= [0,\infty)^n$ with $n\geq 0$ and $M=\RR$. We need to prove that every function $f: \partial \RR^n_+ \to \RR$ whose restrictions to all proper faces of $\RR^n_+$ are smooth extends to a smooth function on $\RR^n_+$. We   exhibit one such extension   explicitly. For a subset $S$ of the set $\{1,\dots,n\}$ and a  point $x=(x_1,\dots, x_n)\in \RR^n_+$ denote by $x_S$ the point of
  $\RR^n_+$ whose $i$-th coordinate is $x_i  $ if $i\in S$ and   zero otherwise.
If  $S\neq \{1,\dots,n\}$, i.e$.$ if $S$ is a proper subset of $ \{1,\dots,n\}$, then  $x_S\in \partial \RR^n_+$.  Set
\begin{equation}\label{f_bar}
\overline f(x)=\sum_{  S\subsetneq \{1,\dots,n\}  } (-1)^{\card (S)+n+1} f(x_S).
\end{equation}
Each function $x\mapsto f(x_S )$ is smooth because it is a composition of $f$ with the projection of $\RR^n_+$  onto its proper face.
Therefore    the function   $\overline f:\RR^n_+\to \RR$ is smooth.
  Moreover,   it satisfies $\overline f(x)=f(x)$ for all  $x\in \partial \RR^n_+$.
  Indeed, pick $i\in\{1,\dots,n\}$ such that  $x_i=0$   and observe that each term in   \eqref{f_bar}
corresponding to $S $ with $i\notin S$  cancels with the term corresponding to  $ {S\cup \{i\}} $ provided the latter set is proper.
This   leaves    only the term   $f(x_S)=f(x)$   determined by  $S=\{1,\dots,n\} \setminus \{i\}$.
\end{proof}

\begin{lemma}\label{smo2}
Let   $ \kappa:K\to M $ be a   continuous map from a     manifold with faces~$K$ to    a smooth manifold~$M$.
Then any homotopy of $\kappa\vert_{\partial K}: \partial K\to M$ to a smooth map  extends to a homotopy of $\kappa$ to a smooth map $K\to  M$.
\end{lemma}

\begin{proof}
Observe first that if $\kappa\vert_{\partial K}: \partial K\to M$ is smooth,
then there is a  homotopy of $\kappa$ rel $\partial K$ to a smooth map $K \to  M$.
Indeed,  Lemma~\ref{smo1} yields an extension of $\kappa\vert_{\partial K}: \partial K\to M$ to a smooth map $U\to M$
where $U$ is a collar   of $\partial K$  in $K$.
The  latter map obviously extends to a continuous map $\kappa': K\to M$   homotopic to $\kappa$ rel $\partial K$.
Since   $\kappa'\vert_U$ is smooth and $\overline {K\setminus U}$ is a compact subset of the smooth manifold $K\setminus \partial K$,
there is a homotopy of $\kappa'$ to a smooth map $K\to M$, and this homotopy may be chosen to be  constant in a neighborhood of  $\partial K$ in $U$.
The resulting   smooth map $K\to M$ is homotopic to $\kappa$ rel $\partial K$.

  To prove the lemma,  take an arbitrary (continuous) extension of the given homotopy of $\kappa\vert_{\partial K}$ to a homotopy of
  $\kappa $, and compose it with a homotopy   $\operatorname{rel} \partial K$
  of the resulting map   $  K\to M$ to a smooth map as in  the previous paragraph.  Next, using a collar of $\partial K$ in $K$,
  deform the composed homotopy of $\kappa$ into a homotopy   satisfying the conditions of the  lemma.
  \end{proof}

  As an exercise, the reader may deduce from Lemma~\ref{smo2}
  (by an inductive construction on the faces of $K$)
  that any  continuous map $K \to  M$ is homotopic to a smooth map.

    A polychain $ (K,\varphi,u,\kappa)$ in   $M$ is   \index{polychain!smooth} \emph{smooth} if the map $\kappa:K\to M$ is smooth.
    We explain now how to deform arbitrary  polychains in $M$   into   smooth polychains.
    For  the notion of  a  deformation  of a polychain   in $M$, see  Section~\ref{homotopy_singular}. We explain first how to extend deformations.
  Let $N\subset K$ consist of some faces of $K$.
    We say that a homotopy $\{(\kappa\vert_{{N}})^t:{{N}}\to M\}_{t\in I}$   of    $\kappa\vert_N$  is  \index{homotopy!compatible with the partition} \emph{compatible}
    with  the partition $\varphi$ if $$(\kappa\vert_{{N}})^t\vert_G \circ \varphi_{F,G}= (\kappa\vert_{{N}})^t\vert_F$$ for any $t\in I$
    and any faces of the  same type $F,G \subset {{N}}$.

 \begin{lemma}\label{sssppsmodefo1----}
Let $\calK=(K,\varphi,u,\kappa)$ be a polychain  in   $M$ and let ${{N}}$ be a union of   faces  of $K$.
Let $\{(\kappa\vert_{{N}})^t:{{N}}\to M\}_{t\in I}$ be a homotopy of  $\kappa\vert_{{N}}$ compatible  with~$\varphi$
such that     $(\kappa\vert_{{N}})^1:{{N}}\to M$ is smooth. Then there is a  deformation   $\{ \calK^t  = (K,\varphi,u,\kappa^t) \}_{ t\in I}$ of
$ \calK^0=\calK$ such that     $ \calK^1$ is smooth and for all   $t\in I$,  $$\kappa^t\vert_{{N}}= (\kappa\vert_{{N}})^t:{{N}}\to M.$$
\end{lemma}

\begin{proof} For any integer $r$, denote   (as in Section \ref{Manifolds with faces})   by  $K_r $ the union of all faces of $K$ of dimension $\leq r$.
Recursively in~$r=-1, 0, \ldots$, we   construct a homotopy $(\kappa\vert_{K_r})^t $ of   $ \kappa \vert_{K_r}$   to a smooth map   $(\kappa\vert_{K_r})^1:K_r\to M$.
For $r=-1$, there is nothing to do since $K_{-1}=  \varnothing  $.  The induction step goes as follows.
For each type of $r$-dimensional faces of $K$, select a representative face  $F$ so that if at least one face of the given type lies in ${{N}}$,
then   $F\subset {{N}}$. If $F\subset {{N}}$, then set $(\kappa\vert_{F})^t = (\kappa\vert_{{N}})^t\vert_F$ for all~$t$.
If $F\nsubseteq {{N}}$, then by Lemma~\ref{smo2}  there is  a  homotopy of $\kappa\vert_F$
 to a smooth map  extending the  homotopy of $\kappa$ on $\partial F \subset K_{r-1}$ obtained at the previous step.
 The homotopy  on the selected   $r$-dimensional faces  uniquely extends to a     homotopy of~$\kappa$ on $K_r$ compatible with~$\varphi$.
 For $r=\dim(K) $, we obtain the required deformation of~$\kappa$.     \end{proof}

\begin{lemma}\label{ppsmodefo1----}
For any
  polychain  $\calK=(K,\varphi,u,\kappa)$  in   $M$, there is a  deformation   $\{ \calK^t  = (K,\varphi,u,\kappa^t) \}_{ t\in I}$ of
$ \calK^0=\calK$ such that     $ \calK^1$ is smooth and  $\kappa^t\vert_F= \kappa\vert_F$ for all   $t\in I$  and  all   faces $F$ of $K$
on which    $\kappa   $  is smooth.
\end{lemma}

\begin{proof} This is a special case of   Lemma \ref{sssppsmodefo1----}   where ${{N}}$ is the union  of all   faces  of $K$
on which    $\kappa   $  is smooth and     $\{(\kappa\vert_{{N}})^t\}_{t\in I}$   is the constant homotopy.  \end{proof}

We can now reformulate the face homology of $M$   in terms of smooth polychains. Note that if  a polychain $\calK$ in $M$ is  smooth,
then so are the polychains $\red \calK$, $\partial \calK$, and $\partial^r \calK$. Disjoint unions of smooth polychains are smooth.
Applying the definitions of  Section~\ref{facehomology} to $X=M$ but considering only smooth polycycles and smooth polychains
     we obtain  \index{face homology!smooth}  \emph{smooth face homology} $\widetilde {H}^s_\ast(M)$.

\begin{theor}\label{++smoothdefo1----} The natural linear map  $\widetilde {H}^s_\ast(M)\to \widetilde {H}_\ast(M)$ is an isomorphism.
\end{theor}

\begin{proof} Lemmas~\ref{homotimplieshomol} and~\ref{ppsmodefo1----} imply that any polycycle  in $M$ is homologous to a smooth polycycle.
This proves the surjectivity of the map in the statement of the theorem.
To prove the injectivity it suffices to show that, for any  homologous  reduced smooth $n$-polychains $\calK_1$, $ \calK_2 $ in $M$  there are
smooth $(n+1)$-polychains $\calR'_1,\calR'_2$ in $M$ such that   $   \calK_1 \sqcup \partial^r \calR'_1   \cong  \calK_2 \sqcup \partial^r \calR'_2$.
By assumption,   there are  $(n+1)$-polychains $\calR_1,\calR_2$   in $M$
and a diffeomorphism    $f:   \calK_1 \sqcup \partial^r \calR_1    \to   \calK_2 \sqcup \partial^r \calR_2$.
  For $i=1,2$,   set $(P_i, \varphi_i, u_i, \kappa_i) = \calK_i \sqcup \partial^r \calR_i$
  and let $  K_i  \subset P_i$ be the union of the components  of $P_i$ underlying~$\calK_i$.
     Lemma~\ref{ppsmodefo1----} yields a homotopy $\{\kappa_1^t:P_1\to M\}_{t }$ of $\kappa_1^0=\kappa_1 $
     to a smooth map $\kappa^1_1 $  in the class of maps   compatible with the partition~$\varphi_1$,
       which is constant     on all faces  of $P_1$ on which   $\kappa_1   $  is smooth. In particular, this homotopy   is constant on    $K_1$.
       Consider the homotopy $\{\kappa_2^t=\kappa_1^t f^{-1}:P_2\to M \}_{t }$ of $\kappa_2^0=\kappa_2 $.
   This homotopy is compatible with the partition $\varphi_2$ and is constant on   $K_2$   (because $\kappa_1=\kappa_2f$ is smooth on $f^{-1}(  K_2 )$).
   Clearly,     $f:P_1\to P_2$   is a diffeomorphism of the smooth polychains $(P_1, \varphi_1, u_1, \kappa^1_1)$ and
$(P_2, \varphi_2, u_2, \kappa^1_2= \kappa_1^1 f^{-1})$.
For  $i=1,2$, the polychain  $(P_i, \varphi_i, u_i, \kappa^1_i)$ is obtained from $\calK_i \sqcup \partial^r \calR_i$ by a deformation
which is constant   on $\calK_i$   and transforms $\partial^r \calR_i$ into a smooth polychain.
 To finish the proof, we need only to show that this deformation of $\partial^r \calR_i$ extends to a deformation of $ \calR_i$ into a smooth polychain.
 This is done in 3 steps. First of all, applying Lemma~\ref{sssppsmodefo1----} to $\calK=\red_+ \partial \calR_i$
 and taking   for $N$  the union of all connected components of non-zero weight
 we obtain that our deformation of    $\partial^r \calR_i=\red_0\red_+ \partial \calR_i$   extends to a  deformation of $\red_+ \partial \calR_i$
 into a smooth polychain.    The latter deformation  induces   a deformation of $ \partial \calR_i$
 into a smooth polychain. One more application of Lemma~\ref{sssppsmodefo1----} allows us to  extend the latter deformation to a deformation of $  \calR_i$
 into a smooth polychain $\calR'_i$.  Then $   \calK_1 \sqcup \partial^r \calR'_1   \cong  \calK_2 \sqcup \partial^r \calR'_2$.     \end{proof}

\subsection{Polychains in path spaces}\label{Path homology}

Pick two points $\star , \star'$ in a smooth manifold~$M$ (possibly, $\star = \star'$   and $\partial M\neq \varnothing$ ).
 A    {\it path}  in $M$ from  $\star  $ to  $\star'  $ is a continuous map $I=[0,1]\to M$ carrying  $0 $ to~$\star$ and $1$ to $\star'$.
Let $\Omega=\Omega  (M,\star,\star')$ be the space of such paths   with   compact-open
     topology; we call  $\Omega$  \index{path space}{\it the path space of~$M$.}
     Note that a map $\sigma$ from a topological space $K$ to $\Omega  $
   is continuous if and only if the adjoint map   $\tilde \sigma: K\times I\to M$,
   carrying any pair $(k\in K, s\in I)$   to $\sigma (k)(s)\in M$ is continuous; see, for example,  \cite[Section 1.2.7]{FuR}.

   Given a subspace $X $ of $  \Omega $, we call a  map from a manifold with faces~$K$  to~$X$    
  \index{map!smooth} \emph{smooth} if the adjoint map $ K \times I  \to M$ is smooth  in the sense of Section~\ref{smoothmaps}.
   A  polychain   $\calK =(K, \varphi, u , \kappa)$ in   $X$   is   \index{polychain!smooth} {\it smooth} if  $\kappa:  K    \to    X  $   is smooth.
The definitions of  Section~\ref{facehomology}  restricted to smooth polycycles and smooth polychains in $X$,
     yield  the  \index{face homology!smooth} \emph{smooth face homology} $\widetilde {H}^s_\ast(X)$ of~$X$. In the next theorem,    $X=\Omega$.

\begin{theor}\label{wwdefo1----}
The natural linear map  $\widetilde {H}^s_\ast(\Omega)\to \widetilde {H}_\ast(\Omega)$ is an isomorphism.
\end{theor}

\begin{proof} We follow the  lines of Section~\ref{smoothmaps} with $M$ replaced by $\Omega$. First, we show that given a manifold with faces $K$, any   smooth map
   $ f:  \partial K \to  \Omega$    extends to a smooth map from a neighborhood of $\partial K$ in $K$ to $  \Omega$.
   Indeed,   the adjoint map $  \tilde f :\partial K \times I\to M$   extends to a smooth map $ \partial (K\times I)\to M$
   by sending  $K\times \{0\} $ to~$\star$ and $K \times \{1\}$ to $\star'$.
   By Lemma~\ref{smo1}, the latter map extends to a smooth map  $  \overline f :U\to M$ for  some neighborhood $U$ of $\partial (K\times I)$ in $ K\times I$.
   Clearly,  $U\supset   V \times I  $ for a neighborhood $V$ of $\partial K$ in $K$.
   The map   $\overline f\vert_{V \times I}$   is adjoint to a smooth map $V\to \Omega$ extending $f$.

     Lemmas~\ref{smo2}--\ref{ppsmodefo1----} remain true with $M$   replaced       by $\Omega$.
     The  proofs   above apply  with the only difference that Lemma~\ref{smo1}  should be replaced by the result of the previous paragraph.
     The proof of Theorem~\ref{++smoothdefo1----} also works with $M$ replaced by $\Omega$. This gives the desired result.   \end{proof}

 In the   case where   $\star,\star' \in \partial M$, the path space $ \Omega=\Omega(M, \star, \star') $ is homotopy equivalent to a  smaller space.  Let
  $ \Omega^\circ=\Omega^\circ(M, \star, \star')$ be the subspace of $ \Omega$       consisting of
all   paths $\alpha: I \to M$ from $\star$ to~$\star'$ such that $\alpha^{-1}(\partial M) = \partial I$.
 We  call $\Omega^\circ$  \index{path space!proper}{\it the proper path space} of $(M, \star, \star')$. 

\begin{lemma}\label{loopsp}   The inclusion map $\Omega^\circ \hookrightarrow \Omega$ is a homotopy equivalence.
\end{lemma}

\begin{proof}  We begin with an    observation in set-theoretic
topology. Consider a topological pair $Y \subset X$ and suppose that
there is a homotopy $\{f_t: X\to X\}_{t\in I }$ of the identity map
$f_0=\id_X$ such that $f_t(X)\subset Y$ for all $t>0$. Then the
inclusion $\iota :Y \hookrightarrow X$ is a homotopy equivalence and
its homotopy inverse $g :X\to Y$ is obtained from $f_1 $ by reducing
the image to $Y$. Indeed,     the family $\{f_t: X\to X\}_{t }$ is a
homotopy between $ f_0=\id_{X }$ and $ f_1=\iota g $. Since $f_t
\iota(Y)\subset Y$ for all  $t\in I$, we have the family  of maps $\{f_t \iota: Y\to Y\}_{t }$. This   is a homotopy between $\id_{Y}$ and $g\iota $.

Using a tubular neighborhood of $\partial M$ in $M$, we can easily construct a    (smooth)    family of
embeddings   $\{F_{s,t}:M\hookrightarrow M\}_{s,t\in I}$  such  that $F_{s,t} =\id_M$ if
$s\in \{0,1\}$ or $t=0$, and $F_{s,t} (M)\subset \Int M=M\setminus
\partial M$ for all other pairs   $(s,t)$.   Given $t\in I$ and a path  $\alpha: I \to M$ from $\star$ to $\star'$, we define a path $ \alpha_t:I\to
M$ by $ \alpha_t (s)= F_{s,t} (\alpha (s))$ for all $s\in I$.   This gives a family of paths $\{\alpha_t\}_{t\in I}$
such that $\alpha_0=\alpha$ and $\alpha_t\in \Omega^\circ$ for  all $t>0$.
The formula $f_t(\alpha)=\alpha_t$ defines a homotopy $\{f_t: \Omega \to \Omega \}_{t\in I}$ of   $f_0=\id_\Omega$ such that
$f_t(\Omega)\subset \Omega^\circ$ for all $t>0$. Now, the result of
the previous paragraph implies that the inclusion   $\Omega^\circ \hookrightarrow \Omega$ is a homotopy equivalence.
\end{proof}

Lemma~\ref{loopsp} implies that $\widetilde {H}_\ast(\Omega^\circ)   \simeq   \widetilde {H}_\ast(\Omega)$.
The following theorem  computes the face homology of $\Omega^\circ$  and    $\Omega$    in terms of smooth polychains in $\Omega^\circ$.

\begin{theor}\label{loopsp+++++}   The natural linear map  $\widetilde {H}^s_\ast(\Omega^\circ)\to \widetilde {H}_\ast(\Omega^\circ)$ is an isomorphism.
\end{theor}

\begin{proof}   Consider the homotopy $\{f_t \}_{t\in I}$ of   $f_0=\id_\Omega$  introduced in the proof of   Lemma  \ref{loopsp}
and set $f=f_1:\Omega\to \Omega^\circ\subset \Omega$.
For any manifold with faces $L$ and for any map $\lambda:L \to \Omega$,
the adjoint map $\widetilde{f_t\lambda}: L \times I \to M$ of $f_t\lambda$ is given by
$$
\widetilde{f_t \lambda}(l,s)= f_t(\lambda(l))(s) =\lambda(l)_t(s) = F_{s,t}(\lambda(l)(s)) = F_{s,t} \big(\widetilde{\lambda}(l,s)\big).
$$
Consequently,  if $\lambda$ is smooth, then    $f_t\lambda:L \to \Omega$ is smooth for all $t\in I$.
We conclude that  for any smooth polychain $\calL$ in $\Omega$, the polychain  $f_*( \calL)$    in $\Omega^\circ$ is    smooth.

 By the surjectivity of the map in Theorem~\ref{wwdefo1----},   any polycycle $\calK$ in $\Omega^\circ$ is homologous in $\Omega$
  to a smooth polycycle~$\calK'$ in $\Omega$. Applying $f$, we obtain that   $f_*(\calK)$   is homologous in $\Omega^\circ$ to the smooth polycycle  $f_*(\calK')$.
  The homotopy $\{f_t\}_{t }$ induces a deformation  of $\calK$ into   $f_*(\calK)$   in $\Omega^\circ$.
  Therefore $\calK$ is homologous to    $f_*(\calK)$   in $\Omega^\circ$.
  Thus, $\calK$ is homologous to   $f_*(\calK')$   in $\Omega^\circ$.
  This proves the surjectivity of the    natural   map  $\widetilde {H}^s_\ast(\Omega^\circ)\to \widetilde {H}_\ast(\Omega^\circ)$.
To  prove the injectivity, consider two reduced smooth $n$-polycycles   $\calK_1$, $ \calK_2 $ in~$\Omega^\circ$   that are homologous in $\Omega^\circ$.
Then they  are homologous in $\Omega $. By the injectivity   of the map   in Theorem~\ref{wwdefo1----}, there are
smooth  $(n+1)$-polychains $\calL_1,\calL_2$ in $\Omega $  such that   $   \calK_1 \sqcup \partial^r \calL_1   \cong  \calK_2 \sqcup \partial^r \calL_2$.
Applying $f$ we obtain     $$f_*(\calK_1) \sqcup \partial^r (f_*(\calL_1))   \cong  f_*(\calK_2) \sqcup \partial^r (f_*(\calL_2))$$   where
  $ f_*(\calK_1), f_*(\calL_1), f_*(\calK_2),  f_*(\calL_2)$   are smooth polychains in~$\Omega^\circ$.
So,   $$\lb f_*(\calK_1) \rb =  \lb f_*(\calK_2) \rb \in \widetilde H_n^s(\Omega^\circ).$$
The homotopy $\{f_t  \}_{t }$ induces a smooth deformation of $\calK_i$ into   $f_*(\calK_i)$
and therefore  $\lb  \calK_i \rb = \lb f_*(\calK_i) \rb  \in \widetilde H_n^s(\Omega^\circ)$    for $i=1,2$.
 Hence $\lb \calK_1 \rb =  \lb \calK_2\rb \in \widetilde H_n^s(\Omega^\circ)$.
This completes the proof of the injectivity  of the   natural   map  $\widetilde {H}^s_\ast(\Omega^\circ)\to \widetilde {H}_\ast(\Omega^\circ)$   and of the theorem.
\end{proof}

\chapter {Operations on polychains} \label{Operations on polychains}

  Throughout this chapter, $M$ is 
  an oriented  smooth $n$-dimensional manifold  with boundary,  where $n\geq 2$.
We fix  points  $\star_1,\star_2,\star_3,\star_4 \in \partial M$ and assume,  
unless explicitly stated to the contrary,  
  that $\{\star_1,\star_2\}\cap\{\star_3,\star_4\}=\varnothing$ (possibly,  $\star_1=\star_2$ and/or $\star_3=\star_4$). 
  For   $i,j\in \{1,2,3,4\}$,   let   $\Omega_{ij} =\Omega(M,\star_i,\star_j)$ be the path space 
  and   $ \Omega^\circ_{ij} =\Omega^\circ (M,\star_i,\star_j)\subset \Omega_{ij}$  be the proper path space of  $ (M,\star_i,\star_j)$.

\section{Transversality  in  path spaces}\label{sect-transs}

In this section, we study  transversality of   polychains in the proper path spaces $\Omega^\circ_{12}$ and $\Omega^\circ_{34}$.

  \subsection{Transversal  maps}\label{transs}
  
  The  \index{diagonal} {\it     diagonal} of~$M$
   $$
  \diag_M=  \{(x,x)\, \vert\,  x\in M\}   \subset M\times M
   $$   is a smooth manifold diffeomorphic to $M$. 
   We say that a smooth map~$g$ from  a manifold with faces $N$
    to $M \times M$ is  \index{map!weakly transversal} {\it   weakly   transversal  to $ \diag_M$} 
    if $g(N)$ does not meet $\partial (\diag_M) $ 
     and the restriction of $g$ to  $\Int(N)=N\setminus \partial N$ is transversal to $\Int (\diag_M)$ in the usual sense  of differential topology.
    (The interiors of   $N$ and $ \diag_M $  are smooth manifolds   so   this  condition   makes sense.)
    The map~$g$ is  \index{map!transversal to the diagonal}{\it transversal  to  $ \diag_M$}  
     if its restriction to any face of $N$ is weakly transversal to $ \diag_M$.

Fix   manifolds with faces $K$ and $L$.
Consider smooth maps $\kappa: K  \to    \Omega^\circ_{12}   $,  $\lambda:L \to   \Omega^\circ_{34}   $
and let $\tilde \kappa   : K   \times I  \to M$, $ \tilde \lambda:  L \times I \to M$ be the adjoint maps.
The product  map
$$
\tilde \kappa   \times \tilde \lambda: K   \times I\times L \times I \to M\times M
$$
carries a tuple $(k\in K,  s\in I, l\in L,  t \in I)$   to   the point  $(\kappa(k) (s), \lambda(l) (t))$. The  latter point can lie on $\diag_M$
   only when  $ s,t \in   \Int(I)=(0,1)$   and  never lies  in   $\partial (\diag_M)  $.
We say that  $\kappa$ and $ \lambda$ are \index{map!transversal} {\it transversal}  
 if the map $\tilde \kappa   \times \tilde \lambda$ is transversal to $ \diag_M $ in the  sense above. 
Note that the maps $\kappa$ and $ \lambda$ are  transversal  in our sense 
 if and only if they are transversal in  the sense of  \cite{MrOd},   see Proposition~7.2.2 therein.
(Our  notion of transversality is stronger than the  one   in \cite{Jo},  see Remark~6.3 therein.)
If~$\kappa  $ and~$\lambda  $  are transversal, then their restrictions to arbitrary faces of $K$, $L$ are transversal.
 Clearly, smooth  homotopies of~$\kappa$ and $\lambda$ that are sufficiently $C^1$-small preserve transversality.  

\begin{lemma}\label{defo1}
For any smooth maps $\kappa: K \to    \Omega^\circ_{12}   $ and
  $\lambda:L \to   \Omega^\circ_{34}    $,
there is  an  arbitrarily $C^\infty$-small smooth homotopy   $\{\kappa^t\}_{t\in I}$ of $\kappa^0=\kappa$ such that   $\kappa^1$ and $\lambda$ are  transversal.
\end{lemma}

\begin{proof}    Proceeding by induction on $\dim(L)\geq -1$, we can assume that $\kappa$
is transversal to the restrictions of $\lambda$ to all proper faces of~$L$.   All subsequent homotopies of $\kappa$ are chosen to be
  small enough to preserve   this property.
 Fix a smooth triangulation $T$ of     $K\times I$.   A map   from   a simplex $e$ of $T$ to $M$ is \index{map!smooth} {\it smooth} if it   is smooth as a singular simplex in $M$. A smooth map   $f:e\to M$
      is  \index{map!transversal}  {\it $\lambda$-transversal} if the map  $f \times \tilde \lambda: e \times L \times I \to M\times M$ is   weakly  transversal to $ \diag_M$.
  We call a map $g: K\times I\to M$ \index{map!good} {\it good} if
  the adjoint map  from
      $K$ to the space of paths in $M$ takes values in
    $   \Omega^\circ_{12}  =\Omega^\circ (M,\star_1,\star_2)$. The map $g$ is {\it $T$-good} if it  is good
and  the image of any simplex of $T$ under $g$ lies in  a  closed ball in~$M$.

     Consider the map  $  \tilde \kappa :K\times I \to M$     adjoint to  $  \kappa$.   Clearly,  $
\tilde \kappa$ is good.  Since $\{\star_1, \star_2\}\cap \{\star_3, \star_4\} =  \varnothing  $,   the set
  $\tilde \kappa (K \times \partial I)$ is disjoint from  $\tilde \lambda(L\times I)$.
By continuity, there is a small $\delta>0$ such that the sets
 $  \tilde \kappa (K \times  [0,\delta])$ and $ \tilde \kappa (K \times  [1-\delta, 1])$ are disjoint from  $ \tilde \lambda (L\times I) $.
 Subdividing $T$, we can  assume that   $\tilde \kappa  $ is $T$-good and   $T_\delta=K \times ([0,\delta]\cup [1-\delta, 1])$
is a subcomplex of~$T$.   For any simplex $e$ of $T_\delta$, the   map
        $ \tilde \kappa \vert_e $  is $\lambda$-transversal
  because        $ (\tilde \kappa   \times \tilde \lambda)(e\times L\times I) \cap \diag_M=  \varnothing   $.

Set $p=\dim (K)$. We   shall   construct    $p+2$     homotopies
$ \tilde \kappa= \kappa_{0 }   \leadsto  \kappa_{1 }   \leadsto   \cdots    \leadsto       \kappa_{   p+2 } $  in the class of $T$-good maps $K\times I\to   M$
such   that  for all ${r}\geq 0$ and any simplex $e$ of $T $  of dimension $\leq {r}-1 $, the    map  $\kappa_{{r} }\vert_e $   is $\lambda$-transversal.
         Then, the restriction of   $\kappa_{p+2}$   to   any simplex   of $T $ is $\lambda$-transversal.
   For each face $E$ of $K$, the product $E\times I$ is a subcomplex of $T$. Therefore,   the restriction
of $   \kappa_{ p+2 }   \times \tilde \lambda $ to $ E  \times I \times
 L  \times I $ is  weakly   transversal to $ \diag_M $. By   the beginning of the proof, the same holds when $L$ is replaced with any of its proper
faces.   This shows that  the smooth map $K \to    \Omega^\circ_{12}  $ determined by $\kappa_{p+2}$      is transversal to~$\lambda$.

 The homotopies $\tilde \kappa= \kappa_{0 }   \leadsto  \kappa_{1 }   \leadsto   \cdots$  are constructed recursively.
   For ${r}=0$, the condition on $\kappa_{r}$ is void, and we can take     $     \kappa_{0 }=\tilde
  \kappa$.
      Assume that we    have   required homotopies
      $ \kappa_{0 }  \leadsto   \kappa_{1 } \leadsto   \cdots
    \leadsto    \kappa_{{r}} $ for some ${r}\geq 0$.
Consider an   ${r}$-dimensional simplex $e\in T \setminus  T_\delta  $.
Clearly, $e\subset  K  \times \Int(I)$.  For   $\varepsilon >0$, let $U_\varepsilon $ denote the
(closed) metric    $\varepsilon$-neighborhood  of $\partial e$ in $ K  \times \Int(I)$.
  The inductive assumption  implies   that
\begin{quote}
 $(*)$ for a  sufficiently small $\varepsilon>0$,
the restriction of  the map  $\kappa_{{r} } \times \tilde \lambda$ to
$U_\varepsilon \times     L   \times  I  $ is   weakly   transversal to   $ \diag_M $.
\end{quote}
 Since   $\kappa_{r}$ is       good and   $e\subset  K  \times \Int(I)$, we have
      $
      \kappa_{{r} }(e)\subset \Int (M)$. Since $\kappa_{r}$ is $T$-good,  $
      \kappa_{{r} }(e)\subset B$ for a  closed ball $B \subset M$.   We  can choose $B$ so that $
      \kappa_{{r} }(e)\subset \Int (B)$. We identify $B$ with the closed unit ball in   Euclidean space with center $0$.
  Pick a  small
      neighborhood $S\subset B$   of $0\in B$ so  that $\kappa_{{r}
      }(e)+s\subset \Int (B)$ for all   $s\in S$. Consider   the
  smooth maps $\{f_s:e \to B\subset M\}_{s\in S}$  where   $f_s$
      carries
      any $u\in e$ to $\kappa_{{r}}(u) +s$.   It is obvious that the family of   maps
      $$\{f_s \times \tilde \lambda: e \times    L   \times  {I} \to M\times M\}_{s\in S}$$ is
       weakly   transversal to   $ \diag_M $  in the sense that the adjoint map
      $$e \times  L \times  {I}\times S \to M\times M$$  is   weakly   transversal to   $ \diag_M $.
         By the   classical  transversality  theorem  (see   \cite[Section 2.3]{GP}),  

\begin{quote}
$(**)$ the  set $S_\lambda=\{ s \in S:  \hbox{the   map $f_s \times
\tilde \lambda  $ is   weakly   transversal to $ \diag_M $}\}$ is dense   (and open)   in $S$.
\end{quote}
(This argument is adapted from that of  Laudenbach   \cite[Proof of Lemma~2.6]{La}.)
 In our terminology, $f_s$ is $\lambda$-transversal for any $s\in S_\lambda$.
For  each   $s \in S_\lambda $, we   define  a map $g_s:e \to B\subset M$ by $g_s(u)=\kappa_{r}(u)+ h(u)s $ for all $u\in e$
where  $h:e\to I$ is a smooth function carrying $e\cap U_{\varepsilon/2}$ to $0$
    and
      $e\setminus U_{ \varepsilon}$ to $1$. Then  $g_s=\kappa_{r}$ on  $e\cap U_{\varepsilon/2}$ and $g_s=f_s$ on
      $e\setminus U_\varepsilon$.
 We deduce from $(*)$ and $(**)$ that,
for all   sufficiently small $s\in S_\lambda$,   the
      map $g_s   $ is
      $\lambda$-transversal. Pick such an $s$ and consider the obvious linear  homotopy
    $\kappa_{r}\vert_e \leadsto g_s$   constant on $e\cap
    U_{\varepsilon/2}$. Combining such homotopies corresponding to
    all ${r}$-dimensional simplices   $e\in T \setminus  T_\delta  $ and
    extending by the  constant homotopy on $T_\delta$ we obtain a
    homotopy of   $\kappa_{r}$   on the union of $T_\delta$ with the ${r}$-skeleton of $T$. The latter homotopy
    extends
     to a homotopy   $\kappa_{r}   \leadsto  \kappa_{{r}+1}$ in the class of good maps
    $K\times I\to M$. Taking the vectors  $s$ in this construction small enough, we can always choose
     the homotopy   $\kappa_{r}   \leadsto  \kappa_{{r}+1}$
    so that it proceeds
      in the class of
      $T$-good maps. Since this homotopy   is   constant on $T_\delta  $
      and on the $({r}-1)$-th skeleton of $T$,   the $\lambda$-transversality of $\kappa_{r}$ on the simplices of $T$ of dimension $<{r}$ acquired at the previous
      steps is preserved during the homotopy. By construction,   $\kappa_{{r}+1}  $ is
   $\lambda$-transversal on all ${r}$-dimensional simplices of $T$.
   \end{proof}

 \begin{lemma}\label{defo1+}
 The homotopy in Lemma~\ref{defo1} may be chosen to be constant on the union   of all  faces  $E$ of $K$
such that  $\kappa\vert_E  $  is transversal to $\lambda$.
\end{lemma}

\begin{proof}
The proof proceeds by induction on $\dim(K)$. If $\dim(K)=0$, then we   take the constant homotopy    on  all connected  components   of $K$
on which  $\kappa $  is transversal to $\lambda$ and we take the homotopy provided by Lemma~\ref{defo1} on all other components of $K$.
 Let $p$ be a positive integer  such that the   lemma holds for all~$K$ of dimension $<p$. We prove the lemma for an arbitrary $p$-dimensional manifold with faces~$K$.
 As above, if $\kappa$   is   transversal to $\lambda$ on some connected components of~$K$,
then we  take the constant homotopy    on that components. Thus we can assume without loss of generality that   $\kappa$ may be  transversal to $\lambda$ only on    proper faces of~$K$.

 Let $\Sigma$ be the set of all faces $E$ of $K$ such that $\kappa\vert_E$   is
  transversal to $\lambda$. By  the definition of transversality, if $E\in \Sigma$, then all faces of $E$ also belong to $\Sigma$.
    Set $\vert \Sigma\vert =\cup_{E\in \Sigma} E$   and note that $\vert \Sigma \vert \subset \partial K$ by our  assumption.
    All homotopies of $\kappa:K\to    \Omega^\circ_{12}  $ in the following construction are arbitrarily  $C^\infty$-small smooth homotopies constant on $\vert \Sigma \vert$.
    We   recursively   construct    $p $     homotopies $  \kappa= \kappa_{-1 }   \leadsto  \kappa_{0 }   \leadsto   \cdots       \leadsto       \kappa_{  p-1 } $
    such   that     the restriction of $\kappa_{r}$ to any face   of $K $  of dimension $\leq      {r}  $
        is  transversal to $\lambda$ for all  ${r}  $.  Assume that we already have   homotopies  $\kappa=  \kappa_{-1 }   \leadsto  \kappa_{0 }   \leadsto   \cdots
       \leadsto       \kappa_{  {r}-1 } $ with the required properties where $0\leq {r} < p $.
  Consider an ${r}$-dimensional face $E$ of $K$ not belonging to $\Sigma$. By the assumptions on $\kappa_{{r}-1}$, the restriction of  $\kappa_{{r}-1} $ to any proper face of $E$ is transversal to $\lambda$. Since $\dim(E)={r}<p$,  the inductive assumption guarantees that  there is an arbitrarily    $C^\infty$-small, smooth, constant on $\partial E$ homotopy   of $\kappa_{{r}-1} \vert_E$   into  a map $E\to   \Omega^\circ_{12}  $ transversal to $\lambda$.   Combining  these homotopies over all   $E$ as above together with   the   constant homotopy on $\vert \Sigma \vert$ and extending to a small smooth homotopy of $\kappa_{{r}-1 }$ on the rest of $K$, we obtain a homotopy $\kappa_{{r}-1 }   \leadsto  \kappa_{{r} } $ with   the required properties.

  Next   pick a collar    $U\cong \partial K\times I \subset K$ of $\partial K\cong \partial K\times  \{0\}$ in $K$.
Set $V =\overline{K\setminus U}\subset K$.   Then    $V$ is a manifold  with faces and  $\partial V=U\cap V\cong \partial K $.
By the above,   $\kappa'=\kappa_{  p-1 }: K\to    \Omega^\circ_{12}  $ is a smooth map whose restriction to all proper faces of $K$ is  transversal to $\lambda$.
   Therefore, choosing the collar $U$   sufficiently narrow, we can ensure that the map   $\kappa' \vert_U:U\to    \Omega^\circ_{12}  $ is transversal to $\lambda$. By Lemma~\ref{defo1}, there is an arbitrarily     $C^\infty$-small homotopy $\{\kappa^t\vert_V \}_{t\in I}$ of $\kappa^0\vert_V=\kappa' \vert_V$ such that $\kappa^1\vert_V $ is transversal to $\lambda$. This homotopy   extends to a     small homotopy  $\{\kappa^t\}_{t\in I}$ of $\kappa^0=\kappa'$  constant on a neighborhood of $\partial K$ in $U$.
   If the homotopy  $\{\kappa^t\vert_V \}_{t\in I}$ is sufficiently small, then the extension   may be chosen so that $\kappa^t \vert_U $ is transversal to $\lambda$ for all $t\in I$. Then  $\kappa^1\vert_U$ is transversal to~$\lambda$, and so,   $\kappa^1:K\to    \Omega^\circ_{12}  $ is transversal to $\lambda$. The composite homotopy $$
  \kappa  =\kappa_{-1} \leadsto \cdots
       \leadsto       \kappa_{  p-1 } =\kappa'=\kappa^0\leadsto \kappa^1 $$ satisfies all the   conditions of the lemma.
\end{proof}

\subsection{Transversal  polychains}\label{transversality_polychains}

  We  call   smooth  polychains  $\calK =(K, \varphi, u , \kappa)$ in $   \Omega^\circ_{12}   $
and $\calL =(L, \psi, v , \lambda)$ in $   \Omega^\circ_{34}    $   \index{polychains!transversal} {\it transversal}
if the     maps $ \kappa:K \to    \Omega^\circ_{12}  $
and $ \lambda: L\to   \Omega^\circ_{34}   $ are    transversal. The following two lemmas show  that any smooth  polychain can be made transversal to a given smooth polychain by a small deformation.

\begin{lemma}\label{defo1++eee}
Let $\calK=(K,\varphi,u,\kappa)$    and $\calL=(L,\psi,v, \lambda)$ be smooth
  polychains  in   $   \Omega^\circ_{12}  $ and~$  \Omega^\circ_{34}   $, respectively. Let ${{N}}$ be a union of   faces  of $K$.  Let $$\{(\kappa\vert_{{N}})^t:{{N}}\to    \Omega^\circ_{12}  \}_{t\in I}$$ be a smooth   homotopy of  $\kappa\vert_{{N}}$ compatible  with~$\varphi$   such that     $(\kappa\vert_{{N}})^1:{{N}}\to    \Omega^\circ_{12}  $ is transversal to $\calL$. Then there is a smooth  deformation     $\{ \calK^t  = (K,\varphi,u,\kappa^t) \}_{ t\in I}$ of
$ \calK^0=\calK$ such that     $ \calK^1$ is  transversal to $\calL$ and   for all   $t\in I$,
$$\kappa^t\vert_{{N}}= (\kappa\vert_{{N}})^t:{{N}}\to    \Omega^\circ_{12}  .$$
\end{lemma}

\begin{proof}  We apply the same recursive method as in  the proof of   Lemma~\ref{sssppsmodefo1----}   with~$M$ replaced by $   \Omega^\circ_{12}  $.
The homotopy of $\kappa $ on a representative face~$F$ is obtained in two steps.
First, we take an arbitrary smooth homotopy of $\kappa\vert_F$    extending the   homotopy of $\kappa\vert_{\partial F } $ obtained at the previous step.
Then we compose   with an additional smooth homotopy rel $\partial F$ to a map $F\to    \Omega^\circ_{12}  $ transversal to $\lambda$.    The latter homotopy   is provided by Lemma~\ref{defo1+}.
 \end{proof}

\begin{lemma}\label{defo1++}
  Let $\calK=(K,\varphi,u,\kappa)$    and $\calL=(L,\psi,v, \lambda)$ be smooth
  polychains  in   $   \Omega^\circ_{12}  $ and~$  \Omega^\circ_{34}   $, respectively.
There exists an   arbitrarily $C^\infty$-small   smooth deformation   $\{ \calK^t  = (K,\varphi,u,\kappa^t) \}_{ t\in I}$ of
$ \calK$ such that  the polychain  $ \calK^1$ is transversal to $\calL$ and  $\kappa^t\vert_F= \kappa\vert_F$ for all   $t\in I$  and  all faces $F$ of $K$
on which $\kappa$  is transversal to $\lambda$.
\end{lemma}

\begin{proof}    This   is a special case  of Lemma~\ref{defo1++eee}
where~${{N}}$ is the union of all faces of~$K$ on which~$\kappa$ is transversal to $\lambda$ and   $\{(\kappa\vert_{{N}})^t \}_t$   is the constant homotopy.
That the  deformation $\{ \calK^t   \}_{ t }$  may be chosen   arbitrarily $C^\infty$-small  follows from  Lemmas~\ref{defo1} and~\ref{defo1+}.
\end{proof}

  We say that a pair of face homology classes $(a\in \widetilde{H}_\ast(   \Omega^\circ_{12}  ) , b \in \widetilde{H}_\ast (  \Omega^\circ_{34}   ))$  is \index{face homology!transversely represented} {\it transversely represented} by a pair     ($\calK , \calL$)
  if   $\calK$ is a smooth reduced  polycycle in $   \Omega^\circ_{12}  $ representing $a$, $\calL$ is a smooth reduced   polycycle in $  \Omega^\circ_{34}   $ representing $b$, and
  $\calK$ is transversal to  $\calL$.
The following lemma will play a key role in the sequel.

\begin{lemma}\label{keytheor}
Every pair  $(a\in \widetilde{H}_\ast(   \Omega^\circ_{12}  ) , b \in \widetilde{H}_\ast (  \Omega^\circ_{34}   ))$   can be transversely represented by a pair of polycycles.
Any  two pairs of    polycycles transversely representing    $(a  , b  )$ can be related by a
finite sequence  of transformations   $(\calK, \calL)\mapsto  (\check\calK, \check \calL) $  of the following   types:
\begin{itemize}
\item[(i)]    $\calL \cong  \check \calL$ and  $\check \calK \cong  \calK \sqcup \partial^r \calM$ or $ \calK \cong  \check \calK \sqcup \partial^r \calM$
where $\calM$ is a smooth  polychain in $   \Omega^\circ_{12}  $ transversal to $\calL$;
\item[(ii)]  $\calK \cong  \check \calK$ and   $\check \calL \cong  \calL \sqcup \partial^r \calN$ or $ \calL \cong  \check \calL \sqcup \partial^r \calN$
where $\calN$ is a smooth  polychain in $  \Omega^\circ_{34}   $ transversal to $\calK$.
\end{itemize}
\end{lemma}

\begin{proof}
The first claim   follows from   Lemma~\ref{defo1++} and the surjectivity in   Theorem~\ref{loopsp+++++}.
That we need only reduced polycycles  follows from the fact that the reduction of a (smooth) polycycle gives a homologous (smooth) polycycle.

We   prove the second claim of the lemma.
Consider   pairs of   reduced  polycycles  $(\calK_1, \calL)$ and $(\calK_2, \calL)$ transversely representing    $(a  ,b  )$.
Since $\calK_1$   is homologous to   $\calK_2 $, we have $   \calK_1 \sqcup \partial^r \calR_1   \cong  \calK_2 \sqcup \partial^r \calR_2$
for some   $(n+1)$-polychains $\calR_1,\calR_2$ in $   \Omega^\circ_{12}  $.
The injectivity in Theorem~\ref{loopsp+++++}   ensures that $\calR_1$, $\calR_2$ can   be chosen  to be  smooth.
Then  there are smooth polychains $\calR'_1,\calR'_2$ in $   \Omega^\circ_{12}  $  transversal to $\calL$
such that $\calK_1 \sqcup \partial^r \calR'_1   \cong  \calK_2 \sqcup \partial^r \calR'_2$.
These polychains are obtained from $\calR_1$, $\calR_2$ using
   the same method as in the proof of  Theorem~\ref{++smoothdefo1----}  with the following replacements: 
   $M  \leadsto     \Omega^\circ_{12}  $,   \lq\lq smooth"  $  \leadsto  $   \lq\lq  transversal to $\calL$", \lq\lq homotopy"  $ \leadsto $ \lq\lq smooth homotopy",
   Lemma~\ref{sssppsmodefo1----}  $ \leadsto $  Lemma~\ref{defo1++eee},     Lemma~\ref{ppsmodefo1----}   $ \leadsto $  Lemma~\ref{defo1++}.
  The  move $(\calK_1, \calL) \mapsto (\calK_2,\calL)$  expands as the  composition of the following   type (i) moves:
$$ (\calK_1, \calL) \mapsto (\calK_1 \sqcup \partial^r \calR'_1 , \calL) \mapsto (\calK_2 \sqcup \partial^r \calR'_2 , \calL) \mapsto  (\calK_2,\calL).$$
(The   middle move   is a type (i) move corresponding to $\calM=\varnothing$.)
 A similar argument shows that if     two pairs of polycycles $(\calK, \calL_1)$ and $(\calK,\calL_2)$ transversely represent    $(a  ,b  )$,
 then   the move $(\calK, \calL_1) \mapsto (\calK,\calL_2)$ is   a  composition  of type (ii) moves.

Consider now any   pairs of polycycles  $(\calK_1, \calL_1)$ and $(\calK_2,\calL_2)$ transversely representing    $(a  ,b  )$.  By Lemma \ref{defo1++},
there is an arbitrarily  $C^\infty$-small   smooth  deformation   of $ \calK_1$ into a polycycle $\calK$
  transversal to   $\calL_2$. We assume  the deformation to be so small   that   $ \calK$ is
  transversal to   $\calL_1$ as well. By Lemma~\ref{homotimplieshomol},  $ \calK\simeq \calK_1 $ represents~$a$.
By the previous paragraph, each of the moves $$(\calK_1, \calL_1) \mapsto (\calK,\calL_1) \mapsto (\calK,\calL_2) \mapsto (\calK_2,\calL_2) $$
expands as  a composition  of moves of types (i) and (ii).
\end{proof}

\section{Intersection of polychains} \label{operation_D}

  We define    \lq\lq intersection"    for  transversal polychains in $ \Omega^\circ_{12}$ and
   $ \Omega^\circ_{34}$.

\subsection{The intersection polychain}\label{def_D}

Let  $\calK=(K,\varphi,u,\kappa)$
be  a smooth polychain  of dimension $p$  in $ \Omega^\circ_{12}$ and
let   $\calL=(L,\psi,v,\lambda)$
be  a smooth polychain  of dimension   $q$ in   $ \Omega^\circ_{34}$.
Assume that   $\calK$ and $\calL$ are transversal in the sense of Section~\ref{transversality_polychains}.
We  derive from $\calK$ and $\calL$ an \lq\lq intersection polychain" in $ \Omega_{32} \times  \Omega_{14}$.

Let $\tilde \kappa: K \times I \to M$ and $\tilde \lambda: L \times I \to M$
be the adjoint maps of $\kappa $ and  $\lambda $ respectively. Set $N = K \times I \times L \times I$  
and consider the map $\tilde \kappa \times \tilde \lambda: N\to M\times M$.
  Since $K,L$ and $I$ are manifolds with faces of dimensions $p,q$ and $1$, respectively,~$N$ is a   manifold with faces of dimension $p+q+2$.  
  The transversality of~$\kappa$ and~$\lambda$ implies that the set 
$$
D = \big(\tilde \kappa \times \tilde \lambda\big)^{-1}(\diag_M)  \subset N
$$
 is empty  if  $p+q+2< n$ and  is a $(p+q+2-n)$-dimensional manifold with corners if  $p+q+2\geq n$.
 In the latter case each point   of~$D$ has a neighborhood~$V$ in~$N$ such that $V\cap D$ is homeomorphic to  $\RR^u \times [0,\infty)^{ v}$ for some integers  $u, v \geq 0$ with $u+v=p+q+2-n$    and $V$ is   homeomorphic to $ \RR^n \times (V\cap D)$.
 These claims follow from the general theorems about transversality and about submanifolds of manifolds with corners, see \cite[Propositions 3.1.14 and  7.2.7]{MrOd}.
  Consequently,   $P(D)\subset P(N) $ so that we can consider the   commutative diagram of inclusion maps
$$
\xymatrix{
\pi_0(P(D) \cap V )  \ar[rr]^-{i} \ar[d]_-{j' } && \pi_0(P(N) \cap  V) \ar[d]^-{j  } \\
\pi_0(P(D) ) \ar[rr]^-{i'}&& \pi_0(P(N) )\, .
}
$$
The  structure of   $ V $ described above implies that $i$ is a bijection  between $v$-element sets.
Since~$N$ is a manifold with faces,~$j$ is an injection.
  Therefore $j'$ is injective which implies that $D$ is a manifold with faces.
The faces  of $D$ are    the connected components of the intersections of $D$ with faces of $N$.

We now upgrade~$D$ to a polychain in $ \Omega_{32} \times  \Omega_{14}$.
First of all, we orient $D$ as follows. We  use the orientation of $\diag_M\approx M$
and the product orientation of $M\times M$ to orient the normal vector bundle of $\diag_M$ in $M\times M$ (see the Introduction for our orientation conventions). Next, we pull-back this orientation of the normal vector bundle along $\tilde \kappa \times \tilde \lambda$ to obtain an orientation of the normal vector bundle of~$D$ in~$N$. The latter orientation together with  the product orientation in $N=  K \times I \times L \times I$
induces an orientation of $D$.
We  also  equip   $D$
with the weight $w:\pi_0(D) \to \kk$ which,
for any connected components $X$ of $K$ and   $Y$ of $L$,
carries all connected components of $D$ contained in $X \times I \times Y \times I$ to $u(X) \, v(Y) \in \kk$.

Next, we define a continuous map
$\kappa \, \tilde \losange \, \lambda: D\times I \to M \times M$ by
\begin{equation} \label{bowtie}
(\kappa \, \tilde \losange \, \lambda )  ( x, s, y,  t, u)=
\begin{cases}
  (\lambda (y) (t\ast u), \kappa(x) (s\ast u)   )  & \text{if   $0\leq u \leq 1/2$}\\
(\kappa(x) (s \ast u), \lambda(y) (t \ast u)) & \text{if $1/2  \leq  u \leq 1$}
\end{cases}\notag
\end{equation}
for any $(x,s,y,t)\in D \subset K \times I  \times L \times I$ and $u \in I$, where   we  set
$$
\ell \ast u =
\left\{\begin{array}{ll}
2\ell u & \hbox{for    $\ell \in I, u\in[0,1/2]$}\\
1-2 (1-\ell   )(1-u ) &  \hbox{for   $\ell \in I, u\in [1/2,1]$.}
\end{array}\right.
$$
The key property of the operation $\ast$ is that for any $\ell, u \in I $, we have
$0\leq \ell \ast u \leq \ell  $
 if $  u\in [0,1/2]$ and $ \ell \leq \ell \ast u \leq 1$    if $u\in [1/2, 1]$.
For  a fixed  $(x, s, y,  t)\in D$, the point   $(\kappa \, \tilde \losange \, \lambda ) (x, s,y, t, u)\in M\times M$
moves   along the path $\big(\tilde \lambda(y,t\ast u), \tilde \kappa(x,s\ast u)\big)$
from $(\star_3, \star_1)$ to the diagonal point $\big(\tilde\kappa(x,s ),\tilde \lambda (y,t )\big)$ as $u$ increases from $0$ to $1/2$
and, next, it moves    from that diagonal point  to $(\star_2,\star_4)$
along the path $\big(\tilde \kappa(x,s\ast u), \tilde \lambda(y,t\ast u)\big)$ as $u$    increases from $1/2$ to~$1$: see Figure~\ref{diamond}.
  Thus   the  map $\kappa \, \tilde \losange \, \lambda$ is adjoint to a continuous map
$$
\kappa   \losange \lambda: D \longrightarrow   \Omega \big(M\times M,  (\star_3, \star_1), (\star_2, \star_4)\big)
=  \Omega_{32} \times \Omega_{14}
$$
whose  coordinate maps  are denoted by $\kappa   \triangleleft \lambda: D \to    \Omega_{32}$ and $\kappa   \triangleright \lambda: D \to  \Omega_{14}$.

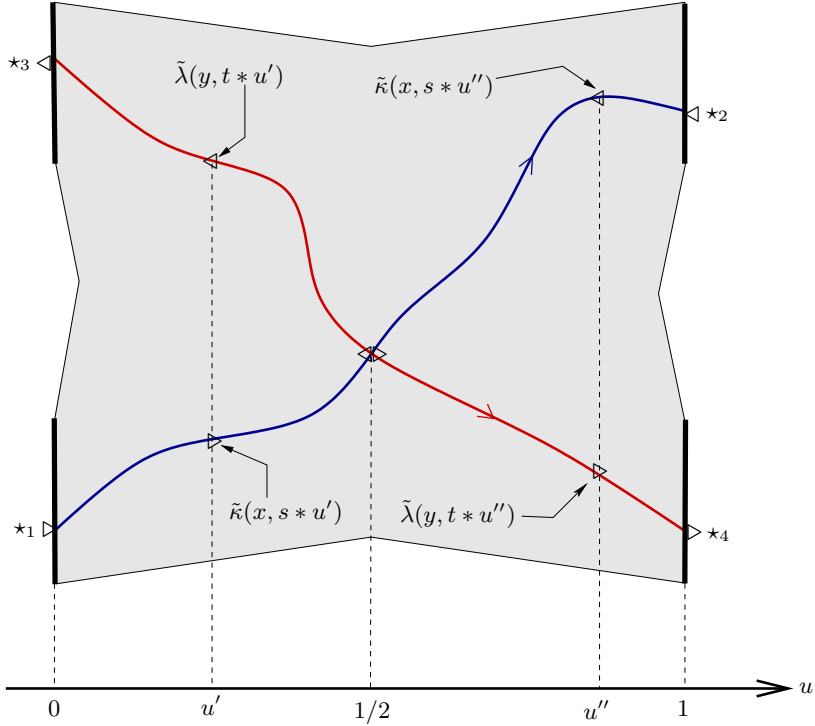
\begin{figure}[h]
\begin{center}
\labellist \small \hair 2pt
\pinlabel {$0$} [t] at 35 3
\pinlabel {$1$} [t] at 466 3
\pinlabel {$1/2$} [t] at 253 3
\pinlabel {$u'$} [t] at 143 6
\pinlabel {$u''$} [t] at 407 3
\pinlabel {$u$} [l] at 542 8
\pinlabel {\Large $\triangleright$} at 32 116
\pinlabel {$\star_1$} [r] at 27 116
\pinlabel {\Large $\triangleright$} at 474 114
\pinlabel {\Large $\triangleright$} at 145 177
\pinlabel {\Large $\triangleright$} at 409 156
\pinlabel {$\star_4$} [l] at 479 113
\pinlabel {\Large $\losange$} at 253 236
\pinlabel {\Large $\triangleleft$} at 28 436
\pinlabel {$\star_3$} [r] at 23 436
\pinlabel {\Large $\triangleleft$} at 472 401
\pinlabel {\Large $\triangleleft$} at 142 369
\pinlabel {\Large $\triangleleft$} at 407 412
\pinlabel {$\star_2$} [l] at 478 401
\pinlabel { $\tilde\lambda(y,t\ast u'')$} [r] at 355 127
\pinlabel { $\tilde\lambda(y,t\ast u')$} [b] at 155 415
\pinlabel { $\tilde\kappa(x,s\ast u'')$} [r] at 340 420
\pinlabel { $\tilde\kappa(x,s\ast u')$} [t] at 194 140
\endlabellist
\includegraphics[scale=0.55]{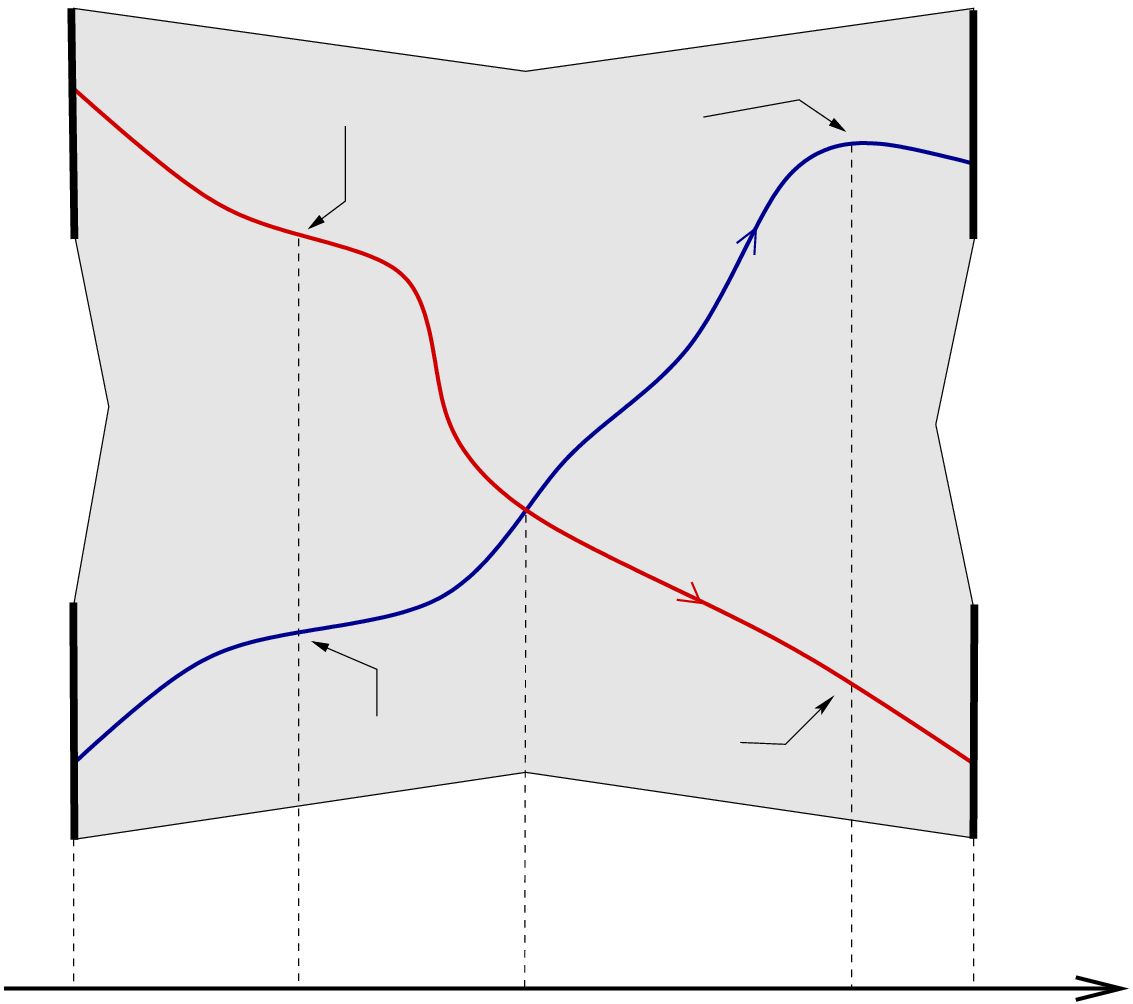}
\end{center}
\caption{  The pair
$(\triangleleft\,,\triangleright)=(\kappa  \, \tilde{\losange} \, \lambda)(x,s,y,t,u)\in M\times M$
  for a fixed $(x, s, y,  t)\in D$ and $u$ running from $0$ to $1$.}
\label{diamond}
\end{figure}

Finally, we define a partition  of $D$.    Here   we need the assumption that the images of $\kappa$ and $\lambda$
lie in  $ \Omega^\circ_{12}\subset \Omega_{12}$ and $ \Omega^\circ_{34}\subset \Omega_{34}$, respectively. This assumption implies that
$$D \subset
 K \times \Int(I) \times L \times \Int(I)\subset N.$$
Therefore each face $F$ of $D$ is contained in a unique smallest face
$$
N_F = A_F \times I \times B_F \times I
$$
of $N$, where $A_F$ is a face  of $K$ and   $B_F$ is a  face of  $L$.
Note that the codimension of~$F$ in~$D$ is equal to the codimension of $N_F$ in~$N$.
We declare two  faces~$F$ and~$F'$ of~$D$   to have the same type
if and only if $A=A_F$ has the same type as $A'=A_{F'}$ in $K$, $B=B_F$ has the same type as $B'=B_{F'}$ in $L$,
and the diffeomorphism
$$
\xymatrix{
N_F= A \times I \times B \times I
\ar[rrr]^-{\varphi_{A,A'} \times \id_I \times \psi_{B,B'} \times \id_I} &&&
 A' \times I \times B' \times I  =N_{F'}
}
$$
carries $F$ onto $F'$.
By restriction, we obtain a diffeomorphism $\theta_{F,F'}:F \to F'$
for any such $F$, $F'$. This defines a partition, $\theta$, of $D$.
Note that, for each face $F$ of~$D$, the faces of $D$ of the same type as $F$
are in one-to-one correspondence with the pairs $(A',B')$
where $A'$ is a face of $K$ of the same type as $A_F$
and $B'$ is a face of $L$ of the same type as $B_F$.

\begin{lemma}
The tuple $
\calD(\calK,\calL) = (D,\theta,w,\kappa\losange \lambda)
$
is a polychain in $\Omega_{32} \times \Omega_{14}$.
\end{lemma}

\begin{proof} We need only to show that the  map $\kappa\losange \lambda$ is compatible with the partition~$\theta$.
Consider two faces $F$ and $F'$  of the same type in $D$ and set $$A=A_F, \,\, B=B_F,  \,\,  A'=A_{F'}, \quad  {\rm {and}} \quad B'=B_{F'}.$$  For any $(x,s,y,t) \in F$ and $u\in [0,1/2]$, we have
\begin{eqnarray*}
(\kappa\losange \lambda)\left(\theta_{F,F'}(x,s,y,t)\right)(u)
&=& (\kappa\losange \lambda)\left(\varphi_{A,A'}(x),s,\psi_{B,B'}(y),t \right)(u)\\
&=&  \left(\lambda\big(\psi_{B,B'}(y) \big)(t\ast u), \kappa\big(\varphi_{A,A'}(x)\big) (s\ast u)  \right) \\
&=&  \left(\lambda\big(y \big)(t\ast u), \kappa\big(x\big) (s\ast u)  \right)  \\
&=& (\kappa\losange \lambda)\left(x,s,y,t\right)(u)
\end{eqnarray*}
where the third equality follows from the compatibility of  $\kappa$ with $\varphi$ and of $\lambda$ with~$\psi$.
A similar argument works for $u\in[1/2,1]$.
Thus,  $(\kappa\losange \lambda) \, \theta_{F, F'} = (\kappa\losange \lambda)\vert_{F}$.
\end{proof}

\subsection{Properties of $\calD$} \label{prop_D}

We study the behavior of  the polychain  $\calD(\calK,\calL)$    under
the operations on $ \calK$ and $\calL$   introduced in Sections~\ref{reducedpolychains--} and   \ref{reducedpolychains}.

\begin{lemma}\label{reduction_D}
Let  $\calK,\calK'$ be smooth $p$-polychains in $\Omega^\circ_{12}$
and  $\calL,\calL'$ be smooth $q$-polychains in $\Omega^\circ_{34}$ such  that $\calK,\calK'$ and $\calL,\calL'$ are pairwise transversal. Then
\begin{itemize}
\item[(i)]  $\calD(k\calK ,\calL)\cong \calD(\calK , k\calL)\cong k \calD(\calK ,\calL)$ for any $k\in \kk$;
\item[(ii)] $\red \calD(\red\calK,\red\calL)=\red \calD(\calK,\calL)$;
\item[(iii)]  $\partial \calD(\calK,\calL)   \cong   (-1)^n \calD(\partial \calK, \calL)\sqcup   (-1)^{n+p+1} \calD( \calK,\partial \calL)$;
  \item[(iv)] $\partial^r \calD(\calK,\calL) = (-1)^n \red \calD(\partial^r \calK, \red\calL) \sqcup   (-1)^{n+p+1} \red \calD( \red\calK,\partial^r \calL)$;
\item[(v)] $\calD(\calK \sqcup \calK',\calL) \cong \calD(\calK ,\calL) \sqcup \calD( \calK',\calL)$, 
 $\calD(\calK,\calL\sqcup \calL') \cong \calD(\calK ,\calL) \sqcup \calD( \calK,\calL')$.
\end{itemize}
\end{lemma}

\begin{proof} Claims (i) and (v) are obvious. Claim  (iv) easily follows from (ii) and~(iii). We prove (ii).
It is clear  that
\begin{eqnarray*}  \red_+ \calD(\red_+(-),\red_+(-))&=&\red_+\calD(-,-)
\end{eqnarray*}
and
\begin{eqnarray*}
  \red_0\calD(\red_0(-),\red_0(-)) &=& \red_0\calD(-,-)   .
\end{eqnarray*}
Using the identities $ \red  \red_0= \red  =\red \red_+$, we conclude that
\begin{eqnarray*}
   \red \calD(\red \calK,\red \calL) &=&  \red  \red_0  \calD(\red_0\red_+\calK,\red_0\red_+\calL)   \\
&=& \red \red_0\calD( \red_+\calK,\red_+\calL)\\
&=&   \red   \red_+ \calD(\red_+\calK,\red_+\calL) \\
&=&   \red   \red_+ \calD(\calK, \calL)  \ = \ \red \calD(\calK,\calL) .
\end{eqnarray*}

  We now prove (iii).  Let $\calK=(K,\varphi,u,\kappa)$, $\calL=(L,\psi,v,\lambda)$ and
$\calD(\calK,\calL)=(D,\theta,w,\kappa\losange \lambda)$   as  in Section \ref{def_D}.
  Consider  the boundary polychain
$$\partial \calD( \calK,\calL)= (D^\partial, \theta^\partial, w^\partial, (\kappa\losange \lambda)^\partial)$$ as defined in  Section~\ref{reducedpolychains},
as well as the polychains
$\calD(\partial \calK,\calL)= (\stl D,\stl \theta,\stl w,\kappa^\partial \losange \lambda)$ and $\calD( \calK,\partial\calL)=(D\str, \theta\str,w\str ,\kappa \losange \lambda^\partial).$
 We verify  that
\begin{equation}\label{D-partial}
D^\partial  \cong   (-1)^n\, \stl D \, \sqcup \, (-1)^{n+p+1} D\str   .
\end{equation}
Consider a principal face $F$ of $D$. Since the codimension of~$F$ in~$D$ is equal to the codimension of  $N_F=A_F \times I \times B_F \times I$  in $N=K \times I \times L \times I$, the
  face $N_F $ is  a principal face of $N$. Therefore either $A_F$ is  a connected component  of $K$  and  $B_F$ is a principal face of $L$,
 or, $A_F$ is a principal face of $K$ and $B_F$ is a connected component of $L$.
 We first analyze the former case. Set  $N\str= K \times I \times L^\partial \times I$.
Then $F\subset D \subset N$ corresponds to  a connected component $F\str$ of~$D\str \subset N\str$
 via the     map $\id_K \times \id_I \times \, \iota \times \id_I:N\str \to N$  where $\iota: L^\partial \to L$ is the natural map as in Section~\ref{reducedpolychains}.
 The  orientation of $F$ induced by that of   $\partial D \subset D$  may differ from the orientation of $F\str$ induced by $D\str$, and
we now compute this difference.
  Let  $\varepsilon^1$ be  the trivial $1$-dimensional  vector bundle
  equipped  with the canonical orientation and let $-\varepsilon^1$ be the same bundle with the opposite orientation. Given a cartesian product of topological spaces,
   let  $\pr_i$ denote the projection onto the $i$-th factor.
   Set $N\str_F =K \times I \times  B_F  \times I$, which is a submanifold with faces of $N^*$ of  codimension~$0$ containing $F\str$.
 We can also view $N\str_F$ as a submanifold of $N$ of codimension $1$, so that $F\str \subset N\str_F$ corresponds to $F \subset N$.
Using the orientation conventions of the Introduction and using the letter $T$ for the tangent vector bundle of a manifold,
we obtain the following  orientation-preserving   isomorphisms of oriented vector bundles:
\begin{eqnarray*}
TN\vert_{  N\str_F  } & = & \pr^*_1(TK)\vert_{  N\str_F  } \oplus \pr^*_2(TI)\vert_{  N\str_F  }
 \oplus \pr^*_3(TL)\vert_{  N\str_F  } \oplus \pr^*_4(TI)\vert_{  N\str_F  }  \\
 & \cong & \pr^*_1(TK) \oplus \pr^*_2(TI) \oplus \varepsilon^1 \oplus \pr^*_3(T   B_F  ) \oplus \pr^*_4(TI)\\
 &\cong & (-1)^{p+1}  \varepsilon^1\oplus
 \underbrace{\pr^*_1(TK) \oplus \pr^*_2(TI) \oplus \pr^*_3(T  B_F ) \oplus \pr^*_4(TI).}_{= T   N\str_F  }
\end{eqnarray*}
Therefore
\begin{eqnarray*}
TN\vert_{{F\str}} & \cong & (-1)^{p+1}  \varepsilon^1\oplus T   N\str_F   \vert_{  {F}\str  }\\
& \cong & (-1)^{p+1}  \varepsilon^1\oplus \bignu_{   N^*_F  } {  F\str  }  \oplus T {  F\str  }
  \ = \ (-1)^{p+1}  \varepsilon^1\oplus \bignu_{ N^* } {F\str}  \oplus T { F\str }
\end{eqnarray*}
where   the letter   $\bignu$ stands for the normal vector bundle of a submanifold in the ambient manifold.
On the other hand, restricting  the   orientation-preserving   isomorphism  of oriented vector bundles $TN\vert_D \cong \bignu_ND \oplus TD$ to $F$ we obtain  that
\begin{eqnarray*}
TN\vert_F \ \cong \  \bignu_ND \vert_F \oplus TD\vert_F
&\cong&  \bignu_ND \vert_F \oplus \varepsilon^1\oplus TF
 \cong (-1)^n \varepsilon^1 \oplus  \bignu_ND \vert_F \oplus  TF.
\end{eqnarray*}
Since the orientations of   $\bignu_{ N\str } {F\str} = \bignu_{ N\str_F } {F\str}$   and $\bignu_ND $
are both induced  by the orientation of the normal bundle of $\diag_M$ inx
 $M\times M$,
  the bundle isomorphism $  \bignu_{N\str_F} {F\str}   \to \bignu_ND \vert_F$ induced by the inclusion $  N^*_F   \subset N$   is   orientation-preserving.
Combining with the computations above, we deduce that    $TF \cong (-1)^{n+p+1}\, TF\str$.
The case where $A_F$ is a principal face of~$K$ and $B_F$ is a connected component of~$L$ is treated   similarly.
In this case,    $F$    corresponds to   a connected component~$\stl F $  of~$\stl D$,
and the orientation of $F $ induced from that of  $D$  differs     by $(-1)^n$ from the orientation of $\stl F$ induced by $\stl D$.
This gives the   diffeomorphism    \eqref{D-partial} of oriented  manifolds with faces,
which   is easily checked to be a diffeomorphism  of polychains as in (iii).
\end{proof}

\section{The operation $\widetilde \Upsilon$} \label{Upsilon-tilde}

  We  introduce an operation $\widetilde \Upsilon$ in the face homology of path  spaces.

\subsection{Definition and properties of  $\widetilde \Upsilon$}  \label{dfp_tildeUpsilon} 

  First, we show that the  intersection  operation   defined in Section \ref{operation_D}
induces an operation in face homology.

\begin{lemma}\label{Upsilon_tilde_def}
For any integers $p,q \geq 0$,
the  intersection $(\calK, \calL)\mapsto \calD(\calK, \calL)$ from  Section~\ref{def_D}  induces a bilinear map
$$
\widetilde H_p (\Omega_{12}) \times \widetilde H_q (\Omega_{34})
\to \widetilde H_{p+q+2-n} ( \Omega_{32} \times \Omega_{14}).
$$
\end{lemma}

\begin{proof}
Consider any face homology classes
${a}\in \widetilde H_p (\Omega_{12})$ and ${b}\in \widetilde H_q(\Omega_{34})$.
By  Lemmas \ref{loopsp}   and    \ref{keytheor}, the pair $({a},{b})$ can be transversely represented
by a smooth reduced  $p$-polycycle $\calK$  in $\Omega_{12}^\circ$
and a smooth reduced $q$-polycycle  $\calL$ in  $\Omega_{34}^\circ$.
It  follows from   Lemma \ref{reduction_D}.(iv)   that  the polychain $\calD(\calK,\calL)$ in $\Omega_{32} \times \Omega_{14}$  is a polycycle.
Consider another such pair  $(\check\calK,\check\calL)$  transversely representing $({a},{b})$.
We claim that the polycycles $\calD(\calK,\calL)$ and $\calD(\check\calK,\check \calL)$ are homologous.
By Lemma \ref{keytheor},   it suffices to prove this claim in  the following two cases:
\begin{itemize}
\item $\calL=\check\calL$ and there exists
a smooth $(p+1)$-polychain $\calM$ in $\Omega^\circ_{12}$ transversal to $\calL$
such that  $\calK  \cong \check\calK  \sqcup \partial^r\calM$
or $\check\calK  \cong \calK  \sqcup \partial^r\calM$;
\item   $\calK=\check\calK$ and there exists
a smooth $(q+1)$-polychain $\calN$ in $\Omega^\circ_{34}$ transversal to $\calK$
such that  $\calL  \cong \check\calL  \sqcup \partial^r\calN$
or $\check\calL  \cong \calL  \sqcup \partial^r\calN$.
\end{itemize}
Assume for concreteness that  $\calL=\check\calL$ and  $\calK  \cong \check\calK  \sqcup \partial^r\calM$ (the other cases can be treated    similarly).
Since $\calL$ is a reduced  polycycle,   Lemma \ref{reduction_D}.(iv)   implies that
$$
\partial^r  \calD\left(\calM,\calL\right)
= (-1)^n  \red  \calD\left(\partial^r \calM,\calL\right).
$$
This and    Lemma  \ref{reduction_D}.(v)   imply   that
$$
 \red  \calD\left(\calK,\calL\right) \cong  \red  \calD\left(\check\calK  \sqcup \partial^r\calM,\check\calL\right)
\cong  \red  \calD\left(\check\calK ,\check\calL\right) \sqcup  \partial^r \calD\left(  (-1)^n \calM,\check\calL\right) .
$$
 We conclude   that
$\calD\left(\calK,\calL\right)$ is homologous to $\calD\left(\check\calK,\check\calL\right)$.
Thus,  the   face homology class $\lb\calD\left(\calK,\calL\right) \rb\in \widetilde  H_{p+q+2-n} ( \Omega_{32} \times \Omega_{14})$ of $\calD\left(\calK,\calL\right)$
  depends only on ${a}\in \widetilde  H_p (\Omega_{12})$ and ${b} \in \widetilde  H_q (\Omega_{34})$. This defines the pairing in the statement of the lemma.
  The bilinearity of this  pairing  follows from    assertions (i) and (v) in Lemma  \ref{reduction_D}.
\end{proof}

The pairing produced by Lemma \ref{Upsilon_tilde_def} induces a linear map
$$
\widetilde H_p (\Omega_{12}) \otimes \widetilde H_q (\Omega_{34})
  \longrightarrow   \widetilde H_{p+q+2-n} ( \Omega_{32} \times \Omega_{14}).
$$ Taking the direct sum over all   $p,q\geq 0$,
we obtain a  linear map  of degree $2-n$
\begin{equation}\label{defoftildeUpsilon}
\widetilde \Upsilon  : \widetilde H_\ast (\Omega_{12}) \otimes \widetilde H_\ast (\Omega_{34}) \longrightarrow   \widetilde H_{\ast} ( \Omega_{32} \times \Omega_{14}).
\end{equation}
To stress the role of the   tuple of  base points $(\star_1,\star_2,\star_3,\star_4)$    we will also denote this map by $\widetilde \Upsilon_{12,34}$.
  Any permutation $(\star_i,\star_j,\star_k,\star_l)$ of  $(\star_1,\star_2,\star_3,\star_4)$    
  such that $\{\star_i,\star_j\}\cap\{\star_k,\star_l\}=\varnothing$ yields   a   map
$$
\widetilde \Upsilon_{ ij,kl}:  \widetilde H_\ast (\Omega_{ij}) \otimes \widetilde H_\ast (\Omega_{kl})
 \longrightarrow   \widetilde H_{\ast} ( \Omega_{kj} \times \Omega_{il}).
 $$
We now  establish the following symmetry  for $\widetilde \Upsilon$.

\begin{lemma}\label{antisymmetry_Upsilon_tilde} Let   $\perm:  \Omega_{32} \times \Omega_{14} \to \Omega_{14} \times \Omega_{32}$
be the map   permuting  the two factors of the cartesian product.
For any   ${a}\in \widetilde H_p (\Omega_{12})$ and ${b}\in \widetilde H_q (\Omega_{34})$ with $p,q\geq 0$,
$$
\perm_*  \widetilde \Upsilon_{12,34}({a} \otimes {b})=   (-1)^{(p+1)(q+1)+n }\,  \widetilde \Upsilon_{34,12} ({b} \otimes {a}).
$$
\end{lemma}

 \begin{proof}
  We  assume that $({a},{b})$ is  transversely represented
by a smooth   reduced   $p$-polycycle $\calK=(K,\varphi,u,\kappa)$  in $\Omega_{12}^\circ$
and a smooth   reduced    $q$-polycycle  $\calL=(L,\psi,v,\lambda)$ in $\Omega_{34}^\circ$.
 Let $\calD(\calK,\calL)=(D,\theta,w,\kappa\losange \lambda)$ and $\calD(\calL,\calK)=(D',\theta',w',\lambda\losange \kappa)$.
Let $\mathsf{q}$ be the permutation map $M \times M\to M \times M, (m_1, m_2)\mapsto  (m_2,m_1)$.
This map  preserves $\diag_M$ pointwise  and preserves (respectively,\ inverts) the orientation of the normal bundle of $\diag_M$ in $M\times M$ if $n$ is even (respectively, odd).
Let  $$\mathsf{h} : K \times I \times L \times I \to  L \times I \times K \times I$$ be the permutation map defined by $(k,s,l,t)\mapsto (l,t,k,s)$.
Clearly,  $$\deg \mathsf{h} = (-1)^{(p+1)(q+1)} \quad {\rm {and}} \quad (\tilde \lambda \times \tilde \kappa)\mathsf{h}   =   \mathsf{q}  (\tilde \kappa \times \tilde \lambda).$$
Thus,  $\mathsf{h}$ restricts to a diffeomorphism $\mathsf{h}\vert_D:D \to D'$  of  degree $ (-1)^{(p+1)(q+1)+n}$.
This diffeomorphism carries   the weight $w$ to $w'$ and the partition $\theta$ to   $\theta'$.
Also,  $(\lambda \losange \kappa)  \circ    \mathsf{h}\vert_D  = \perm   \circ   (\kappa \losange \lambda)$.
Thus, $\mathsf{h}\vert_D$ is a diffeomorphism of  the polychains  $\perm_*\calD(\calK,\calL)$ and  $(-1)^{(p+1)(q+1)+n}\calD(\calL,\calK)$. We conclude that
\begin{eqnarray*}
\perm_*\widetilde \Upsilon_{12,34}({a} \otimes {b}) \  = \ \lb \perm_*\calD(\calK,\calL) \rb &= & (-1)^{(p+1)(q+1)+n} \lb  \calD(\calL,\calK) \rb \\
&=& (-1)^{(p+1)(q+1)+n }   \widetilde \Upsilon_{34,12} ({b} \otimes {a}).
\end{eqnarray*}

\up
\end{proof}

\subsection{Computation of $\widetilde \Upsilon$}\label{computationoftildeupsilon}

To evaluate $\widetilde \Upsilon  $ on a pair of face homology classes in $ \Omega_{12},  \Omega_{34}  $,
we   represent these classes by smooth  reduced   transversal polycycles in
$\Omega_{12}^\circ, \Omega_{34}^\circ$ and    take the face homology class of the   intersection polycycle.
We   now explain    how to compute $\widetilde \Upsilon  $ from  more general   polycycles in $ \Omega_{12},  \Omega_{34}  $.

We say that   polycycles (possibly non-smooth and non-reduced)
$\calK=(K,\varphi, u,\kappa)$ in $\Omega_{12} $ and $\calL=(L,\psi,v,\lambda)$ in $ \Omega_{34} $  are   \index{polycycle!admissible} \emph{admissible}
if there exist  open sets $U\subset K\times \Int (I)$ and $V\subset L\times \Int (I)$ such that
\begin{itemize}
\item[(i)] the maps $\tilde \kappa \vert_U:U\to M$ and $\tilde \lambda \vert_V:V\to M$ are smooth and their images do not meet $\partial M$;
\item[(ii)]  $(\tilde \kappa\times \tilde \lambda )^{-1} (\diag_M) \subset U\times V$;
\item[(iii)] for any face $E$ of $K$ and any face $F$ of $L$, the restriction of $\tilde \kappa\times \tilde \lambda$ to
$$\big (  \Int (E  \times I) \cap U \big ) \times \big (  \Int (F \times  I)  \cap V \big )$$
 is transversal to   $\Int(\diag_M)$ in the usual sense of differential topology.
\end{itemize}
If $\calK$ and $\calL$ are admissible, then   we can define the intersection polycycle $\calD(\calK, \calL)$ in $\Omega_{32}\times \Omega_{14}$
repeating word for word the definitions of Section \ref{operation_D}.
The   polycycle  $\calD(\calK, \calL)$ depends only on   $\calK , \calL$   and does not depend on the choice   of $U, V$.

\begin{lemma}\label{newcomputation} Let $\calK $   and $\calL$ be admissible polycycles in $\Omega_{12} $ and $ \Omega_{34} $ representing, respectively,
$ a\in  \widetilde H_\ast (\Omega_{12})$ and $ b\in  \widetilde H_\ast (\Omega_{34}) $. Then  $\widetilde \Upsilon (a,b)= \lb \calD(\calK, \calL) \rb$.
\end{lemma}

 \begin{proof} Let $\calK=(K,\varphi, u,\kappa)$  and $\calL=(L,\psi,v,\lambda)$.
 The  set $$ (\tilde \kappa\times \tilde \lambda )^{-1} (\diag_M) \subset K\times I\times L\times I$$  is closed and, hence, compact.
 By (ii), there are compact sets $A\subset U$ and $B\subset V$ such that $(\tilde \kappa\times \tilde \lambda )^{-1} (\diag_M) \subset A\times B$.
 Pick a small deformation  of~$\kappa$ and~$\lambda$  into smooth maps (in the class of maps compatible with the partitions).
 The deformation  may be chosen to be  constant on  some open neighborhoods $U'\subset U$, $V' \subset V$ of $A,B$, respectively,
 and  to be so small that the condition  (ii)    with $U \times V$ replaced by $U' \times V'$ is met during the deformation.
 The condition   (iii)  with $U,V$ replaced by $U', V'$ is automatically met during the deformation.
  By Lemma \ref{homotimplieshomol},   the face homology class $\lb \calD(\calK, \calL) \rb$ is preserved under such a deformation.
 Thus, without loss of generality we can assume from the very beginning that the maps $\kappa$ and $\lambda$ are smooth.

 Pick a small neighborhood $W$ of $\partial M$ in $M$ such that   $\tilde \kappa(U)\cup  \tilde \lambda (V) \subset M\setminus W$.
 The  proof  of Lemma \ref{loopsp} and Theorem \ref{loopsp+++++}  provides,  for any $i,j\in \{1,2,3,4\}$,
  a  homotopy  of the identity map $\id:\Omega_{ij}  \to \Omega_{ij}$ into a map  $f_{ij}: \Omega_{ij}  \to \Omega_{ij}^\circ \subset \Omega_{ij} $
such that    smooth polycycles in $\Omega_{ij}$ remain smooth throughout the homotopy.  The homotopy  acts on a path in $M$ from $\star_i$ to $\star_j$
by   pushing the interior points of the path inside~$M$ along a 1-parameter family of embeddings $M\hookrightarrow M$.
We can assume that   these embeddings are constant on $M \setminus W$ and so, the homotopy  fixes all points of the paths  lying in $M\setminus W$.
For $i=1, j=2$ and $i=3, j=4$, these homotopies induce  a smooth deformation of polycycles
$$\{\calK^t\}_{t\in I} = \{(K,\varphi, u, \kappa^t)\}_{t\in I}, \quad \{\calL^t\}_{t\in I} = \{(L,\psi,v, \lambda^t)\}_{t\in I}$$
 where $\kappa^0=\kappa$, $\lambda^0=\lambda$,   $\kappa^1=f_{12} \kappa$, $\lambda^1=  f_{34}\lambda$.
  Our assumptions ensure that  $\tilde \kappa^t\vert_U= \tilde \kappa\vert_U$ and $\tilde \lambda^t\vert_V = \tilde \lambda\vert_V$ for all  $t\in I$.
Thus  the set $(\tilde \kappa^t\times \tilde \lambda^t )^{-1} (\diag_M)$ does not depend on~$t$,
 and $\calK^t,\calL^t$ are admissible for all $t\in I$.
  Then  the polycycle  $\calD( \calK^1, \calL^1)$ is obtained from the polycycle  $\calD( \calK^0, \calL^0)=\calD( \calK, \calL)$ by deformation.
   Hence,   by Lemma \ref{homotimplieshomol},
    $\lb \calD( \calK^1, \calL^1) \rb=\lb \calD( \calK, \calL) \rb$. The polycycles $\red ( \calK^1) $ in $\Omega_{12}^\circ$ and   $  \red (\calL^1) $ in $\Omega_{34}^\circ$
   transversely represent the pair $(a,b)$.    We conclude that
\begin{eqnarray*}
\widetilde \Upsilon (a,b) &= & \lb \calD( \red \calK^1, \red \calL^1)  \rb \\  & =&   \lb \red \calD( \red \calK^1, \red \calL^1)  \rb \\
& = &    \lb \red \calD(  \calK^1,  \calL^1) \rb  \ = \  \lb \calD( \calK^1, \calL^1)  \rb \ = \ \lb \calD( \calK, \calL) \rb
\end{eqnarray*}
where the third equality is given by Lemma \ref{reduction_D}.(ii).
\end{proof}

\subsection{The Leibniz  rule} \label{Leibniz}

We formulate  for $\widetilde \Upsilon$   a Leibniz-type rule in the second variable.
(Since $\widetilde \Upsilon$ is symmetric   in the sense of Lemma \ref{antisymmetry_Upsilon_tilde},  a  Leibniz-type rule in the first variable easily follows.)
Pick  a fifth  base  point $\star_5\in \partial M $. For   any  $i,j,k \in \{1, \ldots , 5\}$,
the concatenation of paths $  \conc  : \Omega_{ij} \times \Omega_{jk}\to \Omega_{ik}$ induces a bilinear \index{concatenation pairing} \emph{concatenation pairing}
\begin{equation}\label{concatenation}
\widetilde H_*(\Omega_{ij})\times \widetilde H_*(\Omega_{jk}) \longrightarrow \widetilde H_*(\Omega_{ik}), \
(a , b) \longmapsto ab=\conc_*(a \times b),
\end{equation}
Similarly, for any $i,j,k,l,m \in \{1, \ldots , 5\}$,
the map  $\conc: \Omega_{ij} \times \Omega_{jk}\to \Omega_{ik}$  induces  bilinear pairings
$$
\widetilde H_{\ast} ( \Omega_{lm}  \times   \Omega_{ij}) \times \widetilde H_\ast(\Omega_{jk})
  \to   \widetilde H_{\ast} ( \Omega_{lm} \times    \Omega_{ik}),  \ (x , a)   \mapsto   xa=(\id \times \conc)_*(x \times a),
$$
$$
\widetilde H_\ast(\Omega_{ij}) \times \widetilde H_{\ast} ( \Omega_{jk} \times \Omega_{lm})
  \to    \widetilde H_{\ast} ( \Omega_{ik}\times \Omega_{lm}), \ (a , x)   \mapsto   ax=(\conc \times \id)_*(a \times x).
$$

\begin{lemma}\label{Leibniz_Upsilon_tilde}
If $\star_5\in \partial M \setminus \{\star_1,\star_2\}$,
then for any   $a\in \widetilde H_p (\Omega_{12})$,  $b\in \widetilde H_q (\Omega_{34})$, and $c \in  \widetilde H_{i} (\Omega_{45})$ with  $p,q,{i} \geq 0$,
$$
\widetilde\Upsilon_{12,35}(a \otimes bc)=  (-1)^{{i} } \widetilde\Upsilon_{12,34}(a \otimes b ) \, c+ (-1)^{(p+n+1)q} \,b \, \widetilde\Upsilon_{12, 45} (a \otimes c).
$$
\end{lemma}

\begin{proof}
Let $\calK=(K,\varphi,u,\kappa),\calL=(L,\psi,v,\lambda),\calR=(R,\chi,z,\rho)$
be  smooth polycycles in $\Omega^\circ_{12},\Omega^\circ_{34}, \Omega^\circ_{45}$ representing $a,b,c$ respectively.
 Applying Lemma \ref{defo1++} twice (and choosing homotopy there sufficiently small), we can  assume that $\calK$ is transversal to both $\calL$ and   $\calR$.
 Then   $bc$   is represented by the following   polycycle in $\Omega_{35}$:
$$
\calN= \conc_*(\calL \times \calR) = (L\times R,\psi \times \chi, v \times z,\eta)
$$
where $\eta =  \conc  (\lambda \times \rho)$ and  the adjoint map $\widetilde \eta: L\times R \times I \to M$   is computed by
$$
\widetilde\eta(l,r,t) = \left\{\begin{array}{ll} \widetilde\lambda(l,2t) & \hbox{for } l\in L, r\in R,  t \in [0,1/2],\\ \widetilde\rho(r,2t-1) & \hbox{for } l\in L, r\in R, t \in [1/2,1]. \end{array}\right.
$$
The polycycles $\calK$ and $\calN$ are admissible in the sense of Section~\ref{computationoftildeupsilon}: we can take $U=K\times \Int (I)$ and $V=N\times (\Int (I) \setminus \{1/2\})$.
  It follows from Lemma \ref{newcomputation} that     $\widetilde\Upsilon_{12,35}(a \otimes bc) =\lb \calD(\calK,\calN)\rb $. Thus, to  prove the lemma,  it is enough to show  that
\begin{eqnarray}
  \label{Leibniz_polycyles}   \calD(\calK,\calN) & \simeq  &  (-1)^{{i} }\,  (\id \times \conc)_*(\calD(\calK,\calL) \times \calR) \\
  \notag &&  \sqcup \  (-1)^{(p+n+1)q}\,  (\conc\times \id)_*(\calL \times \calD(\calK,\calR)).
\end{eqnarray}
To this end, we compare  $\calD(\calK,\calN)=(D, \theta, w,\kappa \losange \eta)$ with
$$
\calD(\calK,\calL) =\left(D',\theta',w',\kappa \losange \lambda\right) \quad \hbox{and} \quad
\calD(\calK,\calR) =   ('\!D,\, '\!\theta,\, '\!w,\kappa \losange \rho).
$$
Consider the  embedding
$$
P':(K \times I \times L \times I) \times R   \hookrightarrow   K \times I \times (L \times R) \times I
$$
defined by $P'(k,s,l,t,r)= (k,s,l,r,t/2)$ and  the  embedding
$$
'\!P: L \times (K \times I \times R \times I)  \hookrightarrow  K \times I \times (L \times R) \times I
$$
defined by $'\!P(l,k,s,r,t)= (k,s,l,r,(t+1)/2)$.
Note   that $P'$ has degree $(-1)^{{i} }$ while~$'\!P$ has degree $(-1)^{(p+1)q}$. Consider   also   the cartesian projections
$$\pr': (K \times I \times L \times I )\times R \longrightarrow  K \times I \times L \times I , $$
$$'\!\pr: L \times (K \times I \times R \times I) \longrightarrow  K \times I \times R \times I.$$
Clearly,  $(\widetilde \kappa \times \widetilde \eta)P' = (\widetilde \kappa \times \widetilde \lambda) \pr'$.
Therefore,   the map   $P'$ restricts to a diffeomorphism $D'\times R \to P'(D'\times R) \subset D$ of degree $(-1)^{{i} }$.
Similarly, since $(\widetilde \kappa \times \widetilde \eta)\, '\!P = (\widetilde \kappa \times \widetilde \rho)\, '\!\pr$,
the map $'\!P$ restricts to a diffeomorphism $L \times\, '\!D  \to\, '\!P(L \times\, '\!D) \subset D$ of degree $(-1)^{(p+1)q+nq}$.
 Here we   use the following general fact  involving  our orientation conventions stated in the Introduction:
if $X$, $Y$ are oriented manifolds and  $S$ is  an oriented submanifold of $X$,
then the  bundle map $\bignu_{X\times Y} (S \times Y) \to \bignu_XS$ induced by the  cartesian  projection $X \times Y \to X$
is an orientation-preserving isomorphism on each fiber,
 while  the bundle map $\bignu_{Y\times X} (Y \times S)  \to \bignu_XS$ induced by the    cartesian   projection $Y \times X \to X$
 is orientation-preserving if and only if the product $(\dim X -\dim S) \cdot \dim(Y)$    is even.

It is clear from the definition of $\calN$ and the computations of degrees above that
$$
P' \sqcup\, '\!P: (-1)^{i}  (D' \times R) \sqcup (-1)^{(p+n+1)q} (L \times\, '\!D) \longrightarrow D
$$
is    an orientation-preserving diffeomorphism.   We claim that   it     transports the polychain structures  of
$(\id \times \conc)_*\left(\calD(\calK,\calL) \times \calR\right)$ and $(\conc \times \id)_*\left(\calL \times \calD(\calK,\calR)\right)$
 into the polychain structure of $ \calD(\calK,\calN)$ up to deformation of the latter. This   will imply   \eqref{Leibniz_polycyles} and the lemma.

To prove our claim, we  need to   verify  that $P' \sqcup\, '\!P$ preserves   the face partitions and  the weights  and
 commutes with the  maps to the path spaces up to deformation. We start with  the face partitions.
Let   $F',G'$   be faces   of $D'$ of the same type   and let $H,J$ be    faces of $R$ of the same type.
Then $ F'\times H $ and $  G'\times J $  are faces of~$D'\times R$ of the same type.
We claim that the faces     $F= P'(F'\times H)$ and $G= P'(G'\times J)$    of~$D$ have the same type. By
  Section \ref{def_D}, $\theta'_{F',G'}:F'\to G'$  is the restriction of the diffeomorphism
$$
\varphi_{A_{F'},A_{G'}} \times \id \times \psi_{B_{F'},B_{G'}} \times \id :
N_{F'}=A_{F'}\times I \times B_{F'} \times I \longrightarrow A_{G'}\times I \times B_{G'} \times I =N_{G'}
$$
to $F'$ where $N_{F'}$ (respectively, $N_{G'}$) is   the smallest face of $K \times I \times L \times I$ containing~$F'$ (respectively, $G'$).
 The   smallest faces $N_F$ and $N_G$ of $K \times I \times (L \times R) \times I$ containing~$F$ and  $G$ respectively  are
$$
N_F = A_{F'} \times I \times (B_{F'} \times H) \times I \quad \hbox{and} \quad
N_G = A_{G'} \times I \times (B_{G'} \times J) \times I.
$$
Clearly, the diagram
$$
\xymatrix{
N_{F'} \times H  \ar[d]_{P'} \ar[rrrrr]^{\left(\varphi_{A_{F'},A_{G'}} \times \id \times \psi_{B_{F'},B_{G'}} \times \id\right) \times \chi_{H,J} } &&& &&N_{G'} \times J   \ar[d]^{P'}  \\
N_F  \ar[rrrrr]^{\varphi_{A_{F'},A_{G'}} \times \id \times (\psi_{B_{F'},B_{G'}}\times \chi_{H,J}) \times \id} &&&&& N_G
}
$$
commutes, so that the bottom diffeomorphism in that diagram  carries $F$ onto $G$.
We deduce that  $F$ and $G$  have the same type in $D$ and the identification map $\theta_{F,G}:F\to G$
(which, by definition, is the restriction of the bottom diffeomorphism to $F$) satisfies
$$
\theta_{F,G} \circ P'\vert_{F'\times H} = P'\vert_{G' \times J} \circ (\theta'_{F',G'} \times \chi_{H,J}).
$$
This proves that   $P'$ carries the partition $\theta' \times \chi$ on $D'\times R$ to the partition $\theta$ restricted to $P'(D' \times R)\subset D$.
 A similar argument shows   that   $'\!P$ carries  the partition $\psi \times\, '\!\theta $ on $L \times\, '\!D$ to the partition $\theta$  restricted to $'\!P(L \times\, '\!D ) \subset D$.
It remains only to observe that a face  of $D$ lying in $P'(D'\times R)$ cannot have the same type as a face of $D$ lying in $'\!P(L\times\, '\!D)$.
To see this, we use the fact that every face $F$ of $D$ determines a smallest face $N_F = A_F \times I \times (B_F \times C_F) \times I$
of $K\times I \times (L \times R) \times I$ such that $F\subset N_F$ and $A_F,B_F,C_F$ are faces of $K,L,R$ respectively.
If $F,G$ are faces of $D$ of the same type, then   $A_F,B_F,C_F$  must have the same type as $A_G,B_G,C_G$ respectively, and the diffeomorphism
$$
\varphi_{A_F,A_G} \times \id \times (\psi_{B_F,B_G} \times \chi_{C_F,C_G}) \times \id:N_F \longrightarrow N_G
$$
carries $F$ onto $G$. Since this diffeomorphism preserves   the last coordinate and
$$
P'(D'\times R) \subset K \times I \times (L\times R) \times [0,1/2], \quad
'\!P(L\times\, '\!D) \subset K \times I \times (L\times R) \times [1/2,1] \
$$
we deduce that  $F$ and $G$ are  both contained either in $P'(D'\times R)$  or in $'\!P(L\times\, '\!D)$.

We  next show that  the diffeomorphism  $P' \sqcup\, '\!P$ preserves the weights.
Let $W' $ be a connected component of $D'$ and let $Z$ be a connected component of $R$.
The weight of the connected component $W'  \times Z$ of $D' \times R$ is
$$
(w' \times z)(W'  \times Z)= w'(W' )z(Z) = u(U) v(V) z(Z)
$$
where $U$ and $V$ are connected components of $K$ and $L$, respectively, such that $W'  \subset U \times I \times V \times I$.
Clearly,  $$P'(W'  \times Z)\subset U \times I \times (V \times Z) \times I$$  so that
$$
w\big(P'(W'  \times Z)\big) = u(U) \cdot  (v \times z)(V \times Z) =   u(U) v(V) z(Z) =  (w' \times z)(W'  \times Z).
$$
This proves that  $P'$ carries  the weight $w'\times z$ on $D'\times R$ to the weight $w$ restricted to $P'(D'  \times R  ) $.
  A similar argument shows   that   $'\!P$ carries    the weight $v \times\, '\!w$  on   $L \times\,  '\!D$ to the weight $w$ restricted to $'\!P(   L \times  '\!D) $.

 We now show that $P' \sqcup\, '\!P$    commutes with the  maps to $\Omega_{32}\times \Omega_{15}$   up to deformation.
 The maps in question are $\kappa \losange \eta:D\to \Omega_{32}\times \Omega_{15}$ and $f \sqcup  g$ where
 \begin{equation}\label{mapf} f=(\id \times \conc) ((\kappa \losange \lambda) \times \rho): D'\times R\to \Omega_{32}\times \Omega_{15},\end{equation}
 \begin{equation}\label{mapg} g= (\conc \times \id)  (\lambda \times (\kappa \losange \rho)) :L\times\, '\!D\to \Omega_{32}\times \Omega_{15}.\end{equation}
We first compute    $(\kappa \losange \eta )    P' $.
Pick any  $(k,s,l,t) \in D'$ and $r\in  R$. For  ${x}\in[0,1/2]$,
\begin{eqnarray*}
(\kappa \losange \eta)(P'(k,s,l,t,r))({x}) &=& (\kappa\ \widetilde \losange\ \eta)(k,s,l,r,t/2,{x}) \\
&=& \big(\widetilde\eta(l,r,(t/2)*{x}),\widetilde \kappa(k,s*{x})  \big)\\
&=& \big(\widetilde\eta(l,r,t{x}),\widetilde \kappa(k,2s{x})  \big)\\
&=& \big(\widetilde\lambda (l,2t{x}),\widetilde \kappa(k,2s{x})  \big)\\
&=& \big(\widetilde\lambda (l,t*{x}),\widetilde \kappa(k,2s{x})  \big).
\end{eqnarray*}
Similarly, for  ${x}\in[1/2,1]$,
\begin{eqnarray*}
(\kappa \losange \eta)(P'(k,s,l,t,r))({x}) &=& (\kappa\ \widetilde \losange\ \eta)(k,s,l,r,t/2,{x}) \\
&=& \big(\widetilde \kappa(k,s*{x}), \widetilde \eta(l,r,(t/2)*{x})   \big)\\
&=&  \big(\widetilde \kappa(k,s*{x}), \widetilde \eta(l,r,1-  (2-t)(1-{x})  )   \big)  .
\end{eqnarray*}
We  separate two cases depending on whether or not $1-  (2-t)(1-{x})  \leq   1/2$ or, equivalently, on whether or not
${x}\leq (3-2t)/(4-2t)$. For ${x}\in[1/2,(3-2t)/(4-2t)]$, we obtain
\begin{eqnarray*}
(\kappa \losange \eta)(P'(k,s,l,t,r))({x})
&=& \big(\widetilde \kappa(k,s*{x}), \widetilde \lambda(l,2-2(2-t)(1-{x}) )   \big);
\end{eqnarray*}
for ${x}\in[(3-2t)/(4-2t),1]$, we obtain
\begin{eqnarray*}
(\kappa \losange \eta)(P'(k,s,l,t,r))({x})
&=& \big(\widetilde \kappa(k,s*{x}), \widetilde \rho(r,1-2(2-t)(1-{x}) )   \big).
\end{eqnarray*}
These computations show that the   first   coordinate map $D' \times R \to \Omega_{32}$  of $(\kappa \losange \eta)P'$
is equal to     $(\kappa \triangleleft \lambda) \circ \pr'\vert_{D'\times R}$,
which is also the  first coordinate    of the map $f$ given by \eqref{mapf}.
 The second   coordinate maps $D' \times R \to \Omega_{15}$  of $(\kappa \losange \eta)P'$    and  $f$ may  differ.
Nonetheless, they   are homotopic   in the following way.  For any $ s,t,{y} \in I $,  consider the numbers
$$
0\, < \     \frac{1}{4}\, \leq\,    a_{y}=\frac{1+{y}}{4} \, \leq \, \frac{1}{2}\, \leq\,  b_{t,{y}}=\frac{1 }{2} +{y} \frac{1-t}{4-2t} \, \leq \,  \frac{3-2t}{4-2t}  \  < \, 1
$$
and let
 $$ \alpha_{s,{y}}: [0,a_{y}] \longrightarrow [0,s], \quad \beta_{t,{y}}: [a_{y}, b_{t,{y}}]  \longrightarrow [t,1], \quad \gamma_{t,{y}}:[b_{t,{y}},1]  \longrightarrow [0,1]$$
 be the    affine    maps carrying the left/right endpoints of segments to the left/right endpoints respectively.
We define a  continuous map  $ {e}: D' \times R \times I \times I \to M$  by
\begin{equation}\label{homotopy_F}
{e}(k,s,l,t,r,{x},{y}) = \left\{ \begin{array}{ll}  \widetilde\kappa(k,\alpha_{s,{y}}({x})) & \hbox{if } {x}\in [0,a_{y}],\\
\widetilde \lambda(l,\beta_{t,{y}}({x})) & \hbox{if } {x}\in [a_{y},b_{t,{y}}],\\
\widetilde \rho(r,\gamma_{t,{y}}({x})) & \hbox{if } {x} \in [b_{t,{y}},1]. \end{array}\right.
\end{equation}
Observing that  $a_{0} = 1/4,   b_{t,0}= 1/2$ and $ a_1= 1/2,   b_{t,1}=(3-2t)/(4-2t)$,
we  conclude that   ${e}$   determines a homotopy    between the second  coordinate  maps of  $f$  and   $(\kappa \losange \eta)P'$ in the class of maps $D' \times R \to \Omega_{15}$.
It   remains to check that this homotopy   is compatible with the partition $\theta'\times \chi$ on  $D' \times R$.
Any   faces $F,G$ of $D' \times R$ of the same type expand  as $F=F'\times H$ and $G=G'\times J$ where
$F',G'$ are faces of $D'$ of the same type and   $H, J$ are faces of $R$ of the same type. Let
$$
N_{F'}=A_{F'} \times I \times B_{F'} \times I \quad \hbox{and} \quad N_{G'}=A_{G'} \times I \times B_{G'} \times I
$$
be the smallest faces of $K \times I \times L \times I$ containing $F'$ and $G'$ respectively.
The identifying map $ (\theta'\times \chi)_{F,G}:F \to G $   is the restriction of the diffeomorphism
$$
\left(\varphi_{A_{F'},A_{G'}}\times \id \times \psi_{B_{F'},B_{G'}}\times \id\right) \times \chi_{H,J}: N_{F'} \times H \longrightarrow N_{G'} \times J.
$$
   Since the maps $\kappa,\lambda, \rho$ are compatible with the partitions $\varphi,\psi,\chi$ respectively,
we deduce from \eqref{homotopy_F} that  for any $(k,s,l,t )\in F'$, $r\in H$, and $x,y\in I$,
\begin{eqnarray*}
&&e\big((\theta'\times \chi)_{F , G }(k,s,l,t,r), x,y\big)\\
&=&e\big(\varphi_{A_{F'},A_{G'}}(k),s,\psi_{B_{F'},B_{G'}}(l),t,\chi_{H,J}(r), x,y\big) \ = \ e(k,s,l,t,r,{x},{y}).
\end{eqnarray*}
Hence for each $y\in I$, the map
$$D'\times R \longrightarrow \Omega_{15}, \quad (k,s,l,t,r) \longmapsto \big(x\mapsto e(k,s,l,t,r,x,y)\big)$$ is compatible
with the partition $\theta' \times \chi$.   We    conclude that  the homotopy of $f$   to   $(\kappa \losange \eta)P'$ determined by $e$
is compatible with the partition $\theta'\times \chi$.
One similarly   constructs a deformation of the map \eqref{mapg} into $(\kappa \losange \eta)\, '\!P $ compatible with the partition.
\end{proof}

\subsection{Change of base points}\label{change}

Consider one more  tuple  $(\star'_1, \star'_2, \star'_3, \star'_4)$ of  points of $\partial M$
such that  $\{\star'_1, \star'_2\} \cap \{\star'_3, \star'_4\}=  \varnothing  $
and set $\Omega'_{ij}=\Omega (M, \star'_i, \star'_j)$.  Section~\ref{dfp_tildeUpsilon}   yields a linear map
$$
\widetilde \Upsilon': \widetilde H_\ast (\Omega'_{12}) \otimes
\widetilde H_\ast (\Omega'_{34}) \longrightarrow   \widetilde H_{\ast} ( \Omega'_{32} \times   \Omega'_{14}).
$$
We compare  $\widetilde \Upsilon'$  to  the map
  $\widetilde \Upsilon: \widetilde H_\ast (\Omega_{12}) \otimes \widetilde H_\ast (\Omega_{34}) \to \widetilde H_{\ast} ( \Omega_{32} \times   \Omega_{14})$
assuming that $\star_i$ and $\star_i'$ belong
to the same connected component of $\partial M$ for all $i\in \{1,2,3,4\}$.

Choose  a   path  $\varsigma_i:I\to  \partial M$ from $\star_i$ to   $\star'_i $ for each $i$.
The formula $\gamma\mapsto \varsigma_i^{-1} \gamma \varsigma_j$ defines a
continuous map   $(\varsigma_i, \varsigma_j)_\#$ from $\Omega_{ij}$ to $\Omega_{i'j'}$.
Homotopic paths yield      homotopic maps,  and   constant paths  yield    maps  homotopic to the identity.
Therefore  $(\varsigma_i, \varsigma_j)_\#$ is a homotopy equivalence with homotopy  inverse $ (\varsigma^{-1}_i, \varsigma^{-1}_j)_\#$.
The homotopy equivalence $(\varsigma_i, \varsigma_j)_\#$ induces an isomorphism in the    face homology which we denote   by the same symbol:
\begin{equation}\label{nu_isomorphism}
(\varsigma_i, \varsigma_j)_\#: \widetilde H_*(\Omega_{ij}) \stackrel{\simeq}{\longrightarrow} \widetilde H_*(\Omega_{ij}').
\end{equation}
Similarly, the  isomorphism  $ \widetilde H_{\ast} ( \Omega_{ij} \times \Omega_{kl})\to \widetilde H_{\ast} ( \Omega'_{ij} \times   \Omega'_{kl})$ induced by the homotopy equivalence
 $(\varsigma_i,\varsigma_j)_\# \times (\varsigma_k, \varsigma_l )_\#$ is also denoted by   $(\varsigma_i,\varsigma_j)_\# \times (\varsigma_k, \varsigma_l )_\#$.

\begin{lemma}\label{naturbracket_tilde}
If $n\geq 3$, then the  following diagram   commutes:
\begin{equation}\label{1234_1234'}
\xymatrix@R=0.8cm @C=3cm {
\widetilde H_\ast (\Omega_{12}) \otimes \widetilde H_\ast (\Omega_{34})
 \ar[r]^-{ \widetilde \Upsilon }  \ar[d]_-{(\varsigma_1,\varsigma_2)_\# \otimes (\varsigma_3, \varsigma_4 )_\#}^-\simeq  & \widetilde H_{\ast} ( \Omega_{32} \times \Omega_{14})
 \ar[d]^-{(\varsigma_3,\varsigma_2)_\# \times (\varsigma_1, \varsigma_4 )_\#}_-\simeq  \\
\widetilde H_\ast (\Omega'_{12}) \otimes \widetilde H_\ast (\Omega'_{34}) \ar[r]^-{\widetilde \Upsilon' } & \widetilde H_{\ast} ( \Omega'_{32} \times   \Omega'_{14}).
}
\end{equation}
 \end{lemma}

\begin{proof} Since the isomorphism  $(\varsigma_i, \varsigma_j )_\#$  depends only on the homotopy classes of the paths $\varsigma_i$, $\varsigma_j$,
and since  composition of the paths leads to composition of the corresponding isomorphisms,
it is enough to consider the case where three of the paths $\varsigma_i$'s are constant.
Assume for concreteness that $\star_1=\star_1'$, $\star_2=\star_2'$, $\star_3=\star_3'$, and $\varsigma_1, \varsigma_2, \varsigma_3$ are constant paths.
The  assumption   $n\geq 3$ implies that deforming if necessary  the path  $\varsigma=\varsigma_4$,
we can ensure that $\varsigma(I) \subset \partial M \setminus \{\star_1, \star_2 \}$.

Let ${a}\in \widetilde H_p (\Omega_{12})$ and  ${b}\in \widetilde H_q(\Omega_{34})$.
Consider  smooth polycycles  $\calK=(K,\varphi, u,\kappa)$  in $\Omega_{12}^\circ$
and   $\calL=(L,\psi, v,\lambda)$ in $\Omega_{34}^\circ$ transversely representing   the pair  $({a}, {b})$.
The class $(1,\varsigma)_\#({b})\in \widetilde H (\Omega_{34}')$ is represented by the polycycle $\calL'=(1,\varsigma)_\#\calL$   in
$\Omega_{34}'$   (but   not in $\Omega_{34}'^\circ$). The polycycles $\calK$ and $\calL'$ are admissible in the sense of Section~\ref{computationoftildeupsilon}:
we can take $U=K\times \Int (I)$ and $V=L\times (0,1/2)$.
Set  $\calD (\calK, \calL)=(D,\theta,w,\kappa \losange \lambda)$  and $\calD (\calK, \calL')=(D',\theta',w',\kappa \losange \lambda')$ where $\lambda'=(1,\varsigma)_\sharp \lambda$.
It is easy to   construct a  diffeomorphism $f:D \to D'$  preserving  the orientation, the weight, and the face partition,
and  such that $(\kappa \losange \lambda') \circ f$ is homotopic to $\left(\id \times (1,\varsigma)_\#\right)\circ (\kappa \losange \lambda)$
in the class of  maps  $D \to \Omega_{32} \times \Omega'_{14}$   compatible with $\theta$.
 Lemma \ref{newcomputation} implies that
\begin{eqnarray*}
\widetilde \Upsilon' \big( {a}\otimes (1,\varsigma)_\# ({b})\big) &=& \lb \calD  (\calK, \calL') \rb\\ &=& \lb \big(\id \times (1,\varsigma)_\#\big) \calD(\calK, \calL) \rb\\
&=&   \left(\id \times (1,\varsigma)_\#\right) \lb \calD(\calK, \calL) \rb \ = \   \left(\id \times (1,\varsigma)_\#\right) \widetilde \Upsilon({a},{b}).
\end{eqnarray*}
This proves the  commutativity of the diagram \eqref{1234_1234'}.
 \end{proof}

\subsection{Extension of $\widetilde \Upsilon$}\label{checkUpsilon---}

  Assuming that $n\geq 3$,
we extend the definition of $\widetilde \Upsilon$ to all  4-tuples of points $\star_1$, $\star_2$, $ \star_3$, $ \star_4\in \partial M$.
Deforming  these points in $\partial M$,   we can obtain    points $ \star'_1, \star'_2, \star'_3, \star'_4 \in \partial M$
such that $\{\star'_1, \star'_2\} \cap \{\star'_3, \star'_4\} =   \varnothing  $. For   $i=1,   \dots   , 4$, pick a
    path  $\varsigma_i:I\to  \partial M$ from $\star_i$ to   $\star'_i $.   Section \ref{dfp_tildeUpsilon}   yields a linear map
$$
\widetilde \Upsilon': \widetilde H_\ast (\Omega'_{12}) \otimes
\widetilde H_\ast (\Omega'_{34}) \longrightarrow   \widetilde H_{\ast} ( \Omega'_{32} \times   \Omega'_{14})
$$
where $\Omega'_{ij}=\Omega (M, \star'_i, \star'_j)$  for all $i,j$.  Then   we define
$$
\widetilde \Upsilon=\widetilde \Upsilon_{12,34} : \widetilde H_\ast (\Omega_{12}) \otimes
\widetilde H_\ast (\Omega_{34}) \longrightarrow   \widetilde H_{\ast} ( \Omega_{32} \times   \Omega_{14})
$$
to be the  unique linear map such that the diagram \eqref{1234_1234'} commutes.
Lemma~\ref{naturbracket_tilde} implies that this map   depends neither on the choice of the paths  $\varsigma_1,\varsigma_2,\varsigma_3,\varsigma_4$
nor on the choice of the points $\star'_1, \star'_2, \star'_3, \star'_4$. If     $\{\star_1, \star_2\} \cap \{\star_3, \star_4\} =   \varnothing  $,
then we   can take $\star'_i=\star_i$ and the constant path  $\varsigma_i $ for all $i$, and recover the same map  $\widetilde \Upsilon$   as before.

The properties of  $\widetilde \Upsilon$ established under the assumption $\{\star_1, \star_2\} \cap \{\star_3, \star_4\} =   \varnothing  $
remain true for arbitrary base points in $\partial M$.
  This easily follows from the definitions and the fact that  the concatenation pairing \eqref{concatenation}   is preserved
  under  the change-of-base-points isomorphism \eqref{nu_isomorphism}.

\subsection{Renormalization}\label{checkUpsilon}

  We will use   a renormalized version
\begin{equation}\label{defoftildeUpsilon+}
\check \Upsilon=  \check  \Upsilon_{12,34}  : \widetilde H_\ast (\Omega_{12}) \otimes \widetilde H_\ast (\Omega_{34})
\longrightarrow   \widetilde H_{\ast} ( \Omega_{32} \times \Omega_{14})
\end{equation}
 of $\widetilde{\Upsilon}$ defined by   $\check \Upsilon (a\otimes b)=(-1)^{\vert b\vert + n\vert a \vert} \widetilde{\Upsilon} (a\otimes b)$
for any homogeneous $a\in \widetilde H_\ast (\Omega_{12})$ and $b\in \widetilde H_\ast (\Omega_{34})$.
 The properties  of~$\widetilde \Upsilon$ can be rephrased   for~$\check\Upsilon$. 
In particular,   Lemma~\ref{antisymmetry_Upsilon_tilde} yields the identity
\begin{equation}\label{checksymm}
\perm_*  \check  \Upsilon_{12,34}({a} \otimes {b})=  - (-1)^{\vert a\vert_n \vert b\vert_n }  \check  \Upsilon_{34,12} ({b} \otimes {a})
\end{equation}
where  $\vert\! -\! \vert_n=\vert\! -\! \vert +n$ is  the   $n$-degree.   Also,
for  any  $\star_5\in \partial M$   (distinct from $\star_1$ and $\star_2$ if $n=2$)
and any homogeneous   $a\in \widetilde H_\ast (\Omega_{12})$,  $b\in \widetilde H_\ast (\Omega_{34})$,  $c \in  \widetilde H_{\ast} (\Omega_{45})$,
  Lemma~\ref{Leibniz_Upsilon_tilde} yields   the  Leibniz rule
\begin{equation}\label{LeibnizforwidetileUpsilon}
\check \Upsilon_{12,35}(a \otimes bc)=    \check \Upsilon_{12,34}(a \otimes b ) \, c+ (-1)^{\vert a\vert_n \vert b\vert} \,b \, \check \Upsilon_{12, 45} (a \otimes c).
\end{equation}
  Finally,   the  diagram  \eqref{1234_1234'} remains commutative with   $\widetilde \Upsilon$ replaced by $\check \Upsilon$.

\section{The operation $ \Upsilon$} \label{Upsilon}

We derive from  $\check \Upsilon$ an operation $\Upsilon$ in singular homology.  
In this section we drop the assumption  $\{\star_1,\star_2\}\cap\{\star_3,\star_4\} =  \varnothing  $   when  $n\geq 3$.

\subsection{Definition and    properties of  $\Upsilon$}\label{Upsilon1111}

Consider the  linear map
\begin{equation}\label{functionz}
 \Upsilon=\Upsilon_{12,34} : H_\ast (\Omega_{12}) \otimes H_\ast (\Omega_{34})
\longrightarrow   H_{\ast} ( \Omega_{32} \times \Omega_{14})
\end{equation}
defined by the   commutative diagram
$$
\xymatrix{
\widetilde H_\ast (\Omega_{12}) \otimes \widetilde H_\ast (\Omega_{34}) \ar[rr]^-{ \check\Upsilon } &&
\widetilde H_{\ast} ( \Omega_{32} \times \Omega_{14}) \ar[d]^-{[-]}\\
H_\ast (\Omega_{12}) \otimes H_\ast (\Omega_{34}) \ar[u]^-{\lb-\rb \times \lb-\rb}
\ar@{-->}[rr]^-{ \Upsilon } && H_{\ast} ( \Omega_{32} \times \Omega_{14}).
}
$$

Formula  \eqref{checksymm} and the naturality of the transformation  $[ - ]:   \widetilde H_\ast \to  H_\ast $
imply the   following   antisymmetry  of    $\Upsilon$: for any  homogeneous $a\in H_\ast (\Omega_{12})$, $b\in H_\ast (\Omega_{34})$,
\begin{equation}\label{Upsilon1}
\perm_* \Upsilon_{12,34}(a  \otimes b)= - (-1)^{\vert a\vert_n \vert b\vert_n } \Upsilon_{34,12} (b\otimes a)
\end{equation}
where $\perm_\ast: H_\ast(\Omega_{32} \times \Omega_{14} )\to  H_\ast(\Omega_{14} \times \Omega_{32} )$
is the linear map induced by the permutation map $\perm:\Omega_{32} \times \Omega_{14} \to \Omega_{14} \times \Omega_{32}$.

If $n\geq 3$, then the   diagram \eqref{1234_1234'}  with $\widetilde \Upsilon$ replaced by $\check \Upsilon$
and   the naturality of the transformations $\lb -\rb$ and $[ -]$ imply that   the following diagram commutes:
\begin{equation}\label{1234_1234'_H}
\xymatrix@R=0.8cm @C=3cm {
H_\ast (\Omega_{12}) \otimes  H_\ast (\Omega_{34})
 \ar[r]^-{  \Upsilon }  \ar[d]_-{(\varsigma_1,\varsigma_2)_\# \otimes (\varsigma_3, \varsigma_4 )_\#}^-\simeq  & H_{\ast} ( \Omega_{32} \times \Omega_{14})
 \ar[d]^-{(\varsigma_3,\varsigma_2)_\# \times (\varsigma_1, \varsigma_4 )_\#}_-\simeq  \\
 H_\ast (\Omega'_{12}) \otimes H_\ast (\Omega'_{34}) \ar[r]^-{ \Upsilon' } &  H_{\ast} ( \Omega'_{32} \times   \Omega'_{14}).
}
\end{equation}
Here, for every $i\in\{1,2,3,4\}$, $\star'_i$ is a  point of $\partial M$ connected to $\star_i$ by a path $\varsigma_i:I \to \partial M$,
$\Upsilon'$ is the map  \eqref{functionz} determined by the  base points $ \star'_1,\star'_2,\star'_3,\star'_4 $,
and, for all  $i,j\in\{1,2,3,4\}$, $(\varsigma_i, \varsigma_j)_\#$ stands for  the  homotopy equivalence
$$\Omega_{ij} \to \Omega'_{ij} = \Omega(M,\star'_i,\star'_j), \gamma\mapsto \varsigma_i^{-1} \gamma \varsigma_j$$ and for the induced isomorphism in  singular   homology.

 The  following    crucial lemma will be  proved in Section~\ref{subs123}.

\begin{lemma}\label{diag-imp}
The following  diagram   commutes:
\begin{equation}\label{Upsilon_Upsilon}
\xymatrix{
\widetilde H_\ast (\Omega_{12}) \times  \widetilde H_\ast (\Omega_{34})  \ar[d]_-{[-] \times [-]}
\ar[rr]^-{\check   \Upsilon} && \widetilde H_\ast ( \Omega_{32} \times \Omega_{14})  \ar[d]^-{[-]}\\
 H_\ast (\Omega_{12}) \times H_\ast (\Omega_{34}) \ar[rr]^-{\Upsilon} &&
H_\ast ( \Omega_{32} \times \Omega_{14}).
}
\end{equation}
\end{lemma}

\subsection{The Leibniz rule for   $\Upsilon$}

As   in Section \ref{Leibniz}  in the case of face homology,
the concatenation of paths  induces three kinds of bilinear pairings  in singular homology:
\begin{equation}\label{concic}
H_*(\Omega_{ij})\times H_*(\Omega_{jk}) \longrightarrow  H_*(\Omega_{ik}), \ (a , b) \longmapsto ab=\conc_*(a \times b),
\end{equation}
$$
 H_{\ast} ( \Omega_{lm}  \times   \Omega_{ij}) \times  H_\ast(\Omega_{jk})
  \to  H_{\ast} ( \Omega_{lm} \times    \Omega_{ik}),  \ (x , a)   \mapsto   xa=(\id \times \conc)_*(x \times a),
$$
$$
H_\ast(\Omega_{ij}) \times  H_{\ast} ( \Omega_{jk} \times \Omega_{lm})
  \to    H_{\ast} ( \Omega_{ik}\times \Omega_{lm}), \ (a , x)   \mapsto   ax=(\conc \times \id)_*(a \times x).
$$

\begin{lemma}\label{Leibniz_Upsilon}
For  any  $\star_5\in \partial M$   (distinct from $\star_1$ and $\star_2$ if $n=2$)   and any homogeneous
$a\in  H_\ast (\Omega_{12})$,  $b\in  H_\ast (\Omega_{34})$,   $c \in   H_{\ast} (\Omega_{45})$,
\begin{equation}  \Upsilon_{12,35}(a \otimes bc)=     \Upsilon_{12,34}(a \otimes b ) \, c+ (-1)^{\vert a\vert_n \vert b\vert} \,b \,   \Upsilon_{12, 45} (a \otimes c).
\end{equation}
\end{lemma}

 \begin{proof}
 For any $x\in \widetilde H_*(\Omega_{ij})$, $ y\in \widetilde H_*(\Omega_{jk})$ with $i,j,k\in \{1,   \dots   ,5\}$, we have
\begin{equation}\label{multimu}
[xy] = [\conc_*(x \times y)]= \conc_*[x\times y] = \conc_*([x]\times [y])=  [x] [y]
\end{equation}
where we use the naturality of $[-]$ and  Lemma \ref{twocrossproducts}.
 In particular,
$bc = [\lb b \rb ] [\lb c \rb ] = [\lb b \rb \lb c \rb ]$.
We deduce that
\begin{eqnarray*}
 \Upsilon_{12,35}(a \otimes bc)
&=& \Upsilon_{12,35}\big([\lb a \rb ] \otimes \left[  \lb b\rb  \lb c\rb \right]\big)\\
&=&  \left[\check   \Upsilon_{12,35}\big(\lb a \rb, \lb b\rb \lb c \rb\big)\right]\\
&=&     \left[\check   \Upsilon_{12,35}\big(\lb a \rb, \lb b\rb\big) \lb c \rb
+ (-1)^{\vert a\vert_n \vert b\vert}   \lb b\rb  \check   \Upsilon_{12,35}\big(\lb a \rb,  \lb c \rb\big) \right] \\
&=&    \left[\check   \Upsilon_{12,35}\big(\lb a \rb, \lb b\rb\big) \right]   c
+ (-1)^{\vert a\vert_n \vert b\vert} b    \left[  \check   \Upsilon_{12,35}\big(\lb a \rb,  \lb c \rb\big) \right] \\
&=&      \Upsilon_{12,35}(a ,  b) c + (-1)^{\vert a\vert_n \vert b\vert} b      \Upsilon_{12,35}(a ,  c)
\end{eqnarray*}
where the second, third, fourth and fifth formulas follow respectively from     \eqref{Upsilon_Upsilon},   \eqref{LeibnizforwidetileUpsilon}, \eqref{multimu}, and the definition of $\Upsilon$.
  \end{proof}

  \subsection{Proof of Lemma \ref{diag-imp}}\label{subs123}

We claim that, for  all  $a\in \widetilde H_\ast (\Omega_{12})$ and  $b\in\widetilde H_\ast  (\Omega_{34})$,
\begin{equation}\label{claim_Upsilon}
\big[\widetilde\Upsilon\big(\lb[a]\rb,b\big)\big] = \big[\widetilde\Upsilon(a,b)\big]
= \big[\widetilde\Upsilon\big(a,\lb[b]\rb\big)\big]  .
\end{equation}
This would imply similar equalities with $\widetilde\Upsilon$ replaced by $\check \Upsilon$. Therefore
$$
\Upsilon([a],[b]) = \left[\check \Upsilon\big(\lb [a]\rb,\lb [b] \rb\big)\right]
= \left[\check \Upsilon\big(a,\lb [b] \rb\big)\right] = \left[\check \Upsilon(a,b)\right]   ,
$$
which   proves the commutativity of the diagram \eqref{Upsilon_Upsilon}.

We  prove the first equality in  \eqref{claim_Upsilon};   
 the   second   equality follows by the   symmetry  of~$\widetilde\Upsilon$   (Lemma \ref{antisymmetry_Upsilon_tilde}).
By Section~\ref{checkUpsilon---},  we can assume that $\{\star_1,\star_2\} \cap \{\star_3,\star_4\} {= \varnothing}$.
We need to prove that, for any smooth  polycycle $\calK=(K,\varphi,u,\kappa)$ in $\Omega_{12}^\circ$
and any  smooth  polycycle $\calL=(L,\psi,v,\lambda)$ in $\Omega_{34}^\circ$ transversal to~$\calK$,
\begin{equation}\label{left_side}
 \big[\widetilde\Upsilon\big(\lb [\calK] \rb,\lb \calL\rb \big)\big]=[\calD(\calK,\calL)] .
\end{equation}

Set $p=\dim(K)$. Pick a locally ordered smooth triangulation $T$ of $K$ which fits~$\varphi$.
 The construction of  such a triangulation  in  Section~\ref{fundcla} (using  Lemma~\ref{partition})
shows that  we can further assume that (*)
$F\cap \tau$ is a face of $\tau$ for any face~$F$ of~$K$ and  any simplex $\tau$ of $T$ 
 (cf$.$ the last paragraph  in the proof of Lemma~\ref{partition}).
  Consider the fundamental   $p$-chain
$$ \sigma= \sigma(T,u)= \sum_\Delta \, \varepsilon_\Delta \, u(K^\Delta) \,  \sigma_\Delta \in C_p(K)  $$
 determined by $T$ as in  Section~\ref{fundcla} (here $\Delta$  runs over all $p$-simplices  of $T$).
Then $\kappa_*(\sigma)$ is a smooth singular  $p$-cycle  in $\Omega_{12}^\circ  $  representing the singular homology class  $[\calK] \in H_p(\Omega_{12}^\circ)$.
Next consider the smooth $p$-polycycle $\calK' =(K',\varphi',u',\kappa')$ in $\Omega_{12}^\circ$ associated
 with the expansion  $\kappa_*(\sigma)=\sum_\Delta \, \varepsilon_\Delta \, u(K^\Delta) \,  \kappa \sigma_\Delta$ as in  Section~\ref{firsttransf}.
 By construction, $K'$ is a disjoint union of copies of the standard $p$-simplex $\Delta^p$  indexed by $p$-simplices $\Delta$ of $T$
and $\kappa'=\kappa   \zeta$
where $\zeta=\coprod_\Delta \sigma_\Delta: K'\to K$.  By the definition of the transformation $\lb -\rb:   H_\ast \to \widetilde H_\ast $,
$$
\lb [\calK] \rb = \lb [\kappa_*(\sigma)]\rb =    \lb \calK' \rb \in  \widetilde H_{p}(\Omega_{12}),
$$
so that \eqref{left_side} is equivalent to
\begin{equation}\label{with_K'}
 \big[\widetilde\Upsilon\big(\lb \calK' \rb,\lb \calL\rb \big)\big]=[\calD(\calK,\calL)] .
\end{equation}

Lemma \ref{defo1++} yields a deformation of $ \calL$ into a polycycle $\calL^1$   transversal to $\calK'$. Such a
polycycle   $\calL^1$ is  also transversal to $\calK$. By Lemma \ref{homotimplieshomol}, $\lb \calL\rb=\lb \calL^1\rb$
  so that $ \big[\widetilde\Upsilon\big(\lb \calK' \rb,\lb \calL\rb \big)\big]=\big[\widetilde\Upsilon\big(\lb \calK' \rb,\lb \calL^1\rb \big)\big]$.
By Lemma \ref{Upsilon_tilde_def}, the polycycles  $\calD(\calK,\calL)$  and    $\calD\left(\calK,\calL^1\right)$ are homologous
  so that    $[\calD(\calK,\calL)]=\left[\calD\left(\calK,\calL^1\right)\right]$.
Thus, in order to prove  \eqref{with_K'},  we may assume  without loss of generality that $\calL$ is transversal to $\calK'$. We need to prove  that
\begin{equation}\label{D(K',L)}
\left[\calD(\calK',\calL)\right]=\left[\calD(\calK,\calL)\right].
\end{equation}

Set $\calD(\calK,\calL )=(D,\theta,w,\kappa\losange\lambda )$, $\calD(\calK',\calL )=(D',\theta',w',\kappa'\losange\lambda )$ and
$$ \overline \zeta=
\zeta\times \id_I \times \id_L \times \id_I:  K' \times I \times L \times I \longrightarrow   K \times I \times L \times I.
$$
It follows from the definition that $D'= \overline \zeta^{-1}(D)$.
Thus~$D'$ is obtained    by cutting~$D$ into pieces, each \lq\lq piece"
being a connected component of  $D \cap (\Delta \times I \times L\times I)$ where~$\Delta$ is a $p$-simplex of $T$.
The map $\overline \zeta\vert_{D'} : D' \to D$ is the obvious gluing map. It is surjective and its restriction to every connected component of~$D'$ is injective.

Consider the equivalence relations $\sim_{\theta}$ on $D$ and $\sim_{\theta'}$  on $D'$ defined by the partitions  (see  Section~\ref{Partitions}). 
We claim that if  some   points   $d_1,d_2 \in  D'$ satisfy   $\overline \zeta (d_1)\sim_{\theta} \overline \zeta(d_2)$, then   $d_1 \sim_{\theta'} d_2$.
 We now check this claim. Assume that  for $i=1,2$,  $$d_i=(k_i',s,l_i,t) \in D'\subset K'\times I \times L \times I,  \quad {\rm{and\,\,  set}} \quad k_i=\zeta(k'_i)\in K.$$
By assumption,  there exist faces $F_1,F_2$ of $D$ of the same type
such that $\overline \zeta (d_i)=(k_i,s,l_i,t)\in F_i $ for $i=1,2$ and  $\theta_{F_1,F_2}:F_1 \to F_2$ carries $\overline \zeta (d_1)$ to $\overline \zeta (d_2)$.
For $i=1,2$, let $A_i$ and $B_i$ be  faces of $K$ and $L$ respectively such that $  A_i \times I \times B_i \times I$
is the smallest face of $K \times I \times L \times I$ containing $F_i$.
Then $k_i\in A_i$, $l_i\in B_i$, $A_1$ has the same type as $A_2$, $B_1$ has the same type as $B_2$, and
$\varphi_{A_1,A_2}(k_1)=k_2$, $ \psi_{B_1,B_2}(l_1)=l_2$.
 To proceed,   let $\Delta_i\approx \Delta^p$ be the connected component of $k'_i$ in $K'$,
and let $\sigma_i=\sigma_{\Delta_i}:\Delta^p \to K$ be the corresponding singular simplex
 (which is a simplicial isomorphism onto a $p$-simplex of the triangulation $T$).   Then $k_i\in A_i\cap \sigma_i(\Delta_i)$.
By the assumption (*)  above,   $A_i\cap \sigma_i(\Delta_i)$ is a face of the  $p$-simplex  $    \sigma_i(\Delta_i)  $.
Since  $T$ fits $\varphi$, the sets
$$
\tau_1=\sigma_1(\Delta_1) \cap\varphi_{A_2,A_1}\big( A_2\cap \sigma_2(\Delta_2)\big),
\qquad
\tau_2=\sigma_2(\Delta_2) \cap\varphi_{A_1,A_2}\big( A_1\cap \sigma_1(\Delta_1)\big)
$$
are faces of the $p$-simplices $\sigma_1(\Delta_1), \sigma_2(\Delta_2)$   containing $k_1,k_2$ respectively. The map
 $\varphi_{A_1,A_2}:A_1\to A_2$ restricts to a simplicial isomorphism $\varphi_{12} : \tau_1\to \tau_2$ preserving the   order of the vertices and carrying $k_1$ to $k_2$.
Set   $r=\dim(\tau_1)=\dim(\tau_2)$.  Then $\tau'_i = \sigma_i^{-1} (\tau_i )$ is an  $r$-dimensional face  of the  $p$-simplex $\Delta_i$   containing $k'_i$ for $i=1,2$.
The map $\varphi'_{12}=\sigma_2^{-1} \varphi_{12}   \sigma_1\vert_{\tau'_1}   :\tau'_1\to \tau'_2$
is an    order-preserving simplicial isomorphism  and $\varphi'_{12} (k'_1)=k'_2$.
Since $\Delta_1,\Delta_2$ are  copies of the standard $p$-simplex $\Delta^p$,   their faces
$\tau'_1,\tau'_2$ correspond to   certain   $(r+1)$-element subsets $S_1,S_2$ of the set $\{0,\dots,p\}$.
Since     $\kappa\varphi_{A_1,A_2} =\kappa\vert_{A_1}$, we have
$$
 (\kappa \sigma_{1})\circ  e_{S_1}  = (\kappa \sigma_{2})\circ e_{S_2} : \Delta^{r}\longrightarrow \Omega_{12}^\circ.
$$
By the definition of   $\calK'$,   the latter equality   implies that the faces $\tau'_1$, $\tau'_2$ of $K'$
have the same type and  $\varphi'_{\tau'_1,\tau'_2}=\varphi'_{12}:\tau'_1\to \tau'_2$.  Consider now the map
 $$\Theta=   \big(\varphi'_{\tau'_1,\tau'_2} \times \id \times \psi_{B_1,B_2} \times \id\big):
 \tau'_1 \times I \times B_1 \times I\to \tau'_2 \times I \times B_2 \times I.$$
 We have
$$
\Theta (d_1) = (\varphi'_{12}(k'_1),s,\psi_{B_1,B_2}(l_1),t) =d_2.
$$
Let $G_i$ be the connected component of $d_i $ in  $D' \cap (\tau'_i \times I \times B_i \times I)$.  The equality $\Theta (d_1)  =d_2$ implies that $\Theta (G_1)=G_2$.
Thus, $G_1$ and $G_2$   are faces of $D'$ of the same type,
  and the identification map $\theta'_{G_1,G_2}=\Theta\vert_{G_1}$ carries $d_1$ to $d_2$. This proves   that $d_1 \sim_{\theta'} d_2$ as claimed.

Consider the  canonical projections $\pi:D\to D_\theta$ and $\pi':D'\to D'_{\theta'}$.
The previous claim implies  that there exists a unique map  $g$ such that the  diagram
$$
\xymatrix{
D_\theta \ar[d]_-g & \ar[l]_-\pi D \\
D'_{\theta'}  & D' \ar[u]_{  \overline \zeta\vert_{D'}   }  \ar[l]_-{\pi'} 
}
$$
 commutes. The map $g$ is continuous because $\pi'$ is continuous and $\overline \zeta,\pi$ are quotient maps.
Set $S= \overline \zeta (\partial D')=D\cap T^{p-1}$ where $T^{p-1}$ is the $(p-1)$-skeleton of $T$. Clearly, $\partial D\subset S$.  Consider   the commutative diagram
$$
\xymatrix{
&& H_*\left(D_\theta,(\partial D)_\theta\right)\ar[d] &  \ar[l]_-{\pi_*} H_*(D,\partial D) \ar[d]
&\!\!\!\!\!\!\!\!\!\!\!\!\!\!\!\! \!\!\!\!\! \!\!\!\!\!\ni [D,w] \\
[D_\theta,w]\in \!\!\!\!\!\!\!\!\!\!\!\!\!\!\!\!\!\!\!& H_*\left(D_\theta \right)  \ar[r] \ar[d]_-{g_*} \ar[ru]& H_*\left(D_\theta,S_\theta \right)   \ar[d]_-{g_*}
&  \ar[l]_-{\pi_*} H_*(D,S) &\!\!\!\!\! \!\!\!\!\! \!\!\!\!\!\!\!\!  \ni  (\overline \zeta\vert_{  D'  })_* ([D',w']) \\
\big[D'_{\theta'},w'\big]\in \!\!\!\!\!\!\!\!\!\!\!\!\!\!& H_*\left(D'_{\theta'}\right)  \ar[r]  & H_*\left(D'_{\theta'},(\partial D')_{\theta'}\right)
& \ar[l]_-{\pi'_*} H_*(D',\partial D') \ar[u]_-{(\overline \zeta\vert_{  D'  })_*  }
&\!\!\!\!\!\!\!\!\!\!\!\!\!\! \!\!\!\!\! \!\!\!\!\!\ni [D',w']
}
$$
where the  unlabelled arrows  are the inclusion maps.
The definition of $\calK'$ implies that the weight $u':\pi_0(K')\to \kk$ of $\calK'$ is the composition of   $   \zeta_\#:\pi_0(K')\to \pi_0(K)$
with the weight $u:\pi_0(  K  )\to \kk$ of $\calK$. Using   the definition of the   operation $\calD$,
we deduce that  the weight $w':\pi_0(D')\to \kk$   is the composition of
$ (\overline \zeta\vert_{  D'  })_\#:\pi_0(D')\to \pi_0(D)$ with the weight $w:\pi_0(D)\to \kk$.
This  fact and the definition of $[D',w'], [D,w]$ imply   that $(\overline \zeta\vert_{  D'  })_* \left([D',w']\right)$
is the image of $[D,w]$ in $H_*(D,S)$. By Lemma~\ref{partit},   the image of  $[D_\theta,w]$ in $H_*\left(D_\theta,(\partial D)_\theta\right)$ is equal to
 $\pi_*([D,w])$. Using this fact, the uniqueness in  Lemma~\ref{partit},   and a simple diagram chasing  we obtain that
\begin{equation}\label{g,D,D'}
g_*([D_\theta,w]) =[D'_{\theta'},w'] \ \in H_\ast(D'_{\theta'}).
\end{equation}
Next, we verify that 
\begin{equation}\label{kappp} (\kappa' \losange \lambda )_{\theta'} g = (\kappa \losange \lambda )_{\theta}:D_\theta\to \Omega_{32}\times \Omega_{14}. 
\end{equation}
   Given  $d=(k,s,l,t) \in D$, we have $g\pi(d)= \pi'(k',s,l,t)$ for any  $k'\in \zeta^{-1}(k) \subset K'$. Then
$$
(\kappa' \losange \lambda )_{\theta'} g \big(\pi(d)\big) = (\kappa' \losange \lambda )(k',s,l,t) = (\kappa \losange \lambda )(k,s,l,t) = (\kappa \losange \lambda )_\theta\big(\pi(d)\big)
$$
where we use the equality $\kappa'(k')=\kappa(k)$.
Since $\pi:D\to D_\theta$ is onto, we conclude that \eqref{kappp} holds. 
This  and \eqref{g,D,D'}     imply  \eqref{D(K',L)}:
\begin{eqnarray*}
\left[\calD(\calK,\calL)\right] &=&  \left((\kappa \losange \lambda )_{\theta}\right)_* \left([D_{\theta},w]\right)\\
&=& \left((\kappa' \losange \lambda )_{\theta'}\right)_*  g_*([D_\theta,w])\\
& =&  \left((\kappa' \losange \lambda )_{\theta'}\right)_* \left([D'_{\theta'},w']\right)
\ = \ \left[\calD(\calK',\calL)\right].
\end{eqnarray*}

\chapter {The intersection bibracket} \label{Intersection bibrackets}

  Throughout this chapter, $M$ is 
  an oriented  smooth $n$-dimensional manifold  with non-void boundary,  where $n\geq 2$.

\section{Construction of the intersection bibracket} \label{bilh}

We introduce the path homology category  of~$M$ and define the intersection bibracket in this category.
 
\subsection{The path homology category} \label{phc}

Let $\calC=\calC(M)$ be the graded category
whose   set of objects is   $\partial M$ and whose graded modules of  morphisms are defined by
$$
\Hom_{\calC} ( \star ,  \star' )=H_\ast\big(\Omega(M, \star , \star' )\big)
$$
for any $ \star , \star'  \in \partial M$.
Composition in  $\calC$ is the pairing \eqref{concic}   defined via concatenation of paths.
For $\star  \in \partial M$, the identity morphism of $  \star$ in $\calC$ is the element of $H_0(\Omega(M,\star,\star))$ represented by the  constant path in $ \star $.
 We call $\calC$ the \index{path homology category}  \emph{path homology category} of $M$.
By Section~\ref{extenscat},  this category  determines a graded algebra
\begin{equation}\label{algebraA}
A= A(\calC)=\bigoplus_{ \star , \star'  \in \partial M}    H_{\ast} \big( \Omega(M, \star,\star' )  \big) .
\end{equation}

The subcategory   ${\calC}^0$   of $\calC$  formed by all objects and  morphisms of degree $0$  can be  formulated in terms of paths in $M$:
for any $ \star , \star'  \in \partial M$,  the   module  of morphisms $\Hom_{  {\calC}^0   }(\star,\star')$  is freely generated by
the set of homotopy classes of paths from $ \star $ to $ \star' $ in $M$.
  Thus   the category   ${\calC}^0$   is the linearization of  the fundamental groupoid  $\pi_1(M,\partial M)$ of $M$ based at $\partial M$,
and the algebra $A(  {\calC}^0   )$   is the corresponding groupoid algebra. Clearly, $A(  {\calC}^0   )$ embeds in $A$ as a subalgebra.

\subsection{The intersection bibracket} \label{interbi}

  Assume  that $n=\dim(M) \geq 3$ and
\begin{equation} \label{cross_iso}
  \begin{array}{c}
\hbox{\it the cross product in the homology $H_*(\Omega_\star)$ } \\
\hbox{\it of the loop space $\Omega_\star=\Omega(M,\star,\star)$ based at $\star \in \partial M$ } \\
\hbox{\it induces an isomorphism } H_*(\Omega_\star) \otimes H_*(\Omega_\star) \simeq  H_*(\Omega_\star \times \Omega_\star).
\end{array}
\end{equation}
 By the K\" unneth theorem,
  the  condition \eqref{cross_iso} holds  if $\kk$ is a principal ideal domain and   $H_*(\Omega_\star)=H_*(\Omega_\star; \kk)$ is a flat $\kk$-module
(this occurs, for instance, when $\kk$ is a field);
it   also holds   for any $\kk$ if   $H_*(\Omega_\star;\ZZ)$ is a free abelian group.
Then  the cross product induces an isomorphism
$\varpi_{32,14}: H_{\ast} ( \Omega_{32}  ) \otimes H_{\ast} (   \Omega_{14}) \to H_{\ast} ( \Omega_{32} \times \Omega_{14})$
for any choice of base points $ \star_1, \star_2, \star_3, \star_4 \in \partial M$.
Composing the inverse isomorphism with the  map   $\Upsilon_{12,34}$ defined in
  Section~\ref{Upsilon1111},  we obtain a  linear map
$$ (\varpi_{32,14})^{-1} \Upsilon_{12,34}:
H_{\ast} ( \Omega_{12}  ) \otimes H_{\ast} (   \Omega_{34}) \to H_{\ast} ( \Omega_{32}  ) \otimes H_{\ast} (   \Omega_{14}).
$$
The direct sum of these maps over all 4-tuples of points in $ \partial M$ is a linear map
$$
\double{-}{-}: A\otimes A \longrightarrow A \otimes A
$$
called the  \index{intersection bibracket} \emph{intersection bibracket}   of $M$.    We can now state our main  result.

\begin{theor}\label{geometricbibracket}
 Under the assumptions above,  $(\calC, \double{-}{-})$  is a double Gerstenhaber category of degree $d=2-n$.
\end{theor}

\begin{proof}  That the bibracket $\double{-}{-}$ has degree $d$ follows from the fact
that the intersection polychain of a $p$-polycycle and a $q$-polycycle has dimension $p+q+2-n=p+q+d$ for any $p,q$.
  The $d$-antisymmetry of $\double{-}{-} $ follows from the formula \eqref{Upsilon1}
and the following well-known  fact:   for   any  topological spaces $X$ and $Y$
such that   the cross product induces an isomorphism $H_*(X) \otimes H_*(Y) \to H_*(X \times Y)$,
the isomorphism $H_\ast(X) \otimes H_\ast(Y) \to H_\ast(Y) \otimes H_\ast(X)$
induced by the interchange of factors $X\times Y\to Y\times X$ and the   cross product isomorphisms,
carries $a\otimes b $ to $(-1)^{\vert a\vert \vert b\vert} b\otimes a$ for any homogeneous $a\in H_\ast(X)$, $b\in H_\ast(Y)$; see, for example,  \cite[Section 4(b)]{FHT}.
The  bibracket  $\double{-}{-} $  satisfies the first  
 Leibniz rule~\eqref{bibi}   as easily follows from  Lemma~\ref{Leibniz_Upsilon}  using the associativity and the naturality of the cross product in singular homology.
By   Lemma~\ref{bibrtranspose} and the $d$-antisymmetry, the bibracket $\double{-}{-}$ also satisfies the second  
Leibniz rule~\eqref{bibi+}. Therefore $\double{-}{-}$ is a $d$-graded bibracket in $A$.

It  is obvious from the definitions that $\double{-}{-}$ annihilates the identity morphisms of all objects.
It remains only to prove  that the associated  tribracket  is equal to  zero;
we postpone the proof to Section \ref{Jacobi}.
\end{proof}

 Since $d=2-n<0$, the restriction of the   intersection    bibracket in  $\calC$ to   ${\calC}^0$   is equal to zero.
Moreover,   the  morphisms in   ${\calC}^0$   represented by paths in $\partial M$ annihilate the bibracket in $\calC$ both on the right and on the left.

Theorem~\ref{geometricbibracket} and Lemma~\ref{bracketcategories}  imply that for every  integer $N\geq 1$,
the associated representation algebra $\calC_N^+$  is   a   unital    Gerstenhaber algebra of degree $ 2-n$.

\subsection{The {Pontryagin} algebra}\label{The {Pontryagin} algebra}

  We now  fix   a base point $\star\in \partial M$.
 By the \index{Pontryagin algebra} \emph{Pontryagin algebra of $M$}, we mean the unital graded algebra
 $$
 A_\star =\End_{{{\calC}(M)}} (\star) =H_\ast (\Omega_\star )
 \quad \   \hbox{where} \ \Omega_\star=\Omega (M, \star, \star).
 $$
Multiplication in $A_\star$ is the \index{Pontryagin product} \emph{Pontryagin  product} given by $ab=\conc_\ast(a\times b)$.

   Pontryagin  algebras have been extensively studied since Serre's thesis \cite{Se}.
They can be explicitly computed using the Adams--Hilton model \cite{AH} or the techniques of rational homotopy theory (at least, in the simply connected case).
We only mention  the   relation with  the homotopy groups, and refer to \cite{FHT} for a  detailed exposition.
Consider the boundary  homomorphism
\begin{equation*}\label{partial}
\partial_i: \pi_i(M)=\pi_i(M, \star)\longrightarrow \pi_{i-1}(\Omega_\star)=\pi_{i-1}(\Omega_\star, e_\star)
\end{equation*}
for the path space fibration of $M$   where   $i\geq 1$ and $e_\star\in \Omega_\star$ is the constant path at~$\star$.
Since the total space of  that fibration   is contractible, $\partial_i$ is an isomorphism for all~$i$.
Composing $\partial_i$ with the Hurewicz homomorphism    $ \pi_{i-1}(\Omega_\star)\to H_{i-1}(\Omega_\star)$,
we obtain an additive map  $  \overline \partial_i: \pi_i(M )\to A_\star^{i-1}$ called the \index{connecting homomorphism} \emph{connecting homomorphism}.
For $i=1$,  this homomorphism extends to a ring isomorphism   $\kk[ \pi_1(M, \star  )]   \simeq  A_\star^0 $.
  For $i=2$, this homomorphism induces an isomorphism from  $\kk \otimes_\ZZ  \pi_2(M )$ onto   $H_1(\Omega_\star^{\operatorname{null}}) \subset A_\star^1$
where $\Omega_\star^{\operatorname{null}}$ is the connected component of $\Omega_\star$ formed by all null-homotopic loops.

 The   group $\pi_1(M,\star)$   acts on $A_\star$  by   graded algebra automorphisms:
the  action    of any $g \in \pi_1(M,\star)$  is the automorphism
$a\mapsto
 a^g = gag^{-1}
$ of $A_\star$
where $g$ is viewed as an invertible element of $A_\star^0$.
The inclusion $\partial M \subset M$ allows us to consider the induced action of $\pi_1(\partial M,\star)$ on $A_\star$.

 \begin{theor}\label{geometricbibracket+++++}
If  $n=\dim(M)\geq 3$   and the condition \eqref{cross_iso} is satisfied,
then the restriction of the intersection  bibracket $\double{-}{-}$ in the category ${\calC}(M)$ to $A_\star$
is a $\pi_1(\partial M,\star)$-equivariant Gerstenhaber  bibra\-cket of degree $ 2-n$.
\end{theor}

\begin{proof}
 Clearly,   $A_\star=A(\calC_\star)$ where $\calC_\star$ is the full subcategory of $\calC$ determined by the object $\star$.
 Therefore our claim  is a consequence of Theorem~\ref{geometricbibracket} and  the results stated   at the end of   Section~\ref{doublecat}. 
We only need to check the equivariance. Let $a,b\in A_\star$ and   let  $\varsigma$ be a loop in $\partial M$ based at $\star$ representing   $g  \in \pi_1(\partial M,\star)$.
We deduce from  the commutativity of the diagram  \eqref{1234_1234'_H}   that
$$
\Upsilon(a^g \otimes b^g) =  \Upsilon\left((\varsigma^{-1},\varsigma^{-1})_\sharp(a), (\varsigma^{-1},\varsigma^{-1})_\sharp(b)\right)
=\left( (\varsigma^{-1},\varsigma^{-1})_\sharp \times( \varsigma^{-1},\varsigma^{-1})_\sharp\right) \Upsilon(a,b).
$$
Using the naturally of the cross product, we conclude that
$$
\double{a^g}{b^g}= \big(\double{a}{b}'\big)^g \otimes \big(\double{a}{b}''\big)^g.
$$

\up
\end{proof}

  By Theorem~\ref{geometricbibracket+++++} and  Lemma~\ref{importantGerspp},
  the   intersection   bibracket in $A_\star$ induces  a   natural   structure of a Gerstenhaber algebra of degree  $2-n $
  in the commutative unital graded algebra $(A_\star)_N^+$ for all $N\geq 1$.
We call   $(A_\star)_N^+$ the \index{representation algebra} \emph{$N$-th   representation algebra of~$M$}.
The action of $\pi_1(\partial M,\star)$ on  $A_\star$ induces an action of $\pi_1(\partial M,\star)$ on   $(A_\star)_N^+$    by   graded algebra automorphisms,
  and the Gerstenhaber bracket  in  $(A_\star)_N^+$ is  $\pi_1(\partial M,\star)$-equivariant.
  The isomorphism classes of the double Gerstenhaber algebra   $A_\star$   and
the Gerstenhaber algebras $\{(A_\star)_N^+\}_N$ depend only on the connected component of $\star$ in $\partial M$.

\subsection{The induced Lie bracket}\label{The induced Lie bracket}
We  keep   notation of Section \ref{The {Pontryagin} algebra}   and  let $\check A_\star$ be the quotient of $A_\star =H_\ast (\Omega_\star )$
by the submodule $[A_\star,A_\star]$ spanned by the vectors $   ab - (-1)^{\vert a \vert \vert b \vert} ba$
where  $ a,b $ run over all homogeneous elements of $ A_\star$. Under the assumptions of Theorem \ref{geometricbibracket+++++},
the  intersection bibracket $\double{-}{-}$ in $A_\star$ composed with the multiplication  of $A_\star$
induces a $(2-n)$-graded Lie bracket $\langle-,-\rangle$ in $\check A_\star$,        see Section~\ref{stdb}.  

  The Lie  bracket $\langle-,-\rangle$   can be  computed using  the map  
 $  \conc_* :H_*( \Omega_\star \times \Omega_\star) \to H_*(\Omega_\star)$ induced by the   concatenation of loops.
Namely, if $h:  A_\star \to \check A_\star$ is the natural projection, then   for   any homogeneous $a, b\in A_\star$,
\begin{eqnarray*}
\langle h(a) , h(b) \rangle &=  &h ( \double{a}{b}' \double{a}{b}'') \\
&=&  h    \conc_*  \Upsilon(a\otimes b)
\ = \ (-1)^{\vert b\vert  + n \vert a\vert   } h      \conc_* \big(\big[  \widetilde\Upsilon(\lb a \rb \otimes \lb b \rb)\big]\big).
\end{eqnarray*}
The resulting expression    may be used as the definition  of $\langle-,-\rangle$  avoiding the use of   $\double{-}{-}$.
This gives a $(2-n)$-graded Lie bracket  in $
\check A_\star$  over an arbitrary    commutative   ring  $\kk$.
  The Jacobi identity for $\langle-,-\rangle$ may be deduced from Lemma~\ref{Jacobi_Upsilon} below.
  Presumably, the   Lie    bracket   $\langle - , - \rangle$  is  related to the operation
  discussed in \cite[Remark 3.2.3]{KK1} using   Chas--Sullivan's   techniques.

\subsection{The   simply connected case} \label{ss_case}

Suppose   that the  manifold $M$ is simply connected and  the ground ring $\kk$ is a field of characteristic zero.
The    classical  Milnor--Moore   theorem  (see  \cite[Theorem 21.5]{FHT})   asserts   that,
 the   {Pontryagin} algebra   $A_\star =H_\ast (\Omega_\star )$  is fully  determined by
$\pi_\ast(M)=  \oplus_{p\geq 0}  \,\, \pi_p (M)$ and the   Whitehead bracket   $[-,-]_{\operatorname{Wh}}$   in $\pi_\ast(M)$.
More precisely,   consider  the  graded module
$$
 L_\star   =\bigoplus_{p\geq 0} \,\, \kk \otimes_\ZZ \pi_p (\Omega_\star)
$$
(obtained from $\pi_\ast(M)$ by   tensorizing with   $\kk$ and shifting the degree by $1$),
  and equip $L_\star$  with the    bracket   defined by
$$
[k \otimes \alpha  , l \otimes \beta  ]=  kl \otimes  (-1)^{p +1} \partial_{p+q+1}
\left ( \left[\partial_{p+1}^{-1} (\alpha), \partial_{q+1}^{-1} (\beta)\right]_{  \operatorname{Wh}   } \right ) \  \in \kk \otimes \pi_{p+q} (\Omega_\star)
$$
for any $k,l\in \kk$, $\alpha\in \pi_p (\Omega_\star)$, $\beta\in \pi_q (\Omega_\star)$. Then $  L_\star  $  is a $0$-graded Lie  algebra   and
the Hurewicz homomorphism   $    L_\star \to A_\star$ extends to an isomorphism  of the universal enveloping algebra $U(L  ) $ onto $ A_\star$.
Moreover,  under this isomorphism, the standard comultiplication in $U(L  )$ carrying any
$\alpha\in  L_\star    $ to $\alpha\otimes 1+1\otimes \alpha$
corresponds to  the  comultiplication in $A_\star$ induced by the diagonal map $\Omega_\star\to \Omega_\star\times \Omega_\star$.
Note that, by the   Poincar\'e--Birkhoff--Witt   theorem for graded Lie algebras \cite[Theorem 21.1]{FHT},
the natural linear map $L_\star\to U(L_\star)$ is   injective so that   $L_\star$ can be treated  as a submodule of $U(L_\star)   \simeq A_\star $.

Recall  from Section   \ref{AMT--1}   that the   $0$-graded   Lie algebra $  L_\star $ gives rise to representation algebras $\{(  L_\star )_N\}_{N\geq 1}$.
The   Milnor--Moore   isomorphism  $U(  L_\star )   \simeq    A_\star$  induces an isomorphism   $(  L_\star )_N    \simeq   (A_\star)^+_N$ for all $N\geq 1$.
In this way, the    algebras $\{(  L_\star )_N\}_{N\geq 1}$ acquire  a structure of Gerstenhaber algebras   of degree $ 2-n$.

\subsection{The 2-dimensional case} \label{dimension_2}
  
 The case  $n=2$ (so far ruled out in this  section by the assumptions of Section \ref{interbi}) has been extensively studied by several authors 
 and  gave the original impetus to this work. We  briefly discuss   this case.

A connected  oriented surface $M$  with $\partial M\neq   \varnothing  $ is an Eilenberg--MacLane space $\hbox{K}(\pi,1)$ where  $\pi $ is the fundamental group of $ M$.
For any points $\star_1,\star_2 \in \partial M$, the space $\Omega(M,\star_1,\star_2)$ is homotopy equivalent  to the underlying discrete set of   $\pi$.
Therefore, in the notation of Section \ref{phc}, we have     $\calC=  {\calC}^0  $
and  $A=A (  {\calC}^0  )$   is the groupoid algebra of $\pi_1(M,\partial M)$.

For any   points $\star_1,\star_2,\star_3,\star_4 \in \partial M$ such that $\{\star_1,\star_2\} \cap \{\star_3,\star_4\}=\varnothing$,
Section~\ref{Upsilon1111} yields  a linear map
$$
\Upsilon_{12,34}:H_0(\Omega_{12}) \otimes H_0(\Omega_{34}) \longrightarrow H_0(\Omega_{32} \times \Omega_{14})=H_0 (\Omega_{32}) \otimes H_0(\Omega_{14})
$$
 (the latter equality holds for all   $\kk$). This construction   extends to arbitrary 4-tuples of points in $\partial M$
by slightly pushing these points in the positive direction along $\partial M$  and  proceeding as in Section~\ref{checkUpsilon---}.
After an appropriate normalization, this yields a $0$-antisymmetric 0-graded bibracket  of degree 0   in the groupoid algebra $A$.
This bibracket is  quasi-Poisson in an appropriate sense, cf$.$ \cite{AKsM,VdB,MT}.
For    $\star \in \partial M$, the restriction of   this   bibracket to the   group    algebra
$A_\star=  \kk[\pi_1(M,\star)]$    is the   double bracket $\double{-}{-}^s$  studied in \cite[Section 7]{MT}.
 It is  closely related to the homotopy intersection form in $\kk[\pi_1(M,\star)]$ introduced in \cite{Tu1}; see also \cite{KK2} for a similar operation.
 The associated  Lie bracket $\langle-,-\rangle$ in   $\check A_\star$    was first introduced by   Goldman   \cite{Go2}.

 Lemma \ref{bracketcategories} implies that for every   integer $N\geq 1$,
 the   above   bibracket in $A$ induces a  
 bracket in the associated representation algebra  $\calC_N^+$. 
  This bracket is quasi-Poisson (and not   Poisson), cf$.$ \cite{MT}.
 Note that  $\calC^+_N$ is the coordinate algebra of the affine scheme (over $\kk$) that associates to any unital commutative algebra $B$
the set of groupoid homomorphisms $\pi_1(M,\partial M) \to \GL_{N}(B)$. Indeed, through   linear extension of groupoid homomorphisms,
 the latter set may be identified with the set of linear functors from $ \calC$ to the category $\Mat_N(B)$   considered  in  Section~\ref{extenscat}. We conclude by applying \eqref{universality}.

\section{The Jacobi  identity} \label{Jacobi}

We  conclude the proof of Theorem \ref{geometricbibracket}
by proving that the tribracket associated with the intersection bibracket   is equal to zero. 
We resume notation  of    Chapter~\ref{Operations on polychains}, 
 i.e.,  fix  points  $\star_1,\star_2,\star_3,\star_4 \in \partial M$ 
such  that $\{\star_1,\star_2\}\cap\{\star_3,\star_4\}=\varnothing$ and, for  any  $i,j\in \{1,2,3,4\}$,  
   let   $\Omega_{ij} =\Omega(M,\star_i,\star_j)$ be the path space and   $ \Omega^\circ_{ij} =\Omega^\circ (M,\star_i,\star_j)$  be the proper path space of  $ (M,\star_i,\star_j)$. 
  We start by developing a parametrized version of the theory of polychains in path spaces.

\subsection{Parametrized versions of $\widetilde \Upsilon$, $\check \Upsilon$, and $\Upsilon$}\label{parar}

Let $Z$ be an arbitrary   topological space.
Given a  polychain $\calL=(L,\psi,v,\lambda)$  in $\Omega^\circ_{34} \times {Z}$, we let~$\lambda' $ and~$\lambda'' $
be the compositions  of $\lambda:L\to \Omega^\circ_{34} \times {Z}$ with the projections to $\Omega^\circ_{34} $ and $ {Z}$, respectively.
We call the polychain $\calL$ \index{polychain!smooth} \emph{smooth} if  the map $ \lambda':L \to \Omega^\circ_{34}$ is smooth in the sense of Section~\ref{Path homology}.
Applying the definitions of  Section~\ref{facehomology}  but considering only smooth   polychains in $\Omega^\circ_{34} \times {Z}$,
 we obtain \index{face homology!smooth} \emph{smooth face homology} $\widetilde {H}^s_\ast(\Omega^\circ_{34} \times {Z})$.
 The proof of Theorem \ref{loopsp+++++}    easily adapts to this setting and yields that the natural linear map
$$
\widetilde {H}^s_\ast(\Omega^\circ_{34} \times {Z})\longrightarrow \widetilde {H}_\ast(\Omega^\circ_{34} \times {Z})\simeq  \widetilde {H}_\ast(\Omega_{34} \times {Z})
$$
is an isomorphism. This computes the face homology of    $\Omega_{34} \times {Z}$  in terms of smooth polychains in $\Omega^\circ_{34} \times {Z}$.

  We say that  smooth  polychains  $\calK =(K, \varphi, u , \kappa)$ in $\Omega^\circ_{12}$
and $\calL =(L, \psi, v , \lambda)$ in $ \Omega^\circ_{34} \times {Z} $  are  \index{polychains!transversal}  {\it transversal}
if the     maps $ \kappa:K \to \Omega^\circ_{12}$
and $ \lambda': L\to \Omega^\circ_{34}  $ are    transversal in the sense of Section~\ref{transs}.
A pair  $(a ,b)\in \widetilde H_p (\Omega_{12}) \times  \widetilde H_q (\Omega_{34} \times {Z})$   with $p,q\geq 0$
is \index{polychains!transversely represented} \emph{transversely represented} by a pair $(\calK,\calL)$ if $\calK$ is a smooth reduced $p$-polycycle in $\Omega_{12}^\circ$
and $\mathscr{L}$ is a smooth reduced $q$-polycycle in $\Omega_{34}^\circ \times {Z}$   transversal to $\calK$.  Adapting the proof of Lemma \ref{keytheor},
we obtain that any pair $(a,b)$ as above can   be transversely represented by a pair of polycycles, and, furthermore,
any  two  such pairs of polycycles  can be related by a
finite sequence  of transformations   $(\calK, \calL)\mapsto  (\check\calK, \check \calL) $  of the following   types:
\begin{itemize}
\item[(i)]    $\calL \cong  \check \calL$ and  $\check \calK \cong  \calK \sqcup \partial^r \calM$ or $ \calK \cong  \check \calK \sqcup \partial^r \calM$
where $\calM$ is a smooth $(p+1)$-polychain in $\Omega_{12}^\circ$ transversal to $\calL$;
\item[(ii)]  $\calK \cong  \check \calK$ and   $\check \calL \cong  \calL \sqcup \partial^r \calN$ or $ \calL \cong  \check \calL \sqcup \partial^r \calN$
where $\calN$ is a smooth $(q+1)$-polychain in $\Omega_{34}^\circ \times Z$
  transversal to $\calK$.
\end{itemize}

We next adapt the construction of the intersection polychain.
Consider   smooth transversal polychains $\calK=(K,\varphi,u,\kappa)$     in $\Omega^\circ_{12}$ and $\calL=(L,\psi,v,\lambda)$   in $\Omega^\circ_{34} \times {Z}$.
Since the  polychain $\calL'=(L,\psi,v,\lambda')$ in $\Omega^\circ_{34}$  is  smooth and transversal to $\calK$, Section \ref{def_D} yields an intersection polychain
$
\calD(\calK, \calL') =(D , \theta, w, \kappa\losange \lambda')
$
in $\Omega_{32} \times \Omega_{14}$. We lift $\calD(\calK, \calL')$ to   a polychain in $\Omega_{32} \times \Omega_{14}\times {Z}$ as follows.

\begin{lemma}
Let $\pr: K \times I \times L \times I \to L$ be the cartesian projection.
 The tuple $\calD^{Z}(\calK,\calL)=(D,\theta, w, \delta)$ with $\delta=(\kappa\losange \lambda', \lambda''\circ \pr\vert_D)$
is a polychain   in $\Omega_{32} \times \Omega_{14}\times {Z}$.
\end{lemma}

\begin{proof}
We need only to check that the  map $\lambda''\circ \pr\vert_D:D\to Z$ is compatible with the partition $\theta$.
Let $F,G$ be two faces of $D$ of the same type and let
$$
N_{F} = A_{F} \times I \times B_{F} \times I,  \quad N_{G} = A_{G} \times I \times B_{G} \times I
$$
be the smallest faces of $K \times I \times L \times I$   containing $F,G$,  respectively.
Since $\lambda''$ is compatible with the partition $\psi$ of $L$, we have for any $(k,s,l,t)\in F$
\begin{eqnarray*}
\lambda'' \pr\big(\theta_{F,G}(k,s,l,t)\big) &=& \lambda'' \pr\big(\varphi_{A_{F},A_{G}}(k),s,\psi_{B_F,B_G}(l),t\big)\\
&=& \lambda'' \big(\psi_{B_F,B_G}(l)\big) \\
&=& \lambda''(l) \ = \  \lambda'' \pr(k,s,l,t).
\end{eqnarray*}

\up
\end{proof}

The next claim  is a parametrized version of Lemma \ref{Upsilon_tilde_def} and is proved similarly.

\begin{lemma}\label{Upsilon^x_tilde_def}
For any integers $p,q \geq 0$,
the  intersection $(\calK, \calL)\mapsto \calD^{Z} (\calK, \calL)$   induces a bilinear map
$
\widetilde H_p (\Omega_{12}) \times \widetilde H_q (\Omega_{34} \times {Z})
\to \widetilde H_{p+q+2-n} ( \Omega_{32} \times \Omega_{14} \times {Z}).
$
\end{lemma}

The direct sum   over all integers $p,q\geq 0$ of the pairings produced by Lemma~\ref{Upsilon^x_tilde_def} is a linear map of degree $2-n$
$$
\widetilde{\Upsilon}_{12,34{Z}}: \widetilde H_*(\Omega_{12})  \otimes \widetilde H_*(\Omega_{34} \times {Z})
\longrightarrow    \widetilde H_*(\Omega_{32} \times \Omega_{14} \times {Z}).
$$
As  in Section \ref{checkUpsilon}, a normalized version  of this map
$$
\check\Upsilon_{12,34{Z}}: \widetilde H_*(\Omega_{12})  \otimes \widetilde H_*(\Omega_{34} \times {Z})
\longrightarrow    \widetilde H_*(\Omega_{32} \times \Omega_{14} \times {Z})
$$
is defined by    $\check \Upsilon_{12,34{Z}} (a\otimes b)=(-1)^{\vert b\vert + n\vert a \vert}\, \widetilde{\Upsilon}_{12,34{Z}} (a\otimes b)$
for any homogeneous   $a\in \widetilde H_\ast (\Omega_{12})$ and $b\in \widetilde H_\ast (\Omega_{34} \times {Z})$.
We  also   define an operation $\Upsilon_{12,34{Z}}$ in singular homology by the  commutative diagram
\begin{equation} \label{Upsilon_Upsilon_bis-}
\xymatrix{
\widetilde H_\ast (\Omega_{12}) \otimes \widetilde H_\ast (\Omega_{34 } \times {Z}) \ar[rr]^-{ \check\Upsilon_{12,34{Z}} } &&
\widetilde H_{\ast} ( \Omega_{32} \times \Omega_{14} \times {Z}) \ar[d]^-{[-]}\\
H_\ast (\Omega_{12}) \otimes H_\ast (\Omega_{34} \times {Z}) \ar[u]^-{\lb-\rb \times \lb-\rb}
\ar@{-->}[rr]^-{ \Upsilon_{12,34{Z}} } && H_{\ast} ( \Omega_{32} \times \Omega_{14} \times {Z}).
}
\end{equation}
The proof of Lemma \ref{diag-imp} extends to this  setting  and    gives the  commutative  diagram
\begin{equation} \label{Upsilon_Upsilon_bis}
\xymatrix{
\widetilde H_\ast (\Omega_{12}) \otimes \widetilde H_\ast (\Omega_{34 } \times {Z}) \ar[rr]^-{ \check\Upsilon_{12,34{Z}} }  \ar[d]_-{[-] \times [-]} &&
\widetilde H_{\ast} ( \Omega_{32} \times \Omega_{14} \times {Z}) \ar[d]^-{[-]}\\
H_\ast (\Omega_{12}) \otimes H_\ast (\Omega_{34} \times {Z}) \ar[rr]^-{ \Upsilon_{12,34{Z}} } && H_{\ast} ( \Omega_{32} \times \Omega_{14} \times {Z}).
}
\end{equation}

The following  two lemmas   will help us to compute   $\widetilde {\Upsilon}_{12,34{Z}}$ and ${\Upsilon}_{12,34{Z}}$.

\begin{lemma}\label{tilde_Upsilon_cross}
For any $a\in \widetilde{H}_*(\Omega_{12})$,   any  $b\in \widetilde{H}_*(\Omega_{34})$ and any homogeneous   $c\in \widetilde{H}_*({Z})$, we have
$
\ \widetilde{\Upsilon}_{12,34{Z}}\big(a\otimes (b\times c)\big) = (-1)^{\vert c \vert}\, \widetilde{\Upsilon}_{12,34}(a\otimes b) \times c.
$
\end{lemma}

\begin{proof}
It suffices to consider the case where $a$ and $b$ are homogeneous.
Let $\calK=(K,\varphi,u,\kappa)$ and $\calL=(L,\psi,v,\lambda)$ be  smooth   polycycles in $\Omega_{12}^\circ$ and $\Omega_{34}^\circ$
  representing $a$ and $b$ respectively, and such  that $\calK$ is transversal to $\calL$.
Let $\calN=(N,\chi,z,\eta)$ be a   polycycle in ${Z}$ representing $c$.
Then
$$ \widetilde{\Upsilon}_{12,34}(a\otimes b) \times c = \lb \calD(\calK,\calL) \times \calN\rb.$$
By the definition of   $\widetilde{\Upsilon}_{12,34{Z}}$,
$$
\widetilde{\Upsilon}_{12,34{Z}}\big(a\otimes (b\times c)\big) = \lb \calD^{{Z}}(\calK,\calL \times \calN) \rb.
$$
Therefore, it is enough to show that
\begin{equation}\label{D{Z}}
\calD(\calK,\calL) \times \calN  =  (-1)^{\vert c \vert}\,  \calD^{{Z}}(\calK,\calL \times \calN).
\end{equation}
We set $\calD(\calK,\calL)=(D,\theta,w,\kappa \losange \lambda)$  so that
$$
\calD(\calK,\calL)\times \calN = \big(D\times N,\theta\times \chi,w\times z,(\kappa \losange \lambda)\times \eta\big).
$$
We  also set
$$
   \calD^{{Z}}(\calK,\calL \times \calN) = \big(D^{Z},\theta^{Z},w^{Z}, \delta^Z\big)
\quad \hbox{with} \ \delta^Z=\big(\kappa \losange (\lambda \pr_L),  \eta \pr_N\vert_{D^{Z}}\big)
$$
where $\pr_L:L \times N \to L$ and $\pr_N: K \times I \times (L \times N) \times I \to N $  are  the cartesian projections.
The map $$(K \times I \times L \times I) \times N \longrightarrow K \times I \times (L \times N) \times I, \quad (k,s,l,t,n) \longmapsto (k,s,l,n,t)$$
restricts to a diffeomorphism $f: D \times N \to D^{Z}$  of degree $(-1)^{\dim(N)}=(-1)^{\vert c\vert}$.
For any  point  $(k,s,l,t,n)$ in $D \times N$, we have
\begin{eqnarray*}
 \delta^Z f(k,s,l,t,n) \ = \ \delta^Z (k,s,l,n,t)
&=&  \big( (\kappa \losange \lambda)(k,s,l,t),\eta(n)\big)\\
&=& \big((\kappa \losange \lambda) \times \eta\big)(k,s,l,t,n).
\end{eqnarray*}
Furthermore, the diffeomorphism $f$ carries the partition $\theta\times \chi$ to $\theta^{Z}$ and the weight $w \times z$ to the weight $w^{Z}$.
Hence, $f$ is  a diffeomorphism of polychains  \eqref{D{Z}}.
\end{proof}

\begin{lemma}\label{Upsilon_cross}
For any $a\in {H}_*(\Omega_{12})$, $b\in {H}_*(\Omega_{34})$ and $c\in {H}_*({Z})$, we have
$$
{\Upsilon}_{12,34{Z}}\big(a\otimes (b\times c)\big) = {\Upsilon}_{12,34}\big(a\otimes b\big) \times c
$$
\end{lemma}

\begin{proof}
It  suffices to consider     homogeneous $a,b,c$.
By Lemma~\ref{twocrossproducts}, we have $b\times c = [\lb b \rb] \times [\lb c \rb]= [ \lb b \rb \times \lb c\rb]$.
We deduce that
\begin{eqnarray*}
{\Upsilon}_{12,34{Z}}\big(a\otimes (b\times c)\big)  &=& {\Upsilon}_{12,34{Z}}\big([\lb a \rb ]\otimes  [ \lb b \rb \times \lb c\rb]\big) \\
&=& \left[\check {\Upsilon}_{12,34{Z}}\big(\lb a\rb \otimes (\lb b \rb \times \lb  c \rb ) \big) \right] \\
&=&  (-1)^{\vert b \vert + \vert c \vert +n \vert a \vert} \big[\widetilde{\Upsilon}_{12,34{Z}}\big(\lb a\rb \otimes (\lb b \rb \times \lb  c \rb ) \big) \big]\\
&=&  (-1)^{\vert b \vert  +n \vert a \vert} \big[\widetilde{\Upsilon}_{12,34{}}(\lb a\rb \otimes \lb b \rb )  \times \lb  c \rb \big] \\
&=&  (-1)^{\vert b \vert  +n \vert a \vert} \big[\widetilde{\Upsilon}_{12,34{}} (\lb a\rb \otimes \lb b \rb ) \big] \times \left[ \lb  c \rb \right] \\
& = & {\Upsilon}_{12,34{}} \big(a\otimes b\big) \times c
\end{eqnarray*}
where the second, fourth and fifth equalities follow from \eqref{Upsilon_Upsilon_bis}, Lemma \ref{tilde_Upsilon_cross} and Lemma \ref{twocrossproducts} respectively.
\end{proof}

Given  two topological spaces $Y$ and $Z$, a  straightforward generalization of the constructions above 
and of Lemma \ref{Upsilon^x_tilde_def} yields a bilinear map
$$
\widetilde \Upsilon_{{Y}12,34Z}: \widetilde H_*(Y \times \Omega_{12}) \otimes \widetilde H_*( \Omega_{34} \times Z)
\longrightarrow  \widetilde H_*(Y \times \Omega_{32} \times \Omega_{14} \times Z).
$$
 A  normalized version $\check \Upsilon_{{Y}12,34Z}$ of this map is defined by
$$\check \Upsilon_{{Y}12,34Z}(a \otimes b) = (-1)^{\vert b\vert+n \vert a \vert}\, \widetilde \Upsilon_{{Y}12,34Z}(a\otimes b)$$
for any homogeneous $a \in  \widetilde H_*(Y \times \Omega_{12})$ and $b \in  \widetilde H_*( \Omega_{34} \times Z)$.
The corresponding map in singular homology is defined by the  commutative diagram
\begin{equation} \label{Upsilon_Upsilon_ter-}
\xymatrix{
\widetilde H_\ast (Y \times \Omega_{12}) \otimes \widetilde H_\ast (\Omega_{34 } \times {Z}) \ar[rr]^-{ \check\Upsilon_{{Y}12,34{Z}} } &&
\widetilde H_{\ast} (Y \times  \Omega_{32} \times \Omega_{14} \times {Z}) \ar[d]^-{[-]}\\
H_\ast (Y \times \Omega_{12}) \otimes H_\ast (\Omega_{34} \times {Z}) \ar[u]^-{\lb-\rb \times \lb-\rb}
\ar@{-->}[rr]^-{ \Upsilon_{{Y}12,34{Z}} } && H_{\ast} ( Y \times \Omega_{32} \times \Omega_{14} \times {Z}).
}
\end{equation}
Then, again, we  have the  commutative  diagram
\begin{equation} \label{Upsilon_Upsilon_ter}
\xymatrix{
\widetilde H_\ast (Y \times \Omega_{12}) \otimes \widetilde H_\ast (\Omega_{34 } \times {Z}) \ar[rr]^-{ \check\Upsilon_{{Y}12,34{Z}} }  \ar[d]_-{[-] \times [-]} &&
\widetilde H_{\ast} ( Y \times \Omega_{32} \times \Omega_{14} \times {Z}) \ar[d]^-{[-]}\\
H_\ast (Y \times \Omega_{12}) \otimes H_\ast (\Omega_{34} \times {Z}) \ar[rr]^-{ \Upsilon_{{Y}12,34{Z}} } && H_{\ast} (Y \times \Omega_{32} \times \Omega_{14} \times {Z}).
}
\end{equation}
Finally,   Lemma \ref{Upsilon_cross} generalizes to the   identity
\begin{equation} \label{double_parameters}
{\Upsilon}_{{Y}12,34{Z}}\big((c \times a)\otimes (b\times d)\big) =   c \times {\Upsilon}_{12,34}\big(a\otimes b\big) \times d
\end{equation}
for any $a\in {H}_*(\Omega_{12})$, $b\in {H}_*(\Omega_{34})$ and   $c\in H_*(Y)$, $d\in {H}_*({Z})$.

\subsection{Half-smooth polychains}\label{parar+++}

We   compute the   intersection   operations of Section~\ref{parar}  via  so-called   ``half-smooth''   polychains.
  Let $Z$ be  a   topological space.
A  $q$-polychain $\calL=(L,\psi,v,(\lambda',\lambda''): L\to \Omega_{34}^\circ  \times Z)$ is \index{polychain!half-smooth} \emph{half-smooth}
if the restrictions of  the map $\tilde {\lambda}': L \times I \to M$ (adjoint to ${\lambda}' $) to the manifolds with faces $L \times [0,1/2]$ and $L\times [1/2,1]$  are smooth.
Furthermore, $\calL$  is   \index{polychain!half-transversal} \emph{half-transversal} to   a smooth  $p$-polychain $\calK=(K,\varphi,u, \kappa)$  in $\Omega^\circ_{12}$
 if for any face $E$ of $K$, any face~$F$ of~$L$, and any of the three sets $J= [0,1/2], [1/2,1], \{1/2\}$
the map   $$\tilde \kappa \times \tilde {\lambda}'  : E \times I \times F \times J\longrightarrow M\times M$$   is   weakly   transversal to $\diag_M $ in the sense of Section~\ref{transs}.
Then   the set
$$
D(J)= \big\{(k,s,l,t) \in K \times I \times L \times J: \tilde\kappa(k,s)=\tilde\lambda(l,t)\big\}
$$
inherits from  $K \times I \times L \times J$ a   structure of a  manifold with faces, and we have
\begin{equation}\label{321f}
D (J)  \subset K\times \Int (I) \times L\times (J\cap \Int (I)).
\end{equation}
Set 
$$
D^-=D([0,1/2]), \quad D^+=D([1/2, 1]) , \quad D^{\nicefrac{1}{2}}=D(\{1/2\}).
$$
It is clear that
$D^{\nicefrac{1}{2}}=D^- \cap D^+=\partial D^- \cap \partial D^+ $ and
$$
\dim D^- = \dim D^+=p+q+2-n, \quad \dim D^{\nicefrac{1}{2}}  =p+q+1-n.
$$

Since  $\calL$ may be non-smooth, we cannot consider the intersection polychain $\calD^{Z}(\calK,\calL)$.
  (A priori,  the  set $ D^- \cup  D^+$  does not have a  structure of a manifold  with faces.)
Instead, we turn the disjoint union $ D^-\sqcup D^+$ into a polychain
which will serve  as a  substitute  for  $\calD^{Z}(\calK,\calL)$.
The inclusion \eqref{321f} allows us to use the same construction as in Section \ref{def_D} in order to upgrade  $D^-$, $D^+$,
and $D^{\nicefrac{1}{2}}$ to polychains in  $\Omega_{32}^\circ \times \Omega_{14}^\circ \times {Z}$ denoted, respectively,
  $\calD^{-}=\calD^{-}(\calK,\calL)$, $\calD^{+}=\calD^{+}(\calK,\calL) $,  and $\calD^{\nicefrac{1}{2}}=\calD^{\nicefrac{1}{2}}(\calK,\calL) $. As can be checked from our conventions, the oriented manifold $D^{\nicefrac{1}{2}}$ has the orientation inherited from $(-1)^{p+q+1+n} \partial D^-$
or, equivalently, the orientation inherited from $(-1)^{p+q+n} \partial D^+$.
The inclusions ${D^{\nicefrac{1}{2}} \subset D^\pm}$ are compatible with the polychain structures (except for  the orientations):
they  map faces of $D^{\nicefrac{1}{2}} $ diffeomorphically onto faces of $D^\pm$, map faces  of the same type onto faces of the same type,
commute with the   identification diffeomorphisms of the faces, commute with the maps to   $\Omega_{32}^\circ \times \Omega_{14}^\circ \times {Z}$,
and  the induced maps in $\pi_0$ commute with the weights.
Also,  a face of  $D^\pm$ having the same type as the image of a face $F$ of $D^{\nicefrac{1}{2}} $
must be the image of a face   of  $D^{\nicefrac{1}{2}} $ of the same type as $F$.
These facts allow us to form  a $(p+q+2-n)$-polychain $  \calD^{h}  =  \calD^{h}  (\calK, \calL)$
in  $\Omega_{32}^\circ \times \Omega_{14}^\circ \times {Z}$ by taking the disjoint union $\calD^{-}  \sqcup \calD^{+} $
and declaring that the images of any face of $D^{\nicefrac{1}{2}}$ in $ D^-$ and  $  D^+$ have the same type
and the identification diffeomorphism  between them is the identity map.
We shall  sometimes write     
$$
\displaystyle{\calD^{-}  \mathop{\cup}_{{\nicefrac{1}{2}}} \calD^{+}}
$$ 
for this  polychain $  \calD^{h} $.

\begin{lemma}\label{half-half}
Let $\calK$ be a smooth $p$-polycycle in  $\Omega^\circ_{12}$ and
let $\calL$ be a half-smooth $q$-polycycle in  $\Omega^\circ_{34} \times {Z}$  half-transversal to $\calK$. Then $   \calD^{h} (\calK,\calL) $ is a polycycle
 in  $\Omega^\circ_{32} \times \Omega^\circ_{14} \times {Z}$ and
 $$\big[\widetilde{\Upsilon}_{12,34Z}(\lb \calK\rb, \lb\calL\rb )  \big] =
 [  \calD^{h}   (\calK,\calL) ] \in  H_{p+q+2-n}(\Omega_{32} \times \Omega_{14} \times {Z}).$$
\end{lemma}

\begin{proof}
Lemma \ref{reduction_D} directly extends to smooth polychains $\calK$, $\calK'$ in   $\Omega^\circ_{12}$
and half-smooth polychains   $\calL,\calL'$ in $\Omega^\circ_{34} \times {Z}$    half-transversal to $\calK$, $\calK'$;
 one should only replace $\calD $  by   $  \calD^{h}  $. This implies the first claim of Lemma~\ref{half-half}.

There is an arbitrarily small  deformation $\left\{\calL^t=\big(L,\psi,v,((\lambda')^t,\lambda'')\big)\right\}_{t\in I}$ of $\calL^0=\calL$ 
into a smooth polycycle $\calL^1$.  
We can assume that the restrictions of the maps $ (\widetilde \lambda')^t  :L\times [0,1] \to M$ to   $L\times [0,1/2] $
and $L\times [1/2,1] $ are smooth maps smoothly depending on $t\in I$.
As   in   the proof of  Lemma \ref{homotimplieshomol},
we derive from the deformation  $\left\{\calL^t \right\}_{t\in I}$   a $(q+1)$-polychain $\calR$ in  $\Omega_{34}^\circ \times {Z}$
such that $\partial^r \calR = \red(\calL^1) \sqcup \red(-\calL)$. The   assumptions on the deformation imply that   $\calR$ is half-smooth.
Taking the  deformation small enough, we can ensure  that $\calR$ is  half-transversal to~$\calK$.
By  the assumption $\partial^r \calK=  \varnothing  $ and the generalized version of Lemma \ref{reduction_D},   \begin{eqnarray*}
 (-1)^{n+p+1} \partial^r   \calD^{h} (\calK,\calR)  &=& \red   \calD^{h} \big(\red \calK,\red(\calL^1) \sqcup \red(-\calL)\big)\\
&=&  \red   \calD^{h} (\red \calK,\red \calL^1) \sqcup  \left(- \red   \calD^{h} (\red \calK, \red \calL)\right)\\
&=& \red   \calD^{h} ( \calK,\calL^1) \sqcup  \left(- \red   \calD^{h} ( \calK,\calL)\right).
\end{eqnarray*}
Therefore
$$
\lb   \calD^{h} (\calK,\calL) \rb= \lb   \calD^{h} (\calK,\calL^1) \rb \in \widetilde{H}_{p+q+2-n}(\Omega_{32} \times \Omega_{14} \times {Z}).
$$
Projecting to singular homology, we obtain  the equality
$$ \left[  \calD^{h}  (\calK,\calL)  \right]=
\left[  \calD^{h} (\calK,\calL^1)\right] \in {H}_{p+q+2-n}(\Omega_{32} \times \Omega_{14} \times {Z}).
$$
Since  the polycycle $\calL^1$ is smooth, the manifold with faces   underlying $  \calD^{h} (\calK,\calL^1)$ is obtained
by cutting out the manifold with faces   underlying $\calD^{Z}(\calK,\calL^1)$ along a smooth compact oriented proper submanifold of codimension 1.
This easily implies the equality
 $ \left[  \calD^{h} (\calK,\calL^1)\right]= \left[\calD^{Z}(\calK,\calL^1)\right] $. Thus,
\begin{eqnarray*}
\big[\widetilde{\Upsilon}_{12,34Z}(\lb \calK\rb, \lb\calL\rb ) \big] &=&   \big[\widetilde{\Upsilon}_{12,34Z}(\lb \calK\rb, \lb\calL^1\rb ) \big] \\
&=& \left[\calD^{Z}(\calK,\calL^1)\right] \ = \ \left[  \calD^{h} (\calK,\calL)\right].
\end{eqnarray*}

\up
\end{proof}

\subsection{A Jacobi-type identity for $\Upsilon$}

 As in Section \ref{checkUpsilon---},
the operations $\widetilde{\Upsilon}_{12,34{Z}}$, $\check{\Upsilon}_{12,34{Z}}$, $ {\Upsilon}_{12,34{Z}}$ generalize
to all     tuples      $\star_1, \star_2, \star_3,\star_4 \in \partial M$.
We  pick two extra points $\star_5, \star_6 \in \partial M$. For ${Z}=\Omega_{56}$,
the maps  $\widetilde{\Upsilon}_{12,34{Z}}$, $\check{\Upsilon}_{12,34{Z}}$, $ {\Upsilon}_{12,34{Z}}$  will be denoted respectively
by $\widetilde{\Upsilon}_{12,3456}$, $\check{\Upsilon}_{12,3456}$, $ {\Upsilon}_{12,3456}$.
Given a permutation $(i,j,k,l,m,n)$ of $(1,2,3,4,5,6)$, we can accordingly  renumber   the  points $\star_1,\dots,\star_6$
and consider  the corresponding maps   $\widetilde{\Upsilon}_{ij,klmn}$, $\check{\Upsilon}_{ij,klmn}$, $ {\Upsilon}_{ij,klmn}$.
 We now   establish  a   Jacobi-type identity for    ${\Upsilon}_{ij,klmn}$.

\begin{lemma}\label{Jacobi_Upsilon} \label{Jacobi_Upsilon_tilde}
Consider  the   permutation maps
\begin{eqnarray*}
\perm_{231}&:& \Omega_{36} \times \Omega_{52} \times \Omega_{14} \longrightarrow  \Omega_{52} \times \Omega_{14} \times \Omega_{36},
\quad (x,y,z) \longmapsto (y,z,x),\\
\perm_{312}&:& \Omega_{14} \times \Omega_{36} \times \Omega_{52}
\longrightarrow \Omega_{52} \times \Omega_{14} \times \Omega_{36}, \quad (x,y,z) \longmapsto (z,x,y).
\end{eqnarray*}
For any  $a\in H_p (\Omega_{12})$,  $b\in  H_q (\Omega_{34})$ and $c \in  H_r (\Omega_{56})$ with  $p,q,r\geq 0$,
we have   the following equality in $H_{p+q+r+4-2n} (\Omega_{52} \times \Omega_{14} \times \Omega_{36})$:
\begin{eqnarray*}
 {\Upsilon}_{12,5436} \left(a \otimes {\Upsilon}_{34,56}(b \otimes c)\right) &&  \\
+ (-1)^{(p+n) (q+r)}  (\perm_{312})_* {\Upsilon}_{34,1652} \left(b \otimes {\Upsilon}_{56,12}(c \otimes a)\right) &&\\
+ (-1)^{(p+q)(r+n)} (\perm_{231})_* {\Upsilon}_{56,3214} \left(c \otimes  {\Upsilon}_{12,34}(a \otimes b)\right)  &= &0.
\end{eqnarray*}
\end{lemma}

\begin{proof}
Set $\varepsilon=(-1)^{ n(q+1)+pr}$. The definition of $\Upsilon_{34,56}$ and \eqref{Upsilon_Upsilon_bis}  imply that
\begin{eqnarray*}
 && {\Upsilon}_{12,5436} \left(a \otimes {\Upsilon}_{34,56}(b \otimes c)\right)  \\
 &=&  {\Upsilon}_{12,5436} \left([\lb a \rb] \otimes \left[\check {\Upsilon}_{34,56}(\lb b \rb \otimes \lb c \rb )\right]\right) \\
 &=&  \left[ \check{\Upsilon}_{12,5436} \left(\lb a \rb \otimes \check {\Upsilon}_{34,56}(\lb b \rb \otimes \lb c \rb ) \right)  \right] \\
 &=&  (-1)^{(r+nq)+(  q+r+n  +np)} \big[ \widetilde{\Upsilon}_{12,5436} \big(\lb a \rb \otimes \widetilde{\Upsilon}_{34,56}(\lb b \rb \otimes \lb c \rb ) \big)  \big]\\
 &=&  \varepsilon  (-1)^{q+p(n+r) } \big[ \widetilde{\Upsilon}_{12,5436} \big(\lb a \rb \otimes \widetilde{\Upsilon}_{34,56}(\lb b \rb \otimes \lb c \rb ) \big)  \big].
\end{eqnarray*}
Using the naturality of the transformation $[-]$, we also  obtain  that
\begin{eqnarray*}
&&(-1)^{(p+n) (q+r)}  (\perm_{312})_* {\Upsilon}_{34,1652} \left(b \otimes {\Upsilon}_{56,12}(c \otimes a)\right)\\
&=& (-1)^{(p+n) (q+r)}   \left[(\perm_{312})_* \check{\Upsilon}_{34,1652} \left(\lb b \rb \otimes \check{\Upsilon}_{56,12}(\lb c \rb \otimes \lb a\rb )\right) \right] \\
&=& (-1)^{(p+n) (q+r)+(p+nr)+(r+p+n +nq) }\big[(\perm_{312})_* \widetilde{\Upsilon}_{34,1652} \big(\lb b \rb \otimes \widetilde{\Upsilon}_{56,12}(\lb c \rb \otimes \lb a\rb )\big)\big]\\
&=&  \varepsilon  (-1)^{r+q(n+p) }\big[(\perm_{312})_* \widetilde{\Upsilon}_{34,1652} \big(\lb b \rb \otimes \widetilde{\Upsilon}_{56,12}(\lb c \rb \otimes \lb a\rb )\big)\big]
\end{eqnarray*}
and
\begin{eqnarray*}
&&  (-1)^{(p+q)(r+n)}  (\perm_{231})_* {\Upsilon}_{56,3214} \left(c \otimes  {\Upsilon}_{12,34}(a \otimes b)\right)\\
& =&   (-1)^{(p+q)(r+n)}  \left[(\perm_{231})_* \check {\Upsilon}_{56,3214} \big(\lb c \rb\otimes  \check {\Upsilon}_{12,34}(\lb a \rb \otimes \lb b \rb)\big) \right] \\
& =&  (-1)^{(p+q)(r+n) + (q+np)+( p+q+n +nr)} \big[(\perm_{231})_* \widetilde{\Upsilon}_{56,3214} \big(\lb c \rb\otimes \widetilde{\Upsilon}_{12,34}(\lb a \rb \otimes \lb b \rb)\big) \big] \\
& =&    \varepsilon  (-1)^{p+r(n+q) } \big[(\perm_{231})_*  \widetilde{\Upsilon}_{56,3214} \big(\lb c \rb\otimes  \widetilde {\Upsilon}_{12,34}(\lb a \rb \otimes \lb b \rb)\big) \big].
\end{eqnarray*}
Thus, it is enough to prove the following identity in $H_{*}\left( \Omega_{52} \times \Omega_{14} \times \Omega_{36}\right)$,
where $a\in \widetilde H_p (\Omega_{12})$,  $b\in \widetilde H_q (\Omega_{34})$ and $c \in  \widetilde H_r (\Omega_{56})$ are now  any face homology classes:
\begin{eqnarray}
\label{half}  (-1)^{q+p(n+r)}\, \big[ \widetilde{\Upsilon}_{12,5436} \big( a \otimes \widetilde{\Upsilon}_{34,56}(b \otimes c)\big)\big]&&\\
\notag + (-1)^{r+q(n+p)}\, \big[(\perm_{312})_* \widetilde{\Upsilon}_{34,1652} \big(b \otimes  \widetilde{\Upsilon}_{56,12}(c \otimes a)\big)  \big]&&\\
\notag + (-1)^{p+r(n+q)}\,  \big[ (\perm_{231})_* \widetilde{\Upsilon}_{56,3214} \big(c \otimes  \widetilde{\Upsilon}_{12,34}(a \otimes b)\big) \big]
&=& 0  .
\end{eqnarray}

Slightly moving the points $\star_1,\dots,\star_6$ in $\partial M$, we can   assume that they  are pairwise distinct.
Let $\calK=(K,\varphi,u,\kappa)$ be a smooth $p$-polycycle  in $\Omega^\circ_{12}$ representing~$a$,
let $\calL=(L,\psi,v,\lambda)$ be a smooth  $q$-polycycle in $\Omega^\circ_{34}$ representing~$b$, and
let $\calN=(N,\chi,z,\eta)$ be a smooth $r$-polycycle in $\Omega^\circ_{56}$ representing~$c$.
 We will assume that $\calK,\calL,\calN$ are pairwise transversal  in the sense of Section \ref{transs}.
 This assumption   and other  transversality conditions   imposed below in the course of the proof
are always achieved by a  small deformation  of  $\calK, \calL, \calN$.

Let  $\calD_{bc}=\calD(\calL,\calN)$ be  the intersection polycycle  as  defined in Section \ref{def_D}.
  Recall that   its underlying manifold with faces, $D_{bc}$,   consists of all tuples
$(l,h,n,i)\in   L\times I \times N\times I$  such that $ \tilde \lambda(l,h) =\tilde\eta(n,i)$.
Let $(cb,bc) $ stand for the underlying continuous map  $\lambda \losange \eta: D_{bc} \to \Omega^\circ_{54} \times \Omega^\circ_{36}$ of $\calD_{bc}$.
The map     $cb=\lambda \triangleleft \eta:D_{bc} \to \Omega_{54}^\circ$   carries a point $(l,h,n,i)$  to the path $I \to M$
which runs from $\star_5$ to $ \tilde\eta(n,i)$   along $\tilde\eta(n,-)$   in the first half-time
and then runs from $\tilde \lambda(l,h)$ to $\star_4$ along $\tilde \lambda(l,-)$   in the second half-time.
(Here and below,  the time parameter   of paths   always  increases along subintervals of $I$ with  constant speed.)
The map $bc=\lambda \triangleright \eta:D_{bc} \to  \Omega_{36}^\circ$    carries
$(l,h,n,i)$  to the path $ I   \to M$ which runs from $\star_3$ to $ \tilde \lambda(l,h)$  along $\tilde \lambda(l,-)$   in the first half-time
and then runs  from $\tilde \eta(n,i)$  to  $\star_6$  along $\tilde \eta(n,-)$   in  the second half-time.
  Thus   the  paths  $cb(l,h,n,i)$ and $bc(l,h,n,i)$ are obtained from  the paths  $ \tilde\eta(n,-)$   and   $\tilde \lambda(l,-)$
by switching direction at the  intersection   point $ \tilde \lambda(l,h) =\tilde\eta(n,i)$,  see Figure~\ref{Dbc}.

 \begin{figure}[htb]
\labellist
\small\hair 2pt
 \pinlabel {$\star_5$} [r] at 2 215
 \pinlabel {$\star_3$} [r] at 2 121
 \pinlabel {$\star_1$} [r] at 2 24
 \pinlabel {$\star_6$} [l] at 195 21
 \pinlabel {$\star_4$} [l] at 195 121
 \pinlabel {$\star_2$} [l] at 195 216
 \pinlabel {$c$} [b] at 57 180
 \pinlabel {$b$} [b] at 46 101
 \pinlabel {$b$} [b] at 149 118
 \pinlabel {$c$} [b] at 143 36
\endlabellist
\centering
\includegraphics[scale=0.6]{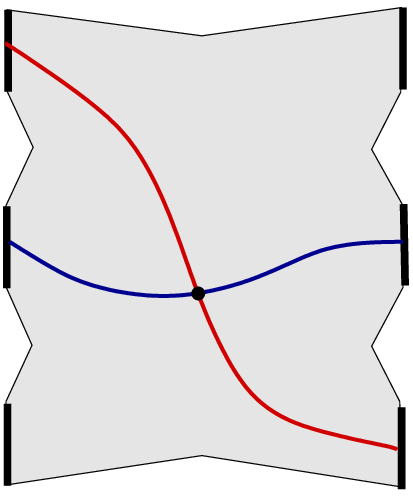}
\caption{The polycycle $\calD_{bc}$ in $\Omega_{54}^\circ \times \Omega_{36}^\circ$.}
\label{Dbc}
\end{figure}
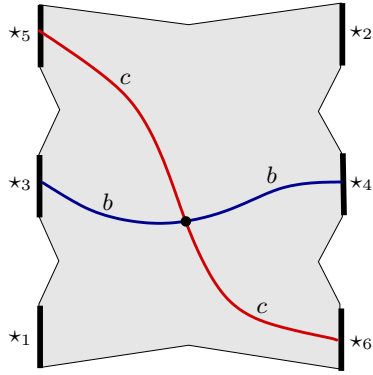

We set $I^\circ =\Int (I)=(0,1)$,  $Z= \Omega^\circ_{36}$ and view   $\calD_{bc}$ as a polycycle in $\Omega^\circ_{54} \times Z$.
It    is half-smooth in the sense of Section \ref{parar+++}. Slightly deforming  the map $\tilde \kappa:K\times I \to M$ adjoint to $\kappa$,
we can assume  $\calD_{bc}$ to be half-transversal to $\calK$  in the sense of Section~\ref{parar+++}.
  In the sequel, we  consider  the associated $(p+q+r+4-2n)$-polychains $\calD_{abc}^{-} =   \calD^{-}  (\calK,\calD_{bc} )$
and $\calD_{abc}^{+} =   \calD^{+}  (\calK,\calD_{bc} )$  in   $\Omega^\circ_{52} \times \Omega^\circ_{14} \times Z $.

  On the one hand,   the  manifold with faces  $D_{abc}^{-}$ underlying   the polychain   $\calD_{abc}^{-}$ consists of all tuples
\begin{equation}\label{firstpoint}
(k,s,l,h,n,i,t) \in \, K \times I^\circ \times  L \times I^\circ \times N \times I^\circ   \times (0,1/2]
\end{equation}
such that $\tilde \lambda(l,h) = \tilde \eta(n,i)$ and  $\tilde \kappa(k,s) = \tilde\eta(n,i*t)$.  The map
\begin{equation}\label{firstmap}
(ca,a(cb),bc): D_{abc}^{-}  \longrightarrow  \Omega^\circ_{52} \times \Omega^\circ_{14} \times Z
=\Omega^\circ_{52} \times \Omega^\circ_{14} \times \Omega^\circ_{36}
\end{equation}
underlying $\calD_{abc}^{-}$  is schematically shown in Figure \ref{Dabc} where one   switches direction at the dotted intersections.
The first coordinate $ca:D_{abc}^{-} \to \Omega^\circ_{52}$  sends any point \eqref{firstpoint} to the path $ I    \to M$
which goes from $\star_5$ to $\tilde\eta(n,i\ast t)$ along $\tilde\eta(n,-)$ in half-time
and, next, goes from $\tilde \kappa(k,s)$ to $\star_2$  along $\tilde \kappa(k,-)$ in half-time.
The map $a(cb):D_{abc}^{-}\to \Omega^\circ_{14}$
carries  a  point \eqref{firstpoint} to the path $I \to M$ which goes from $\star_1$ to $\tilde\kappa(k,s)$ along $\tilde\kappa(k,-)$
in half-time, next, goes from $\tilde\eta(n,i\ast t)$ to $\tilde\eta(n,i)$ along $\tilde\eta(n,-)$ in time $\big[\frac{1}{2},1-\frac{1}{4(1-t)}\big]$
and, finally, goes from $\tilde \lambda(l,h)$ to $\star_4$  along $\tilde \lambda(l,-)$ in time $\big[1-\frac{1}{4(1-t)},1\big]$.
The map  $bc:  D_{abc}^{-} \to  \Omega^\circ_{36}$  sends a  point \eqref{firstpoint} to the path $  I    \to M$
which goes from $\star_3$ to $\tilde\lambda(l,h)=\tilde \eta (n,i)$ along $\tilde\lambda(l,-)$ in half-time
and, next, goes from $\tilde \eta (n,i)$ to $\star_6$  along $\tilde \eta(n,-)$ in half-time.

 On the other hand,  the   manifold with faces  $D_{abc}^{+}$ underlying  the polychain  $\calD_{abc}^{+}$ consists of all tuples
\begin{equation}\label{secondpoint}
(k,s,l,h,n,i,t) \in \, K \times I^\circ \times  L \times I^\circ \times N \times I^\circ   \times [1/2,1)
\end{equation}
such that $\tilde \lambda(l,h) = \tilde \eta(n,i)$ and $\tilde \kappa(k,s) =  \tilde\lambda(l,h*t)$. The map
\begin{equation}\label{firstmap+}
((cb)a, ab,bc): D_{abc}^{+}  \longrightarrow  \Omega^\circ_{52} \times \Omega^\circ_{14} \times \Omega^\circ_{36}
\end{equation}
is computed   similarly to \eqref{firstmap}    and is schematically shown in  Figure \ref{Dabc}.
We only note that the map $(cb)a : D_{abc}^{+} \to \Omega^\circ_{52} $
carries  a  point \eqref{secondpoint}  to the path $I \to M$ which goes from $\star_5$ to $\tilde\eta(n,i)$ along $\tilde\eta(n,-)$
in  time   $\left[0, \frac{1}{4t}\right]$,  next, goes from $\tilde\lambda(l,h)$ to $\tilde\lambda(l,h\ast t)$
along $\tilde\lambda(l,-)$ in time   $\left[ \frac{1}{4t}, \frac{1}{2}\right]$
and, finally, goes from $\tilde \kappa(k,s)$ to $\star_2$  along $\tilde \kappa(k,-)$ in the remaining half-time.

Consider also the polychain $\calD^{\nicefrac{1}{2}}_{abc}=    \calD^{\nicefrac{1}{2}}   \big(\calK,\calD_{bc}\big)$ in  $\Omega^\circ_{52} \times \Omega^\circ_{14} \times Z $.
Its  underlying   $(p+q+r+3-2n)$-manifold with faces $D^{\nicefrac{1}{2}}_{abc}=D_{abc}^{-} \cap D_{abc}^{+} $
consists of the  tuples $(k,s,l,h,n,i,1/2)$ such that $\tilde \kappa (k,s)=\tilde \lambda(l,h) = \tilde \eta(n,i)$.
The underlying map
$$
(ca, ab, cb):D^{\nicefrac{1}{2}}_{abc}  \longrightarrow  \Omega^\circ_{52} \times \Omega^\circ_{14} \times \Omega^\circ_{36}
$$
is the restriction of  the maps \eqref{firstmap} and  \eqref{firstmap+}, see  Figure \ref{Dabc}.

\begin{figure}[h]
\begin{center}
\labellist \small
\pinlabel {$\star_5$} [r] at 5 213
\pinlabel {$\star_5$} [r] at 292 218
\pinlabel {$\star_5$} [r] at 586 221
\pinlabel {$\star_3$} [r] at 5 120
\pinlabel {$\star_3$} [r] at 292 123
\pinlabel {$\star_3$} [r] at 586 122
\pinlabel {$\star_1$} [r] at 5 22
\pinlabel {$\star_1$} [r] at 292 25
\pinlabel {$\star_1$} [r] at 586 24
\pinlabel {$\star_2$} [l] at 193 215
\pinlabel {$\star_4$} [l] at 193 119
\pinlabel {$\star_6$} [l] at 192 19
\pinlabel {$\star_2$} [l] at 480 215
\pinlabel {$\star_4$} [l] at 480 119
\pinlabel {$\star_6$} [l] at 480 19
\pinlabel {$\star_2$} [l] at 773 215
\pinlabel {$\star_4$} [l] at 773 119
\pinlabel {$\star_6$} [l] at 773 21
\pinlabel {$c$} [b] at 29 199
\pinlabel {$a$} [b] at 128 193
\pinlabel {$a$} [b] at 25 42
\pinlabel {$c$} [bl] at 85 125
\pinlabel {$b$} [b] at 140 110
\pinlabel {$c$} [b] at 142 38
\pinlabel {$b$} [b] at 20 112
\pinlabel {$c$} [b] at 344 180
\pinlabel {$c$} [b] at 435 36
\pinlabel {$b$} [b] at 323 105
\pinlabel {$b$} [b] at 438 115
\pinlabel {$a$} [b] at 327 50
\pinlabel {$a$} [b] at 424 202
\pinlabel {$c$} [b] at 633 185
\pinlabel {$c$} [b] at 730 36
\pinlabel {$b$} [b] at 629 100
\pinlabel {$b$} [b] at 699 106
\pinlabel {$b$} [b] at 749 122
\pinlabel {$a$} [b] at 627 35
\pinlabel {$a$} [b] at 731 190
\endlabellist
\includegraphics[scale=0.44]{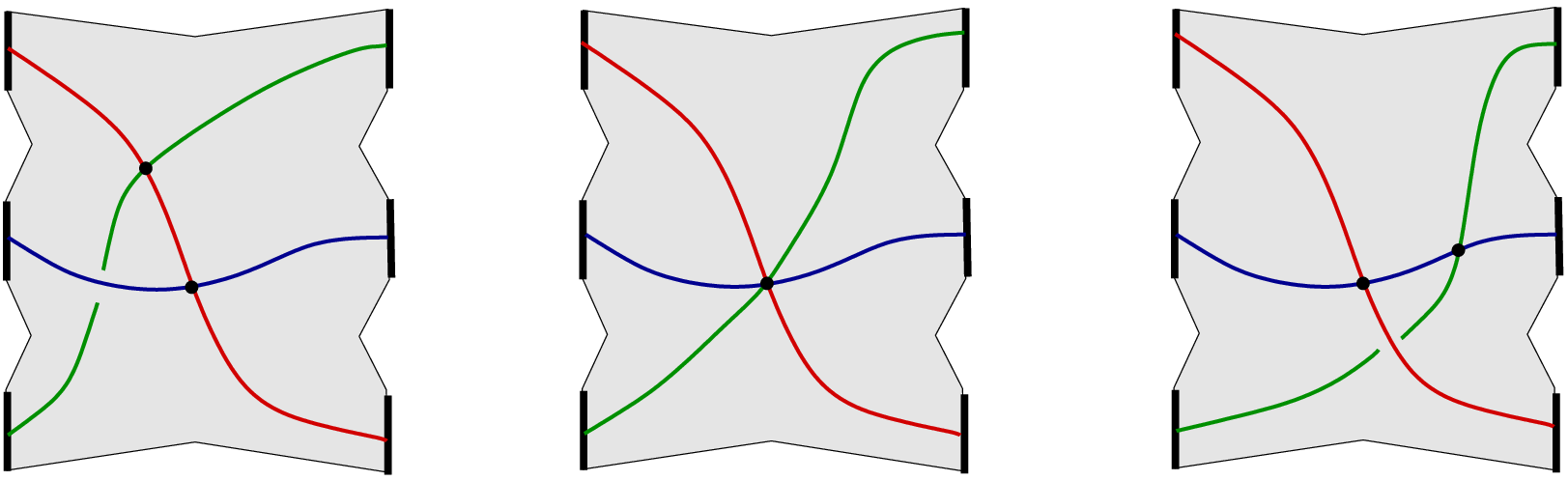}
\end{center}
\caption{The polychains $\calD_{abc}^{-}$, $\calD^{\nicefrac{1}{2}}_{abc}$ and $\calD_{abc}^{+}$.}
\label{Dabc}
\end{figure}

  Cyclically permuting $a,b,c$, we similarly obtain polychains
$\calD_{bca}^{-}$, $\calD^{\nicefrac{1}{2}}_{bca}, \calD_{bca}^{+}$ and $\calD_{cab}^{-}$, $\calD^{\nicefrac{1}{2}}_{cab}, \calD_{cab}^{+}$.
Lemma \ref{half-half} allows us to rewrite \eqref{half} as the identity  
\begin{eqnarray}
\notag  (-1)^{q+p(n+r)}\, \big[  \calD_{abc}^{-} \mathop{\cup}_{\nicefrac{1}{2}}   \calD_{abc}^{+} \big]
  + (-1)^{r+q(n+p)}\, (\perm_{312})_*  \big[  \calD_{bca}^{-}\mathop{\cup}_{\nicefrac{1}{2}}   \calD_{bca}^{+} \big]&&\\
\label{abc...}  + (-1)^{p+r(n+q)}\, (\perm_{231})_*   \big[  \calD_{cab}^{-} \mathop{\cup}_{\nicefrac{1}{2}}  \calD_{cab}^{+} \big]  & =  & 0
\end{eqnarray}
 in $H_*(\Omega_{52} \times \Omega_{14} \times \Omega_{36})$.
 The idea of the proof is to show that the six polychains on the left-hand side of \eqref{abc...} cancel each other pairwise.

  We first explain how to relate the  polychains $\calD^-_{abc}$ and $\calD_{bca}^{+}$.
  Observe that   the   manifold with faces  $D_{bca}^{+}$ underlying $\calD_{bca}^{+}$ consists of all tuples
\begin{equation}\label{secondpoint+}
(l,h, n,  i' , k,s,  t' ) \in \, L \times I^\circ \times  N \times I^\circ \times K \times I^\circ   \times   [1/2,1)
\end{equation}
such that   $\tilde \eta(n, i' )  =  \tilde \kappa(k,s)$   and $\tilde \lambda(l,h) =  \tilde\eta(n, i'*t' )$.
We define a smooth map
$$
F: K \times I^\circ \times  L \times I^\circ \times N \times I^\circ   \times (0,1/2] \to  L \times I^\circ \times  N \times I^\circ \times K \times I^\circ   \times [1/2,1)
$$
by the formula $$F(k,s,l,h,n,i,t)=(l,h, n, i'(i,t), k,s, t'(i,t))$$
where  $i':I^\circ \times (0,1/2] \to I^\circ$ and $t':I^\circ \times (0,1/2] \to [1/2,1)$ are given by
\begin{equation}\label{i'_t'}
i'(i,t)= 2it \quad \hbox{and}  \quad t'  (i,t) =1-\frac{1-i}{2-4it}.
\end{equation}
Observe that the functions $i',t'$  satisfy   the  equations $i'=i *t$ and $i=i'*t'$.
It easily follows  that the transformation $(i',t'): I^\circ \times (0,1/2] \to I^\circ \times [1/2,1)$ is a diffeomorphism,
so that $F$ is a diffeomorphism carrying $D_{abc}^{-}$ onto $D_{bca}^{+}$.
The resulting diffeomorphism
$$
F_{abc}: D_{abc}^{-}  \longrightarrow  D_{bca}^{+}
$$
is compatible with the partitions  and the weights of  the polychains  $\calD_{abc}^{-} $, $\calD_{bca}^{+}$. Moreover,
\begin{equation}\label{mapF}
({ca,a(cb),bc}) =   \perm_{312}   ((ac)b,bc,ca ) F_{abc}: D_{abc}^{-}  \longrightarrow  \Omega_{52}^\circ  \times \Omega_{14}^\circ  \times \Omega_{36}^\circ
\end{equation}
up to homotopy of the second coordinate map compatible with the partitions.
The map $F_{abc}$ carries  $ D^{\nicefrac{1}{2}}_{abc}\subset D_{abc}^{-} $   diffeomorphically onto $ D^{\nicefrac{1}{2}}_{bca}\subset D_{bca}^{+} $
via the permutation $(k,s,l,h,n,i,1/2)\mapsto (l,h, n, i , k,s, 1/2)$ and  \eqref{mapF} holds on  $ D^{\nicefrac{1}{2}}_{abc}$
as an equality of maps (no homotopy needed).
One easily constructs a homotopy of the map $a(cb) : D^-_{abc} \to \Omega^\circ_{14}$ into $((ac)b)\circ F_{abc}$ constant on $D^{\nicefrac{1}{2}}_{abc}$.
Since the left-hand side of \eqref{abc...} is preserved under such a homotopy of  $a(cb)$, we can assume that \eqref{mapF} is an equality of maps.

We prove now that \begin{equation} \label{deg_F_abc}
\deg F_{abc} = (-1)^{1+ pn+qn+ (p+1)(q+r)}.
\end{equation}
The  diffeomorphism $ F_{abc}$ carries  the open subset
$  R^-  =D_{abc}^{-} \setminus D^{\nicefrac{1}{2}}_{abc}$ of $D_{abc}^{-}$ onto the open subset $  R^+  =D_{bca}^{+}  \setminus D^{\nicefrac{1}{2}}_{bca}$ of $D_{bca}^{+} $,
and $\deg  F_{abc}$ is equal  to the degree   of the restricted diffeomorphism $  R^-   \to    R^+  $. Clearly,
$$  R^-  = D^-_{abc}\cap X^-_{abc} \quad {\rm{where}} \quad X^-_{abc}=K \times I^\circ \times  L \times I^\circ \times N \times I^\circ   \times (0,1/2)  $$
and
$$  R^+   = D^+_{bca} \cap X^+_{bca}  \quad {\rm{where}} \quad
X^+_{bca} = L \times I^\circ \times  N \times I^\circ \times K \times I^\circ   \times (1/2,1).
$$
Consider  the maps
$$
X^-_{abc}    \stackrel{G^-}{\longrightarrow}M^4, \,\,
(k,s,l,h,n,i,t)  \longmapsto   (\tilde\kappa(k,s),\tilde\eta(n,i*t),\tilde\lambda(l,h),\tilde\eta(n,i) )
$$
and
$$
 X^+_{bca}   \stackrel{G^+}{\longrightarrow}   M^4, \,\,
(l,h,n,i',k,s,t')  \longmapsto     (\tilde\kappa(k,s),\tilde\eta(n,i'),\tilde\lambda(l,h),\tilde\eta(n,i'*t')  ).
$$
Since $\calN$ is transversal to both $\calK$ and $\calL$,
the map $G^-$ is transversal to $\diag_M \times \diag_M$ in the following sense:
for any faces $A,B,C$ of $K,L,N$ respectively, the restriction of $G^-$ to the interior of $A \times I^\circ \times  B  \times I^\circ \times C \times I^\circ   \times (0,1/2)$
is transversal to the interior of $\diag_M \times \diag_M$ (in the usual sense of differential topology).
Similarly,  the map $G^+$ is transversal to $\diag_M \times \diag_M$.
Observe that  $  G^-=   G^+ F\vert_{{X^-_{abc}}}$    and that  the inverse images of $\diag_M \times \diag_M$
under the maps   $G^-,G^+$ are, respectively, the sets $  R^-  ,   R^+  $.
We identify
\begin{equation}\label{normal_bundles}
\bignu_{{M^4}}(\diag_M \times \diag_M) = \pr_{12}^*\bignu_{{M^2}}(\diag_M) \oplus  \pr_{34}^*\bignu_{{M^2}}(\diag_M)
\end{equation}
where   $\pr_{ij}: M^4 \to M^2$ is the cartesian projection defined by $\pr_{ij}(m_1,m_2,m_3,m_4)=(m_i,m_j)$.
As above,    $\bignu_{{M^2}}(\diag_M)$ carries  the orientation induced by that of $\diag_M \approx M$ using our orientation convention,
and we give to  \eqref{normal_bundles} the product orientation.
 Pulling back the latter  orientation along $G^-$, we obtain an orientation on the normal bundle of $  R^-  $   in $X^-_{abc}$;
this oriented normal vector bundle   is   denoted by   $\bignu^-$. The   normal bundle of $  R^+  $ in $X^+_{bca}$
is oriented similarly and denoted by~$\bignu^+$.
 Let $T^-$ be  the tangent  bundle of  $  R^-  $  with the orientation induced by that of   $\bignu^-$. Similarly,
let $T^+$ be  the tangent  bundle of $  R^+  $  with the orientation induced by  that of   $\bignu^+$.
Clearly, the diffeomorphism $(i',t'): I^\circ \times (0,1/2] \to I^\circ \times [1/2,1)$ defined by \eqref{i'_t'} is orientation-reversing. Hence
  $\deg F = (-1)^{1+(p+1)(q+r)}$, and since $F$ carries $  R^-   $ onto $  R^+  $ and induces an orientation-preserving map $\bignu^-\to \bignu^+$, we have
\begin{equation}\label{T+-}
F_{abc}^*(T^+) = (-1)^{1+(p+1)(q+r)}\,  T^-.
\end{equation}
Next,  consider the following isomorphisms of oriented vector bundles over $  R^-  $,
where $T$ stands for the tangent  bundle,  $\bignu$ stands for the normal bundle,  and $\pr$ denotes the appropriate cartesian projection:
\begin{eqnarray*}
&&\!\!\!  T( K \times I \times L \times I \times N \times I \times I)\vert_{{  R^-  }} \\
&\cong &\!\!\! \pr^* T( K \times I)\vert_{{  R^-  }}  \oplus \pr^* T( L \times I \times N \times I)\vert_{{  R^-  }}  \oplus \pr^* T  (I)\vert_{{  R^-  }}\\
&\cong &\!\!\!   \pr^* T( K \times I)\vert_{{  R^-  }}  \oplus \pr^*\!\big( \bignu_{L \times I \times N \times I} {(D_{bc})}
\oplus  T{(D_{bc})} \big)\big\vert_{{  R^-  }}  \oplus \pr^* T  (I)\vert_{{  R^-  }}  \\
&\cong &\!\!\!  (-1)^{n(p+1)}\, \pr^*\! \bignu_{L \times I \times N \times I} ({D_{bc}} )\vert_{{  R^-  }} \oplus  \pr^*  T( K \times I \times {D_{bc}} \times I)\vert_{{  R^-  }}  \\
&\cong & \!\!\!  (-1)^{n(p+1)}\,  \pr^*\! \bignu_{L \times I \times N \times I} ({D_{bc}} )\vert_{{  R^-  }}
 \oplus  \bignu_{K \times I \times {D_{bc}} \times I} {(  R^-  )} \oplus T{(  R^-  )}\\
&\cong & \!\!\!  (-1)^{np}\,   \underbrace{\bignu_{K \times I \times {D_{bc}} \times I} {(  R^-  )}
\oplus \pr^*\! \bignu_{L \times I \times N \times I} ({D_{bc}} )\vert_{{  R^-  }}}_{\bignu^-}\oplus\, T{(  R^-  )}  .
\end{eqnarray*}
It follows that $T^- = (-1)^{np}\, T(  R^-  )$. Similarly,
\begin{eqnarray*}
&&\!\!\!  T( L \times I \times N \times I \times K \times I \times I)\vert_{{  R^+  }} \\
&\cong &\!\!\! \pr^* T( L \times I)\vert_{{  R^+  }}  \oplus \pr^* T( N \times I \times K \times I)\vert_{{  R^+  }}  \oplus \pr^* T  (I)\vert_{{  R^+  }}\\
&\cong &\!\!\!   \pr^* T( L \times I)\vert_{{  R^+  }}  \oplus \pr^*\! \big( \bignu_{N \times I \times K \times I} {(D_{ca})}
\oplus  T{(D_{ca})} \big)\big\vert_{{  R^+  }}  \oplus \pr^* T  (I)\vert_{{  R^+  }}  \\
&\cong &\!\!\!  (-1)^{n(q+1)}\, \pr^*\! \bignu_{N \times I \times K \times I} ({D_{ca}} )\vert_{{  R^+  }}
\oplus  \pr^*  T( L \times I \times {D_{ca}} \times I)\vert_{{  R^+  }}  \\
&\cong & \!\!\!  (-1)^{n(q+1)}\,  \underbrace{\pr^*\! \bignu_{N \times I \times K \times I} ({D_{ca}} )\vert_{{  R^+  }}
 \oplus  \bignu_{L \times I \times {D_{ca}} \times I} {(  R^+  )}}_{(-1)^n \bignu^+ } \oplus\, T{(  R^+  )}  .
\end{eqnarray*}
Here the sign $(-1)^n$  accompanying $\bignu^+$ is  the degree of the  permutation map    ${ M^2\to M^2}, (m_1,m_2)\mapsto  (m_2,m_1)$.
It follows that $T^+ = (-1)^{nq}\,T(  R^+ )$. Formula   \eqref{T+-}  and the  computations of $T^+, T^-$  imply \eqref{deg_F_abc}.

Cyclically permuting $a,b,c$, we   obtain diffeomorphisms $F_{bca}:D^-_{bca} \to D^+_{cab}$ and $F_{cab}:D^-_{cab} \to D^+_{abc}$ such that
\begin{equation} \label{deg_F_bca}
\deg F_{bca} = (-1)^{1+ qn+rn+ (q+1)(r+p)}
\end{equation}
and
\begin{equation} \label{deg_F_cab}
\deg F_{cab} = (-1)^{1+ rn+pn+ (r+1)(p+q)}.
\end{equation}

To conclude the proof, we set
$$
D^{\pm}= D^{\pm}_{abc} \sqcup  D^{\pm}_{bca} \sqcup  D^{\pm}_{cab}
\quad {\rm {and}} \quad  D^{\nicefrac{1}{2}}=D^{+} \cap D^{-}=D^{\nicefrac{1}{2}}_{abc} \sqcup  D^{\nicefrac{1}{2}}_{bca} \sqcup  D^{\nicefrac{1}{2}}_{cab} .
$$
Clearly,  $F_{cab}F_{bca}F_{abc}=\id$  on $D^{\nicefrac{1}{2}}_{abc}$.
Therefore any triangulation  of  $D^{\nicefrac{1}{2}}_{abc}$  extends uniquely to  a triangulation, $ T^{\nicefrac{1}{2}} $,  of
 $ D^{\nicefrac{1}{2}}$ invariant under  $F_{abc} \sqcup F_{bca}  \sqcup F_{cab}$.
  (All triangulations in this argument are supposed to be locally ordered and to  fit the given partitions, cf$.$  Sections~\ref{Partitions} and~\ref{fundcla}.)
Subdividing, if necessary,  $ T^{\nicefrac{1}{2}} $ we can   assume that it  extends to a triangulation, $T^-$, of $ D^{-} $.
Transferring $T^-$ along the diffeomorphism $F_{abc} \sqcup F_{bca} \sqcup F_{cab}:  D^{-}  \to  D^{+}$
we obtain a triangulation, $T^+$, of $D^+$ also extending $ T^{\nicefrac{1}{2}} $.
We use   the triangulations $T^-$ and $T^+$   to represent the left-hand side of \eqref{abc...} by a $(p+q+r+4-2n)$-dimensional  singular chain.
  According to  \eqref{deg_F_abc}, \eqref{deg_F_bca} and \eqref{deg_F_cab},
every  singular simplex  contributed by a   top-dimensional   simplex of  $T^-$ cancels with the corresponding singular simplex in~$T^+$.
Therefore the singular chain in question is equal to zero and so is the left-hand side of \eqref{abc...}.
\end{proof}

\subsection{Proof of Theorem \ref{geometricbibracket} (the end)}

  Let $\triple{-}{-}{-}\in \End(A^{\otimes 3})$ be the tribracket induced by the intersection bibracket $\double{-}{-}$ in $A=A(\calC)$.
Pick any points $\star_1,\dots, \star_6\in \partial M$ and any homology classes  $a\in H_p (\Omega_{12})$,
$b\in  H_q (\Omega_{34})$ and $c \in  H_r (\Omega_{56})$.
We need to show that the tensor
\begin{eqnarray}
\label{expanded_tribracket} \triple{a}{b}{c} &=& \double{a}{\double{b}{c}'}\otimes \double{b}{c}'' \\
\notag &&+  (-1)^{(p+n) (q+r)} \Perm_{312} \left (\double{b}{\double{c}{a}'} \otimes \double{c}{a}''\right) \\
\notag && + (-1)^{(p+q)(r+n)} \Perm_{231} \left (\double{c}{\double{a}{b}'} \otimes \double{a}{b}''\right)
\end{eqnarray}
vanishes, where    $\Perm_{312}, \Perm_{231}\in \End(  A^{\otimes 3})$
are the graded permutations defined in Section \ref{conventions}.
For any $i,j,k,l,u,v\in \{1,\dots, 6\}$, let
$$
\varpi_{ij,kl}: H_*(\Omega_{ij}) \otimes H_*(\Omega_{kl})  \longrightarrow H_*(\Omega_{ij} \times \Omega_{kl}),
$$
$$
\varpi_{ij,kl,uv}: H_*(\Omega_{ij}) \otimes H_*(\Omega_{kl}) \otimes H_*(\Omega_{uv}) \longrightarrow H_*(\Omega_{ij} \times \Omega_{kl} \times \Omega_{uv})
$$
be  the linear maps induced by the cross product.
By   definition of the   intersection bibracket     and Lemma \ref{Upsilon_cross},
\begin{eqnarray*}
\varpi_{52,14,36}\left( \double{a}{\double{b}{c}'}\otimes \double{b}{c}'' \right)
&=&   \varpi_{52,14} \left ( \double{a}{\double{b}{c}'} \right ) \times \double{b}{c}''  \\
&=&   \Upsilon_{12,54}\left(a\otimes{\double{b}{c}'} \right) \times \double{b}{c}''  \\
&=&  \Upsilon_{12,5436}\big(a\otimes  \left({\double{b}{c}'} \times \double{b}{c}''\right) \big) \\
&=&  \Upsilon_{12,5436}\left(a\otimes \varpi_{54,36}\left({\double{b}{c} }  \right) \right)  \\
&=&  \Upsilon_{12,5436}\left(a\otimes \Upsilon_{34,56}(b\otimes c) \right).
\end{eqnarray*}
Cyclically permuting $a,b,c$, we also obtain
\begin{eqnarray*}
&& \varpi_{52,14,36} \Perm_{312} \left (\double{b}{\double{c}{a}'} \otimes \double{c}{a}''\right) \\
&=& (\perm_{312})_*  \varpi_{14,36,52}  \left (\double{b}{\double{c}{a}'} \otimes \double{c}{a}''\right) \\
&=& (\perm_{312})_*  \Upsilon_{34,1652}\left(b \otimes \Upsilon_{56,12}(c\otimes a) \right)
\end{eqnarray*}
and
\begin{eqnarray*}
&& \varpi_{52,14,36} \Perm_{231} \left (\double{c}{\double{a}{b}'} \otimes \double{a}{b}''\right) \\
&=& (\perm_{231})_*  \varpi_{36,52,14}  \left (\double{c}{\double{a}{b}'} \otimes \double{a}{b}''\right) \\
&=& (\perm_{231})_*  \Upsilon_{56,3214}\left(c \otimes \Upsilon_{12,34}(a\otimes b) \right).
\end{eqnarray*}
Combining the last three identities, formula  \eqref{expanded_tribracket} and Lemma \ref{Jacobi_Upsilon},
we obtain that   $\varpi_{52,14,36} \left ( \triple{a}{b}{c} \right )=0$.  We conclude that $\triple{a}{b}{c}=0$.

\section{Computations and examples}  \label{sec-examples}

We   compute $\Upsilon$ for spherical homology classes   of complementary dimensions and for $0$-dimensional homology classes.
We  use these results to determine   the intersection bibracket in two examples.

\subsection{Intersection of spheres} \label{Reidemeister_pairing}

Assume that  $n=\dim(M) \geq 4$.
We    compute the operation  $\Upsilon$   on the   loop   homology classes  arising from spheres of complementary dimensions.
Let us fix a base point $s_k$ in the $k$-sphere $S^k$ for every $k\geq 1$.  For $x\in \partial M$, we let $ \pi_k(M,x) = \big[(S^k,s_k),(M,x)\big]$  be the $k$-th homotopy group of $M$ at~$x$.
For $x,y \in \partial M$,  we set $\pi_{1}(M,x,y) =  \pi_0 (\Omega(M,x,y))$.

Consider   base points $\star,\star'$ in $\partial M$ and   integers $p,q\geq 2$ such that $p+q=n=\dim(M)$. Let $\Upsilon^\pi_{\star,\star'}$ be the  following  composition:
$$
{\small \xymatrix {
\pi_p(M,\star)  \times    \pi_q(M,\star') \ar[d]_-{\overline \partial_p  \times   \overline \partial_q } \ar@/^0.5cm/@{-->}[drr]^-{\Upsilon^\pi_{\star,\star'}} &&\\
H_{p-1}(\Omega_\star) \otimes H_{q-1}(\Omega_{\star'})  \ar[r]^-\Upsilon  & H_0(\Omega_{\star'\!\star} \times \Omega_{\star\star'})  \simeq  \hspace{-1cm} & \ \kk[\pi_1(M,\star',\star)] \otimes  \kk[\pi_1(M,\star,\star')].
}}
$$
Here  $\Omega_{\star}=\Omega(M,\star,\star)$, $\Omega_{\star\star'}=\Omega(M,\star,\star')$,
$\overline \partial_*:\pi_*(M,\star)  \to     H_{*-1}(\Omega_\star)$ is the connecting homomorphism of Section~\ref{The {Pontryagin} algebra},
and similar notation  applies with $\star$ and $\star'$ exchanged.
The following lemma computes $\Upsilon^\pi_{\star,\star'}$ when $\star \neq \star'$.

\begin{lemma}\label{intersecting_spheres}
  Assume   $\star\neq \star'$.   Let $\alpha: (S^p,s_p) \to (M,\star)$ and $\beta:(S^q,s_q)\to (M,\star')$ be    continuous maps
such that $\alpha^{-1}(\partial M)=\{s_p \}$, $\beta^{-1}(\partial M)=\{s_q \}$
and $\alpha\vert_{S^p\setminus \{s_p\}}$, $\beta\vert_{S^q \setminus \{s_q\}}$ are transversal smooth maps.
Then
\begin{equation}\label{spheres}
\Upsilon^\pi_{\star,\star'}([\alpha],[\beta]) =
(-1)^{n(p+1)+1} \sum_{(x,y)} \varepsilon(x,y) \, [\beta_{ y} \alpha_{x  }^{-1}] \otimes  [\alpha_{  x} \beta_{y  }^{-1} ].
\end{equation}
Here:    the sum runs over all $(x,y)\in S^p \times S^q$ such that   $\alpha(x)=\beta(y)$;
 $\varepsilon(x,y)$ is the sign of the product orientation in $\alpha_*(T_x S^p) \oplus \beta_*(T_y S^q)=T_{\alpha (x)} M$ with respect to the orientation of $M$;
$\alpha_{  x}$  is the composition of  $\alpha$ with  a  path from  $s_p$ to $x$ in $S^p$
and $\beta_{ y}$  is the composition of $\beta$ with a  path from $s_q$ to $y$  in $S^q$.
\end{lemma}

\begin{proof}
For $k\geq 1$, let $h_k:I^k \to S^k$ be  a continuous map such that $h_k ( \partial I^k)=\{s_k \}$,
  $h_k\vert_{\Int(I^k)}$ is smooth   and the quotient map $\bar h_k: I^k/\partial I^k\to S^k$ is a degree~1 homeomorphism.
Then $\alpha h_p: I^p= I^{p-1} \times I \to M$ is   adjoint to a   continuous   map ${\omega_\alpha}: I^{p-1} \to  \Omega_\star $
which carries $\partial I^{p-1}$ to the constant path $e_\star$.
Let  ${\bar \omega_\alpha}:I^{p-1}/\partial I^{p-1} \to \Omega_\star$ be the quotient map. Then
$$
\overline \partial_{p}\left([\alpha]\right)
 = (\bar \omega_\alpha)_*\big(\big[I^{p-1}/\partial I^{p-1}\big]\big) = \left[\calK \right]
$$
where $\calK=\left(I^{p-1},\theta_{p-1},1, {\omega_\alpha}\right)$   is the  polycycle in $\Omega_\star$  with weight $1$ and with partition $\theta_{p-1}$
defined as   the product of $ p-1 $ copies of the partition of $I$  identifying $\{0\}$ to $\{1\}$.
Similarly,  $  \overline \partial_{q}\left([ \beta  ]\right) =\left[\calL\right]$
for  $\calL=\left(I^{q-1},\theta_{q-1},1, {\omega_{\beta}}\right)$.

 The polycycles $\calK$ and $\calL$ are admissible in the sense of Section \ref{computationoftildeupsilon}
where  $\star_1=\star_2=\star$, $\star_3=\star_4=\star'$, $U= \Int(I^{p-1}) \times \Int(I)$ and $V= \Int(I^{q-1}) \times \Int(I)$.
 We can therefore consider the intersection polychain $\calD(\calK,\calL)$
and  by Lemma \ref{newcomputation},  it represents $\widetilde\Upsilon(\lb \calK \rb \otimes \lb \calL \rb)$.
Then, using  Lemma \ref{diag-imp}, we get
\begin{eqnarray*}
 \Upsilon^\pi_{\star,\star'}([\alpha],[  \beta  ]) \ = \ \Upsilon([\calK] ,[\calL] ) &=&  (-1)^{(q-1)+n(p-1)} \big[\widetilde \Upsilon\left(\, \lb \calK \rb \otimes \lb \calL \rb\, \right) \big]\\
 &=&  (-1)^{q+n(p+1)+1} \big[ \calD( \calK, \calL) \big] .
\end{eqnarray*}
The intersection polycycle $\calD( \calK,   \calL  )$ is 0-dimensional,
and its points bijectively correspond to the pairs $(x,y)\in S^p \times S^q$ such that   $\alpha(x)=\beta(y)$.   Such a    pair $(x,y)$  contributes
$$
\tilde\varepsilon(x,y)\,    \left(\beta_{ y} \bar \alpha_{x  } ,  \alpha_{  x} \bar\beta_{y }\right)    \in \Omega_{\star'\!\star} \times \Omega_{\star \star'}
$$
to $\calD( \calK,  \calL)$ where  $\alpha_x,\beta_y$ are paths
as in the statement of the lemma,  $\bar\alpha_{  x}$  is the composition of  $\alpha$ with  a  path from  $x$ to $s_p$ in $S^p$,
and  $\bar\beta_{ y}$  is the composition of $\beta$ with a  path from $y$ to $s_q$  in $S^q$.
Here $\tilde\varepsilon(x,y)$ is the sign of the linear isomorphism
\begin{equation*}\label{alpha_beta}
\xymatrix{
 T_{(x,y)} \left(S^p \times S^q\right) \ar[r]^-{(\alpha  \times \beta)_*} & T_{(z,z)} \left(M \times M \right)\ar[r]&
  \frac{T_{(z,z)} (M \times M)}{T_{(z,z)} \diag_M} = \nu_{M \times M}(\diag_M)_{(z,z)},
}
\end{equation*}
where $z=\alpha(x)=\beta(y)$, $T_{(x,y)} (S^p \times S^q)= T_x S^p \oplus T_yS^q $ has the product orientation
and $\nu_{M \times M}(\diag_M)$ has the orientation induced  from  that   of $\diag_M \approx M$.
The linear map $T_{(z,z)}(M \times M) = T_zM\oplus T_zM \to T_zM$ defined by $(u,v)\mapsto u-v$
induces an orientation-preserving isomorphism  $\nu_{M \times M}(\diag_M)_{(z,z)} \to T_zM$.
Composing   with the linear isomorphism above,   we obtain
the   map   $\alpha_* \oplus (-\beta_*): T_x S^p \oplus T_yS^q \to T_{z}M$ whose degree is $(-1)^q \varepsilon(x,y)$.
Therefore $\tilde\varepsilon(x,y)=(-1)^q \varepsilon(x,y)$.  Thus,
\begin{eqnarray*}
\Upsilon^\pi_{\star,\star'}([\alpha],[ \beta ])
&=&(-1)^{n(p+1)+1} \sum_{(x,y)} \varepsilon(x,y)\big[  \left(  \beta_{ y} \bar \alpha_{x  } ,  \alpha_{  x} \bar\beta_{y } \right)   \big]\\
&=&(-1)^{n(p+1)+1} \sum_{(x,y)} \varepsilon(x,y) \big[  \beta_{ y}  \bar \alpha_{x  }   \big] \otimes  \big[ \alpha_{  x}  \bar \beta_{y }    \big].
\end{eqnarray*}
Since   $p \geq 2$, the path   $\alpha_x$ is well defined up to homotopy rel $\partial I$ and  $\bar \alpha_{x  }$ is homotopic to $\alpha_x^{-1}$.
Similar claims hold for   $\beta$  since $q\geq 2$.     This yields \eqref{spheres}.
 \end{proof}

  If we consider a single point  $\star$ in the boundary of $ M$,   then we can similarly compute  the linear map
$$
\Upsilon^\pi = \Upsilon^\pi_{\star,\star}: \pi_p(M, \star)   \times   \pi_q(M, \star) \longrightarrow  \kk[\pi_1(M,\star) ] \otimes \kk[\pi_1(M,\star) ].
$$
Fix a path $\varsigma:I\to  \partial M $ from   $ \star$ to  a different  point $\star'\in \partial M$,
and consider   maps  $\alpha: (S^p,s_p) \to (M,\star)$ and $\beta:(S^q,s_q)\to (M,\star')$  satisfying the conditions of  Lemma~\ref{intersecting_spheres}.
Transporting  $\beta $ along  $\varsigma^{-1}$, we obtain a map $\varsigma^{-1}\beta: (S^q,s_q) \to (M,\star)$. Applying Lemmas~\ref{Leibniz_Upsilon} and~\ref{intersecting_spheres}, we obtain that
\begin{eqnarray}
\notag  \Upsilon^\pi([\alpha],[ \varsigma^{-1} \beta ]) &=& \Upsilon\!\left( \overline \partial_p [\alpha] ,  \varsigma ( \overline \partial_q [\beta]) \varsigma^{-1} \right) \\
\notag &=& \varsigma\,  \Upsilon\!\left(  \overline  \partial_p  [\alpha] ,   \overline \partial_q   [\beta]\right)  \varsigma^{-1} \\
\notag &=& \varsigma\,   \Upsilon^\pi_{\star,\star'}([\alpha],[  \beta ])  \varsigma^{-1}  \\
\label{Upsilon_pi} &=& (-1)^{n(p+1)+1} \sum_{(x,y)} \varepsilon(x,y) \, [\varsigma \beta_{ y} \alpha_{x  }^{-1}] \otimes  [\alpha_{  x} \beta_{y  }^{-1}\varsigma^{-1} ].
\end{eqnarray}
This  computation  implies that  the map   $\Upsilon^\pi$ is   determined by the  pairing
$$ ( \aug \otimes \id)   \Upsilon^\pi: \pi_p(M,\star)   \times   \pi_q(M,\star) \longrightarrow  \kk[\pi_1(M,\star) ]  ,$$
where $\aug:   \kk[\pi_1(M,\star) ]   \to \kk$ is the addition of coefficients. Note that
\begin{equation} \label{Reidemeister pairing}
  ( \aug \otimes \id)  \Upsilon^\pi ([\alpha],[ \varsigma^{-1} \beta ]) =
(-1)^{n(p+1)+1} \sum_{(x,y)} \varepsilon(x,y) \,   [\alpha_{  x} \beta_{y  }^{-1}\varsigma^{-1} ].
\end{equation}
 The   pairing on the right-hand side  is the well known      \lq\lq geometric intersection" of spherical cycles or the  \lq\lq Reidemeister pairing",
 see   \cite{Ke}  or    \cite[Section 5]{Wa}.

\subsection{Intersection of arcs with spheres}

 Assume that  $n= \dim(M) \geq 3$. 
We fix three points $\star_1, \star_2,\star_3 \in \partial M$ and consider the map $ \Upsilon^\pi_{12,3}$ defined by the following composition:
$$
{\small \xymatrix {
\pi_1(M,\star_1,\star_2)  \times  \pi_{n-1}(M,\star_3) \ar[d]_-{ \overline \partial_1 \times \overline \partial_{n-1} } \ar@/^0.5cm/@{-->}[drr]^-{\Upsilon^\pi_{12,3}} &&\\
H_{0}(\Omega_{12})  \otimes H_{n-2}(\Omega_{33}) \ar[r]^-\Upsilon  & H_0(\Omega_{32} \times \Omega_{13})  \simeq  \hspace{-1cm} & \ \kk\big[\pi_1(M,\star_3,\star_2)\big] \otimes  \kk\big[\pi_1(M,\star_1,\star_3)\big].
}}
$$
As in the previous sections,   $\Omega_{ij} = \Omega(M,\star_i,\star_j)$ for any $i,j \in\{1,2,3\}$.
 Lemma \ref{intersecting_spheres} easily adapts to this setting and  yields the following   computation of~$\Upsilon^\pi_{12,3}$.

\begin{lemma}\label{intersecting_arc_and_spheres}
 Let $\alpha\in \Omega^\circ_{12}$ and let  $\beta:(S^{n-1},s_{n-1})\to (M,\star_3)$ be   a continuous map such that  $\beta^{-1}(\partial M)=\{s_{n-1} \}$.
Assume that $\star_1 \neq \star_3$, $\star_2 \neq \star_3$ and  that $\alpha\vert_{(0,1)}$, $\beta\vert_{S^{n-1} \setminus \{s_{n-1}\}}$ are transversal smooth maps.
Then
\begin{equation*}\label{arc_and_sphere}
\Upsilon^\pi_{12,3}([\alpha],[\beta]) =
- \sum_{(x,y)} \varepsilon(x,y) \, [\beta_{ y} \alpha_{x 1}] \otimes  [\alpha_{ 0x} \beta_{y  }^{-1} ].
\end{equation*}
Here:   the sum runs over all $(x,y)\in [0,1] \times S^{n-1}$ such that   $\alpha(x)=\beta(y)$;
 $\varepsilon(x,y)$ is the sign of the product orientation in $\alpha_*(T_x [0,1]) \oplus \beta_*(T_y S^{n-1})=T_{\alpha (x)} M$ with respect to the orientation of $M$;
$\alpha_{0x}$  (respectively $\alpha_{x1}$) is the path running along $\alpha$ from  $\star_1$ to $\alpha(x)$ (respectively from $\alpha(x)$ to $\star_2$) in the positive direction
and $\beta_{ y}$  is the composition of $\beta$ with a  path from $s_{n-1}$ to $y$  in $S^{n-1}$.
\end{lemma}

Lemma \ref{intersecting_arc_and_spheres} can be adapted to the cases where $\star_1=\star_3$ and/or $\star_2=\star_3$.
Besides,  we can similarly define an operation
$$
\Upsilon_{1,23}^\pi: \pi_{n-2}(M,\star_1) \times \pi_1(M,\star_2,\star_3) \longrightarrow  \kk\big[\pi_1(M,\star_2,\star_1)\big] \otimes  \kk\big[\pi_1(M,\star_1,\star_3)\big]
$$
and compute it as in Lemma \ref{intersecting_arc_and_spheres}.

\subsection{A simply connected example} \label{sc_example}

Fix     $2g$ integers $p_1,q_1,\dots,p_g,q_g\geq 2$   such that  $p_1+q_1= \cdots =p_g +q_g =n$. Consider the closed smooth $n$-manifold
\begin{equation} \label{M_hat}
{W} = \left(S^{p_1} \times S^{q_1}\right) \sharp \cdots  \sharp \left(S^{p_g} \times S^{q_g}\right)
\end{equation}
with the product orientation on each summand, and assume that   $M=  {W} \setminus \Int({D})$ where~${D}$ is a closed $n$-ball smoothly embedded in ${W}$.
Fix   a point $\star \in \partial M$ and consider the {Pontryagin} algebra $A_\star=H_*(\Omega_\star)$ where $\Omega_{\star}=\Omega(M,\star,\star)$.
We  now  compute   the   intersection bibracket   in $A_\star$.

For  an appropriate choice of ${D}$,  of the base points $\{s_k\in S^k\}_{k}$,   and of the balls along which the connected sums are performed in \eqref{M_hat}, the  sets
$$
X_i=S^{p_i} \times \{s_{q_i}\} \subset S^{p_i} \times S^{q_i}  \quad {\text{and}} \quad Y_i=\{s_{p_i}\} \times S^{q_i}\subset S^{p_i} \times S^{q_i}
$$
are  embedded spheres in $M$, for all $i=1,\dots,g$.
 Since $M$ is simply connected, these spheres define  certain   elements  $x^{\pi}_i  \in\pi_{p_i}(M,\star)$ and $y^{\pi}_i \in \pi_{q_i}(M,\star)$, respectively.
Consider the  corresponding  elements of the {Pontryagin} algebra
$$
x_i =  \overline \partial_{p_i} (x^{\pi}_i ) \in A^{p_i-1}_\star, \quad y_i= \overline \partial_{q_i} (y^{\pi}_i ) \in A^{q_i-1}_\star.
$$
Since $M$ deformation retracts to a wedge of $2g$ spheres isotopic to $X_1,Y_1,\dots,X_g,Y_g$,
it follows from \cite[III.1.B]{BS} (or, alternatively,  from  \cite[Corollary 2.2]{AH})
that  $x_1,y_1,$ $\dots, $ $x_g,y_g$ freely generate the   unital  graded algebra  $A_\star$.

  In particular, if    $\kk=\ZZ$, then  $A_\star$ is a free abelian group. Therefore the condition  \eqref{cross_iso} is satisfied for any ground ring $\kk$.
Hence  the intersection bibracket $\double{-}{-}$ in $A_\star$ is defined for any $\kk$, and is  fully determined by its values on the generators.
 These values can be computed  from the   formula \eqref{Upsilon_pi}: for any $i,j=1, \ldots , g$,
\begin{equation} \label{xy_yx}
\double{x_i}{y_j}=   \delta_{ij} (-1)^{q_i(p_i+1)+1}\, 1 \otimes 1, \quad \double{y_j}{x_i}= \delta_{ij} (-1)^{p_i+1}\, 1 \otimes 1,
\end{equation}
\begin{equation} \label{xx_yy}
\double{x_i}{x_j}=0, \qquad \quad  \double{y_i}{y_j}=0.
\end{equation}
 Here we use    the assumption that the spheres $X_i,Y_j $ have codimension $\geq 2$ in~$M$ and so can be made disjoint from the interiors of  arcs connecting them to $\star$.
As a consequence, we observe that the  bibracket  $\double{-}{-}$ is a graded version of the bibracket associated by Van den Bergh \cite{VdB}
with  the (double of the) quiver   $Q_g$ having  a single vertex and~$g$ edges.

The graded module $\check A_\star= A_\star /[A_\star,A_\star]$ is freely generated by  words in the letters $x_1,y_1,\dots,x_g,y_g$,
subject to the \index{cyclic relations}  \emph{cyclic relations} $w_1 w_2 =(-1)^{\vert w_1\vert\, \vert w_2 \vert} w_2 w_1$ for any words $w_1,w_2$
where   $\vert w_i\vert$  is the sum of the degrees of the letters appearing in~$w_i$.
The  $(2-n)$-graded Lie  bracket $\langle -,-\rangle$ in $\check A_\star$ induced by $\double{-}{-}$ is a graded version of the  necklace  Lie bracket    associated to $Q_g$, see \cite{BLb,G}.

For  any  integer $N\geq 1$, the Gerstenhaber bracket $\{-,-\}$  in $(A_\star)^+_N$  induced by $\double{-}{-}$
can be computed from  \eqref{xy_yx}, \eqref{xx_yy}.
In particular,   $(A_\star)^+_1= \Com(A_\star)$ is
the unital commutative graded algebra  with free generators $x_1,y_1,\dots,x_g,y_g$ in degrees $\vert x_i \vert =p_i-1,\vert y_i \vert =q_i -1$,
and for any $i,j=1, \ldots, g$,
$$
\{x_i,y_j\}=  (-1)^{q_i(p_i+1)+1} \delta_{ij}, \ \{y_j,x_i\}=  (-1)^{p_i+1}\delta_{ij},  \ \{x_i,x_j\}= 0, \ \{y_i,y_j\}=0.
$$
The bracket $\{-,-\}$  is a graded version of the  standard  Poisson bracket in the symmetric algebra
of a free module of rank $2g$ equipped with a symplectic form.

\subsection{A non-simply connected example} \label{nsc_example}

  Let $n \geq 3$.   We  compute   the   intersection bibracket   in the {Pontryagin} algebra of the exterior of a ball in ${W}=S^1 \times S^{n-1}$.
We endow ${W} $ with the product orientation and set $$X=S^1 \times \{s_{n-1}\}\subset {W} \quad {\text {and}} \quad Y=\{s_1\} \times S^{n-1}\subset {W}$$
where $s_1\in S^1$ and $s_{n-1}\in S^{n-1}$   are the base points.
As above, assume  that   $M={W} \setminus \Int({D})$  where~${D}$ is a closed $n$-ball smoothly embedded in ${W}\setminus (X\cup Y)$.
Pick a point $\star \in \partial M =\partial {D}$ and connect it to the point 
$  s=   (s_1,s_{n-1})\in \Int(M)$ by  a path  $\gamma:I\to M$ such that $\gamma^{-1}(X \cup Y)= \{1\}$.
Up to homotopy   relative to the  endpoints, there are two such paths;
  we take the path $\gamma$ such that   a positive tangent vector of~$\gamma$
followed by a   positively   oriented basis of $T_{s_{n-1}}S^{n-1}$ yields a   positively   oriented  basis of   $T_sM$, see Figure \ref{S1Sq}.
Transporting~$X$ and~$Y$ along $\gamma$,
we obtain certain homotopy classes    $x^\pi\in  \pi_1(M,\star)$ and   $y^{\pi} \in   \pi_{n-1}(M,\star)$.
Consider the corresponding   elements
$$
x =  \overline \partial_{1}   (x^{\pi})  \in A_\star^0 \quad {\text {and}} \quad y=  \overline \partial_{n-1}  (y^{\pi}) \in A^{n-2}_\star
$$
of  the algebra $A_\star=H_*(\Omega_\star)$. Note that $x$ is   invertible in $A_\star^0\simeq \kk[\pi_1(M, \star)]$.
We claim that the   unital  graded algebra $A_\star$ is generated by $x^{\pm 1}$   and $y$  subject to the only
  relation $x x^{- 1}=1$.   Indeed,
$$
A_\star = \bigoplus_{i\in \ZZ} x^i H_*(\Omega_\star^{\operatorname{null}})
$$
where $\Omega_\star^{\operatorname{null}}$ is the connected component of $\Omega_\star$ consisting of null-homotopic loops.
The space $\Omega_\star^{\operatorname{null}}$ can be identified  with the loop space of the universal cover of~$M$.
This cover has the homotopy type of a wedge of countably many copies of $S^{n-1}$ since~$M$ deformation retracts to $X\cup Y \cong S^{1} \vee S^{n-1}$.
Therefore, the   unital   graded algebra $H_*(\Omega_\star^{\operatorname{null}})$
is freely generated by the elements  $\{x^i y   x^{-i}\}_{i\in \ZZ}$,  and the   claim  above easily  follows.

\begin{figure}[h]
\begin{center}
\labellist \small
 \pinlabel {$\star$} [ ] at 149 113
 \pinlabel {$\star'$} [ ] at 147 22
 \pinlabel {\small $\zeta$} [ ] at 152 70
 \pinlabel {$\gamma$} [t] at 263 114
 \pinlabel {$Y$} [ ] at 370 217
  \pinlabel {$s$} [ ] at 370 157
  \pinlabel {$\circlearrowright$} [ ] at 370 100
 \pinlabel {$X$} [t] at 665 143
 \pinlabel {$\partial D$} [tr] at 92 29
\endlabellist
\includegraphics[scale=0.35]{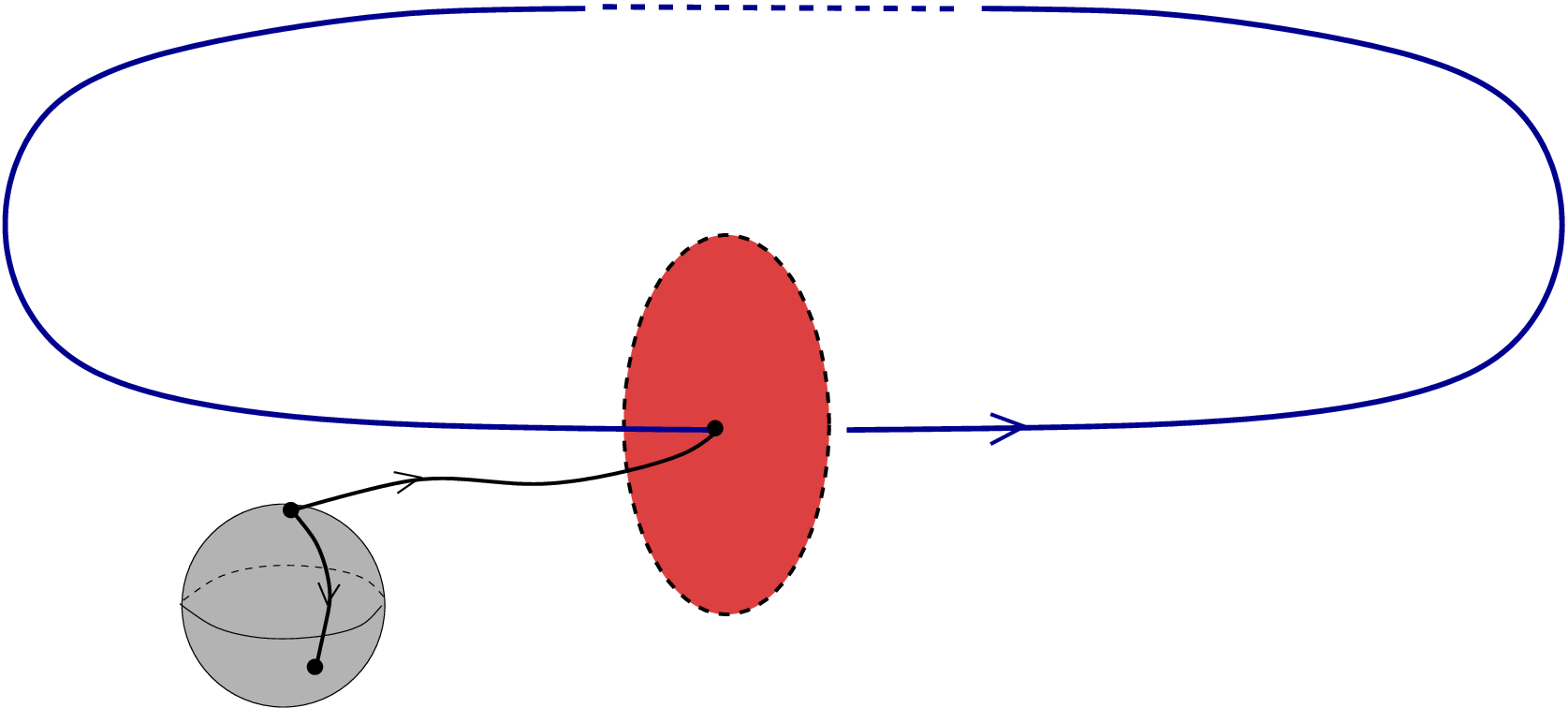}
\end{center}
\caption{The manifold $M=(S^1\times S^{n-1})\setminus \Int(D)$.}
\label{S1Sq}
\end{figure}
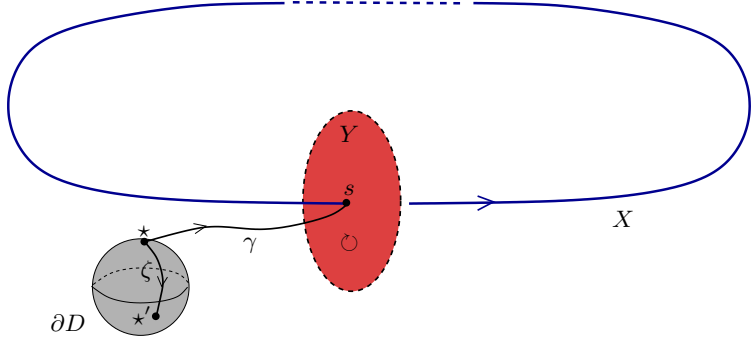

In particular, if $\kk=\ZZ$, then $A_\star$ is a free abelian group. Therefore the intersection bibracket $\double{-}{-}$  in $A_\star$  is defined for any ground ring $\kk$.
To determine $\double{-}{-}$, it suffices to compute its values on the generators $x,y$.
For degree reasons,
\begin{equation}\label{xx}
\double{x}{x}=0.
\end{equation}
Let $\varsigma$ be an arc in $\partial M$ connecting $\star$ to another point $\star'$.
  By Lemma \ref{intersecting_arc_and_spheres},  we obtain
$$
\double{x}{\varsigma^{-1} y\varsigma} = - \varsigma^{-1}x \otimes \varsigma, \quad
\double{\varsigma^{-1} y \varsigma}{x} = \varsigma \otimes \varsigma^{-1}x .
$$
 This implies the equalities
\begin{equation}\label{xy,yx}
\double{x}{y}= - x \otimes 1, \quad \double{y}{x} = 1 \otimes x.
\end{equation}
Observe next that $x^{-1}yx$ and  $\varsigma^{-1}y\varsigma$ are images  under the connecting homomorphism  of   certain elements of $\pi_{n-1}(M,\star)$ and $\pi_{n-1}(M,\star')$
that can be represented by disjoint embedded spheres. It follows that $\double{x^{-1}yx}{\varsigma^{-1}y\varsigma}=0$ which implies that $\double{x^{-1}yx}{y}=0$.
Using the Leibniz rules and \eqref{xy,yx}, we deduce that
\begin{equation}\label{yy}
\double{y}{y}= 1 \otimes y - y \otimes 1.
\end{equation}

Using \eqref{xx}--\eqref{yy}, one can   also   compute the  graded   Lie bracket $\langle-,-\rangle$ in~$\check A_\star$
and the Gersthenhaber bracket $\{-,-\}$ in   $(A^+_\star)_N$   for any integer $N\geq 1$.

\chapter{Properties of the intersection  bibracket}\label{More on brackets   and bibrackets}

  In this chapter, $M$  is  a smooth oriented   connected manifold  of dimension $n\geq 2$ 
   such that $\partial M  \neq \varnothing$ and the condition   \eqref{cross_iso} is satisfied.

\section{The  scalar intersection form} \label{scalar_intersection_form}

We  derive from the intersection bibracket   of $M$   a scalar intersection form    and compute it  in terms of
usual  homology intersections. We begin with algebraic preliminaries.

 \subsection{The scalar form  induced by  a bibracket}\label{scalar_form}

 Consider an arbitrary   graded category~$\mathscr{C}$ and   the associated graded algebra $A=A(\mathscr{C})$, see  Section~\ref{extenscat}.
  Given an augmentation  $\varepsilon :A \to \kk$ and a $d$-graded bibracket $\double{-}{-}$ in~$\mathscr{C}$ with $d\in \ZZ$,
we define the  \index{scalar form} \emph{induced scalar form}
  $\bullet: A \times A \to \kk$ by   $a\bullet b=  (\varepsilon \otimes \varepsilon)(\double{a}{b })$ for any $a,b \in A$.
 Observe  that
$$
a \bullet (bc) = (a \bullet b)\, \varepsilon(c) +  \varepsilon(b)\, (a \bullet c),
$$
for any $a\in A$, $b\in \Hom_{\mathscr{C}}(X,Y)$, $c\in \Hom_{\mathscr{C}}(Y,Z)$  with   $X,Y,Z \in \Ob(\mathscr{C})$;
similarly,
$$
(ab) \bullet c = \varepsilon(a) \, (b \bullet c) +    (a \bullet c) \, \varepsilon(b)
$$
for any $c\in A$, $a\in \Hom_{\mathscr{C}}(X,Y)$, $b\in \Hom_{\mathscr{C}}(Y,Z)$  with  $X,Y,Z \in \Ob(\mathscr{C})$.
 Furthermore,
if the bibracket $\double{-}{-}$ is $d$-antisymmetric,
then $a \bullet b = -(-1)^{\vert a\vert_d\, \vert b \vert_d }\, b \bullet a$  for   any   homogeneous $a,b \in A$.

\subsection{ The scalar form induced by the intersection bibracket}

The path homology category
  $\mathscr{C}=\mathscr{C}(M)$  of the manifold $M$  
  has a canonical augmentation $\varepsilon: A (\mathscr{C}) \to \kk$ obtained as the direct sum over all $\star, \star'\in \partial M$ of the  compositions
$$H_*\big(\Omega(M,\star,\star')\big)   \longrightarrow    H_0\big(\Omega(M,\star,\star')\big)   \longrightarrow    \kk,$$
where the left arrow is the obvious projection and the right arrow   carries  the homology classes of all points to $1$.
 By the previous  subsection,  this augmentation together with the intersection bibracket induce  a bilinear form \index{intersection bibracket!scalar form of} $\bullet: A (\mathscr{C}) \times A(\mathscr{C}) \to \kk$.

 We compute $\bullet$  in terms of standard  homological intersections in~$M$.
  For simplicity, we assume in the rest of this section that $n \geq 3$, though the case $n=2$ may be considered similarly.
 For any points   $\star_1, \star_2, \star_3, \star_4 \in \partial M$,  we define a   linear map
\begin{equation} \label{ccdott}
  H_\ast (M,\{\star_1, \star_2\}) \otimes H_\ast (M,\{\star_3, \star_4\})\stackrel{\mediumdot}{ \longrightarrow} \kk .
\end{equation}
  It suffices to define the restriction   of $\mediumdot$   to $H_k \otimes H_l$ for any   $k,l \geq 0$. If $k+l \neq~n$, then this restriction is equal to zero.
Suppose  now  that $k+l=n$.
When   $\{\star_1, \star_2\} \cap \{\star_3, \star_4\}=\varnothing $,   the form   $\mediumdot$ is the standard homological intersection, see, for example,  \cite{Br}.
When   $\{\star_1, \star_2\} \cap \{\star_3, \star_4\} \neq \varnothing $,  we separate two cases. If $k\geq 2$,  then
 $H_k(M,\{\star_1, \star_2\})$   is canonically isomorphic to   $H_k(M)$ and the pairing $\mediumdot$ is induced by the homological intersection $ H_k(M ) \otimes H_l (M,\{\star_3, \star_4\}) \to \kk$.
The case $l\geq 2$ is treated similarly using that
 $H_l(M,\{\star_3, \star_4\})$   is canonically isomorphic to   $H_l(M)$. Note that the assumption $k+l=n \geq 3$ guarantees that $k \geq 2$ or $l \geq 2$.
 If both these inequalities hold true, then the two definitions above give the same pairing.

The next  lemma yields a   version  of the  homological suspension homomorphism   due to Serre      \cite[\S IV.5]{Se}.

\begin{lemma}\label{Serre}
Let  $\Omega=\Omega (M, \star, \star')$ with  $\star,\star' \in \partial M$. There is a   unique    homomorphism  $\Sigma: H_*(\Omega) \to H_{*+1}(M, \{\star, \star'\})$ such that for every polycycle $\calK=(K,\varphi,u,\kappa)$ in $\Omega  $, we have
\begin{equation} \label{Sigma}
\Sigma([\calK])= \big[(K \times I, \varphi \times    \tau, u \times 1, \tilde\kappa )\big]
\end{equation}
where $\tau$ is   the trivial partition on $I=[0,1]$.
\end{lemma}

\begin{proof} The uniqueness of~$\Sigma$ is a direct consequence of Theorem~\ref{mainonpolychains}.
 To prove the existence,  define a  continuous  map   $\ev: \Omega \to M  $
   by $\ev(\alpha)= \alpha(1/2)$ and   set $\Omega^\partial=\ev^{-1}(\{\star,\star'\})$.  The formula
 $$
d(\alpha,s)(t) = \left\{\begin{array}{ll}
\star & \hbox{if } t\in \left[0,1/2-s/{2}\right], \\ \alpha(s + 2t-1 ) & \hbox{if } t \in\left[1/2-s/{2},1-s/{2} \right], \\ \star' &  \hbox{if } t\in\left[1-s/{2} ,1\right]
\end{array}\right.
$$
defines  a  continuous    map $  d:(\Omega \times  I, \Omega \times \partial I)\to \left(\Omega ,\Omega^\partial \right)$.
Let  $$\Delta: H_*(\Omega ) \longrightarrow  H_{*+1}\left(\Omega ,\Omega^\partial \right)$$ be the  linear map
sending any $x \in H_*(\Omega )$ to $  d_*(x \times [I,\partial I])  $   (the definition of~$\Delta$ is inspired by \cite[\S 5]{CS} and \cite[Remark 3.2.3]{KK1}).
Finally, we  set $\Sigma = \ev_* \Delta$. To check \eqref{Sigma}, observe that
  the fundamental class $[I,\partial I]\in H_1(I,\partial I)$ is represented by the $1$-dimensional polycycle   $\mathscr{I}=(I,   \tau  ,1,\id:I\to I)$  relative to $\partial I$.
Lemma \ref{twocrossproducts} implies that  for any  polycycle $\calK=(K,\varphi,u,\kappa)$ in $\Omega$,
\begin{eqnarray*}
\Sigma([\calK]) &=&  \ev_*   d_*   ([\calK] \times [\mathscr{I}])  \\
&=&  (\ev d)_* \left[  \calK \times \mathscr{I } \right] = \big[(K \times I, \varphi \times    \tau, u \times 1, \tilde\kappa )\big].
\end{eqnarray*}

\up
\end{proof}

 We  can  now state the main result of this section.

\begin{theor}\label{bulletvscdot} For any   $\star_1,\star_2 , \star_3,\star_4\in \partial M$,  the following  diagram commutes:
$$
\xymatrix{
H_*\big(\Omega(M, \star_1, \star_2)\big) \otimes H_*\big(\Omega(M, \star_3, \star_4)\big)  \ar[r]^-{\bullet} \ar[d]_-{- (-1)^{n\vert-\vert}   \Sigma\,   \otimes\,    \Sigma  }  & \kk \,  . \\
H_*\big(M, \{\star_1,\star_2\}\big) \otimes H_*\big(M,  \{\star_3,\star_4\}\big)  \ar@/_1cm/[ru]^-{ \mediumdot} &
}
$$
 \end{theor}

The proof   of  Theorem \ref{bulletvscdot}  proceeds in three steps.
First we consider arbitrary  disjoint subsets $\partial_- M$, $ \partial_+ M$ of $\partial M$ and    the standard  homology intersection form
$$
 H_*(M,\partial_- M) \otimes H_*(M,\partial_+ M) \longrightarrow H_*(M).
$$
We denote this form by $\odot$ and compute it  in terms of  polycycles. Secondly, we   relate
  $\odot$ to the operation $\Upsilon$. Finally, we deduce Theorem \ref{bulletvscdot}.

\begin{lemma}\label{bullet} Let $\partial_- M$, $ \partial_+ M$ be disjoint subsets of $\partial M$.
Let  $\calK=(K,\varphi,u,\kappa)$ be a smooth $p$-polycycle in $M$ relative to $\partial_- M$,
  let $\calL=(L,\psi,v,\lambda)$ be  a smooth   $q$-polycycle  in $M$ relative to $\partial_+ M$
such that the map  $\kappa \times \lambda:K \times L \to M \times M$ is transversal to $\diag_M$   in the sense of  Section~\ref{transs}.
Let  $D=(\kappa \times \lambda)^{-1}(\diag_M)$ and let  $\pr_K:K \times L \to K$ be the cartesian projection.
  Then~$D$ is a manifold with faces and, for some  orientation, partition $\theta$, and weight~$w$ on~$D$,   the polychain $\calD=  (D,\theta,w, \kappa \pr_K \vert_D)$ is a polycycle   such that
\begin{equation}\label{homological_intersection}
[\calK] \odot [\calL] = (-1)^{q(p+n)} [\calD] \in H_{p+q-n}(M).
\end{equation}
\end{lemma}

\begin{proof}
The transversality assumption ensures that $D$ inherits from $K\times L$ a structure of   a manifold  with corners, see \cite{MrOd}.
The same argument  as  at the beginning of Section \ref{def_D} shows that $D$ is a manifold with faces. We orient~$D$
  so that the  induced  orientation of its normal bundle  in $ K\times L $
is the pull-back of  the orientation of the normal bundle of $\diag_M \approx M $ in $ {M \times M}  $ via   $(\kappa \times \lambda)\vert_D$.
The partition $\theta$ of $D$ is defined as follows:
the faces of $D$ are the connected components of the intersections $(F\times  G) \cap D$ where $F$ and $G$ range over faces of $K$ and $L$ respectively;
two such faces $C\subset (F\times  G)\cap D$ and $C'\subset(F'\times  G')\cap D$ are of the same type if $F, F'$ are of the same type, $G, G'$ are of the same type,
and $(\varphi_{F,F'} \times \psi_{G,G'})(C)=C'$. Then   $\theta_{C,C'}= (\varphi_{F,F'} \times \psi_{G,G'})\vert_{C}$.
The weight $w$ of $D$ carries a connected component $Z$ of $D$ to $u(X)v(Y)$ where $X,Y$ are   connected components  of $K,L$ respectively, such that $Z\subset X\times Y$.
Then $\calD=  (D,\theta,w, \kappa \pr_K \vert_D)$ is a polycycle satisfying \eqref{homological_intersection}.

We leave the general case of this claim to the reader and prove it only under the following assumptions:
$K$ and $L$ are transversal compact oriented   smooth   submanifolds of $M$  such that
  $\partial K= \partial M \cap K \subset \partial_-M$ and $\partial L=\partial M \cap L  \subset \partial_+M$;
the partitions $\varphi$  of $K$ and   $\psi$ of $L$ are trivial;
the weights $u:\pi_0(K)\to \kk$ and $v:\pi_0(L)\to \kk$ send all connected components  to $1\in \kk$;
the maps $\kappa:K \to M$ and $\lambda:L \to M$ are the inclusions.   Under these assumptions, we have
$$
[\calK] \odot  [\calL] = [K] \odot  [L]= [K \cap L]
$$
where $[K]\in H_*(M,\partial_-M)$, $[L]\in H_*(M,\partial_+M)$ and $[K \cap L]\in H_*(M)$ are  the fundamental classes,
and   $K\cap L$ is oriented so that
\begin{equation}\label{nu_K_L}
\nu_{M}(K\cap L)=\nu_{M}(K)\vert_{K\cap L} \oplus \nu_{M}(L)\vert_{K\cap L}
\end{equation}
(this agrees with the   orientation   rule in \cite[p.\ 375]{Br}).
Since  $D= (K\times L) \cap  \diag_M$ corresponds to  $K\cap L \subset M$
under the standard identification $\diag_M\approx M$, we need only  to compare the orientation of $D$ with that of $K \cap L$.
Note the following orientation-preserving  isomorphisms of oriented vector bundles:
\begin{eqnarray*}
T(M^2)\vert_{K \times L} &=& \pr_K^*\big(T(M)\vert_K\big) \oplus\pr_L^*\big(T(M)\vert_L\big) \\
&\cong &\pr_K^* \nu_M(K) \oplus\pr_K^* T(K)   \oplus\pr_L^*  \nu_M(L)   \oplus\pr_L^*  T(L)   \\
&\cong & (-1)^{p(q+n)} \pr_K^* \nu_M(K)  \oplus\pr_L^* \nu_M(L) \oplus\pr_K^* T(K)   \oplus\pr_L^*  T(L)   \\
&\cong & (-1)^{p(q+n)}  \pr_K^* \nu_M(K)  \oplus\pr_L^* \nu_M(L) \oplus T(K \times L)
\end{eqnarray*}
where  $\pr_K:K\times L \to K$ and $\pr_L:K\times L \to L$ are the cartesian projections. Restricting to $D \subset K\times L $, we obtain
\begin{eqnarray*}
T(M^2)\vert_{D} &\cong & (-1)^{p(q+n)}  \left(\pr_K^* \nu_M(K) \right)\vert_D  \oplus \left(\pr_L^* \nu_M(L)\right) \vert_D    \oplus T(K \times L)\vert_D \\
&\cong & (-1)^{p(q+n)}   p^*(\nu_M(K) \vert_{K\cap L})  \oplus  p^*(\nu_M(L) \vert_{K\cap L})    \oplus T(K \times L)\vert_D \\
  &\cong & (-1)^{p(q+n)}   p^*  \nu_{ M}(K\cap L)   \oplus \nu_{K\times L}(D) \oplus T(D)
\end{eqnarray*}
where $p$ is   the identification diffeomorphism $D \to K\cap L$.
On the other hand,
\begin{eqnarray*}
T(M^2)\vert_D &\cong & \nu_{M\times M}(\diag_M)\vert_D \oplus T(\diag_M)\vert_D\\
 &\cong & \nu_{M\times M}(\diag_M)\vert_D \oplus p^*\left( T(M)\vert_{K \cap L}\right) \\
&\cong & \nu_{M\times M}(\diag_M)\vert_D \oplus p^*   \nu_M(K \cap L)   \ \oplus p^* T(K\cap L)\\
&\cong& (-1)^{(p+q)n}  p^*   \nu_M(K \cap L)   \oplus \nu_{M\times M}(\diag_M)\vert_D \oplus  p^* T(K\cap L) .
\end{eqnarray*}
Since $\nu_{K\times L}(D) =  \nu_{M\times M}(\diag_M)\vert_D$ as oriented vector bundles, we deduce that
$$
T(D) =  (-1)^{p(q+n)} \cdot  (-1)^{(p+q)n}\,  p^* T(K\cap L)  = (-1)^{q(p+n)}  p^* T(K\cap L)
$$
and \eqref{homological_intersection} follows.
\end{proof}

\begin{lemma} \label{Upsilon_epsilon}
Let,  under the  assumptions of Lemma~\ref{bullet},  $\star_1,\star_2 \in \partial_- M$ and $\star_3,\star_4\in \partial_+ M$.
Let  $    \varepsilon_1  $ be the composition of the  augmentation   $  \varepsilon: H_*(\Omega_{32} \times \Omega_{14})  \to \kk$  with
  the  linear map $ \kk \to H_*(M)$ sending $1\in \kk$ to $[\star_1]\in H_0(M)$.   Then the following  diagram commutes:
$$
\xymatrix{
H_*(\Omega_{12}) \otimes H_*(\Omega_{34})  \ar[r]^-{\Upsilon_{12,34}} \ar[d]_-{- (-1)^{n\vert-\vert}   \Sigma\,   \otimes\,    \Sigma  }
& H_*(\Omega_{32} \times \Omega_{14}) \ar[d]^-{    \varepsilon_1   } \\
H_*(M,\partial_- M) \otimes H_*(M,\partial_+ M) \ar[r]^-{ \odot} & H_*(M) .
}
$$
\end{lemma}

\begin{proof} Let  $\pr_{32}: \Omega_{32} \times \Omega_{14} \to \Omega_{32}$  be  the cartesian projection.
Clearly, the map $\ev: \Omega_{32} \to M$ is homotopic to the constant map $\alpha \mapsto \star_3$
so that, in homology,  $(\ev \pr_{32})_* =     \varepsilon_1    $.
 Pick  now  any  $a \in H_p(\Omega_{12})$ and $b\in H_q(\Omega_{34})$ with $p,q\geq 0$.
Let $\calK=(K,\varphi,u,\kappa)$ be a smooth   reduced   $p$-polycycle in   $\Omega_{12}^\circ$   and let $\calL=(L,\psi,v,\lambda)$
be a smooth   reduced   $q$-polycycle in   $\Omega_{34}^\circ$     transversely representing the pair of face homology classes  $(\lb a\rb,\lb b\rb )$.
Set $\calD(\calK,\calL)=(D,\theta,w,\kappa\losange \lambda)$. Then
\begin{eqnarray*}
    \varepsilon_1     \Upsilon_{12,34}(a \otimes b) &=&   (-1)^{q+np} (\ev \pr_{32})_*(\left[\calD(\calK,\calL)\right]) \\
&=&  (-1)^{q+np} \ev_*\big[(D,\theta,w,\kappa\triangleleft \lambda)\big] \\
&=& (-1)^{q+np} \big[(D,\theta,w, \tilde \kappa \circ \pr\vert_D)\big]
\end{eqnarray*}
where $\pr:K \times I \times L \times I \to K \times I$ is the cartesian projection.
We  deduce from Lemmas~\ref{Serre} and   \ref{bullet} that
\begin{eqnarray*}
&& -(-1)^{np} \, \Sigma(a) \odot \Sigma(b) \\
&=& -(-1)^{np} \big[(K \times I, \varphi \times    \tau, u \times 1, \tilde\kappa ) \big ]  \odot \big[(L \times I, \psi \times   \tau, v \times 1, \tilde\lambda ) \big ]    \\
 &=& (-1)^{1+np+(q+1)(p+1+n)} \big[(D,\theta,w, \tilde \kappa \circ \pr\vert_D)\big] \\
&=&  (-1)^{1+q +(q+1)(p+1+n)}   \varepsilon_1      \Upsilon_{12,34}(a \otimes b) \ = \   (-1)^{ (q+1)(p+n)}   \varepsilon_1     \Upsilon_{12,34}(a \otimes b).
\end{eqnarray*}
Since $(q+1)(p+n)$ is even if $p+q=n-2$ and   $ \varepsilon_1     \Upsilon_{12,34}(a \otimes b)=0$ otherwise, we obtain the claim of the lemma.
\end{proof}

We can now complete the proof of Theorem \ref{bulletvscdot}. Set $\partial_-M=\{\star_1, \star_2\} $ and $\partial_+M=\{\star_3, \star_4\} $.
Suppose first that  $\partial_-M \cap \partial_+M=\varnothing$.
The desired claim is obtained by combining  the diagram in Lemma \ref{Upsilon_epsilon} with the obvious diagram
$$
\xymatrix{
H_*(\Omega_{32} \times \Omega_{14})  \ar[r]_-{  \simeq  }^-{ \varpi_{32,14}  } \ar[d]_-{ \varepsilon_1  }
& H_*(\Omega_{32})  \otimes  H_*(\Omega_{14}) \ar[d]^-{   \varepsilon \otimes \varepsilon  } \\
H_*(M) \ar[r]^-{} & \kk
}
$$
  where $ \varpi_{32,14}$ denotes the inverse of the cross product isomorphism as before, and  
 the bottom horizontal arrow is the standard augmentation.
To handle the case  $\{\star_1, \star_2\} \cap \{\star_3, \star_4\} \neq \varnothing$,
consider   a smooth isotopy $\{\phi^t:M\to M\}_{t\in I}$ of $ \phi^0 =\id_M$    which is constant outside of a small neighborhood of the points $\star_1,  \star_2$
  and such that   the point  $\star_i' =  \phi^1(\star_i) $   lies in   $ \partial M \setminus \{\star_3, \star_4\}$ for $i=1,2$.
The  diffeomorphism $  \phi^1:   (M, \star_1, \star_2) \to  (M, \star'_1, \star'_2)$   induces   horizontal isomorphisms in the  commutative diagram
$$
\xymatrix{
H_*\big(\Omega(M, \star_1, \star_2)\big)  \big)  \ar[r]^-{ \simeq   } \ar[d]_-{  - (-1)^{n\vert-\vert}   \Sigma         }
& H_*\big(\Omega(M, \star'_1, \star'_2)\big)     \ar[d]_-{  - (-1)^{n\vert-\vert}   \Sigma          } \\
H_*\big(M, \{\star_1,\star_2\}\big)   \ar[r]^-{ \simeq   }    & H_*\big(M, \{\star'_1,\star'_2\}\big)    \, .
}
$$
Note that  the upper horizontal arrow   coincides with the isomorphism $(\varsigma_1, \varsigma_2)_\#  $ defined in  Section~\ref{Upsilon1111}   where
  $\varsigma_i:I\to  \partial M$ is the path   $t\mapsto \phi^t(\star_i)$.
Tensoring this diagram by  the obvious commutative diagram
$$
\xymatrix{
H_*\big(\Omega(M, \star_3, \star_4)\big)  \big)  \ar[r]^-{ \id   } \ar[d]_-{   \Sigma   }
& H_*\big(\Omega(M, \star_3, \star_4)\big)     \ar[d]_-{ \Sigma       } \\
H_*\big(M, \{\star_3,\star_4\}\big)   \ar[r]^-{ \id   }    & H_*\big(M, \{\star_3,\star_4\}\big)
}
$$  we obtain a commutative diagram
$$
\xymatrix{
H_*\big(\Omega(M, \star_1, \star_2)\big) \otimes H_*\big(\Omega(M, \star_3, \star_4)\big)  \ar[r]^-{ \simeq   } \ar[d]_-{- (-1)^{n\vert-\vert}   \Sigma\,   \otimes\,    \Sigma  }
& H_*\big(\Omega(M, \star'_1, \star'_2)\big) \otimes H_*\big(\Omega(M, \star_3, \star_4)\big)  \ar[d]_-{- (-1)^{n\vert-\vert}   \Sigma\,   \otimes\,    \Sigma  } \\
H_*\big(M, \{\star_1,\star_2\}\big) \otimes H_*\big(M,  \{\star_3,\star_4\}\big)  \ar[r]^-{ \simeq   }    & H_*\big(M, \{\star'_1,\star'_2\}\big) \otimes H_*\big(M,  \{\star_3,\star_4\}\big) \, .
}
$$
 By the first part of the proof, we have the diagram in Theorem \ref{bulletvscdot} for the  points $\star'_1, \star'_2, \star_3, \star_4$.
  Combining it  with    the   diagram  above
 we obtain the required diagram. Indeed,   according to \eqref{1234_1234'_H},   the upper line  represents the  scalar form
 $\bullet: H_*\big(\Omega(M, \star_1, \star_2)\big) \otimes H_*\big(\Omega(M, \star_3, \star_4)\big) \to \kk$.
 In the bottom line we obviously get $ \mediumdot   : H_*\big(M, \{\star_1,\star_2\}\big) \otimes H_*\big(M,  \{\star_3,\star_4\}\big)\to \kk$.

\section{The   reducibility} \label{reduction}

 The path homology category $\mathscr{C}=\mathscr{C}(M)$ has a natural  structure of a graded Hopf category,
which generalizes the usual Hopf algebra structure on the Pontryagin algebra.  
  The comultiplication   $\Delta$  in~$\mathscr{C}$ is   the direct sum over all $\star, \star'\in \partial M$ of the linear maps
$$H_*\big(\Omega(M,\star,\star')\big)  \longrightarrow
H_*\big(\Omega(M,\star,\star')\big) \otimes  H_*\big(\Omega(M,\star,\star')\big)$$
  induced by the diagonal maps  $\Omega(M,\star,\star') \to \Omega(M,\star,\star') \times \Omega(M,\star,\star')$.  
  (Note that we  use here the  condition~\eqref{cross_iso}.) 
  The counit $\varepsilon$ in $\mathscr{C}$  is the  augmentation  defined   in Section \ref{scalar_form}.
For     $ \star , \star'  \in \partial M$, the inversion of paths  induces a homeomorphism  $\Omega(M,\star,\star')  \to    \Omega(M,\star',\star)$
which in its turn induces a graded linear isomorphism $H_\ast(\Omega(M,\star,\star'))  \to    H_\ast(\Omega(M,\star',\star))$;
the direct sum of these isomorphisms over all $ \star , \star'  \in \partial M$  defines an antipode  $s$  in  $\mathscr{C}$.
  It is well-known that  the path homology category $\mathscr{C} $ with this data is a  cocommutative Hopf category.

\begin{lemma}\label{redred} The intersection bibracket in  $\mathscr{C}=\mathscr{C}(M)$ is reducible.
\end{lemma}

\begin{proof}
Let $A=A(\mathscr{C})$ be the graded algebra associated with $\mathscr{C}$   and let $\Lambda= \Lambda(\double{-}{-})$
be the map \eqref{Lambda} associated with the intersection bibracket  $\double{-}{-}$ in~$\mathscr{C}$.
We must show that $ \Lambda(a,b) \in \Delta(A)$ for any     $a,b \in A$.
Since  $\Lambda$  is bilinear, it suffices to  consider the case where  $a\in H_p (\Omega_{12} )$ and
  $b\in H_q\big(\Omega_{34}\big)$ for some  $p,q \geq 0$. Here  $\Omega_{ij}= \Omega(M,\star_i,\star_j)$,
  and   $\star_1,\star_2  , \star_3,\star_4 $ are four points in $ \partial M$.
  Observe that  a  path in $\partial M$ starting from $\star_2$ represents a certain element   $v\in A^0$ and,  by Lemma~\ref{F},
  $$
  \Lambda(av,b) = \Lambda(a,b) \, \varepsilon(v)  + \Delta( a) \Lambda(v, c)   = \Lambda(a,b) .
  $$
  Similarly,   a  path in $\partial M$ ending at $\star_1$ represents a certain   $u\in A^0$ and
  $$
  \Lambda(ua, b) = \Lambda(u, b) \, \varepsilon(a) + \Delta(u) \Lambda(a,b)  =  \Delta(u) \Lambda(a,b).
  $$
  Thus it suffices to consider the case where  $\{\star_1,\star_2\} \cap \{\star_3,\star_4\}=\varnothing$.

  Pick transversal smooth polycycles  $\calK=(K,\varphi,u,\kappa)$ in $ \Omega^\circ_{12}$ and
    $\calL=(L,\psi,v,\lambda)$ in   $ \Omega^\circ_{34}$ representing respectively   $\langle a \rangle \in \widetilde H_p(\Omega_{12}) $
 and $\langle b \rangle  \in   \widetilde  H_q  (\Omega_{34})$. We form the intersection polycycle  $\calD = (D,\theta,w,\kappa\losange \lambda)$
  in $\Omega_{32} \times \Omega_{14}$ as in Section~\ref{def_D}.   By definition, $\double{a}{b}\in H_{\ast}  (\Omega_{32})  \otimes  H_{\ast}(\Omega_{14} )$   corresponds to
  the homology class $(-1)^{q+np} [\calD ]\in H_{p+q+2-n} \big(\Omega_{32} \times \Omega_{14} \big)$
   under the  isomorphism $$\varpi_{32,14}:H_{\ast}  (\Omega_{32})  \otimes  H_{\ast}(\Omega_{14} ) \longrightarrow
  H_{\ast} \big(\Omega_{32} \times \Omega_{14} \big) $$  induced by the cross product in homology.   Consider the tensor
 $$
 T =    a^{(1)} \otimes \double{a^{(2)}}{b^{(1)}} \otimes b^{(2)}   \\
   \in H_\ast \big(\Omega_{12}\big) \otimes H_\ast \big(\Omega_{32}\big)
   \otimes H_\ast \big(\Omega_{14}\big)  \otimes H_\ast \big(\Omega_{34}\big).
  $$
 Applying \eqref{double_parameters} with $\star_1$, $\star_3$ exchanged and with $Y= \Omega_{12}$, $Z=\Omega_{34}$, we obtain
   \begin{eqnarray*}
  \varpi_{12,32,14,34}(T)
    &=&  a^{(1)} \times \Upsilon_{12,34}\big(a^{(2)} \otimes b^{(1)}\big) \times b^{(2)}   \\
    &=& \Upsilon_{Y12,34Z}\big(\diag_*(a),\diag_*(b) \big) \  \in    H_\ast \big(\Omega_{12}  \times  \Omega_{32} \times  \Omega_{14} \times  \Omega_{34}\big)
   \end{eqnarray*}
   where   $\varpi_{12,32,14,34}$ is the isomorphism induced by the cross product in homology
   and $\diag_*: H_*(\Omega_{ij}) \to H_*(\Omega_{ij} \times \Omega_{ij})$ is induced by the diagonal map $M \to M \times M$.
    The homology class $ \Upsilon_{Y12,34Z}\big(\diag_*(a),\diag_*(b) \big)$  is represented by the  polycycle
   $$ (-1)^{q+np}\, \big(D,\theta,w, \kappa' \times (\kappa\losange \lambda) \times \lambda':D \to \Omega_{12} \times
   \Omega_{32}\times \Omega_{14}  \times  \Omega_{34} \big)$$
   where $\kappa':D \to \Omega_{12} $ is obtained by projecting $D\subset K \times I \times L \times I$  onto   $K $ and applying~$  \kappa $,
   whereas $\lambda':D \to \Omega_{34} $ is obtained by projecting   onto    $L $ and applying~$  \lambda $.
 Consider now the homeomorphisms $\{   J_{i}   :\Omega_{3i}\to \Omega_{i3} \}_{i=2,4}$   induced by the inversion of paths,
the concatenation maps $\{ \conc_i: \Omega_{1i}\times \Omega_{i3}   \to \Omega_{13} \}_{i=2,4}$, and the map
$$
\mu= (\conc_2 \times \conc_4) (\id_{\Omega_{12}} \times  J_{2}   \times \id_{\Omega_{14}} \times  J_{4}   )
(\kappa' \times (\kappa\losange \lambda) \times \lambda') :D \longrightarrow \Omega_{13} \times \Omega_{13}.
$$
It follows that  the image of    the homology class
$$
\Lambda(a,b)=   a^{(1)}  s\Big(\double{a^{(2)}}{b^{(1)}}'  \Big) \otimes  \double{a^{(2)}}{b^{(1)}}''  s(b^{(2)})
$$
 under the cross product isomorphism $\varpi_{13,13}: H_*(\Omega_{13}) \otimes H_*(\Omega_{13}) \to H_*(\Omega_{13} \times \Omega_{13})$
is represented by the  polycycle $(-1)^{q+np}\, (D,\theta,w, \mu )$.
To analyze this polycycle, let $\mu_{1}, \mu_{2}:D  \to  \Omega_{13}$ be  the first   and the second coordinates    of~$\mu$.
For any point $(k,s,l,t)\in D$, the path $\mu_{1}(k,s, l,t)$ is obtained  by concatenation of the following three paths:
(i) the  path $  \kappa(k )$ from $\star_1$ to    $\star_2 $;
(ii)  the initial segment of the path   $  (\kappa(k ))^{-1}$ from $\star_2$ to     the point  $   \kappa(k)(s )= \lambda (l)(t )$;
 (iii) the terminal segment of the path  $ ( \lambda(l))^{-1} $ from the latter point to~$\star_3$.
 This concatenated path goes along a terminal segment of  the  path $  \kappa(k )$  twice in opposite directions.
 Therefore the path  $\mu_1(k,s, l,t)$ is homotopic to a path $\nu (k,s, l,t)$  obtained  by concatenation of just two paths:
  the initial segment of the path  $  \kappa(k )$ from $\star_1$ to        the point  $   \kappa(k)(s )= \lambda (l)(t )$
  and  the terminal segment of the path  $ ( \lambda(l))^{-1} $ from the latter point to $\star_3$.
  The homotopy in question may be defined by an explicit formula which applies to all points $(k,s,l,t)\in D$.
  Therefore, it determines a homotopy of the polycycle  $ (D,\theta,w, \mu_1 )$ into the polycycle $ (D,\theta,w, \nu )$.
  A similar argument applies to  the path $\mu_{2}(k,s, l,t)$ and yields a homotopy of the polycycle  $ (D,\theta,w, \mu_2 )$ into   $ (D,\theta,w, \nu )$.
  Applying these two homotopies coordinatewise we obtain a homotopy of the polycycles  $ (D,\theta,w, \mu )$    and    $ (D,\theta,w,  (\nu,\nu)  )$.
  It is obvious that the homology class represented by the latter polycycle belongs to the image of the map $\diag_*: H_*(\Omega_{13}) \to H_*(\Omega_{13} \times \Omega_{13})$.
 We conclude that   $ \Lambda(a,b) \in \Delta(A)$.
 \end{proof}

Lemma~\ref{redred} and the $(2-n)$-antisymmetry of  the intersection bibracket  of~$M$ implies that it shares all the properties   established in Lemma~\ref{Fcoc}.
Note that the associated pairing~$\lambda$       generalizes  the  Reidemeister pairing  \eqref{Reidemeister pairing}.

 The results of this section are analogues of the known properties of the intersection bibracket in dimension two, see  \cite{MT}.
In dimension two, the role of~$\lambda$ is played by the homotopy intersection form introduced in \cite{Tu1}.

\section{The string bracket} \label{string}

 In this section, we   relate   the intersection bibracket  of~$M$  to
   the Chas--Sullivan string bracket in  loop homology.  
    By a   \emph{loop} in $M$ we mean  a continuous map $S^1\to M$ where $S^1=\RR/\ZZ$.
Let $\Loop=\Loop(M)$ be the space of loops in $M$ with compact-open topology.
The  \index{loop homology} \emph{loop homology} of $M$ is $\LoopH=H_\ast (\Loop)$.
The \index{string homology} \emph{string homology}, $\StringH$,  of $M$ is the  $S^1$-equivariant homology of $\Loop$
where   $S^1$ acts on $ \Loop$  by    $(s \gamma) (t) = \gamma(s+t)$ for any $s, t\in S^1$ and   $\gamma \in \Loop$.
Thus,  $\StringH = H_*\left(  {{\mathscr E}   \times_{S^1} \Loop}\right)$
where ${\mathscr E}$ is the total space of the universal $S^1$-principal fiber bundle   and ${{\mathscr E} \times_{S^1} \Loop}$
is the quotient of ${\mathscr E}  \times \Loop$ by the diagonal action of $S^1$.
Since $\mathscr E$ is contractible, the projection ${\mathscr E}  \times \Loop \to  {{\mathscr E} \times_{S^1} \Loop}$
induces a linear map $\operatorname{E}: \LoopH \to \StringH$.

Chas and Sullivan  \cite{CS} defined    a  degree $2-n$  Lie bracket  in $\StringH$  called the \index{string bracket} \emph{string bracket}.
For $n=2$, this  is the Goldman bracket   discussed in  Section \ref{dimension_2}.   We assume that  $n \geq 3$ and
  relate the string bracket to the Lie bracket $\langle-,-\rangle$ in $
\check A_\star =A_\star/[A_\star,A_\star] $ defined in Section~\ref{The induced Lie bracket}.

\begin{lemma}\label{maprR}
Let      $\star \in \partial M$, $\Omega_\star = \Omega(M,\star,\star)$,   and let   $r: \Omega_\star \hookrightarrow \Loop$ be the  inclusion map.
The   induced homology homomorphism   $r_\ast:A_\star =H_*(\Omega_\star) \to \LoopH=H_\ast (\Loop)$    annihilates $[A_\star,A_\star] $
and  induces a linear map  $\operatorname{R}: \check A_\star\to \LoopH$.
The composition  $  (-1)^n  \operatorname{E} \operatorname{R}:\check A_\star\to \StringH$ is a graded Lie algebra homomorphism.
\end{lemma}

\begin{proof} Let   $ \conc   :\Omega_\star \times \Omega_\star \to \Omega_\star$ be the concatenation of loops.
For $a\in A_\star^p$, $b\in A^q_\star$,
$$
ab-(-1)^{pq}ba =   \conc_*   (a \times b) - (-1)^{pq}   \conc_*   (b \times a) =   \conc_*   (a \times b) -   \conc_*     \perm_* (a \times b)
$$
where $\perm: \Omega_\star \times \Omega_\star \to \Omega_\star \times \Omega_\star$ is the transposition.
  Therefore,   to show that $r_\ast (ab-(-1)^{pq}ba)=0$, it suffices  to prove  that $r_\ast   \conc_*   = r_\ast   \conc_*    \perm_* $.
Clearly,  $r   \conc     \perm=(\frac{1}{2} \cdot\!  ) r  \conc   $
where $\big(\frac{1}{2} \cdot\! \big): \Loop \to \Loop$  stands for the action of $1/2 \in \RR/\ZZ=S^1$.
Since $\big(\frac{1}{2} \cdot\! \big)$ is homotopic to the identity,     $r  \conc      \perm $ is homotopic to  $ r    \conc   $.
  We deduce that   $r_\ast \left([A_\star,A_\star]\right)=0$.

Recall  the definition of the  string bracket $[ -,- ]$ in $\StringH$.  Let   $\operatorname{M}: \StringH \to \LoopH$ be the degree 1 lift  map
in the Gysin sequence of the   $S^1$-bundle    $ \mathscr E \times \Loop\to \mathscr E \times_{S^1} \Loop$.
Chas and Sullivan define a   linear map  $\bullet_{ \operatorname{CS}  }: \LoopH \otimes \LoopH \to \LoopH $  of degree $-n$
called the \index{loop product} \emph{loop product} (and   denoted by $\bullet$ in \cite{CS}).
For a detailed exposition, the reader is referred   for instance   to  Cieliebak \cite{Ci}.
For  a  homogoneous  $x\in \StringH$ and any $y \in \StringH$,
\begin{equation}\label{def_string_bracket}
[x,y]= (-1)^{\vert x \vert+n} \operatorname{E} \big(\operatorname{M}(x) \bullet_{ \operatorname{CS}  } \operatorname{M}(y)\big) .
\end{equation}
This formula implies that to  prove the second claim of the lemma, it is enough to show the commutativity of the   diagram
\begin{equation} \label{MER}
\xymatrix{
\check A_\star \otimes \check A_\star
\ar[d]_-{(-1)^{\vert-\vert}\operatorname{M} \!   \operatorname{E}\! \operatorname{R}  \otimes  \operatorname{M}  \!  \operatorname{E}\! \operatorname{R}}
 \ar[r]^-{\langle -,-\rangle} & \check A_\star  \ar[d]^-{\operatorname{R}}\\
\LoopH \otimes \LoopH \ar[r]^-{ \bullet_{ \operatorname{CS}  }} & \LoopH .
}
\end{equation}

Let $a\in H_p(\Omega_\star)$ and $b\in H_q(\Omega_\star)$ with $p,q \geq 0$. Let $h:A_\star \to \check A_\star$ be the canonical projection.
To compute $\langle h(a) , h(b) \rangle $, we pick  a   path~$\varsigma$ in $\partial M$ connecting~$\star$ to another   point~$\star'$. By definition,
\begin{eqnarray*}
\langle h(a) , h(b) \rangle &=& (-1)^{q  + n p  } h    \conc_*    \big(\big[  \widetilde\Upsilon(\lb a \rb \otimes \lb b \rb)\big]\big) \\
&=& (-1)^{q  + n p  } h    \conc_*    \big(\big[   \left((\varsigma^{-1},e_\star)_\sharp \times (e_\star,\varsigma^{-1})_\sharp\right)
\widetilde\Upsilon(\lb a \rb \otimes (\varsigma,\varsigma)_\sharp \lb b \rb)\big]\big).
\end{eqnarray*}
Let $\calK=(K,\varphi,u,\kappa)$ be a   reduced  smooth polycycle in $\Omega_\star^\circ=\Omega^\circ(M,\star,\star)$
and let $\calL=(L,\psi,v,\lambda)$ be a   reduced   smooth polycycle in $\Omega_{\star'}^\circ= \Omega^\circ(M,\star',\star')$
such that $(\calK,\calL)$ transversely represents the pair $(\lb a \rb , (\varsigma,\varsigma)_\sharp \lb b \rb)$.
Consider the intersection polychain $\calD(\calK,\calL)=(D,\theta,w,\kappa \losange \lambda)$.  Then
\begin{eqnarray}
\label{ER}&&  \quad \operatorname{R}\langle h(a) , h(b) \rangle  \\
\notag &=& (-1)^{q  + n p  } \operatorname{R}h    \conc_*    \big[ \left(  (\varsigma^{-1},e_\star)_\sharp \times (e_\star,\varsigma^{-1})_\sharp\right)  \calD(\calK,\calL)\big] \\
\notag &=& (-1)^{q  + n p  }    \big[ r_*    \conc_*    \left(  (\varsigma^{-1},e_\star)_\sharp \times (e_\star,\varsigma^{-1})_\sharp \right) \calD(\calK,\calL)\big] \\
\notag &=& (-1)^{q  + n p  }    \big[\big(D,\theta,w, r \conc    \left(  (\varsigma^{-1},e_\star)_\sharp \times (e_\star,\varsigma^{-1})_\sharp \right)(\kappa \losange \lambda)\big) \big] \\
\notag &=& (-1)^{q  + n p  }    \big[\big(D,\theta,w, r  (\varsigma^{-1}, \varsigma^{-1})_\sharp    \conc      (\kappa \losange \lambda)\big) \big] \\
\notag &=& (-1)^{q  + n p  }    \big[\big(D,\theta,w, r'   \conc     (\kappa \losange \lambda)\big) \big]
\end{eqnarray}
where,     in the last two lines,  $ \conc$  is the concatenation of paths
$$ \Omega(M,\star',\star) \times \Omega(M,\star,\star') \longrightarrow \Omega(M,\star',\star')=\Omega_{\star'}$$
and  $r':\Omega_{\star'} \hookrightarrow \Loop$ is the inclusion.
On the other hand,
$$
\operatorname{MER} h(a) =  \operatorname{ME}r_*(a) =  \operatorname{ME} \left[ r_*\! \lb a \rb \right]
= \operatorname{ME} \left[(K,\varphi,u,r\kappa)\right].
$$
 Using the computation  of the map $\operatorname{ME}: \LoopH \to \LoopH$   in \cite{CS,Ci} (where this map is denoted by $\Delta$),  we obtain
\begin{equation}\label{MERMER}
\operatorname{MER} h(a) = (-1)^p \left[(K  \times S^1, \bar \varphi, \bar u, \bar \kappa) \right]
\end{equation}
where we use the following notation:
$\bar \varphi$ is the partition on $K  \times S^1$ induced by $\varphi$
(by identifying $F \times S^1$ to  $G\times S^1$ via $\varphi_{F,G}\times \id_{S^1}$ for any faces $F,G$ of the same type in~$K$);
$\bar u$ is the weight on $K  \times S^1$ induced by   $u$  via the   equality  $\pi_0(K  \times S^1)=\pi_0(K)$;
the   map  $\bar \kappa: K  \times S^1 \to \Loop$ is defined using the action of $S^1$ on $\Loop$ by $(k,s)\mapsto s\big(r\kappa(k)\big)$ for   $k\in K$ and $s\in S^1$.
The sign $(-1)^p$  in \eqref{MERMER} is caused by a permutation of the two factors of $K  \times S^1$ with respect to \cite{CS,Ci}.
Similarly,
$$
\operatorname{MER} h(b) =   \operatorname{ME} \left[ r_*\! \lb b \rb \right] =  \operatorname{ME} \left[ r'_*\!   (\varsigma,\varsigma)_\sharp  \lb b \rb \right]
= (-1)^q \left[(L \times  S^1, \bar \psi, \bar v, \bar \lambda) \right]
$$
where the map  $\bar \lambda:L \times S^1  \to \Loop$ is defined by $(l,s)\mapsto s\big(r'\lambda(l)\big)$ for   $l\in L$ and $s\in S^1$.

The loop product $\bullet_{\operatorname{CS}} $ can be computed in terms  of face homology. This gives
\begin{eqnarray}
\label{(MER,MER)} && \operatorname{MER} h(a) \bullet_{\operatorname{CS}} \operatorname{MER} h(b)  \\
\notag &=& (-1)^{p+q}  \left[(K  \times S^1, \bar \varphi, \bar u, \bar \kappa) \right] \bullet_{\operatorname{CS}} \left[(L \times  S^1, \bar \psi, \bar v, \bar \lambda) \right]\\
\notag &=& (-1)^{ p+q  +np } \left[(\bar D, \bar \theta, \bar w, \bar \kappa \infty \bar \lambda) \right] .
\end{eqnarray}
Here $\bar D$ is the inverse image of $\diag_M$   under the map
$$K \times S^1 \times L \times S^1\to M \times M, \quad  (k,s,l,t)   \longmapsto   \big(r\kappa(k)(s),r'\lambda(l)(t)\big).$$
Note that $\bar D$ has a structure of a   manifold with faces inherited from $K \times S^1 \times L \times S^1$.
The orientation, the partition $\bar \theta$ and the weight $\bar w$ of $\bar D$ are as  in the definition of  the intersection   operation $\calD$    in Section \ref{def_D}.
The map $\bar\kappa \infty \bar\lambda: \bar D \to \Loop$ sends a point $(k,s,l,t) $
to the loop that first goes along the loop $\bar \kappa(k,s)$ and then along   the loop $\bar \lambda(l,t)$.
  Note the sign $(-1)^{np}$ in   \eqref{(MER,MER)}, which   arises from the difference between our orientation conventions  and those  in \cite{Ci}.

 The map $K \times I \times L \times I \to K \times S^1 \times L \times S^1$ determined    by the canonical projection  $I\to S^1$
 induces an   orientation-preserving   diffeomorphism  $D \cong    \bar D$ which carries the partition $\theta$ into $\bar \theta$ and the weight $w$ into $\bar w$.
Using the action of $1/4 \in S^1$ on $\Loop$, one   easily constructs a homotopy between the maps $r'  \conc    (\kappa \losange \lambda) $
and   $\bar\kappa \infty \bar\lambda$ from $ D  \cong   \bar D$ to $ \Loop$.
It  follows that
$$
 \big[\big(D,\theta,w, r'   \conc      (\kappa \losange \lambda)\big) \big]
 =  \left[(\bar D, \bar \theta, \bar w, \bar \kappa \infty \bar \lambda) \right] \in \LoopH,
$$
and  we deduce \eqref{MER} from  \eqref{ER} and \eqref{(MER,MER)}.
\end{proof}

\section{Moment maps and Hamiltonian reduction} \label{kill_boundary}

We show that a spherical boundary component of  the manifold~$M$   determines a moment map for the intersection bibracket.
This allows us to define an $H_0$-Poisson structure on the Pontryagin algebras of certain manifolds without boundary.

\subsection{The moment map} \label{moment}

  Assume that $n=\dim(M) \geq 3$ and that~$S$ is a component   of $\partial M$     homeomorphic  to the sphere $S^{n-1}$.
Fix a    point $\star \in S$ and set   $A_\star = H_*(\Omega_\star)$  where $\Omega_\star=\Omega(M,\star,\star)$.
  The orientation-preserving homeomorphisms  $ S^{n-1}   \cong   S $ represent an element $\mu^\pi=\mu^\pi_S  $ of $ \pi_{n-1}(M,\star)$.
 Recall the   connecting homomorphism $\overline \partial_{n-1}:\pi_{n-1}(M,\star)  \to   A_\star^{n-2}=  H_{n-2}(\Omega_\star)$    of Section~\ref{The {Pontryagin} algebra} and set
$$
\mu=\mu_S  = \overline \partial_{n-1}(\mu^\pi ) \in  A_\star^{n-2}\, .
$$

\begin{lemma} \label{spherical_boundary}
 The element   $\mu $ is a moment map for    the intersection bibracket $\double{-}{-}$ in $A_\star$
  in the sense of Section~\ref{moment_maps}. 
\end{lemma}

\begin{proof}
 Consider the path homology category $\calC=\calC(M)$ of $M$ and the intersection bibracket $\double{-}{-}$
in   the associated   graded algebra $A=A(\calC) $.
 Pick a  smooth closed $n$-ball $D\subset \Int(M)$  and consider the smooth  manifold  $  {P}=   M \setminus \Int(D)$.
 As above, we can consider the path homology category of ${P}$ and the intersection bibracket $\double{-}{-}_{P}$ in the associated   graded algebra.
  Consider the  restriction of $\double{-}{-}_P$  to the  algebra
$$
B =\bigoplus_{ \star_1 , \star_2  \in S}    H_{\ast} \big( \Omega({P}, \star_1,\star_2 )  \big) .
$$
The inclusion ${P}\hookrightarrow M$ induces a graded algebra homomorphism $\iota: B \to A$. The definition of the intersection bibracket
implies that   the following diagram commutes:
\begin{equation}\label{AB}
\xymatrix{
B \otimes B \ar[rr]^{\double{-}{-}_{P} } \ar[d]_-{\iota \otimes \iota} && B\otimes B \ar[d]^-{\iota \otimes \iota} \\
A \otimes A \ar[rr]^{\double{-}{-} } && A \otimes A.
}
\end{equation}

\begin{figure}[h]
\begin{center}
\labellist \small
 \pinlabel {$\star'$} [ ] at 60 22
 \pinlabel {$\alpha$} [r] at 75 65
 \pinlabel {$\star$} [ ] at 58 117
 \pinlabel {$\star''$} [ ] at 519 129
 \pinlabel {$\beta$} [b] at 256 119
 \pinlabel {$\partial D$} [tl] at 552 16
 \pinlabel {$S $} [tr] at 7 13
\endlabellist
\includegraphics[scale=0.4]{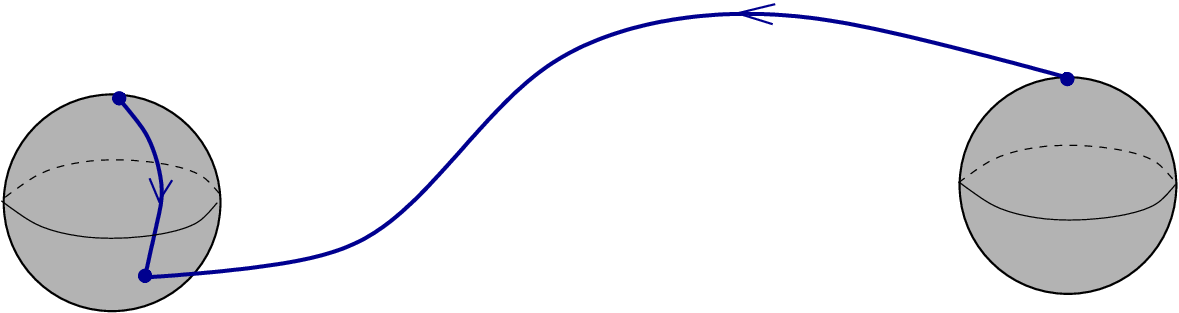}
\end{center}
\caption{The manifold $P=M\setminus \Int(D)$.}
\label{moment_map}
\end{figure}
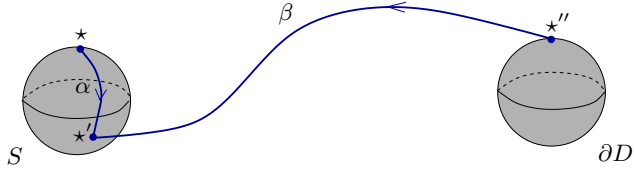

  We must prove that $\double{\mu}{a}=a\otimes 1- 1 \otimes a$ for any $a\in A_\star \subset A$.
  To  this end, fix a path $\alpha$ in $ S$ leading from $\star$ to a distinct point $\star'\in S$: see Figure \ref{moment_map}.
  The path $\alpha$ represents an element in $H_0(\Omega, \star, \star')\subset A^0$ denoted also by~$\alpha$.
  This element is invertible, and its inverse $\alpha^{-1}\in A^0$ is represented by the inverse path. Set  $$a'= \alpha^{-1} a \alpha  \in H_*\big(\Omega(M,\star',\star')\big).$$
  (In the  notation of Section~\ref{change}, $a'=(\alpha, \alpha)_\#( a )$.)
Clearly, $\double{\mu}{a'}= \alpha^{-1} \double{\mu}{a} \alpha$.   Hence it suffices to prove that
\begin{equation} \label{with_a'}
\double{\mu}{a'} = a'\alpha^{-1} \otimes \alpha - \alpha^{-1} \otimes \alpha a'.
\end{equation}
  The homology class $a'$ can be represented by   a polycycle $\mathcal K$  in $\Omega^\circ(M,\star',\star')$.
  Choosing the ball~$D$   close  enough to $\partial M$, we can ensure that it does not meet  the image of $\mathcal K$.
Then $a'=\iota(b)$ for the homology class   $b\in H_*(\Omega({P},\star',\star'))$ represented by~$\mathcal K$.   Similarly,   $\mu=\iota(\tau)$ for some  $\tau \in H_{n-2}(\Omega({P},\star,\star))$.
We deduce from \eqref{AB} that
$$
\double{\mu}{a'} = (\iota\otimes \iota) \left (\double{\tau}{b}_{P} \right ).
$$
To proceed, we   pick an embedded path in~${P}$ leading from a point $\star'' \in \partial D$
  to $\star' $ and meeting $\partial {P}$ only in the endpoints. This path defines an invertible element $\beta \in   H_0(\Omega({P},\star'',\star'))$.
  Set $${c}=\beta b \beta^{-1} \in H_*\big(\Omega({P},\star'',\star'')\big).$$
Note that $\double{\tau}{{c}}_{P}=0$ since ${c}$ can be represented by a polycycle whose image does not meet $S$.
Therefore
\begin{eqnarray*}
\double{\mu}{a'} &=&   (\iota\otimes \iota) \left (\double{\tau}{b}_{P} \right )= (\iota\otimes \iota)   \left( \double{\tau}{\beta^{-1} {c} \beta}_{P} \right ) \\
&=&  (\iota\otimes \iota) \left(\double{\tau}{\beta^{-1}}_{P} {c}\beta + \beta^{-1} {c} \double{\tau}{\beta}_{P} \right) \\
&=& (\iota\otimes \iota) \left(-\beta^{-1}\double{\tau}{\beta}_{P} \beta^{-1} {c}\beta + \beta^{-1} {c} \double{\tau}{\beta}_{P} \right).
\end{eqnarray*}
 By Lemma \ref{intersecting_arc_and_spheres},   we obtain $\double{\beta}{\tau}_{P}=-  \alpha \otimes \beta \alpha^{-1}$.
Therefore $\double{\tau}{\beta}_{P} = \beta  \alpha^{-1} \otimes \alpha$. We conclude that
$$
  \double{\mu}{a'}   =  (\iota\otimes \iota) \left(-\alpha^{-1} \otimes \alpha b + b \alpha^{-1}\otimes \alpha \right),
$$
which proves \eqref{with_a'}.
\end{proof}

 We deduce from  Lemma \ref{spherical_boundary} that the following three conditions are equivalent: 
 \begin{itemize}
\item[(i)] $A_\star=\kk$;
\item[(ii)]    $\mu=0  $;
\item[(iii)]  the intersection bibracket in $ A_\star$ is zero.
\end{itemize}

\subsection{Intersections in manifolds without boundary}

Let ${W}$ be a smooth   connected     oriented manifold of dimension $n\geq 3$  without  boundary.
Our construction of a bibracket in the {Pontryagin} algebra of a manifold requires the base point to lie in the boundary,   so that it  does not   apply to~${W}$.
However, under   certain assumptions on~${W}$ we can use the Hamiltonian reduction  of  Section~\ref{sectionActions-} 
to define   an   $H_0$-Poisson structure on the Pontryagin algebra of ${W}$.
To this end, pick a base point $\star \in  {W}$ and  a smooth closed $n$-ball $D\subset {W}$ with $\star \in  \partial D$.
Consider the smooth manifold ${M} ={W} \setminus \Int(D)$ with $\partial M=\partial D=S^{n-1}$;
  as everywhere  in this chapter, we  assume that the condition \eqref{cross_iso} is satisfied.  
Let
$$A_\star=  H_*\big(\Omega({M},\star,\star)\big)\quad {\text {and}} \quad B_\star=  H_*\big(\Omega({W},\star,\star)\big)$$
be the  Pontryagin algebras of~${M}$ and~${W}$, respectively.
 The   inclusion ${M}\hookrightarrow {W}$ induces a graded algebra homomorphism $ p: A_\star \to B_\star$,
 Clearly, $p(\mu)=0$ where $\mu= \mu_{\partial {M}}\in A^{n-2}_\star$. Therefore,   $\Ker p \supset A_\star  \mu  A_\star $.

\begin{theor} \label{no_boundary}
  Assume that   the   homomorphism  $p:  A_\star \to B_\star$   is onto  and  $\Ker p= A_\star  \mu   A_\star$.
Then the intersection bibracket $\double{-}{-}$ in $A_\star$ induces an $H_0$-Poisson structure of degree $2-n$ on $B_\star$. This structure does not depend on the choice of the ball~$D$.
\end{theor}

\begin{proof}
The first claim   follows from    Lemma \ref{quotient}. The independence of the choice of the ball is a consequence of the naturality of the intersection  bibracket   under diffeomorphisms,
and the fact that for any   balls $D_1, D_2 \subset {W}$ with $\star \in  \partial D_1 \cap \partial D_2$  there is a   diffeomorphism $f:{W}\to {W}$ such that $f(D_1)=D_2$, $f(\star)= \star$,
and~$f$ is isotopic to $\id_{W}$ in the class of diffeomorphisms ${W}\to {W}$ fixing $\star$.  Such an~$f$ acts   on $B_\star$   as the identity, and the result follows.
\end{proof}

  Recall from Theorem \ref{H0_to_G} that 
an $H_0$-Poisson structure  on $B_\star$ induces   Gerstenhaber brackets  on the trace  algebras of $B_\star$.   Hence Theorem \ref{no_boundary}  allows us to associate    Gerstenhaber algebras with~$W$.

Some   manifolds  do not satisfy the assumptions of Theorem \ref{no_boundary}, for example, ${W}=S^n$ (in this  case $ A_\star=\kk$ and $B_\star=\kk[x]$ where the generator~$x$ has degree $n-1$, cf$.$ \cite{BS}).
Nonetheless, according to \cite{HL} and \cite{FT}, these assumptions are satisfied if
${W}$ is a closed simply connected manifold whose cohomology algebra $H^*({W})=H^*({W}; \kk)$ is not   generated by a single element  and
  $\kk$ is a field  whose characteristic is equal to zero or is  sufficiently   large.

\subsection{Example}

We consider the example of Section~\ref{sc_example} and keep the same notation.
Thus
$
{W}= \left(S^{p_1} \times S^{q_1}\right) \sharp \cdots  \sharp \left(S^{p_g} \times S^{q_g}\right)
$
and $M={W} \setminus \hbox{(an open ball)}$.
The element $ \mu=\mu_{\partial M} \in A_\star=H_\ast(\Omega(M, \star, \star))$ can be computed as follows.
As a topological manifold, $M$ is the boundary-connected sum of the manifolds $  M_j    = \left(S^{p_j} \times S^{q_j}\right) \setminus \hbox{(an open ball)}$ where $j=1,\dots,g$.
Hence
$$
\mu^\pi=   \mu^\pi_{\partial M}   =   \sum_{ j=1  }^{g}     {\rm {in}} _j \big(  \mu^\pi_{\partial M_j}  \big)   \in \pi_{n-1}(M)
$$
where ${\rm {in}} _j:\pi_{n-1}(M_j) \to \pi_{n-1}(M)$ is the inclusion homomorphism (we can ignore the base point because $M_j$ and $M$ are simply-connected).
By  the definition of the Whitehead bracket $[-,-]_{ \operatorname{Wh}  }$ in $\pi_\ast(M)$, we have  ${\rm {in}} _j(\mu^\pi_{\partial M_j}) = [x_j^\pi,y_j^\pi]_{  \operatorname{Wh}  }$
where $x_j^\pi \in \pi_{p_j}(M)$ and $y^\pi_j\in \pi_{q_j}(M)$ are represented by the two  factors of $M_j$. Thus,
$$
\mu^\pi  = [x_1^\pi,y_1^\pi]_{  \operatorname{Wh}  } + \cdots + [x_g^\pi,y_g^\pi]_{ \operatorname{Wh}  } \in \pi_{n-1}(M).
$$
Recall that the bracket $[-,-]$  in~$A_\star$ induced by the Pontryagin  multiplication is related to the Whitehead bracket in $\pi_\ast(M)$ by the formula
$$ \left[\,  x_i, y_i  \right] =
\left[\, \overline \partial_{p_i}(x_i^\pi), \overline \partial_{q_i}(y_i^\pi) \right] = (-1)^{p_i}\, \overline \partial_{p_i+q_i-1} \big( [x_i^\pi,y_i^\pi]_{ \operatorname{Wh}  } \big) \in A_\star.
$$
Therefore
$$
\mu = (-1)^{p_1} [x_1,y_1] + \cdots + (-1)^{p_g} [x_g,y_g] \in A_\star.
$$
A direct computation on the generators $x_1,y_1,\dots,x_g,y_g$ of $A_\star$  using \eqref{xy_yx}--\eqref{xx_yy}
confirms that $\mu$ is a moment map of the  intersection bibracket $\double{-}{-}$ of $M$,    as claimed by Lemma \ref{spherical_boundary}.

Consider in more detail the case $g=1$ and set $$p=p_1, \quad  q=q_1, \quad   x=x_1 \in A_\star^{p-1}, \quad  y=y_1\in A_\star^{q-1}.$$
The loop space of $W=S^p\times S^q$ based at  $\star $ is the product  of the loop spaces of $S^p$ and $S^q$.
By   the K\"unneth theorem,  the Pontryagin algebra  $B_\star = H_*(\Omega(W,\star,\star))$ is (as a graded algebra) the tensor product of the Pontryagin algebras of $S^p$ and $S^q$.
Since  the graded algebra  $A_\star$ is freely generated by $x,y$, the quotient $A_\star/ A_\star \mu A_\star$ is the  commutative  graded algebra freely generated by $x,y$. It is clear that
  the assumptions of Theorem~\ref{no_boundary} are satisfied here for any ground ring $\kk$.
  Theorem~\ref{no_boundary} yields an  $H_0$-Poisson structure  $\langle-,-\rangle$ of degree  $2-n$ on the  (commutative)  graded algebra $B_\star$.
  The bracket $\langle-,-\rangle$ in $\check B_\star = B_\star$ is then  a Gerstenhaber bracket  of degree  $2-n$.
It coincides with the Gerstenhaber bracket $\bracket{-}{-}$ in $\Com(A_\star)$   computed in  Section~\ref{sc_example}.

\subsection{Example}

We   consider the example of   Section~\ref{nsc_example} and keep the same notation.
Thus, ${W}=S^1 \times S^{n-1}$ and $M = {W} \setminus \hbox{(an open ball)}$.
The element $ \mu=\mu_{\partial M} \in A_\star=H_\ast(\Omega(M, \star, \star))$ can be computed as follows. Consider the  cylinder  $I  \times S^{n-1}$ with the product orientation and pick a closed $n$-ball ${D}$ in its interior.
Then   $$ [\partial {D}]  = [\{1\} \times S^{n-1}] - [\{0\} \times S^{n-1}]  \in \pi_{n-1}\big(( I \times S^{n-1}) \setminus \Int({D})\big).$$
  (Here $D$  carries  the orientation induced by $M$ and $\partial D$  carries   the orientation inherited from $D$.)
It follows that $\mu^\pi(M) =  (x^\pi)^{-1} \cdot y^\pi -y^{\pi}$ where the dot denotes the action of $\pi_1(M,\star)$ on $\pi_{n-1}(M,\star)$.  We  deduce that
$$
\mu = x^{-1}y x - y \in A_\star.
$$
A direct computation  using \eqref{xx}--\eqref{yy} confirms  that~$\mu$ is a moment map of  the intersection bibracket $\double{-}{-}$ of $M$, as  claimed by Lemma \ref{spherical_boundary}.

By the K\" unneth theorem,  the Pontryagin algebra $B_\star$ of~$W$ is   (as a graded algebra)    the tensor product
of the {Pontryagin} algebras of~$S^1 $ and $ S^{n-1} $.
Thus,   $B_\star$ is   the commutative graded algebra    freely generated by $x^{\pm 1}\in B_\star^0$   and $y \in  B_\star^{n-2}$.
As a consequence, the assumptions of Theorem \ref{no_boundary} are satisfied  (for any ground ring~$\kk$)
so that the intersection bibracket  $\double{-}{-}$ of $M$ induces  an $H_0$-Poisson structure  $\langle -,-\rangle$ of degree $2-n$ on  $B_\star$.
  Since $B_\star$ is commutative, this structure is  a Gerstenhaber bracket  of degree $2-n$.
According to  \eqref{xx}--\eqref{yy}, it is given by
$$
\langle x,x\rangle=0, \quad \langle x,y \rangle =-x, \quad \langle y,y \rangle =0.
$$

\subsection{Remark}
The  results of this section are high-dimensional analogues of the well-known properties of surfaces. The
 Pontryagin algebra of a closed connected  oriented surface  is the group algebra   $B=\kk[   \pi]$ where $\pi $ is the fundamental group of the surface.
 Then $\check B=B/[B,B]= \kk[ \check \pi]$ is  the module freely generated by  the set  $\check \pi$   of conjugacy classes    in~$\pi$.
 The   Goldman Lie bracket in $\check B$ is the canonical $H_0$-Poisson structure on $B$,  see \cite[Section 9]{MT}.

\printindex


\begin{thebibliography}{CJKLS}


\bibitem[AH]{AH} J. F. Adams,  P. J. Hilton,
\emph{On the chain algebra of a loop space.}
Comment. Math. Helv. 30 (1956), 305--330.


\bibitem[AKsM]{AKsM} A. Alekseev, Y. Kosmann-Schwarzbach, E. Meinrenken,
\emph{Quasi-Poisson manifolds.} Canad. J. Math. 54 (2002), no.~1, 3--29.


\bibitem[BCER]{BCER} Y. Berest, X. Chen, F. Eshmatov, A. Ramadoss,
\emph{Noncommutative Poisson structures, derived representation schemes and Calabi--Yau algebras.}
  Mathematical aspects of quantization, 219--246,
Contemp. Math., 583, Amer. Math. Soc., Providence, RI, 2012.


\bibitem[BLb]{BLb} R. Bocklandt, L. Le Bruyn,
\emph{Necklace Lie algebras and noncommutative symplectic geometry.}
 Math. Z. 240 (2002), no.~1, 141--167.


\bibitem[BS]{BS} R. Bott, H. Samelson,
\emph{On the Pontryagin product in spaces of paths.}
Comment. Math. Helv. 27 (1953), 320--337.


\bibitem[Br]{Br}
G. E. Bredon,
Topology and Geometry. Graduate Texts in Mathematics, 139. Springer-Verlag, New York, 1993.


\bibitem[CD]{CD} E. Castillo, R.  D\'iaz,
\emph{Homology and manifolds with corners}.
Afr. Diaspora J. Math. (N.S.) 8 (2009), no.~2, 100--113.


\bibitem[Ce]{Ce} J. Cerf,
\emph{Topologie de certains espaces de plongements}.
Bull. Soc. Math. France 89 (1961), 227--380.


\bibitem[CS1]{CS} M. Chas, D. Sullivan,
\emph{String Topology}.
Preprint (1999) \href{http://arxiv.org/abs/math/9911159}{arXiv:math/9911159}.

 \bibitem[CS2]{CS+} M. Chas, D. Sullivan,
\emph{String topology in dimensions two and three}. Algebraic topology, 33--37,
Abel Symp., 4, Springer, Berlin, 2009. 


\bibitem[Ci]{Ci} K. Cieliebak,
\emph{Lectures on String Topology}.
Draft  (January 2013)  available at \href{http://www.math.uni-augsburg.de/prof/geo/Dokumente/string.pdf}{www.math.uni-augsburg.de/prof/geo/Dokumente/string.pdf}.


\bibitem[CJ]{CJ} R. Cohen, J. D.  Jones,
\emph{A homotopy theoretic realization of string topology. }
Math. Ann. 324 (2002), no.~4, 773--798.


\bibitem[Cb]{Cb} W. Crawley-Boevey,
\emph{Poisson structures on moduli spaces of representations.}
J. Algebra 325 (2011), 205--215.


\bibitem[Do]{Do} A. Douady,
\emph{Vari\'et\'es \`a bord anguleux et voisinages tubulaires.}
S\'eminaire Henri Cartan 14 (1961-2), exp. 1, 1--11.


\bibitem[FHT]{FHT} Y. F\'elix, S. Halperin, J-C. Thomas,
Rational homotopy theory. Graduate Texts in Mathematics, 205.
Springer-Verlag, New York, 2001.


\bibitem[FT]{FT} Y. F\'elix, D. Tanr\'e,
\emph{Sur l'homologie de l'espace des lacets d'une vari\'et\'e compacte.}
Ann. Sci. \'Ecole Norm. Sup. (4) 25 (1992), no.~6, 617--627.


\bibitem[FoR]{FoR} V. V. Fock, A. A. Rosly,
\emph{Poisson structure on moduli of flat connections on Riemann surfaces and the r-matrix.}
(Russian) Moscow Seminar in Math. Physics. English translation:
Amer. Math. Soc. Transl. Ser. 2, 191, 67--86 (1999).


\bibitem[FuR]{FuR} D. Fuks, V.  Rokhlin,
Beginner's course in topology. Geometric chapters.
Springer Series in Soviet Math. Springer-Verlag, Berlin, 1984.


\bibitem[G]{G} V. Ginzburg,
\emph{Non-commutative symplectic geometry, quiver varieties, and operads.}
Math. Res. Lett. 8 (2001), no.~3, 377--400.


\bibitem[Go1]{Go1} W. M. Goldman,
 \emph{The symplectic nature of fundamental groups of surfaces.}
 Adv.  
 Math. 54 (1984), no.~2, 200--225.


\bibitem[Go2]{Go2}
W. M. Goldman,
\emph{Invariant functions on Lie groups and Hamiltonian flows of surface group representations.}
Invent. Math. 85 (1986), no.~2, 263--302.


\bibitem[GP]{GP} V. Guillemin,  A. Pollack,
Differential topology. Prentice-Hall, Inc., Englewood Cliffs, N.J., 1974.


\bibitem[GHJW]{GHJW} K. Guruprasad, J. Huebschmann, L. Jeffrey, A. Weinstein,
\emph{Group systems, groupoids, and moduli spaces of parabolic bundles.}
Duke Math. J. 89 (1997), no.~2, 377--412.


\bibitem[HL]{HL} S. Halperin, J. Lemaire,
\emph{Suites inertes dans les alg\`ebres de Lie gradu\'ees (``Autopsie d'un meurtre. II'')}.
Math. Scand. 61 (1987), no.~1, 39--67.


\bibitem[J\"a]{Ja} K. J\"anich,
\emph{On the classification of $O(n)$-manifolds.}
Math. Ann. 176 (1968), 53--76.

\bibitem[Jo]{Jo} D. Joyce,
\emph{On manifolds with corners.} Advances in geometric analysis, 225--258, Adv. Lect. Math.   21, Int.\ Press, Somerville, MA, 2012.


\bibitem[Ka]{Ka}
C. Kassel,
Quantum groups.
Graduate Texts in Mathematics, 155. Springer-Verlag, New York, 1995.

\bibitem[KK1]{KK1} N. Kawazumi, Y. Kuno,
\emph{The logarithms of Dehn twists.} Quantum Topol. 5 (2014), no.~3, 347--423.

\bibitem[KK2]{KK2} N. Kawazumi, Y. Kuno,
\emph{Intersections of curves on surfaces and their applications to mapping class groups.}
 Ann. Inst. Fourier (Grenoble) 65 (2015), no.~6, 2711--2762. 


\bibitem[Ke]{Ke} M. Kervaire,
\emph{Geometric and algebraic intersection numbers.}
Comment. Math. Helv. 39 (1965),  271--280.

  
\bibitem[LS]{LS} P. Lambrechts, D. Stanley, 
\emph{Poincar\'e duality and commutative differential graded algebras.} 
Ann. Sci. \'Ec. Norm. Sup\'er. (4) 41 (2008), no. 4, 495--509. 
 

\bibitem[La]{La} F. Laudenbach,
\emph{A note on the Chas--Sullivan product.}
Enseign. Math. (2) 57 (2011), no.~1-2, 3--21.


\bibitem[LbP]{LbP}
L. Le Bruyn, C. Procesi,
\emph{Semisimple representations of quivers.}
Trans. Amer. Math. Soc. 317 (1990), no.~2, 585--598.


 \bibitem[LbW]{LbW} L. Le Bruyn, G. Van de Weyer,
\emph{Formal structures and representation spaces.}
 J. Algebra 247 (2002), no.~2, 616--635.


\bibitem[MrOd]{MrOd}
J. Margalef-Roig, E. Outerelo Dom\'inguez, Differential topology.
North-Holland Mathematics Studies, 173. North-Holland Publishing Co., Amsterdam, 1992.


 \bibitem[MT1]{MT} G. Massuyeau, V. Turaev,
\emph{Quasi-Poisson structures on representation spaces of surfaces.}
  Int. Math. Res. Not. 2014  (2014), no.~1, 1--64.
  
 
  
  \bibitem[MT2]{MT+} G. Massuyeau, V. Turaev,
\emph{Brackets in representation algebras of Hopf algebras.}
  Preprint (2015) \href{https://arxiv.org/abs/1508.07566}{arXiv:1508.07566}, to appear in J$.$ Noncommut. Geom.  
  



\bibitem[Mu]{Mu} J. R.  Munkres,
Elementary differential topology.  Annals of Math. Studies,  54.
Princeton University Press, Princeton N.J., 1963.


\bibitem[Pr]{Pr} C. Procesi,
\emph{A formal inverse to the Cayley-Hamilton theorem.}
J. Algebra 107 (1987), no.~1, 63--74.


 \bibitem[Se]{Se} J.-P. Serre,
\emph{Homologie singuli\`ere des espaces fibr\'es. Applications.}
Ann. of Math. (2) 54 (1951), 425--505.


\bibitem[Tu1]{Tu1} V. G. Turaev,
\emph{Intersections of loops in two-dimensional manifolds.}
(Russian) Mat. Sb. 106(148) (1978),   566--588.
English translation: Math. USSR, Sb. 35 (1979), 229--250.


\bibitem[VdB]{VdB} M. Van den Bergh,
\emph{Double Poisson algebras.}
Trans. Amer. Math. Soc. 360 (2008), no.~11, 5711--5769.

\bibitem[Wa]{Wa}  C. T. C. Wall,
Surgery on compact manifolds. Second edition.
  Math.  Surveys and Monographs, 69.  Amer. Math. Soc.,  Providence, RI, 1999.

\end{thebibliography}
\end{document}